\documentclass[dvipsnames]{article}
\usepackage[utf8]{inputenc}
\usepackage[T1]{fontenc}
\usepackage[final]{neurips_2022}
\usepackage{xcolor}

\usepackage[toc,page,header]{appendix}
\usepackage{titletoc}

\usepackage{microtype}
\usepackage{mathtools}
\usepackage{graphicx}
\usepackage{amsmath,amssymb,amsfonts,amsxtra,bm}
\usepackage{amsthm}
\usepackage{xspace}
\usepackage{paralist}
\usepackage{enumerate}
\usepackage{enumitem}

\definecolor{darkblue}{rgb}{0.0,0.0,0.65}
\definecolor{darkred}{rgb}{0.68,0.05,0.0}
\definecolor{darkgreen}{rgb}{0.0,0.29,0.29}
\definecolor{darkpurple}{rgb}{0.47,0.09,0.29}

\usepackage{hyperref}
\hypersetup{
   colorlinks = true,
   citecolor  = darkblue,
   linkcolor  = darkred,
   filecolor  = darkblue,
   urlcolor   = darkblue,
 }

\usepackage{url}
\usepackage{nicefrac}       
\usepackage{bbm}
\usepackage{tikz}

\numberwithin{equation}{section}

\newtheorem{theorem}{Theorem}
\newtheorem{lemma}{Lemma}

\newtheorem{assumption}{Assumption}

\theoremstyle{definition}
\newtheorem{definition}{Definition}
\newtheorem{remark}{Remark}

\newcommand*\at[2]{\left.#1\right|_{#2}}

\newcommand*\del[0]{\partial}

\newcommand*\Z[0]{\mathbb{Z}}

\newcommand*\ddt[0]{\frac{d}{d t}}

\newcommand*\lin[1]{\bm{\left\langle} #1 \bm{\right\rangle}}
\newcommand*\linp[1]{\bm{\langle} #1 \bm{\rangle}}

\newcommand*\E[1]{\mathbb{E}[{#1}]}
\newcommand*\Ep[2]{\mathbb{E}_{{#1}}\left[{#2}\right]}

\renewcommand*\t[1]{\tilde{#1}}

\newcommand*\N[0]{\mathcal{N}}

\newcommand\numberthis{\addtocounter{equation}{1}\tag{\theequation}}
\newcommand*\lrb[1]{{\left[#1\right]}}
\newcommand*\lrbb[1]{\left\{#1\right\}}
\newcommand*\lrp[1]{{(#1)}}
\newcommand*\lrn[1]{{\left\|#1\right\|}}
\newcommand*\lrabs[1]{{\left|#1\right|}}

\newcommand*\cvec[2]{\begin{bmatrix} #1\\#2\end{bmatrix}}

\newcommand*\bmat[1]{\begin{bmatrix} #1 \end{bmatrix}}

\newcommand*\ind[1]{{\mathbbm{1}\lrbb{#1}}}

\renewcommand*{\Re}{\mathbb{R}}

\newcommand*{\R}{\mathcal{R}}

\newcommand*\circled[1]{\tikz[baseline=(char.base)]{
\node[shape=circle,draw,inner sep=1pt] (char) {#1};}}

\newcommand*{\twocase}[4]{\left\{\begin{array}{ll}
        #1 & \text{for } #2\\
        #3 & \text{for } #4
        \end{array}\right.}
        
\newcommand*{\threecase}[6]{\left\{\begin{array}{ll}
        #1, & \text{for } #2\\
        #3, & \text{for } #4\\
        #5, & \text{for } #6
        \end{array}\right.}

\renewcommand*{\Pr}[1]{\mathbb{P}\lrp{{#1}}}

\renewcommand{\L}{\mathcal{L}}

\renewcommand{\epsilon}{\varepsilon}

\newcommand{\JJ}{\mathbf{J}}
\newcommand{\KK}{\mathbf{K}}
\newcommand{\RR}{\mathbf{R}}

\newcommand{\MM}{\mathbf{M}}
\newcommand{\NN}{\mathbf{N}}

\newcommand{\xx}{\mathbf{x}}

\newcommand{\yy}{\mathbf{y}}

\newcommand{\nnu}{\bm{\nu}}

\renewcommand{\AA}{\mathbf{A}}
\newcommand{\BB}{\mathbf{B}}
\newcommand{\WW}{\mathbf{W}}
\newcommand{\CC}{\mathbf{C}}
\newcommand{\DD}{\mathbf{D}}
\renewcommand{\aa}{\mathbf{a}}

\newcommand{\vv}{\mathbf{v}}

\newcommand{\zz}{\mathbf{z}}

\newcommand{\xib}{{\bm\xi}}

\newcommand{\zzeta}{{\bm\zeta}}

\newcommand{\F}{{\mathcal{F}}}
\newcommand{\C}{{\mathcal{C}}}

\newcommand{\elb}[1]{\numberthis \label{#1}}

\DeclareMathOperator{\tc}{{\zeta}}

\newcommand{\dist}{\mathrm{d}}

\DeclareMathOperator{\Exp}{Exp}
\DeclareMathOperator{\Ric}{Ric}
\DeclareMathOperator{\vol}{vol}
\DeclareMathOperator{\Law}{Law}

\DeclareMathOperator{\emat}{\mathbf{exp_{mat}}}

\newcommand{\cheng}[1]{\noindent{\textcolor{red}{\{{\textbf{Cheng:}} \em #1\}}}}

\newcommand*{\party}[2]{\Gamma_{#1}^{#2}}

\definecolor{cdarkred}{rgb}{0.55,0.0,0.0}
\definecolor{darkblue}{rgb}{0.0,0.0,0.55}
\definecolor{cgray}{gray}{0.55}
\definecolor{cdarkblue}{RGB}{30,30,200}
\definecolor{black}{rgb}{0,0,0}

\title{Efficient Sampling on Riemannian Manifolds via Langevin MCMC}

\author{%
  Xiang Cheng \\
  Massachusetts Institute of Technology\\
  \texttt{x.cheng@berkeley.edu} \\
  \And
  Jingzhao Zhang \\
  Tsinghua University
  \texttt{jzhzhang@mit.edu}\\
  \And
  Suvrit Sra \\
  Massachusetts Institute of Technology\\
  \texttt{suvrit@mit.edu}\\
}

\begin{document}

\maketitle

\begin{abstract}
  We study the task of efficiently sampling from a Gibbs distribution $d \pi^* = e^{-h} d {\vol}_g$ over a Riemannian manifold $M$ via (geometric) Langevin MCMC; this algorithm involves computing exponential maps in random Gaussian directions and is efficiently implementable in practice. The key to our analysis of Langevin MCMC is a bound on the discretization error of the geometric Euler-Murayama scheme, assuming $\nabla h$ is Lipschitz and $M$ has bounded sectional curvature.  Our error bound matches the error of Euclidean Euler-Murayama in terms of its stepsize dependence.  Combined with a contraction guarantee for the geometric Langevin Diffusion under Kendall-Cranston coupling, we prove that the Langevin MCMC iterates lie within $\epsilon$-Wasserstein distance of $\pi^*$ after $\tilde{O}(\epsilon^{-2})$ steps, which matches the iteration complexity for Euclidean Langevin MCMC. Our results apply in general settings where $h$ can be nonconvex and $M$ can have negative Ricci curvature. Under additional assumptions that the Riemannian curvature tensor has bounded derivatives, and that $\pi^*$ satisfies a $CD(\cdot,\infty)$ condition, we analyze the stochastic gradient version of Langevin MCMC, and bound its iteration complexity by $\tilde{O}(\epsilon^{-2})$ as well.
  
\end{abstract}
\vspace*{-5pt}
\section{Introduction}
Stochastic differential equations (SDEs) offer a powerful formalism for studying diffusion processes, Brownian motion, and algorithms for sampling and optimization. We study in particular the following \emph{geometric stochastic differential equation}:
\begin{equation}
    d x(t) = \beta(x(t)) dt + dB^g_t,
    \label{e:intro_sde}
\end{equation}
that evolves on a $d$-dimensional Riemannian manifold $(M, g)$. Analogous to the Euclidean setting, the map $\beta: x \to T_xM$ denotes a drift ($T_xM$ is the tangent space at $x\in M$), and $dB^g_t$ denotes the standard Brownian motion on $M$~\citep[Ch.~3]{hsu2002stochastic}. The notation in~\eqref{e:intro_sde} is a shorthand; we can define~\eqref{e:intro_sde} more precisely as the unique diffusion process whose generator is the operator $L f = \lin{\nabla f, \beta} + \frac{1}{2} \Delta f$, where $\Delta$ is the Laplace-Beltrami operator (see~\citep[Proposition.~3.2.1]{hsu2002stochastic}). 

When the drift $\beta(x) = - \frac{1}{2} \nabla h(x)$ (the gradient is taken w.r.t.\ the manifold metric), the SDE~\eqref{e:intro_sde} has an invariant distribution $\pi^*$ with density $e^{-h(x)}$, with respect to the Riemannian volume measure ${\vol}_g(x)$; for $M= \Re^d$ with metric $g(x) = I$, this reduces to the familiar result for Euclidean Langevin diffusion. Using the \emph{Kendall-Cranston Coupling} technique (Chapter 6.5, \cite{hsu2002stochastic}) together with a carefully constructed Lyapunov function by \cite{eberle2016reflection}, we can quantify the mixing rate for \eqref{e:intro_sde} in the 1-Wasserstein distance (w.r.t.\ the manifold distance) under suitable regularity conditions. 

However, the exact SDE is not implementable in practice, so we consider an MCMC algorithm using the \emph{Geometric Euler-Murayama}~\citep{piggott2016geometric,muniz2021higher} discretization of~\eqref{e:intro_sde}:
\begin{equation}
  x_{k+1} = \Exp_{x_k}\bigl(\delta \beta(x_k) + \sqrt{\delta} \zeta_k\bigr),
  \label{e:intro_euler_murayama}
\end{equation}
where $\Exp_x$ denotes the exponential map for $M$, and $\zeta_k$ is a standard Gaussian with respect to \emph{any} orthonormal basis of $T_{x_k}M$. It has long been known that the discrete-time process~\eqref{e:intro_euler_murayama} converges to the SDE \eqref{e:intro_sde} in the limit as $\delta \to 0$~\citep{gangolli1964construction,jorgensen1975central,hsu2002stochastic}.

\subsection*{Motivation for studying the SDE~\eqref{e:intro_sde}}
Geometric SDEs such as~\eqref{e:intro_sde} play a crucial role in the design and analysis of MCMC algorithms \citep{girolami2011riemann,patterson2013stochastic} that have had much success in solving Bayesian problems on statistical manifolds. Recently, it has been shown that the \emph{mirror Langevin Algorithm} on the Hessian manifold of a function can be significantly faster than its Euclidean counterpart \citep{zhang2020wasserstein,chewi2020exponential,li2022mirror,gatmiry2022convergence}. 

These SDEs also directly relate to the tasks of sampling and optimization on manifolds, where often a Lie group structure helps capture symmetries (e.g., the Grassmann manifold, $\text{SO}(n)$, $O(n)$, etc.)~\citep{moitra2020fast,piggott2016geometric,muniz2021higher}. Furthermore, in \cite{lee2017geodesic,lee2018convergence}, the authors propose fast algorithms for constrained sampling and volume estimation on polytopes by sampling from the Hessian manifold of a barrier function.

\section{Overview of Main Contributions}
\label{s:contributions}
\textbf{Our first contribution} in this paper is to provide a quantitative, non-asymptotic bound on the discretization error between \eqref{e:intro_sde} and \eqref{e:intro_euler_murayama}. We present this bound in Lemma \ref{l:informal_discretization-approximation-lipschitz-derivative} in Section \ref{Brownian Motion Section}. In our bound, the expected-squared-distance between a single step of \eqref{e:intro_euler_murayama} and \eqref{e:intro_sde} over $\delta$ time is bounded by $O\lrp{\delta^3}$. This $\delta^3$ error scaling matches that of Euler Murayama in Euclidean space \citep{durmus2017nonasymptotic}. We highlight that our bound is entirely explicit, and depends polynomially on dimension, sectional curvature of $M$, and the Lipschitz parameter for $\beta$ -- quantities which are intrisic to the manifold, and  invariant to the choice of coordinate system. 

Two recent papers, \citep{wang2020fast,li2022mirror} also look into a similar problem of bounding the discretization error of \eqref{e:intro_euler_murayama}. \citet{wang2020fast} analyze the error of \eqref{e:intro_euler_murayama}, but rely on uniformly bounding certain derivatives of the densities along the sample path (see Assumption 3~\citep{wang2020fast}); even in Euclidean space, it is not clear whether such a bound exists, and whether it depends only polynomially on parameters such as dimension. \citet{li2022mirror} provide a quantitative bound on the bias of \eqref{e:intro_euler_murayama} on Hessian manifolds; they assume a bound on a ``modified self-concordance'' parameter that is not affine invariant and can be made arbitrarily large. 

We discuss Lemma \ref{l:informal_discretization-approximation-lipschitz-derivative} in more detail , and sketch its proof in Section \ref{Brownian Motion Section}. The proof relies on a careful construction of geometric Langevin Diffusion as the limit of \eqref{e:intro_euler_murayama}, and may be of independent interest.

\textbf{Our second contribution} is to show that after $\t{O}(\epsilon^{-2})$ steps, the distribution of \eqref{e:intro_euler_murayama} is within $\epsilon$ 1-Wasserstein distance from the stationary distribution of \eqref{e:intro_sde}. We present this result in Theorem \ref{t:langevin_mcmc}. This $\epsilon$ dependence matches that of Euclidean Langevin MCMC~\citep{durmus2017nonasymptotic}. 

The $\t{O}$ in our iteration complexity hides polynomial dependency on dimension, sectional curvature, Lipschitz-parameter of $\beta$, and $1/\alpha$, where $\alpha$ can be viewed as ``mixing rate of the exact SDE.'' Theorem \ref{t:langevin_mcmc} requires that $\beta$ satisfy a manifold analog of the distant-dissipativity assumption (Assumption~\ref{ass:distant-dissipativity} with $m > L_{Ric}/2$, where $-L_{Ric}$ is a lower bound on the Ricci curvature of $M$). Assumption \ref{ass:distant-dissipativity} is general enough to include cases when $\beta = -\nabla h$  for some nonconvex $h$ or when $M$ has negative Ricci curvature. The catch is that in such cases, $1/\alpha$ can become very large; this is generally unavoidable, even on Euclidean space. Distant-dissipativity assumption has often been used in the analysis of Euclidean Langevin diffusion for non-log-concave distributions \citep{eberle2016reflection,bou2020coupling,gorham2019measuring,cheng2020stochastic}.

\textbf{We highlight a computational sub-contribution that} \eqref{e:intro_euler_murayama} only requires computing exponential maps in the direction of $\beta$ plus a uniform Gaussian direction. In many cases, exponential maps are efficiently computable, and this Theorem \ref{t:langevin_mcmc} leads to a computationally efficient sampling algorithm. Our analysis extends naturally if one replaces \eqref{e:intro_euler_murayama} by a retraction step; if the retraction is of order $3/2$ or higher, the one-step error still scales as $O(\delta^3)$. The question of how to construct good retraction maps is well studied in literature, and is somewhat orthogonal to our main objective, so we do not provide details here, and instead refer readers to \citep{absil2012projection,absil2009optimization}. 

Recently, \cite{ahn2021efficient,gatmiry2022convergence,li2020riemannian} have considered a different discretization of \eqref{e:intro_sde}, where one assumes that \emph{the endpoint of the geometric Brownian motion $dB^g_t$ can be sampled exactly} (so that only the drift $\beta$ is discretized). Sampling the exact geometric Brownian motion can be done efficiently in special settings such as the sphere \citep{li2020riemannian}, but on a general manifold, this can be much more expensive than computing an exponential map; in such settings the Langevin MCMC algorithm based on \eqref{e:intro_euler_murayama} may be prefereable. We do highlight, however, that a number of the above results provide error bounds in KL-divergence, which is tighter than the Wasserstein bound in Theorem \ref{t:langevin_mcmc}.

We discuss assumptions and consequences of Theorem \ref{t:langevin_mcmc} in more detail and sketch its proof in Section \ref{s:langevin MCMC}. The proof essentially combines our discretization analysis with a mixing result for \eqref{e:intro_sde} based on the \textit{Kendall-Cranston} coupling and a Lyapunov function from \citep{eberle2016reflection}.

\textbf{For our third contribution,} we analyze the manifold analog of SGLD \citep{welling2011bayesian}:
\begin{equation}
  x_{k+1} = \Exp_{x_k}\bigl(\delta \t{\beta}_k(x_k) + \sqrt{\delta} \zeta_k\bigr),
  \label{e:intro_sgld}
\end{equation}
The the difference between \eqref{e:intro_sgld} and \eqref{e:intro_euler_murayama}, is that at each step, $\beta$ is replaced by a random vector field $\t{\beta}_k$ which satisfies: (i) $\mathbb{E}[\t{\beta}_k] = \beta$; and (ii) $\|\t{\beta}_k(x) - \beta(x)\|\leq \sigma$ almost surely. We show that after $\t{O}\lrp{\epsilon^{-2}}$ steps, $\eqref{e:intro_sgld}$ is within $\epsilon$-2-Wasserstein distance from the stationary distribution of \eqref{e:intro_sde}; thus the $\epsilon$ dependency does not degrade when replacing $\beta$ by its stochastic estimate. We present this result in Theorem \ref{t:SGLD}. For this analysis, we require Assumption \ref{ass:distant-dissipativity}, with $m > L_{\Ric}/2$ and $\R = 0$; this is a more restrictive condition than Theorem \ref{t:langevin_mcmc}; when $\beta = -\frac{1}{2} \nabla h$, this restriction is equivalent to the $CD(\cdot,\infty)$ condition \citep{bakry2014analysis}: $\nabla^2 h + \Ric \succ 0$, where $\Ric$ is the Ricci curvature tensor\footnote{note that $\beta=-\nicefrac{1}{2}\nabla h$, therefore $m$ is $\nicefrac{1}{2}$ times strong-convexity parameter of $h$.}.

Though restrictive, the $CD(\cdot,\infty)$ class of distributions is nonetheless interesting; it is the manifold analog of the class of log-concave densities on Euclidean space, and sampling from this class has been a topic of much recent interest. In machine learning, the stochastic estimate of $\beta$ can often be computed much more quickly than the exact $\beta$, e.g., when $\beta = -2\nabla f$, where $f$ is the empirical average of some loss over a large number of observations. In such cases, \eqref{e:intro_sgld} can be much faster than \eqref{e:intro_euler_murayama}. In Section \ref{s:SGLD} we discuss the assumptions and consequences of Theorem \ref{t:SGLD} in greater detail.

Though Theorem \ref{t:SGLD} requires more restrictive assumptions than Theorem \ref{t:langevin_mcmc}, its proof differs from the proof of Theorem \ref{t:langevin_mcmc} in a significant way: we show mixing of the discrete-time process directly (instead of relying on mixing of the exact SDE). This approach is necessary in order to take advantage of the fact that $\t{\beta}_k$ is an unbiased estimate $\beta$. The key intermediate result for showing mixing under \eqref{e:intro_sgld} is Lemma \ref{l:discrete-approximate-synchronous-coupling-ricci}. We believe this lemma to be of independent interest as it quantifies the distance evolution between two arbitrary discrete-time stochastic processes (that may be unrelated to Gaussian noise and Brownian motion). As an example, \citet{mangoubi2018rapid} showed that ball walk mixes quickly on manifolds with positive sectional curvature. Using Lemma \ref{l:discrete-approximate-synchronous-coupling-ricci}, one can show the more general statement ``ball walk mixes quickly on compact manifolds with positive Ricci curvature.'' 

\section{Preliminaries: Key Assumptions and Notation}
\label{s:preliminaries}
We state in this section the four key assumptions of this paper. Our first assumption involves lower bounding Ricci curvature of the manifold $M$. Let $\Ric$ denote the Ricci curvature tensor:
\begin{assumption}\label{ass:ricci_curvature_regularity}
  We assume that for all $x\in M$, $u,u\in T_x M$, $\Ric(u,u) \geq - L_{\Ric}$, for some $L_{\Ric} \in \Re$.
\end{assumption}
Intuitively, the more positive the Ricci curvature (i.e. the smaller the value of $L_{\Ric}$), the faster geometric Brownian motion mixes.

Our second assumption is a natural generalization of the distant-dissipativity condition in the Euclidean setting. It helps ensure that the drift traps the variable within a bounded region. For any $x,y\in M$, let $\dist(x,y)$ denote their Riemannian distance.
\begin{assumption}
    \label{ass:distant-dissipativity}
    We call a vector field $\beta$ $(m,q,\R)$-\emph{distant-dissipative} if there exist constants $m>0$, $\R \geq 0$, and $q\in \Re$ such that, for all $x,y$ satisfying $\dist\lrp{x,y} \geq \R$, there exists a minimizing geodesic $\gamma: [0,1] \to M$ with $\gamma(0) = x$ and $\gamma(1) = y$, such that the inequality
    \begin{alignat*}{1}
      \linp{\party{}{}\lrp{\beta(y);{y}\to{x}}-\beta(x), \gamma'(0)} \leq -m \dist\lrp{x,y}^2,
    \end{alignat*}
    holds, where $\party{}{}\lrp{\cdot; {y}\to{x}}$ denotes parallel transport from $T_yM$ to $T_xM$ along $\gamma$. In addition, for all $x,y$ satisfying $\dist(x,y) \leq \R$, there exists a minimizing geodesic $\gamma: [0,1] \to M$ with $\gamma(0) = x$ and $\gamma(1) = y$, such that we have instead the inequality
    \begin{alignat*}{1}
        \linp{\party{}{}\lrp{\beta(y);{y}\to {x}}-\beta(x), \gamma'(0)} \leq q \dist\lrp{x,y}^2.
    \end{alignat*}
\end{assumption}
For some intuition: a strictly convex function will have $m > 0$, $\R = 0$, and $q$ arbitrary. Note that we do not require a unique geodesic between $x,y$.

For some intuition about Assumption \ref{ass:distant-dissipativity}: in the Euclidean setting, the first condition simplifies to $\lin{\beta(y) - \beta(x), y-x} \leq - m \lrn{y-x}_2^2$, and the second condition simplifies to $\lin{\beta(y) - \beta(x), y-x} \leq q \lrn{y-x}_2^2$.

We need our third and fourth assumptions for bounding the discretization error of \eqref{e:intro_euler_murayama}. Assumption \ref{ass:beta_lipschitz} upper bounds the Lipschitz constant of $\beta$; Assumption \ref{ass:sectional_curvature_regularity} lower bounds sectional curvature of $M$.

\begin{assumption}\label{ass:beta_lipschitz}
  A vector field $\beta$ is $L_{\beta}'$-Lipschitz if, for all $x\in M$ and all $v\in T_x M$, $\lrn{\nabla_v \beta(x)} \leq L_\beta'\lrn{v}$.
\end{assumption}

In the (flat) Euclidean setting, Assumption \ref{ass:beta_lipschitz} is equivalent to saying that $\beta(x)$ being a $L_\beta'$ Lipschitz vector field.

\begin{assumption}\label{ass:sectional_curvature_regularity}
  Let $R$ be the Riemannian curvature tensor of the manifold $M$. We assume that there exists $L_R \in \Re^+$ such that for all $x\in M$, and for all $u,v,w,z\in T_x M$, $\lin{R(u,v)v,u} \leq L_R\lrn{u}^2\lrn{v}^2$.
\end{assumption}
We now introduce additional notation that will be used throughout this paper. We assume some background in Riemannian geometry, and freely use standard notation; we refer the reader to~\citep{jost2008riemannian,lee2006riemannian,petersen2006riemannian} for an in depth treatment. Readers may also find some works on Riemannian optimization useful as additional context: \citep{bacak2014convex,udriste2013convex,absil2009optimization,zhang2016first,boumal2022intro}.

We use $\nabla$ to denote the Levi Civita connection. Given $x,y\in M$ and $v\in T_x M$, we use $\party{}{}(v;x\to y)$ to denote parallel transport of $v$ from $x$ to $y$ along their minimizing geodesic (if such a choice is not unique we will specify); we sometimes also use the more concise alternative notation $\party{x}{y} v := \party{}{}(v;x\to y)$. Given a general curve (possibly non-geodesic) $\gamma:[0,1] \to M$, we will also use $\party{\gamma(t)}{} v$ to denote the parallel transport of $v$ from $\gamma(0)$ to $\gamma(1)$, along $\gamma$.

A basis $F$ of the tangent space $T_x M$ at some point $x\in M$ is an ordered tuple  $(F^1,\ldots,F^d)$ of vectors that span $T_x M$. We use $\party{}{}\lrp{F;x\to y} := \lrp{\party{}{}\lrp{F^1;x\to y}...\party{}{}\lrp{F^d;x\to y}}$ to denote the ordered tuple of parallel transport of the each of the basis vectors in $F$. Given $v\in \Re^d$ and basis $F$ of some $T_x M$, we use $v \circ F$ as shorthand for $\sum_{i=1}^d v_i F^i$. A distribution that we will see frequently in this paper is the one given by $\xib \circ E^x$, where $\xib \sim \N(0,I)$ is a random vector in $\Re^d$, and $x \in M$ and $E^x$ is an orthonormal basis of $T_x M$. We use $\N_x\lrp{0,I}$ to denote the distribution of $\xib \circ E^x$. One can verify that $\N_x\lrp{0,I}$ does not depend on the choice of basis $E^x$. 


\subsection{An illustrative example: sampling from a sphere}
To give some intuition about Assumptions \ref{ass:ricci_curvature_regularity} - \ref{ass:sectional_curvature_regularity}, we present a simple example for sampling from $S^{d-1}$, the unit d-sphere in $\Re^d$, which is a positively curved Riemannian manifold.\footnote{The sphere is one of the special cases when Brownian motion \emph{can be efficiently sampled exactly}, and so if one's actual goal is to sample from a sphere, prior work such as \cite{li2020riemannian} provide a better algorithm and analysis.} The subsequent bounds for the $S^d$ case also generalize with minor modifications to other closely-related manifolds such as SO$(d)$.

Let $U(x) : \Re^d \to Re$ be a potential function, and suppose that we wish to sample from $d p(x) \propto e^{-U(x)} d vol_g(x)$ defined over $S^{d-1}$. It is known that for any $x\in S^{d-1}, v\in T_x M$, $\Ric(v,v)>0$, so that Assumption \ref{ass:ricci_curvature_regularity} is satisfied with $L_{Ric}  = 0$. The sectional curvature is bounded by $1$, so that Assumption \ref{ass:sectional_curvature_regularity} holds with $L_R = 1$. For the remainder of this example, for $x\in S^{d-1}$ we will identify $T_x M$ with $\lrbb{v\in \Re^d : v^T x = 0}$. Since the sphere is a submanifold of $\Re^d$, the metric is simply given by the Euclidean dot product, i.e. $g(u,v) = u^T v$.


Next, we will verify Assumption \ref{ass:beta_lipschitz}. Let us assume that for all $x$ on the sphere, $U(x)$ satisfies $\lrn{\bm{\nabla} U(x)} \leq L_1$ and $\lrn{\bm{\nabla^2} U(x)}_2 \leq L_2$, where $\bm{\nabla}$ and $\bm{\nabla^2}$ are the usual Euclidean gradient and Hessian respectively (bolded to distinguish from $\nabla$, the covariant derivative). In order to sample from $dp(x)$, $\beta(x) = (I - x x^T)\bm{\nabla} U (x)$. It is known that a geodesic $\gamma(t)$ corresponds to a great arc on the sphere, and that parallel transport of a vector $v$ involves a rotation of the tangential component of $v$. To be precise, for any $x \in S^d$ and $u,v \in T_x M$ such that $\lrn{u}=1$, we verify that, in Cartesian coordinates, $\nabla_u v = (u u^T) v$. Thus $\nabla \beta(x) = (uu^T) (I - xx^T) \bm{\nabla} U(x) + \bm{\nabla^2} U(x) u \leq L_1 + L_2$. Thus Assumption \ref{ass:beta_lipschitz} holds with $L_{\beta}' = L_1 + L_2$.

Finally, we will verify Assumption \ref{ass:distant-dissipativity}. Since $S^{d-1}$ has diameter $\pi$, the first part of Assumption \ref{ass:distant-dissipativity} is satisfied with arbitrary $m$ and $R = \pi$. The second part of Assumption \ref{ass:distant-dissipativity} is satisfied with $q = L_\beta'$. This is because $\frac{d}{dt} \lin{\party{}{}\lrp{\beta(\gamma(t)) ; \gamma(t)\to \gamma(0)} - \beta(\gamma(0)), \gamma'(0)} = \frac{d}{dt} \lin{\beta(\gamma(t)), \gamma'(t)} = \lin{\nabla_{\gamma'(t)}\beta(\gamma(t)),\gamma'(t)} \leq L_\beta' \lrn{\gamma'(t)}^2$.

\section{Bounding the Discretization Error in Geometric Euler-Murayama}
\label{Brownian Motion Section}

One of our main technical contributions is an error bound for the Euler-Murayama discretization. We provide an informal statement of this error bound as Lemma \ref{l:informal_discretization-approximation-lipschitz-derivative} below, and provide the formal statement as Lemma \ref{l:discretization-approximation-lipschitz-derivative} in Appendix \ref{ss:langevin_mcmc_on_manifold}.
\begin{lemma}[Informal version of Lemma \ref{l:discretization-approximation-lipschitz-derivative}]
  \label{l:informal_discretization-approximation-lipschitz-derivative}
  Let $x(t)$ denote the solution to \eqref{e:intro_sde} initialized at some $x(0)$. Let $x^0(t)$ denote one step geometric Euler Murayama discretization: $x^0(t) := \Exp_{x(0)}\lrp{t \beta(x(0)) + \sqrt{t} \zeta}$, where $\zeta \sim \N_{x(0)}(0,I)$. Under Assumptions \ref{ass:beta_lipschitz} and \ref{ass:sectional_curvature_regularity}, for sufficiently small $t$, there exists a coupling between $x(t)$ and $x^0(t)$ such that 
  \begin{alignat*}{1}
    \E{\dist\lrp{x(t),x^0(t)}^2} \leq O\lrp{t^3},
  \end{alignat*}
  where $O()$ hides polynomial dependence on the Lipschitz constant $L_\beta'$, the sectional curvature of $M$, and dimension $d$.
\end{lemma}
\subsection*{Proof Sketch for Lemma \ref{l:informal_discretization-approximation-lipschitz-derivative}}
Given a time interval $T$, we construct a sequence of processes $\{x^i(t)\}_{i \ge 0}$ indexed by $i$ (exact definitinion given in \eqref{d:x^i(t):0} below). Marginally, each $x^i(t)$ corresponds to the linear interpolation of a sequence of Euler-Murayama steps with stepsize $\delta^i = 2^{-i} T$. More specifically, for $k\in \Z^+$,
\begin{alignat*}{1}
  & x^i_{k+1} = \Exp_{x^i_k}\lrp{\delta^i \beta(x^i_k) + \sqrt{\delta^i} \zeta^i_k} \qquad \qquad \zeta^i_k \sim \N_{x^i_k} (0,I)\\
  & x^i(t) = \Exp_{x^i_k} \lrp{\frac{t-k\delta^i}{\delta^i} \lrp{\delta^i\beta(x^i_k) + \sqrt{\delta^i} \zeta^i_k}} \qquad \qquad \text{for }t\in[k\delta^i, (k+1)\delta^i]
\end{alignat*}
Notice that for $i=0$, $x^0(T)$ is a single step of \eqref{e:intro_euler_murayama} with stepsize $T$. On the other hand, $x(t):=\lim_{i\to\infty} x^i(t)$ is exactly the SDE in \eqref{e:intro_sde} (see Lemma \ref{l:Phi_is_diffusion} at the end of this section). We will soon see that $\E{\dist\lrp{x^i(T), x^{i+1}(T)}^2} = O\lrp{T^2 \delta^i}$. We can then bound $\E{\dist\lrp{x^0(T),x(T)}^2} = O\lrp{T^3}$ by summing over the pairwise distance between adjacent $x^i$ and $x^{i+1}$ for $i\in \Z^+$ (this diminishes geometrically). This proves Lemma \ref{l:informal_discretization-approximation-lipschitz-derivative}.

To bound the distance between $x^i$ and $x^{i+1}$, we introduce a crucial additional structure: adjacent pairs of processes $(x^i_k, x^{i+1}_k)$ are coupled using the \textbf{manifold analog of synchronous coupling}, together with the \textbf{discrete-time analog of "rolling without slipping"}. (see Remark \ref{remark:sde_construction}),

The exact formula for how $x^{i+1}_k$ is to be constructed from $x^i_k$ is given in \eqref{d:x^i_k} below; it is a little dense, so we provide a pictorial illustration in Figure \ref{fig:euler-murayama-coupling}. For simplicity, we assume that $\beta = 0$ (the dominating error in Lemma \ref{l:informal_discretization-approximation-lipschitz-derivative} is due to Brownian motion). Figure \ref{fig:euler-murayama-coupling} takes place over a period of time from $k\delta^i$  to $(k+1)\delta^i$ (equivalently, $2k\delta^{i+1}$  to $(2k+2)\delta^{i+1}$). The black squiggle denotes $\lrp{\BB(t)-\BB(k\delta^i)}\circ E^i_k$ for $t\in [k\delta^i, (k+1)\delta^i]$, where $\BB(t)$ is a standard $d$-dimensional Brownian motion, and $E^i_k$ is some basis at $x^i_k$. Let $b^i_k := \lrp{\BB((k+1)\delta^i)-\BB(k\delta^i)}\circ E^i_k$ (black solid arrow). Bounding the distance between $x^i$ and $x^{i+1}$ consists of two steps:

\vspace*{-3pt}
\setlength{\abovecaptionskip}{-40pt}
\setlength{\belowcaptionskip}{-5pt}
\begin{figure}[h]
  \includegraphics[width=8cm]{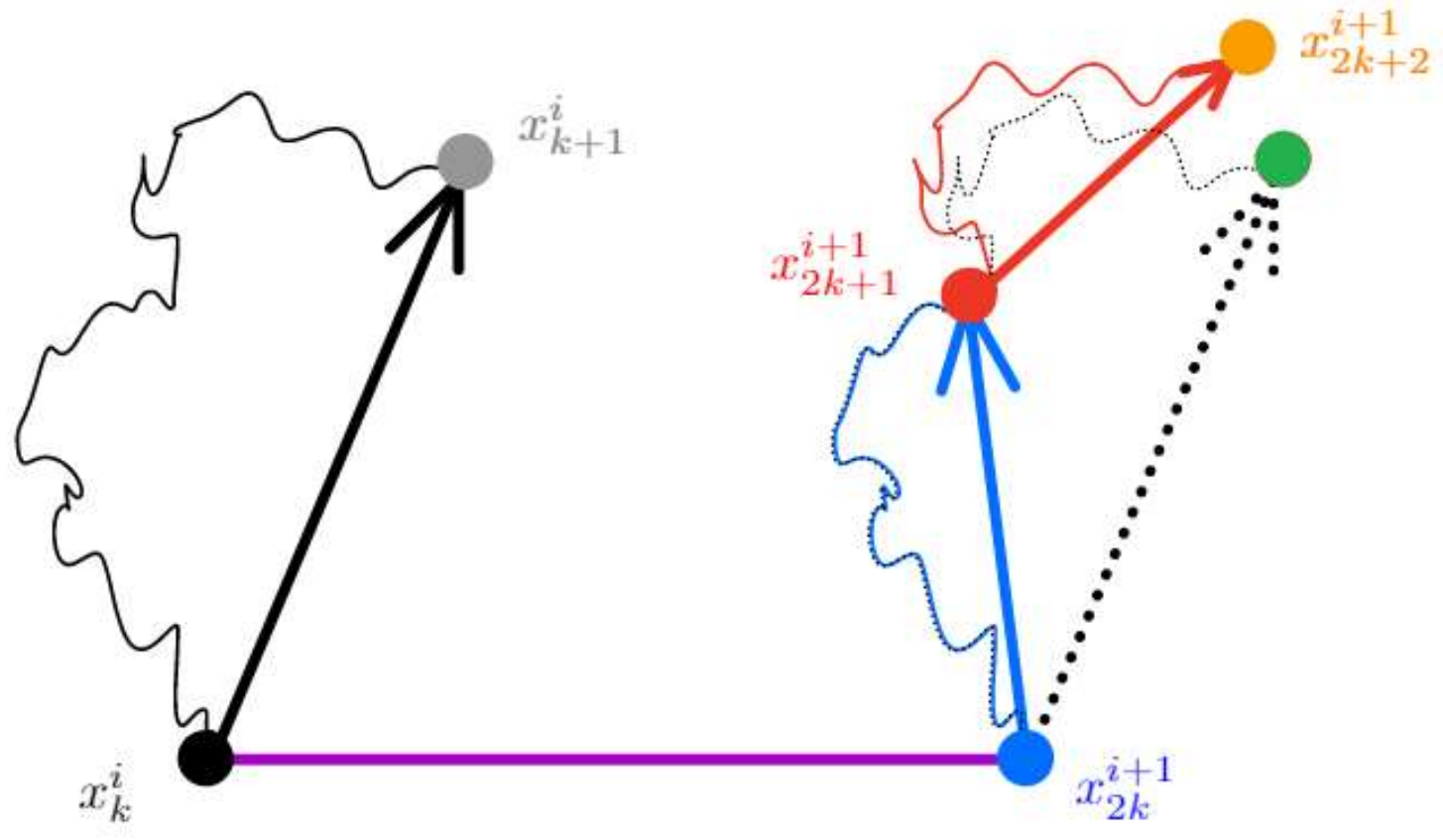}
  \centering
  \caption{An illustration of the coupling between paths for $x^i$ and $x^{i+1}$ when $\beta = 0$ over the time interval $[k\delta^i, (k+1)\delta^i]$.} 
  \label{fig:euler-murayama-coupling}
\end{figure}
\textbf{Step 1: Bounding $\dist\lrp{x^i_{k+1}, \Exp_{x^{i+1}_{2k}}(\party{}{}\lrp{b^i_k; {x^i_k}\to {x^{i+1}_{2k}}})}$} [Solid black arrow vs dotted black arrow]\\
The purple line represents the geodesic from $x^i_k$ to $x^{i+1}_{2k}$; by construction, the basis $E^{i+1}_{2k}$ is the parallel transport of $E^i_k$ along the purple geodesic. The dotted black arrow denotes the $\lrp{\BB((k+1)\delta^i)-\BB(k\delta^i)}\circ E^{i+1}_{2k}$; we verify that it is the paralle transport of $b^i_k$ (black arrow) along the purple geodesic. Using Lemma \ref{l:discrete-approximate-synchronous-coupling}, we can bound the distance between $x^i_{k+1} = \Exp_{x^i_k}(b^i_k)$ (gray point) and $\Exp_{x^{i+1}_{2k}}(\party{}{}\lrp{b^i_k;{x^i_k}\to{x^{i+1}_{2k}}})$ (green point) as
\begin{alignat*}{1}
  \E{\dist\lrp{\Exp_{x^i_k}(b^i_k), \Exp_{x^{i+1}_{2k}}(\party{}{}\lrp{b^i_k;{x^i_k}\to{x^{i+1}_{2k}}})}^2} \leq \lrp{1 + O(L_R d \delta^i)} \E{\dist\lrp{x^i_k, x^{i+1}_{2k}}^2}
\end{alignat*}

\textbf{Step 2: Bounding $\dist\lrp{\Exp_{x^{i+1}_{2k}}(\party{}{}\lrp{b^i_k;{x^i_k}\to{x^{i+1}_{2k}}}),x^{i+1}_{2k+2}}$} [Dotted black arrow vs blue + red arrow]\\
Let 
\begin{alignat*}{1}
  & b^{i+1}_{2k} := \lrp{\BB((2k+1)\delta^{i+1})-\BB(2k\delta^{i+1})}\circ E^{i+1}_{2k}\\
  & b^{i+1}_{2k+1} := \lrp{\BB((2k+2)\delta^{i+1})-\BB((2k+1)\delta^{i+1})}\circ E^{i+1}_{2k}
\end{alignat*}
We verify that $\party{}{}\lrp{b^i_k;{x^i_k}\to{x^{i+1}_{2k}}}$ (black dotted arrow) is equal to $b^{i+1}_{2k} + b^{i+1}_{2k+1}$ (note that $\delta^{i+1} = 1/2 \delta^i$). On the other hand, $x^{i+1}_{2k+2}$ (orange point) is obtained from taking a step in $b^{i+1}_{2k}$ direction (blue arrow) followed by a step in the $\party{}{} \lrp{b^{i+1}_{2k+1};{x^{i+1}_{2k}}\to{x^{i+1}_{2k+1}}}$ direction (red arrow). Due to curvature, this is not the same as taking a single step in the $b^i_k = b^{i+1}_{2k} + b^{i+1}_{2k+1}$ direction (green point). However, we can bound the distance between the orange point and the green point using Lemma 3 of \cite{sun2019escaping} (restated as Lemma \ref{l:triangle_distortion} for ease of reference):
\begin{alignat*}{1}
  \E{\dist\lrp{\Exp_{x^{i+1}_{2k}}(\party{}{}\lrp{b^i_k;{x^i_k}\to{x^{i+1}_{2k}}}),x^{i+1}_{2k+2}}^2} \leq O\lrp{\lrn{b^{i+1}_{2k}}^2\lrn{b^{i+1}_{2k+1}}^2\lrp{\lrn{b^{i+1}_{2k}}+\lrn{b^{i+1}_{2k+1}}}^2} = O\lrp{(\delta^i)^{3}}
\end{alignat*}
Summing the bounds of Step 1 and Step 2 gives $\E{\dist\lrp{x^i_{k+1}, x^{i+1}_{2k+2}}^2} \leq e^{O\lrp{L_R d \delta^i}} \E{\dist\lrp{x^i_k, x^{i+1}_{2k}}^2} + O\lrp{(\delta^i)^3}$. Recursing over $k=0...2^i$, and assuming that $T\leq \frac{1}{L_R d}$, we get $\E{\dist\lrp{x^i_{2^i}, x^{i+1}_{2^{i+1}}}^2} = \E{\dist\lrp{x^i(T), x^{i+1}(T)}^2} \leq O\lrp{T^2 \delta^i}$. Summing over $i=0...\infty$ gives $\E{\dist\lrp{x^i(T), x(T)}^2} = O\lrp{T^{3}}$.



\subsection*{Details for Construction of Brownian Motion}
\label{ss:discrete_gaussian_walk_construction}
To bound the discretization error, we first characterize the manifold SDE~\eqref{e:intro_sde} as the limit of a family of random processes. Let $x_0 \in M$ be an initial point and $E = \lrbb{E^1,\ldots,E^d}$ be an orthonormal basis of $T_{x^0}$. Let $\BB(t)$ denote a standard Brownian Motion in $\Re^d$. Let $T\in \Re^+$. Define
\begin{alignat*}{1}
    & x^0_0 = x_0, \qquad E^{0}_0 = E,\\
    & x^0_1 = \Exp_{x^0_0}(T \beta(x^0_0) + {\lrp{\BB\lrp{T} - \BB\lrp{0}}} \circ E^{0}_0).
\end{alignat*}
For any $i\in \Z^+$, let $\delta^i := 2^{-i}T$. We will now define points $x^i_k \in M$ and orthonormal basis $E^{i}_k$ of $T_{x^i_k}$ for all $i$ and all $k\in \{0,\ldots,\nicefrac{T}{\delta^i}\}$. Our construction is inductive: Suppose we have already defined $x^i_k$ and $E^{i}_k$ for some $i$ and for all $k\in \{0,\ldots,\nicefrac{T}{\delta^i}\}$. Then, we construct $x^{i+1}_k$, for all $k=\{0,\ldots,\nicefrac{T}{\delta^{i+1}}\}$, as follows: 
\begin{alignat*}{1}
    & x^{i+1}_0 := x_0, \qquad E^{i+1}_{0} := E,\\
    & x^{i+1}_{2k+1} := \Exp_{x^{i+1}_{2k}}\lrp{\delta^{i+1}\beta\lrp{x^{i+1}_{2k}} + {\lrp{\BB\lrp{\lrp{2k+1}\delta^{i+1}} - \BB\lrp{2k\delta^{i+1}}}} \circ E^{i+1}_{2k}},\\
    & E^{i+1}_{2k+1} := \party{}{}\lrp{E^{i+1}_{2k}; {x^{i+1}_{2k}} \to {x^{i+1}_{2k+1}} },\\
    & x^{i+1}_{2k+2} := \Exp_{x^{i+1}_{2k+1}}\lrp{\delta^{i+1}\beta\lrp{x^{i+1}_{2k+1}} + {\lrp{\BB\lrp{\lrp{2k+2}\delta^{i+1}} - \BB\lrp{\lrp{2k+1}\delta^{i+1}}}} \circ E^{i+1}_{2k+1}},\\
    & E^{i+1}_{2k+2} := \party{}{} \lrp{E^{i}_{k+1}; {x^{i}_{k+1}}\to{x^{i+1}_{2k+2}}}
    \elb{d:x^i_k}.
\end{alignat*}

The above display defines points $x^{i+1}_k$ for all $k = \lrbb{0,\ldots,\nicefrac{T}{\delta^{i+1}}}$. For parallel transport, if the minimizing geodesic is not unique, any arbitrary choice will do. We verify that for any $i$, for all $k$, $x^i_k$ is indeed of the form \eqref{e:intro_euler_murayama}, i.e. $x^i_{k+1} = \Exp_{x^i_k} \lrp{\delta^i \beta(x^i_k) + \sqrt{\delta^i} \zeta^i_k}$, where $\zeta^i_k := \frac{1}{\sqrt{\delta^i}} \lrp{\BB((k+1)\delta^i) - \BB(k\delta^i) \circ E^i_k} \sim \N_{x^i_k}(0,I)$. Finally, for any $i$, any $k$, and any $t\in [k\delta^i, (k+1)\delta^i)$, we define $x^i(t)$ to be the ``linear interpolation'' of $x^i_k$ and $x^i_{k+1}$, i.e.,
\begin{alignat*}{1}
  x^{i}(t) := \Exp_{x^{i}_{k}}\bigl(\tfrac{t-k\delta^i}{\delta^i}\lrp{\delta^i \beta\lrp{x^i_k} + \lrp{\BB\lrp{(k+1)\delta^i} - \BB\lrp{k\delta^{i}}}} \circ E^{i}_{k}\bigr).
  \elb{d:x^i(t):0}
\end{alignat*}
\begin{remark}\label{remark:sde_construction}
  The choice of basis $E^i_k$ in \eqref{d:x^i_k} can be seen as a combination of ``synchronous coupling'' and ``rolling without slipping'' (see Chapter 2 of \citep{hsu2002stochastic}). In particular, $E^{i+1}_{2k} := \party{}{}\lrp{E^{i}_{k};{x^{i}_{k}}\to {x^{i+1}_{2k}} }$ corresponds to ``synchronous coupling''---the step from $x^{i+1}_{2k}$ to $x^{i+1}_{2k+1}$ is (roughly) parallel to the step from $x^i_{k}$ to $x^i_{k+1}$. On the other hand, $E^{i+1}_{2k+1} := \party{}{} \lrp{E^{i+1}_{2k};{x^{i+1}_{2k}}\to{x^{i+1}_{2k+1}}}$ corresponds to ``rolling without slipping''---the step from $x^{i+1}_{2k+1}$ to $x^{i+1}_{2k+2}$ is with respect to an orthonormal basis that is parallel-transported from $x^{i+1}_{2k}$ to $x^{i+1}_{2k+1}$.
\end{remark}

We verify in the following Lemma that the limit, $i\to \infty$, of $x^i(t)$ is the SDE \eqref{e:intro_sde}:

\begin{lemma}\label{l:Phi_is_diffusion}
  For any $T$, for $t\in [0,T]$, let $x(t) := \lim_{i \to \infty} x^i(t)$. This limit exists uniformly almost-surely, and $x(t)$ is a diffusion process generated by the operator $L$ whose action on any smooth function $f$ is given by $L f = \lin{\nabla f, \beta} + \frac{1}{2} \Delta(f)$, where $\Delta$ denotes the Laplace Beltrami operator. Thus $x(t)$ is equal to \eqref{e:intro_sde} in distribution.

\end{lemma}
Lemma~\ref{l:Phi_is_diffusion} follows immediately from \citep{gangolli1964construction,jorgensen1975central}, but we provide a proof in Appendix \ref{ss:sde_construction_appendix} for completeness.

\section{Langevin MCMC on Riemannian Manifolds}
\label{s:langevin MCMC}
With the one-step discretization error in Lemma \ref{l:informal_discretization-approximation-lipschitz-derivative}, we can bound the iteration complexity of Langevin MCMC \eqref{e:intro_euler_murayama}:
\begin{theorem}[Convergence of Langevin MCMC on Riemannian Manifold]
    \label{t:langevin_mcmc}
    Assume the manifold $M$ satisfies Assumptions~\ref{ass:ricci_curvature_regularity} and~\ref{ass:sectional_curvature_regularity}. Assume in addition that there exists a constant $L_R'$ such that for all $x\in M$, $u,v,w,z,a\in T_x M$, $\lin{(\nabla_a R)(u,v)w,z} \leq L_R'\lrn{u}\lrn{v}\lrn{w}\lrn{z}\lrn{a}$ (this last assumption is for analytical convenience; $L_R'$ does not show up in the quantitative bounds).  Let $\beta$ be a vector field satisfying Assumptions \ref{ass:distant-dissipativity} and \ref{ass:beta_lipschitz}; assume in addition that $m > L_{\Ric}/2$ and that $q+ L_{\Ric}/2 \geq 0$. Let $x^*$ be some point with $\beta(x^*) = 0$.
    
    Let $y(t)$ denote the exact geometric SDE given in \eqref{e:intro_sde} initialized at some $y(0)$. Let $K \in \Z^+$ be some iteration number and $\delta$ be some stepsize. Let $x_k$ denote the Euler Murayama discretization of \eqref{e:intro_sde}, defined by $x_{k+1} = \Exp_{x_k}\lrp{\delta \beta(x_k) + \sqrt{\delta} \zeta_k}$, where $\zeta_k \sim \N_{x_k}\lrp{0,I}$, initialized at some $x_0$ satisfying $\dist\lrp{x_0,x^*} \leq 2\R$. 
    
    Then, there exists a constant $\C_0 = poly\lrp{L_\beta', d, L_R, \R, \frac{1}{m-L_{\Ric}/2}, \log K}$, such that if $\delta \leq \frac{1}{\C_0}$, then there is a coupling between $x_K$ and $y(K\delta)$ satisfying the distance bound
    \begin{alignat*}{1}
        \E{\dist\lrp{y(K\delta),x_{K}}} 
        \leq e^{-\alpha K \delta + \lrp{q + L_{\Ric}/2}\R^2/2}\E{\dist\lrp{y(0),x_{0}}} + \exp\lrp{\lrp{q + L_{\Ric}/2}\R^2} \cdot \t{O}\lrp{\delta^{1/2}}.
    \end{alignat*}
    $\t{O}$ hides polynomial dependence on $L_\beta', d, L_R, \R, \frac{1}{m-L_{\Ric}/2}, \log K, \log\frac{1}{\delta}$, and\\
    $\alpha := \min\lrbb{\frac{m-L_{\Ric}/2}{16}, \frac{1}{2 \R^2}} \cdot e^{{- \frac{1}{2}\lrp{q + L_{\Ric}/2} \R^2}}$.
\end{theorem}

We defer the proof of Theorem \ref{t:langevin_mcmc} to Appendix \ref{ss:proof_of_t:langevin_mcmc}, where we state the explicit expressions for $\C_0$.

\subsection*{Discussion of Theorem \ref{t:langevin_mcmc}}
\label{ss:theorem_1_discussion}
\textbf{To sample from $d\pi^*(x) = e^{-h(x)} d{\vol}_g(x)$}, we let $\beta(x) := - \frac{1}{2} \nabla h(x)$; under this choice of $\beta$, $\pi^*(x)$ is invariant under the SDE for $y(t)$. Picking $y(0) \sim \pi^*(x)$, we thus ensure that ${\Law}(y(t)) = \pi^*$ for all $t$. \textbf{Therefore, the $W_1$ distance between ${\Law}(x_K)$ and $\pi^*$ is upper bounded by $\E{\dist\lrp{y(K\delta), x_K}}$}, which is in turn upper bounded in Theorem \ref{t:langevin_mcmc}.

The distance bound consists of two terms: the first term is exponentially small in $K\delta$, so it goes to $0$ as the number of steps tends to infinity; the second term is proportional to $\delta^{1/2}$. We first pick $\delta = \t{O}\lrp{\epsilon^{2}}$ so that the second term is bounded by $\epsilon$. It then suffices to let $K = \frac{1}{\alpha \delta} \log\lrp{\frac{\E{\dist\lrp{y(0),x_0}}}{\epsilon}} + \lrp{q + L_{\Ric}/2}\R^2 = \t{O}\lrp{\epsilon^{-2}}$ in order for the first term to be bounded by $\epsilon$. Note that by our assumptions and by Lemma \ref{l:far-tail-bound-l2-brownian} $\E{\dist\lrp{y(0),x_0}}\leq poly\lrp{L_R, L_\beta', d, \R, \frac{1}{m}}$.

\textbf{We now discuss a few specific cases for Assumption \ref{ass:distant-dissipativity}.} The easiest setting is when $m > L_{\Ric}/2$ and $\R = 0$ (in this case the value of $q$ does not matter); this occurs, for example, if $\beta = - \frac{1}{2} \nabla h$ for some $c$-strongly-convex $h$, and $M$ has positive Ricci curvature. The mixing rate is $\alpha = \frac{c - L_{\Ric}}{32}$, and this closely relates to a well known result by Bakry and Emery on the Log-Sobolev Inequality of $CD(\cdot,\infty)$ distributions:

\emph{\textbf{Theorem} (\citep{bakry2014analysis}) Let $Ric$ denote the Ricci curvature tensor. If $d \pi^*(x) = e^{-h(x)} d {\vol}_g(x)$, and $\nabla^2 h + \Ric \succ \rho I$, then $\pi^*$ satisfies the Log-Sobolev Inequality (LSI) with parameter $\rho$.}

When $\pi^*$ satisfies the $LSI(\rho)$, it has been shown in \citep{bakry2014analysis} that for any initialization, the KL divergence of ${\Law}(x(t))$ of \eqref{e:intro_sde} with respect to its stationary distribution converges to $0$ with rate $\rho$. In our example in the preceding paragraph, $d\pi^* = e^{-h}d{\vol}_g$ satisfies LSI with $\rho = c-L_{\Ric}$, which is, up to a constant factor, equal to our mixing rate.

The requirement $q+L_{\Ric}/2 \geq 0$ is without loss of generality. If $q + L_{\Ric}/2 < 0$, we can take $q' = -L_{\Ric}/2$ and verify that $\beta$ satisfies Assumption \ref{ass:distant-dissipativity} with $(m,q',\R)$. Since $q'+L_{\Ric}/2 = 0$, the mixing rate is then $\alpha = \min\lrbb{\frac{m-L_{\Ric}/2}{16}, \frac{1}{2 \R^2}} $, which corresponds to the easy ``$CD(0,\infty)$'' setting.

\textbf{A harder, but more general setting} is when $\beta = \frac{1}{2} \nabla h$ for some non-convex $h$, and $L_{\Ric} > 0$, i.e. the manifold can have negative Ricci curvature. We will still need to assume that $\nabla h(x)$ is contractive for points which are further away than some radius $\R$; otherwise, the SDE \eqref{e:intro_sde} may drift off to infinity and $e^{-h}$ may not be integrable. Under Assumption \ref{ass:beta_lipschitz}, $\beta$ satisfies Assumption \ref{ass:distant-dissipativity} with $m, L_\beta', \R$. The mixing rate $\alpha$ is then proportional to $e^{-(L_\beta' + L_{\Ric}/2)\R^2}$; this can be very small if $h$ is nonconvex and highly nonsmooth, and the manifold $M$ has large negative Ricci curvature, leading to very slow mixing. This is generally unavoidable, even when $M$ is the flat Euclidean space.

Readers familiar with the Holly-Stroock perturbation may find the $\exp\lrp{(q + L_{Ric}/2) \R^2}$ term familiar: Let $M$ have diameter $\R$ and let $U$ have $2q$-Lipschitz gradient,, so that $\beta = \frac{1}{2} \nabla U$ satisfies Assumption \ref{ass:distant-dissipativity} with $q$. Then we can decompose $U = U_1 + U_2$, where $U_1$ is $-L_{\Ric}$-strongly-convex and $U_2$ is a "perturbation" with magnitude $(2q + L_{Ric}) \R^2$. $e^{2 U_1} d vol_g$ satisfies $CD(0,\infty)$ and thus $e^{2U} d vol_g$ has Log-Sobolev constant scaling with $\exp\lrp{(2q + L_{Ric}) \R^2}$. 

\textbf{Finally, on compact manifolds without boundary,} one may take $\R$ to be the diameter of the manifold and upper bound the mixing rate $\alpha$ by $\alpha = \frac{1}{2\R^2} e^{- \frac{1}{2} \lrp{q + L_{\Ric}/2}\R^2}$.



\vspace*{-5pt}
\subsection{Proof Sketch of Theorem \ref{t:langevin_mcmc}}
The proof of Theorem \ref{t:langevin_mcmc} consists of two steps. The first step is bounding the discretization error of \eqref{e:intro_euler_murayama}; we have already done this in Lemma \ref{l:informal_discretization-approximation-lipschitz-derivative} in the previous section. 

The second step is showing that two paths of \eqref{e:intro_sde} converge. To do so, we consider $f\lrp{\dist\lrp{x(t),y(t)}}$, where $f$ is a Lyapunov function taken from \citep{eberle2016reflection} (its exact form is given in Definition \ref{d:f} in the Appendix). Two key property of $f$ are that (i) $\frac{1}{2}\exp\lrp{- \lrp{q + L_{\Ric}/2} \R^2/2} r \leq f(r) \leq r$; and (ii) $f'(r) \leq 1$, so that convergence in $f(\dist\lrp{x(t),y(t)})$ implies convergence in $\dist\lrp{x(t),y(t)}$.

Under the \emph{Kendall-Cranston Coupling} \footnote{This is the manifold analog of reflection coupling of Euclidean Brownian motions; see~\citep[Ch.~6.5]{hsu2002stochastic}} of $x(t)$ and $y(t)$, one can show that $f\lrp{\dist\lrp{x(t),y(t)}}$ contracts with rate $\alpha$. We quantify the contraction rate in Lemma \ref{l:g_contraction_without_gradient_lipschitz} below. We stress that Lemma \ref{l:g_contraction_without_gradient_lipschitz} is not new; it was first presented by \cite{eberle2016reflection} (with minor variation), and its proof combines Theorem 6.6.2 of~\citep{hsu2002stochastic} with the Lyapunov function analysis of~\citep{eberle2016reflection}. For completeness, we provide a proof of Lemma \ref{l:g_contraction_without_gradient_lipschitz} in Appendix~\ref{ss:Evolution_of_Lyapunov_Function_under_Kendall_Cranston_Coupling}.

\begin{lemma}
  \label{l:g_contraction_without_gradient_lipschitz}
  Assume $\beta$ is $(m,q,\R)$-distant dissipative as per Assumption~\ref{ass:distant-dissipativity}, and that is also satisfies Assumption~\ref{ass:beta_lipschitz}. Further assume that $m > L_{\Ric}/2$ and $q + L_{\Ric}/2 \geq 0$. Let $x(t)$ and $y(t)$ denote solutions to \eqref{e:intro_sde}. Then there exists a Lyapunov function $f$ satisfying 1. $f(r) \geq \frac{1}{2}\exp\lrp{- \lrp{q + L_{\Ric}/2} \R^2/2} r$ and 2. $\lrabs{f'(r)} \leq 1$, and a coupling between $x(t)$ and $y(t)$, such that for all time $T$, 
  \begin{alignat*}{1}
    \E{f\lrp{\dist\lrp{x(T),y(T)}}}
    \leq& \exp\lrp{-\alpha T}f\lrp{\dist\lrp{x_0,y_0}},
  \end{alignat*}
  where $\alpha := \min\lrbb{\frac{m-L_{\Ric}/2}{16}, \frac{1}{2 \R^2}} \cdot \exp\lrp{- \frac{1}{2}\lrp{q + L_{\Ric}/2} \R^2}$.
\end{lemma}

\textbf{Given these two steps}, we can now sketch the proof of Theorem \ref{t:langevin_mcmc}. Consider an arbitrary step $k$ of \eqref{e:intro_euler_murayama}. For $t\in[k\delta, (k+1)\delta)$, let $\bar{x}(t)$ denote the solution to \eqref{e:intro_sde}, initialized at $\bar{x}(k\delta)$. Then by Lemma \ref{l:g_contraction_without_gradient_lipschitz}, $\E{f\lrp{\dist\lrp{y((k+1)\delta), \bar{x}((k+1)\delta)}}} \leq e^{-\alpha \delta } \E{f\lrp{\dist\lrp{y(k\delta), x_k}}}$. On the other hand, by Lemma \ref{l:informal_discretization-approximation-lipschitz-derivative} $\E{\dist\lrp{\bar{x}((k+1)\delta), {x}_{k+1}}} \leq O\lrp{\delta^{3/2}}$. Summing these two bounds, applying triangle inequality, and using the fact that $f'(r) \leq 1$, we get
\begin{alignat*}{1}
  \E{f\lrp{\dist\lrp{y((k+1)\delta),x_{k+1}}}} \leq e^{-\alpha \delta } \E{f\lrp{\dist\lrp{y(k\delta), x_k}}} + O\lrp{\delta^{3/2}}
  \elb{e:t:one-step-langevin-mcmc-sketch}
\end{alignat*}
Applying the above recursively for $k=0...K$, we get
\begin{alignat*}{1}
  \E{f(\dist\lrp{x_K, y(K\delta)})} \leq e^{-\alpha K\delta} \E{f(\dist\lrp{x_0, y(0)})} + \frac{1}{\alpha} O\lrp{\delta^{1/2}}
\end{alignat*}
Theorem \ref{t:langevin_mcmc} thus follows from the bounds $\frac{1}{2}\exp\lrp{- \lrp{q + L_{\Ric}/2} \R^2/2} r \leq f(r) \leq r$.

Lastly, we briefly mention a complication in the full proof that we omitted from the above sketch: the discretization error between $\bar{x}((k+1)\delta)$ and $x_{k+1}$ grows with $\lrn{\beta(x_k)}$, which does not have a global upper bound. In practice, we can verify that $\dist\lrp{x_k,x^*}$ is sub-Gaussian, so that with high probability $\dist\lrp{x_k,x^*} \leq O\lrp{\log(k\delta)}$; this in turn allows us to bound $\lrn{\beta(x_k)}$ via Assumption \ref{ass:beta_lipschitz}. This is the reason for the dependency on $\log K$ and $\log(1/\delta)$ in the iteration complexity of Theorem \ref{t:langevin_mcmc}. 

\section{Stochastic Gradient Langevin MCMC}
Finally, we bound the iteration complexity of process~\eqref{e:intro_sgld}, which takes the Euler-Murayama scheme \eqref{e:intro_euler_murayama}, and replaces $\beta(x_k)$ at step $k$ by a stochastic estimate $\t{\beta}_k (x_k)$.
\label{s:SGLD}
\begin{theorem}[Convergence of SGLD on Riemannian Manifold]
  \label{t:SGLD}
  Assume the manifold $M$ satisfies Assumptions~\ref{ass:ricci_curvature_regularity} and~\ref{ass:sectional_curvature_regularity}. Assume in addition that there exists a constant $L_R'$ such that for all $x\in M$, $u,v,w,z,a\in T_x M$, $\lin{(\nabla_a R)(u,v)w,z} \leq L_R'\lrn{u}\lrn{v}\lrn{w}\lrn{z}\lrn{a}$.  Let $\beta$ be a vector field satisfying Assumptions \ref{ass:distant-dissipativity} and \ref{ass:beta_lipschitz}; assume in addition that $m > L_{\Ric}/2$ and that $\R=0$.
  
  Let $y(t)$ denote the exact geometric SDE given in \eqref{e:intro_sde}. Let $x_k$ denote the stochastic Euler Murayama discretization \eqref{e:intro_sgld}, where $\t{\beta}_k$ denote independent random vector fields, with $\E{\t{\beta}(x)} = \beta(x)$ for all $x$ and $\|\beta(x) - \t{\beta}_k(x)\| \leq \sigma$ with probability 1. 
  
  Let $K \in \Z^+$ be some iteration number and $\delta$ be some stepsize. There exists a constant $\C_1 = poly\lrp{L_\beta', d, L_R, L_R', \frac{1}{m}, \log K}$, such that if $\delta \leq \frac{1}{\C_1}$, then there is a coupling between $x_K$ and $y(K\delta)$ satisfying the distance bound
  \begin{alignat*}{1}
      \E{\dist\lrp{y(K\delta),x_{K}}^2} 
      \leq e^{-\frac{1}{8}(m-L_{\Ric}/2) K \delta}\E{\dist\lrp{y(0),x_{0}}^2} + \t{O}\lrp{\delta}.
  \end{alignat*}
  $\t{O}$ hides polynomial dependence on $L_\beta', d, L_R, L_R', \sigma, \frac{1}{m-L_{\Ric}/2}, \log K, \log\frac{1}{\delta}$.
\end{theorem}
We defer the proof of Theorem \ref{t:SGLD} to Appendix \ref{ss:proof_of_t:sgld}, where we state the explicit expressions for $\C_1$.
\subsection*{Discussion of Theorem \ref{t:SGLD}}
\textbf{To sample from $d\pi^*(x) = e^{-h(x)} d{\vol}_g(x)$}, we let $\beta(x) := - \frac{1}{2} \nabla h(x)$ and let $y(0) \sim \pi^*(x)$. The $W_2$ distance between ${\Law}(x_K)$ and $\pi^*$ is upper bounded by $\sqrt{\mathbb{E}[\dist\lrp{y(K\delta), x_K}^2]}$. To achieve $\epsilon$ error in $W_2$, Theorem \ref{t:SGLD} requires $\t{O}\lrp{\epsilon^{-2}}$ steps. The reasoning is very similar to Theorem \ref{t:langevin_mcmc}. 

As already discussed in Section \ref{s:contributions}, \textbf{the Assumption on Theorem \ref{t:SGLD} is considerably more restrictive than Theorem \ref{t:langevin_mcmc}.} We also highlight that the error bound in Theorem \ref{t:SGLD} \emph{does depend on $L_R'$, the derivative of the Riemannain curvature tensor}. This is in contrast to Theorem \ref{t:langevin_mcmc}, where $L_R'$ does not appear in any of the bounds.

Though the iteration complexity for Theorem \ref{t:SGLD} can be larger than Theorem \ref{t:langevin_mcmc} as it depends on additional parameters $\sigma$ and $L_R'$, \textbf{\eqref{e:intro_sgld} may nonetheless be faster than \eqref{e:intro_euler_murayama}}. For example, let $h(x) = \frac{1}{N} \sum_{i=1}^N h_i(x)$ and assume $\lrn{\nabla h_i - \nabla h_j} \leq \sigma$ for all $i,j$, one step of \eqref{e:intro_sgld} can be performed with a single gradient computation for a single uniformly sampled $h_i$, whereas \eqref{e:intro_euler_murayama} would require $N$ gradient computations.

\subsection*{Proof Sketch of Theorem \ref{t:SGLD} and Theoretical Highlights}
For any step $k$, let us define $\bar{y}_{k+1} := \Exp_{y(k\delta)}\lrp{\delta {\beta}(y(k\delta)) + \sqrt{\delta} \bar{\zeta}_k}$, where $\bar{\zeta}_k \sim \N_{y_k}(0, I)$. We show, in Lemma \ref{l:sgld-lemma}, that $\E{\dist\lrp{x_{k+1}, \bar{y}_{k+1}}^2} \leq \lrp{1  - \frac{\delta}{4} (m-L_{\Ric}/2)} \E{\dist\lrp{x_k,y(k\delta)}^2}  + 16 \delta^2 \sigma^2$.
On the other hand, once again applying Lemma \ref{l:informal_discretization-approximation-lipschitz-derivative}, we can verify that $\E{\dist\lrp{y((k+1)\delta), \bar{y}_{k+1}}^2} \leq O\lrp{\delta^3}$. By Young's inequality and triangle Inequality,
\begin{alignat*}{1}
  \E{\dist\lrp{x_{k+1}, y((k+1)\delta)}^2} \leq \lrp{1  - \frac{\delta}{8} (m-L_{\Ric}/2)} \E{\dist\lrp{x_k,y(k\delta)}^2}  + O\lrp{\delta^2}
\end{alignat*}
The bound in Theorem \ref{t:SGLD} follows immediately from applying the above recursively for $k=1...K$. We note that The contraction in Lemma \ref{l:sgld-lemma} is in turn derived from Lemma \ref{l:discrete-approximate-synchronous-coupling-ricci}, which may be of independent interest as it quantifies the distance evolution between two general discrete-time stochastic processes (which do not have to be diffusions or related to the Gaussian noise). 

It is elucidative to compare the proof structure of Theorem \ref{t:SGLD} with that of Theorem \ref{t:langevin_mcmc} above. On a high level, at step $k$, Theorem \ref{t:langevin_mcmc} first approximates the discrete MCMC step ($x_{k+1}$) by an exact SDE initialized at $x_k$ ($\bar{x}((k+1)\delta)$). It then uses two facts: 1. $f(\dist\lrp{\bar{x}(t),y(t)})$ \emph{contracts under the \textbf{exact SDE}}, and 2. the approximation error between $x_{k+1}$ and $\t{x}((k+1)\delta)$ is small. In contrast, at step $k$, Theorem 2 first approximates the \emph{exact SDE} ($y((k+1)\delta)$) by a Euler-Murayama step ($\bar{y}_{k+1}$). It then uses two facts: 1. $\dist\lrp{x_{k+1}, \bar{y}((k+1)\delta)}$ \emph{contracts under the \textbf{stochastic Euler-Murayama step}}, and 2. the approximation error between $\bar{y}_{k+1}$ and $y((k+1)\delta)$ is small. The reason for this change is to make use of the fact that $\E{\t{\beta}_k} = \beta$, \emph{conditioned on the randomness up to time $k\delta$}. Suppose we had followed the proof of Theorem \ref{t:langevin_mcmc} and defined $\bar{x}(t)$ to be the solution to the exact SDE, with drift $\t{\beta}_k$, then showing contraction of the exact SDE becomes very tricky as $\E{\t{\beta}_k}$ \emph{is not equal to $\beta$} when conditioned on $\bar{x}(t)$ for any $t > k\delta$.


\section{Acknowledgements}
Xiang Cheng acknowledge support from NSF BIGDATA grant (1741341) and NSF CCF-2112665 (TILOS AI Research Institute). Jingzhao Zhang acknowledges support by Tsinghua University Initiative Scientific Research Program.

\bibliographystyle{plainnat}
\setlength{\bibsep}{3pt} 
\bibliography{references} 

\begin{thebibliography}{36}
\providecommand{\natexlab}[1]{#1}
\providecommand{\url}[1]{\texttt{#1}}
\expandafter\ifx\csname urlstyle\endcsname\relax
  \providecommand{\doi}[1]{doi: #1}\else
  \providecommand{\doi}{doi: \begingroup \urlstyle{rm}\Url}\fi

\bibitem[Absil and Malick(2012)]{absil2012projection}
P-A Absil and J{\'e}r{\^o}me Malick.
\newblock Projection-like retractions on matrix manifolds.
\newblock \emph{SIAM Journal on Optimization}, 22\penalty0 (1):\penalty0
  135--158, 2012.

\bibitem[Absil et~al.(2009)Absil, Mahony, and Sepulchre]{absil2009optimization}
P-A Absil, Robert Mahony, and Rodolphe Sepulchre.
\newblock \emph{Optimization algorithms on matrix manifolds}.
\newblock Princeton University Press, 2009.

\bibitem[Ahn and Chewi(2021)]{ahn2021efficient}
Kwangjun Ahn and Sinho Chewi.
\newblock Efficient constrained sampling via the mirror-langevin algorithm.
\newblock \emph{Advances in Neural Information Processing Systems}, 34, 2021.

\bibitem[Bac{\'a}k(2014)]{bacak2014convex}
Miroslav Bac{\'a}k.
\newblock \emph{Convex analysis and optimization in {H}adamard spaces}.
\newblock de Gruyter, 2014.

\bibitem[Bakry et~al.(2014)Bakry, Gentil, Ledoux, et~al.]{bakry2014analysis}
Dominique Bakry, Ivan Gentil, Michel Ledoux, et~al.
\newblock \emph{Analysis and geometry of Markov diffusion operators}, volume
  103.
\newblock Springer, 2014.

\bibitem[Bou-Rabee et~al.(2020)Bou-Rabee, Eberle, and Zimmer]{bou2020coupling}
Nawaf Bou-Rabee, Andreas Eberle, and Raphael Zimmer.
\newblock Coupling and convergence for hamiltonian monte carlo.
\newblock \emph{The Annals of applied probability}, 30\penalty0 (3):\penalty0
  1209--1250, 2020.

\bibitem[Boumal(2022)]{boumal2022intro}
Nicolas Boumal.
\newblock An introduction to optimization on smooth manifolds.
\newblock To appear with Cambridge University Press, Jan 2022.
\newblock URL \url{http://www.nicolasboumal.net/book}.

\bibitem[Cheng et~al.(2020)Cheng, Yin, Bartlett, and
  Jordan]{cheng2020stochastic}
Xiang Cheng, Dong Yin, Peter Bartlett, and Michael Jordan.
\newblock Stochastic gradient and {Langevin} processes.
\newblock In \emph{International Conference on Machine Learning}, pages
  1810--1819. PMLR, 2020.

\bibitem[Chewi et~al.(2020)Chewi, Le~Gouic, Lu, Maunu, Rigollet, and
  Stromme]{chewi2020exponential}
Sinho Chewi, Thibaut Le~Gouic, Chen Lu, Tyler Maunu, Philippe Rigollet, and
  Austin Stromme.
\newblock Exponential ergodicity of mirror-langevin diffusions.
\newblock \emph{Advances in Neural Information Processing Systems},
  33:\penalty0 19573--19585, 2020.

\bibitem[Durmus and Moulines(2017)]{durmus2017nonasymptotic}
Alain Durmus and Eric Moulines.
\newblock Nonasymptotic convergence analysis for the unadjusted {Langevin}
  algorithm.
\newblock \emph{The Annals of Applied Probability}, 27\penalty0 (3):\penalty0
  1551--1587, 2017.

\bibitem[Eberle(2016)]{eberle2016reflection}
Andreas Eberle.
\newblock Reflection couplings and contraction rates for diffusions.
\newblock \emph{Probability theory and related fields}, 166\penalty0
  (3):\penalty0 851--886, 2016.

\bibitem[Gangolli(1964)]{gangolli1964construction}
Ramesh Gangolli.
\newblock On the construction of certain diffusions on a differentiable
  manifold.
\newblock \emph{Zeitschrift f{\"u}r Wahrscheinlichkeitstheorie und Verwandte
  Gebiete}, 2\penalty0 (5):\penalty0 406--419, 1964.

\bibitem[Gatmiry and Vempala(2022)]{gatmiry2022convergence}
Khashayar Gatmiry and Santosh~S Vempala.
\newblock Convergence of the riemannian langevin algorithm.
\newblock \emph{arXiv preprint arXiv:2204.10818}, 2022.

\bibitem[Girolami and Calderhead(2011)]{girolami2011riemann}
Mark Girolami and Ben Calderhead.
\newblock Riemann manifold {L}angevin and {H}amiltonian {M}onte {C}arlo
  methods.
\newblock \emph{Journal of the Royal Statistical Society: Series B (Statistical
  Methodology)}, 73\penalty0 (2):\penalty0 123--214, 2011.

\bibitem[Gorham et~al.(2019)Gorham, Duncan, Vollmer, Mackey,
  et~al.]{gorham2019measuring}
Jackson Gorham, Andrew~B Duncan, Sebastian~J Vollmer, Lester Mackey, et~al.
\newblock Measuring sample quality with diffusions.
\newblock \emph{Annals of Applied Probability}, 29\penalty0 (5):\penalty0
  2884--2928, 2019.

\bibitem[Hsu(2002)]{hsu2002stochastic}
Elton~P Hsu.
\newblock \emph{Stochastic analysis on manifolds}.
\newblock Number~38. American Mathematical Soc., 2002.

\bibitem[J{\o}rgensen(1975)]{jorgensen1975central}
Erik J{\o}rgensen.
\newblock The central limit problem for geodesic random walks.
\newblock \emph{Zeitschrift f{\"u}r Wahrscheinlichkeitstheorie und verwandte
  Gebiete}, 32\penalty0 (1):\penalty0 1--64, 1975.

\bibitem[Jost(2008)]{jost2008riemannian}
J{\"u}rgen Jost.
\newblock \emph{Riemannian geometry and geometric analysis}.
\newblock Springer, seventh edition, 2008.

\bibitem[Karcher(1977)]{karcher1977riemannian}
Hermann Karcher.
\newblock Riemannian center of mass and mollifier smoothing.
\newblock \emph{Communications on pure and applied mathematics}, 30\penalty0
  (5):\penalty0 509--541, 1977.

\bibitem[Lee(2006)]{lee2006riemannian}
John~M Lee.
\newblock \emph{Riemannian manifolds: an introduction to curvature}, volume
  176.
\newblock Springer Science \& Business Media, 2006.

\bibitem[Lee and Vempala(2017)]{lee2017geodesic}
Yin~Tat Lee and Santosh~S Vempala.
\newblock Geodesic walks in polytopes.
\newblock In \emph{Proceedings of the 49th Annual ACM SIGACT Symposium on
  theory of Computing}, pages 927--940, 2017.

\bibitem[Lee and Vempala(2018)]{lee2018convergence}
Yin~Tat Lee and Santosh~S Vempala.
\newblock Convergence rate of {Riemannian Hamiltonian Monte Carlo} and faster
  polytope volume computation.
\newblock In \emph{Proceedings of the 50th Annual ACM SIGACT Symposium on
  Theory of Computing}, pages 1115--1121, 2018.

\bibitem[Li and Erdogdu(2020)]{li2020riemannian}
Mufan~Bill Li and Murat~A Erdogdu.
\newblock Riemannian {L}angevin algorithm for solving semidefinite programs.
\newblock \emph{arXiv preprint arXiv:2010.11176}, 2020.

\bibitem[Li et~al.(2022)Li, Tao, Vempala, and Wibisono]{li2022mirror}
Ruilin Li, Molei Tao, Santosh~S Vempala, and Andre Wibisono.
\newblock The mirror langevin algorithm converges with vanishing bias.
\newblock In \emph{International Conference on Algorithmic Learning Theory},
  pages 718--742. PMLR, 2022.

\bibitem[Mangoubi and Smith(2018)]{mangoubi2018rapid}
Oren Mangoubi and Aaron Smith.
\newblock Rapid mixing of geodesic walks on manifolds with positive curvature.
\newblock \emph{The Annals of Applied Probability}, 28\penalty0 (4):\penalty0
  2501--2543, 2018.

\bibitem[Moitra and Risteski(2020)]{moitra2020fast}
Ankur Moitra and Andrej Risteski.
\newblock Fast convergence for {Langevin} diffusion with manifold structure.
\newblock \emph{arXiv preprint arXiv:2002.05576}, 2020.

\bibitem[Muniz et~al.(2021)Muniz, Ehrhardt, G{\"u}nther, and
  Winkler]{muniz2021higher}
Michelle Muniz, Matthias Ehrhardt, Michael G{\"u}nther, and Renate Winkler.
\newblock Higher strong order methods for {It\^{o}} {SDE}s on matrix {L}ie
  groups.
\newblock \emph{arXiv preprint arXiv:2102.04131}, 2021.

\bibitem[Patterson and Teh(2013)]{patterson2013stochastic}
Sam Patterson and Yee~Whye Teh.
\newblock Stochastic gradient {R}iemannian {L}angevin dynamics on the
  probability simplex.
\newblock In \emph{NIPS}, pages 3102--3110, 2013.

\bibitem[Petersen(2006)]{petersen2006riemannian}
Peter Petersen.
\newblock \emph{Riemannian geometry}, volume 171.
\newblock Springer Science \& Business Media, 2006.

\bibitem[Piggott and Solo(2016)]{piggott2016geometric}
Marc~J Piggott and Victor Solo.
\newblock Geometric {Euler-Maruyama} schemes for stochastic differential
  equations in $\mathrm{SO}(n)$ and $\mathrm{SE}(n)$.
\newblock \emph{SIAM Journal on Numerical Analysis}, 54\penalty0 (4):\penalty0
  2490--2516, 2016.

\bibitem[Sun et~al.(2019)Sun, Flammarion, and Fazel]{sun2019escaping}
Yue Sun, Nicolas Flammarion, and Maryam Fazel.
\newblock Escaping from saddle points on {Riemannian} manifolds.
\newblock \emph{arXiv preprint arXiv:1906.07355}, 2019.

\bibitem[Udriste(2013)]{udriste2013convex}
Constantin Udriste.
\newblock \emph{Convex functions and optimization methods on {R}iemannian
  manifolds}, volume 297.
\newblock Springer Science \& Business Media, 2013.

\bibitem[Wang et~al.(2020)Wang, Lei, and Panageas]{wang2020fast}
Xiao Wang, Qi~Lei, and Ioannis Panageas.
\newblock Fast convergence of langevin dynamics on manifold: Geodesics meet
  log-sobolev.
\newblock \emph{Advances in Neural Information Processing Systems},
  33:\penalty0 18894--18904, 2020.

\bibitem[Welling and Teh(2011)]{welling2011bayesian}
Max Welling and Yee~W Teh.
\newblock Bayesian learning via stochastic gradient langevin dynamics.
\newblock In \emph{Proceedings of the 28th international conference on machine
  learning (ICML-11)}, pages 681--688. Citeseer, 2011.

\bibitem[Zhang and Sra(2016)]{zhang2016first}
Hongyi Zhang and Suvrit Sra.
\newblock First-order methods for geodesically convex optimization.
\newblock In \emph{Conference on Learning Theory}, pages 1617--1638. PMLR,
  2016.

\bibitem[Zhang et~al.(2020)Zhang, Peyr{\'e}, Fadili, and
  Pereyra]{zhang2020wasserstein}
Kelvin~Shuangjian Zhang, Gabriel Peyr{\'e}, Jalal Fadili, and Marcelo Pereyra.
\newblock Wasserstein control of mirror langevin monte carlo.
\newblock In \emph{Conference on Learning Theory}, pages 3814--3841. PMLR,
  2020.

\end{thebibliography}

\newpage

\newpage
\appendix

\appendix
 \begin{center}
   \Large\bfseries Contents of the Appendices
 \end{center}

\startcontents[sections]
\printcontents[sections]{l}{1}{\setcounter{tocdepth}{2}}
\newpage

\renewcommand*\E[1]{\mathbb{E}\left[{#1}\right]}
\section{Manifold SDE}
\label{appendix:manifold_sde}
An outline of this section is as follows:
\begin{enumerate}
    \item In Section \ref{ss:sde_construction_appendix}, we prove Lemma \ref{l:Phi_is_diffusion}, which guarantees that $x^i(t)$ defined in \eqref{d:x^i(t):0-appendix} has a limit $x(t)$ that equals the solution of the exact Langevin Diffusion in \eqref{e:intro_sde}.
    \item In Section \ref{ss:langevin_mcmc_on_manifold}, we prove Lemma \ref{l:discretization-approximation-lipschitz-derivative}, which bounds the distance between $x^0(t)$ from \eqref{d:x^i(t):0-appendix} and the limit $x(t)$. This is equivalent to bounding the distance between the Euler Murayama discretization \eqref{e:intro_euler_murayama} and the exact Langevin Diffusion \eqref{e:intro_sde}.
    \item In Section \ref{ss:proof_of_t:langevin_mcmc}, we prove Theorem \ref{t:langevin_mcmc}. 
    \item In Section \ref{ss:proof_of_t:sgld}, we prove Theorem \ref{t:SGLD}. 
\end{enumerate}

We also list below the key lemmas which are used to prove the results above.
\begin{enumerate}
    \item Theorem \ref{t:langevin_mcmc} relies on Lemma \ref{l:g_contraction_without_gradient_lipschitz} (contraction of Lyapunov function under exact SDE) and Lemma \ref{l:discretization-approximation-lipschitz-derivative} (bound on Euler Murayama discretization error).
    \begin{enumerate}
        \item Lemma \ref{l:discretization-approximation-lipschitz-derivative} essentially sums the bound from Lemma \ref{l:key_brownian_limit_lemma}. 
        \item Lemma \ref{l:key_brownian_limit_lemma} relies on Lemma \ref{l:discrete-approximate-synchronous-coupling} and Lemma \ref{l:triangle_distortion}.
        \item Lemma \ref{l:g_contraction_without_gradient_lipschitz} relies on Lemma \ref{l:discrete-approximate-synchronous-coupling-ricci}
    \end{enumerate}
    \item Theorem \ref{t:SGLD} relies on Lemma \ref{l:sgld-lemma} (contraction of Lyapunov function under stochastic gradient Euler-Murayama step) and Lemma \ref{l:discretization-approximation-lipschitz-derivative} (bound on Euler Murayama discretization error).
    \begin{enumerate}
        \item Lemma \ref{l:sgld-lemma} relies on Lemma \ref{l:discrete-approximate-synchronous-coupling-ricci}.
    \end{enumerate}
\end{enumerate}

\subsection{SDE Construction}
\label{ss:sde_construction_appendix}
In this section, we state and prove key lemmas related to our construction in Section \ref{ss:discrete_gaussian_walk_construction}, which we reproduce below for ease of reference:
Let $x_0 \in M$ be an initial point and $E = \lrbb{E^1,\ldots,E^d}$ be an orthonormal basis of $T_{x^0}$. Let $\BB(t)$ denote a standard Brownian Motion in $\Re^d$. Let $T\in \Re^+$. Define
\begin{alignat*}{1}
    & x^0_0 = x_0, \qquad E^{0}_0 = E,\\
    & x^0_1 = \Exp_{x^0_0}(T \beta(x^0_0) + {\lrp{\BB\lrp{T} - \BB\lrp{0}}} \circ E^{0}_0).
\end{alignat*}
For any $i\in \Z^+$, let $\delta^i := 2^{-i}T$. We will now define points $x^i_k \in M$ and orthonormal basis $E^{i}_k$ of $T_{x^i_k}$ for all $i$ and all $k\in \{0,\ldots,\nicefrac{T}{\delta^i}\}$. Our construction is inductive: Suppose we have already defined $x^i_k$ and $E^{i}_k$ for some $i$ and for all $k\in \{0,\ldots,\nicefrac{T}{\delta^i}\}$. Then, we construct $x^{i+1}_k$, for all $k=\{0,\ldots,\nicefrac{T}{\delta^{i+1}}\}$, as follows: 
\begin{alignat*}{1}
    & x^{i+1}_0 := x_0, \qquad E^{i+1}_{0} := E,\\
    & x^{i+1}_{2k+1} := \Exp_{x^{i+1}_{2k}}\lrp{\delta^{i+1}\beta\lrp{x^{i+1}_{2k}} + {\lrp{\BB\lrp{\lrp{2k+1}\delta^{i+1}} - \BB\lrp{2k\delta^{i+1}}}} \circ E^{i+1}_{2k}},\\
    & E^{i+1}_{2k+1} := \party{}{}\lrp{E^{i+1}_{2k}; {x^{i+1}_{2k}} \to {x^{i+1}_{2k+1}} },\\
    & x^{i+1}_{2k+2} := \Exp_{x^{i+1}_{2k+1}}\lrp{\delta^{i+1}\beta\lrp{x^{i+1}_{2k+1}} + {\lrp{\BB\lrp{\lrp{2k+2}\delta^{i+1}} - \BB\lrp{\lrp{2k+1}\delta^{i+1}}}} \circ E^{i+1}_{2k+1}},\\
    & E^{i+1}_{2k+2} := \party{}{} \lrp{E^{i}_{k+1}; {x^{i}_{k+1}}\to{x^{i+1}_{2k+2}}}
    \elb{d:x^i_k-appendix}.
\end{alignat*}

The above display defines points $x^{i+1}_k$ for all $k = \lrbb{0,\ldots,\nicefrac{T}{\delta^{i+1}}}$. For any $i$, any $k$, and any $t\in [k\delta^i, (k+1)\delta^i)$, we define $x^i(t)$ to be the ``linear interpolation'' of $x^i_k$ and $x^i_{k+1}$, i.e.,
\begin{alignat*}{1}
  x^{i}(t) := \Exp_{x^{i}_{k}}\bigl(\tfrac{t-k\delta^i}{\delta^i}\lrp{\delta^i \beta\lrp{x^i_k} + \lrp{\BB\lrp{(k+1)\delta^i} - \BB\lrp{k\delta^{i}}} \circ E^{i}_{k}}\bigr).
  \elb{d:x^i(t):0-appendix}
\end{alignat*}

Let us define two convenient notation that we will use throughout our proofs in this Appendix. First, let
\begin{alignat*}{1}
  \overline{\Phi}(t; x, E, \beta, \BB, i)
  \elb{d:x^i(t)}
\end{alignat*}
denote the solution to the interpolated process in \eqref{d:x^i(t):0} (reproduced in \eqref{d:x^i(t):0-appendix}), initialized at $x_0 = x$, i.e. $\overline{\Phi}(t; x_0, E, \beta, \BB, i) = x^i(t)$ as defined in \eqref{d:x^i(t):0-appendix}. (This notation becomes convenient later on when we need to refer to \eqref{d:x^i(t):0} but with different initial points, or with drift vector fields other than $\beta$, or with specific choices of $\BB$.)

We also let
\begin{alignat*}{1}
  \Phi(t; x, E, \beta, \BB) := \lim_{i \to \infty} \overline{\Phi}(t; x, E, \beta, \BB, i)
  \elb{d:x(t)}.
\end{alignat*}
Below, we prove Lemma \ref{l:Phi_is_diffusion} (stated at the end of Section \ref{Brownian Motion Section}), which guarantees that $x(t) = \Phi(t; x, E, \beta, \BB)$ exists, and that $x(t)$ is a solution to the exact Langevin diffusion SDE in \eqref{e:intro_sde}.

\begin{proof}[Proof of Lemma \ref{l:Phi_is_diffusion}]
    The existence of the almost-sure, uniform limit $x(t)$ is proven in Lemma~\ref{l:x(t)_is_brownian_motion}. For the rest of this proof, we verify that $x(t)$ has the generator $L f = \lin{\nabla f, \beta} + \frac{1}{2} \Delta(f)$, where $\Delta$ denotes the Laplace Beltrami operator. 
    By Proposition~3.2.1 of~\citep{hsu2002stochastic}, this implies that $x(t)$ is the solution to~\eqref{e:intro_sde}.
    

    Let $\F_t$ denote the sigma field generated by $\BB(s) : s\in[0,t]$.

    Consider any $f: M \to \Re$ with $\lrn{f'}\leq C$, $\lrn{f''}\leq C$, $\lrn{f'''}\leq C$ globally. Let $x^i(t)$ be as defined in \eqref{d:x^i(t):0}. We will verify that $f(x(t)) - f(x(0)) - \int_0^t L f(x(t)) dt$ is a martingale.
  
    To begin, let $s,t\in [0,T]$ be such that $s = j \delta^a$ and $t = j' \delta^a$ for some positive integers $j\leq j'$, and $a$, (recall that $\delta^i = T/2^i$). We will show that conditioned on $x(s)$, $f(x(t)) - f(x(s)) - \int_s^t L f(x(t)) dt$ is a martingale. Let us define
    \begin{alignat*}{1}
      u^i_k = \delta^i \beta(x^i_k) + \lrp{\BB((k+1)\delta^i) -\BB(k\delta^i)} \circ E^i_k
    \end{alignat*}
    so that $x^i(t) = \Exp_{x^i_k}\lrp{\frac{t-k\delta^i}{\delta^i}u^i_k}$, where $t\in[k\delta^i,(k+1)\delta^i]$.
    
    Consider an arbitrary $\ell \geq a$. Consider the sum
    \begin{alignat*}{1}
      \sum_{k=s/\delta^i}^{t/\delta^i-1} f(x^{\ell}((k+1)\delta^\ell))
      - f(x^{\ell}(k\delta^\ell)) - \lin{u^\ell_k, \nabla f(x^{\ell}(k\delta^\ell))} - \nabla^2 f(x^\ell_k)[u^\ell_k,u^\ell_k].
      \elb{e:t:alkdalksm:1}
    \end{alignat*}
    By Taylor's theorem, 
    \begin{alignat*}{1}
      & \lrabs{f(x^{\ell}((k+1)\delta^\ell))
      - f(x^{\ell}(k\delta^\ell)) - \lin{u^\ell_k, \nabla f(x^{\ell}(k\delta^\ell))} - \nabla^2 f(x^\ell_k)[u^\ell_k,u^\ell_k]}\\
      \leq& C\lrn{u^\ell_k}^3 \\
      \leq& 8 C^3{\delta^\ell}^3 \lrn{\beta(x^\ell_k)}^3 + 8C^3 \lrn{\BB((k+1)\delta^\ell)-\BB(k\delta^\ell)}_2^3,
    \end{alignat*}
    where $C\lrn{u^\ell_k}^3$ captures the third-and-higher order Taylor terms.

    The first order Taylor term can be decomposed as
    \begin{alignat*}{1}
      \lin{u^\ell_k, \nabla f(x^{\ell}(k\delta^\ell))}
      =& \delta^{\ell} \underbrace{\lin{\beta(x^\ell), \nabla f(x^{\ell}(k\delta^\ell))}}_{\lrn{\cdot}^2 \leq {\delta^{\ell}}^2 C^2 \lrn{\beta(x^\ell(k\delta^{\ell}))}} + \underbrace{\lin{\nabla f(x^{\ell}(k\delta^\ell)), \lrp{\BB((k+1)\delta^\ell)-\BB(k\delta^\ell)}\circ E^\ell_k}}_{\Ep{\F_{k\delta^\ell}}{\cdot}=0}.
    \end{alignat*}
    We now simplify the second order Taylor term. Let $v:= \lrp{\BB((k+1)\delta^i) -\BB(k\delta^i)} \circ E^i_k$. We verify that  $\Ep{\F_     {k\delta^\ell}                                                                                                                                       }{\nabla^2 f(x^{\ell}(k\delta^\ell)) [v,v] - \delta^\ell \Delta f(x^{\ell}(k\delta^\ell))}=0$, because $\lrp{\BB((k+1)\delta^i) -\BB(k\delta^i)} \circ E^i_k$ has identity covariance, and the Laplace Beltrami operator is the trace of the Hessian. We can also bound, using Young;s inequality,
    \begin{alignat*}{1}
      & \E{\lrabs{\nabla^2 f(x^{\ell}(k\delta^\ell)) [u^\ell_k,u^\ell_k] - \delta^\ell \Delta f(x^{\ell}(k\delta^\ell))}^2} \\
      \leq& {\delta^\ell}^4 C^2 \lrn{\beta(x^\ell(k\delta^\ell))}^4 + {\delta^\ell}^2 \lrn{\BB((k+1)\delta^\ell)-\BB(k\delta^\ell)}_2^4 + {\delta^{\ell}}^2 C^2 d
    \end{alignat*}
  
    Finally, note that there exists a constant $C'$, which depends on $T,d,L_\beta'\lrn{\beta(x_0)}$, such that for all $\ell$, for all $t\in [0,T]$, $\E{\lrn{\beta(x^\ell(t)}^6} \leq C'$. The proof is similar to Lemma \ref{l:near_tail_bound_L4} and we omit it here.
  
    Plugging into \eqref{e:t:alkdalksm:1} and taking expectation conditioned on the Brownian motion $\BB(t) : t\in[0,s]$, we get that
    \begin{alignat*}{1}
      &\Ep{\F_s}{\lrabs{f(x^\ell(t)) - f(x^\ell(s)) + \sum_{k=s/\delta^\ell}^{t/\delta^\ell-1} - \delta^{\ell} \lin{\beta(x^{\ell}(k\delta^\ell)), \nabla f(x^{\ell}(k\delta^\ell))} - \frac{\delta^{\ell}}{2} \Delta f(x^{\ell}(k\delta^\ell))}^2} \\
      =& \Ep{\F_s}{\lrabs{\sum_{k=s/\delta^\ell}^{t/\delta^\ell-1} f(x^{\ell}((k+1)\delta^\ell)) - f(x^{\ell}(k\delta^\ell)) - \delta^{\ell} \lin{\beta(x^{\ell}(k\delta^\ell)), \nabla f(x^{\ell}(k\delta^\ell))} - \frac{\delta^{\ell}}{2} \Delta f(x^{\ell}(k\delta^\ell))}^2} \\
      \leq& \text{poly}(C,C',d) \sum_{k=s/\delta^\ell}^{t/\delta^\ell-1} {\delta^{\ell}}^{2}\\
      \leq& \text{poly}(C,C',d,T) {\delta^{\ell}}.
    \end{alignat*}
    where the first line is because $f(x^\ell(t)) - f(x^\ell(s)) = \sum_{k=s/\delta^\ell}^{t/\delta^\ell-1} f(x^{\ell}((k+1)\delta^\ell)) - f(x^{\ell}(k\delta^\ell))$, noting that $t,s$ are multiples of $\delta^\ell$ by definition, and the second line uses our first and second Taylor approximation bounds above. 

    Next, define $g^\ell_k := \delta^{\ell} \lin{\beta(x^{\ell}(k\delta^\ell)), \nabla f(x^{\ell}(k\delta^\ell))} + \frac{\delta^{\ell}}{2} \Delta f(x^{\ell}(k\delta^\ell))$. Using the smoothness of $\beta$ and $f$, and the fact that $\dist\lrp{x^\ell(t),x(k\delta^\ell)} \leq \delta^{\ell}\lrn{\beta(x(k\delta^\ell))} + \lrn{\BB((k+1)\delta^\ell)-\BB(k\delta^\ell)}_2$, we verify that for any $k$,
    \begin{alignat*}{1}
      \Ep{\F_s}{\lrabs{g^\ell_k - \int_{k\delta^\ell}^{(k+1)\delta^\ell} \lin{\beta(x(r)), \nabla f(x(r))} + \frac{1}{2} \Delta f(x(r)) dr}} \leq \text{poly} \lrp{C,C',d} {\delta^{\ell}}^{3/2}
    \end{alignat*}
    Putting everything together, we obtain the bound
    \begin{alignat*}{1}
      \Ep{\F_s}{\lrabs{f(x^\ell(t)) - f(x^\ell(s)) + \int_s^t - \lin{\beta(x^{\ell}(r)), \nabla f(x^{\ell}(r))} - \frac{1}{2} \Delta f(x^{\ell}(r)) dr}} \leq  \text{poly}(C,C',d,T) {\delta^{\ell}}^{1/2}.
    \end{alignat*}
    By Lemma \ref{l:x(t)_is_brownian_motion}, $\sup_{t\in[0,T]} \dist\lrp{x^\ell(t), x(t)}$ converges to $0$ almost surely as $\ell \to \infty$. By Dominated Convergence Theorem, and by smoothness of $f$ and $\beta$, 
    
    \begin{alignat*}{1}
        & \Ep{\F_s}{\lrabs{f(x(t)) - f(x(s)) + \int_s^t - \lin{\beta(x(r)), \nabla f(x(r))} - \frac{1}{2} \Delta f(x(r)) dr}}\\
        = & \lim_{\ell \to \infty} \Ep{\F_s}{\lrabs{f(x^\ell(t)) - f(x^\ell(s)) + \int_s^t - \lin{\beta(x^{\ell}(r)), \nabla f(x^{\ell}(r))} - \frac{1}{2} \Delta f(x^{\ell}(r)) dr}}\\
        = & 0.
    \end{alignat*}
    The last equality holds because $\lim_{\ell\to \infty} \text{poly}(C,C',d,T) {\delta^{\ell}}^{1/2} = 0$.
  
    Recall that we assumed that $s$ and $t$ are integral multiples of $T/2^a$ for some positive integer $a$. To extend to general $s,t$, we note that the set of dyadic points (i.e. multiples of $T/2^a$, for some integer $a$) is uniformly dense on the real line. 
    
  \end{proof}

\subsection{Existence of Limit}
\label{ss:Existence}
We present below Lemma \ref{l:key_brownian_limit_lemma}, which bounds the distance between two adjacent trajectories $x^i(t)$ and $x^{i+1}(t)$ as defined in \eqref{d:x^i(t):0} (or equivalently \eqref{d:x^i(t):0-appendix}). The proof of Lemma \ref{l:key_brownian_limit_lemma} works by combining Lemma \ref{l:discrete-approximate-synchronous-coupling} (which bounds distance evolution under "synchronous coupling"), and Lemma \ref{l:triangle_distortion} (which bounds distance evolution under "rolling without slipping"; Lemma \ref{l:triangle_distortion} is taken from \citep{sun2019escaping}). The proof of Lemma \ref{l:key_brownian_limit_lemma} corresponds to Step 1 and Step 2 of the proof sketch of Lemma \ref{l:informal_discretization-approximation-lipschitz-derivative} in Section \ref{Brownian Motion Section}.

Lemma \ref{l:key_brownian_limit_lemma} plays a key role bounding the Euler Murayama discretization error in Lemma \ref{l:discretization-approximation-lipschitz-derivative} in Section \ref{ss:langevin_mcmc_on_manifold}. The proof of Lemma \ref{l:discretization-approximation-lipschitz-derivative} essentially involves summing the bound from Lemma \ref{l:key_brownian_limit_lemma}, for $i=0...\infty$.

Another application of Lemma \ref{l:key_brownian_limit_lemma} is to verify the existence of $x(t) = \lim_{i\to\infty} x^i(t)$ as defined in \eqref{d:x(t)} in Lemma \ref{l:x(t)_is_brownian_motion}. 

\begin{lemma}\label{l:key_brownian_limit_lemma}
  Let $T$ be any positive constant. Let $x^i(t)$ be the (interpolation) of the Euler Murayama discretization with stepsize $\delta^i = T/2^i$ as defined in \eqref{d:x^i(t):0} (or equivalently \eqref{d:x^i(t):0-appendix}). Let $K:= 2^i$ so that $T = K\delta^i$.
  
  Assume that there is are constants $L_\beta, L_\beta'$ such that for all $x,y\in M$, $\lrn{\beta(x)} \leq L_\beta$ and \\
  $\lrn{\beta(x) - \party{y}{x} \beta(y)} \leq L_\beta' \lrn{x-y}$. Then
  \begin{alignat*}{1}
    & \E{\sup_{t\in [0,K\delta^i]} \dist\lrp{x^i(t), x^{i+1}(t)}^2}\\
        \leq& 2^{10} \cdot e^{40K{\delta^i}^2 L_R L_\beta^2 + 2K\delta^i L_R d + K\delta^i L_\beta'} \lrp{K {\delta^i}}^2 \lrp{{\delta^i}^4 L_R^2 L_\beta^6 + {\delta^i} L_R^2 d^3 + {\delta^i}^2 {L_\beta'}^2 L_\beta^2 + {\delta^i} {L_\beta'}^2d},
  \end{alignat*}
  and
  \begin{alignat*}{1}
      & \Pr{\sup_{t \in [0,T]} \dist\lrp{x^i(t), x^{i+1}(t)} \geq 2^{- \frac{i}{4} - 2}} \\
      & \leq e^{\lrp{2^{6-i}T {L_R} L_\beta^2 + 2 L_R d + L_\beta'}T}  T^2 \cdot \lrp{{\delta^i}^3 L_R^2 L_\beta^6 + L_R^2 d^3 + {\delta^i} {L_\beta'}^2 L_\beta^2 +{L_\beta'}^2d} \cdot 2^{-i/2 + 14}.
  \end{alignat*}
\end{lemma}

\textbf{Remark:} Lemma \ref{l:key_brownian_limit_lemma} is usually applied with $T$ being the step-size of a a single Euler-Murayama discretization step (i.e. $\delta$ in \eqref{e:intro_euler_murayama}). Therefore, by taking $T$ to be sufficiently small, the exponential term can be made small, e.g. $\leq 2$. 

\begin{proof}
    Recall that $x^i(t)$ is the linear interpolation of $x^i_k$ as defined in \eqref{d:x^i_k}. 

  Let us define
  \begin{alignat*}{1}
      & a_k := \delta^{i+1} \beta\lrp{x^{i+1}_{2k}} + {\lrp{\BB\lrp{\lrp{2k+1}\delta^{i+1}} - \BB\lrp{2k\delta^{i+1}}}} \circ E^{i+1}_{2k}\\
      & b_k := \delta^{i+1} \beta\lrp{x^{i+1}_{2k}} + {\lrp{\BB\lrp{\lrp{2k+2}\delta^{i+1}} - \BB\lrp{\lrp{2k+1}\delta^{i+1}}}} \circ E^{i+1}_{2k}
      \elb{e:t:rjlqkwn:0}
  \end{alignat*}

  Our proof breaks down the bound of $\dist\lrp{x^i_k,x^{i+1}_{2k+2}}$ into two parts: by Young's inequality,
  \begin{alignat*}{1}
      \dist\lrp{x^i_{k+1}, x^{i+1}_{2k+2}}^2
      \leq& \lrp{\dist\lrp{x^i_{k+1}, \Exp_{x^{i+1}_{2k}}\lrp{a_k + b_k}} + \dist\lrp{\Exp_{x^{i+1}_{2k}}\lrp{a_k + b_k},x^{i+1}_{2k+2}}}^2\\
      \leq& \lrp{1 + \frac{1}{2K}}\dist\lrp{x^i_{k+1}, \Exp_{x^{i+1}_{2k}}\lrp{a_k + b_k}}^2 + K\dist\lrp{\Exp_{x^{i+1}_{2k}}\lrp{a_k + b_k},x^{i+1}_{2k+2}}^2
      \elb{e:t:rjlqkwn:1}
  \end{alignat*}
  We now bound the first term of \eqref{e:t:rjlqkwn:1}. From definition in \eqref{d:x^i_k} and \eqref{e:t:rjlqkwn:0},
  \begin{alignat*}{1}
      x^i_{k+1} =& \Exp_{x^i_k}\lrp{\delta^i \beta\lrp{x^i_k} + {\lrp{\BB\lrp{\lrp{k+1}\delta^{i}} - \BB\lrp{k\delta^{i}}}} \circ E^{i}_{k}}\\
      \Exp_{x^{i+1}_{2k}}\lrp{a_k + b_k} 
      =& \Exp_{x^i_k}\lrp{\delta^i \beta\lrp{x^{i+1}_{2k}} + {\lrp{\BB\lrp{\lrp{2k+2}\delta^{i+1}} - \BB\lrp{2k\delta^{i+1}}}} \circ E^{i+1}_{2k}}
  \end{alignat*}

  We thus apply Lemma \ref{l:discrete-approximate-synchronous-coupling}, with $x := x^i_k$, $y := x^{i+1}_{2k}$, $u := \delta^{i} \beta\lrp{x^i_k} + {\lrp{\BB\lrp{\lrp{k+1}\delta^{i}} - \BB\lrp{k\delta^{i}}}} \circ E^{i}_{k}$, $v := \delta^{i} \beta\lrp{x^{i+1}_{2k}} + {\lrp{\BB\lrp{\lrp{2k+2}\delta^{i+1}} - \BB\lrp{2k\delta^{i+1}}}} \circ E^{i+1}_{k}$. Let $\gamma(t), u(t), v(t)$ be as defined in Lemma \ref{l:discrete-approximate-synchronous-coupling}. Then Lemma \ref{l:discrete-approximate-synchronous-coupling} bounds
  \begin{alignat*}{1}
      \dist\lrp{\Exp_{x}(u), \Exp_y(v)}^2 \leq& \lrp{1+ 4 \C_k^2 e^{4\C_k}} \dist\lrp{x,y}^2 + 32 e^{\C_k} \lrn{v(0) - u(0)}^2 + 2\lin{\gamma'(0), v(0) - u(0)} 
      \elb{e:t:pmcqwd:1}
  \end{alignat*}
  where $\C_k := \sqrt{L_R} \lrp{\lrn{u} + \lrn{v}} \leq 2\sqrt{L_R}\lrp{\delta^i L_\beta + \lrn{\BB((k+1)\delta^i) - \BB(k\delta^i)}_2}$. 
  
  Some of the terms above can be simplified. We begin by bounding the $\lrn{u(0) - v(0)}$ term. By assumption that $\beta$ is Lipschitz, $\lrn{\delta^i \beta(x^i_k) - \party{x^{i+1}_{2k}}{x^i_k}\delta^i \beta(x^{i+1}_{2k})} \leq \delta^i L_\beta' \dist\lrp{x^i_k, x^{i+1}_{2k}}$. By definition of $E^{i+1}_{2k}$ from \eqref{d:x^i_k}, 
  \begin{alignat*}{1}
      & \party{x^{i+1}_{2k}}{x^i_k} \lrp{{\lrp{\BB\lrp{\lrp{2k+2}\delta^{i+1}} - \BB\lrp{2k\delta^{i+1}}}} \circ E^{i+1}_{2k}}\\
      =& {\lrp{\BB\lrp{\lrp{2k+2}\delta^{i+1}} - \BB\lrp{2k\delta^{i+1}}}} \circ \lrp{\party{x^{i+1}_{2k}}{x^i_k}E^{i+1}_{2k}}\\
      =& {\lrp{\BB\lrp{\lrp{k+1}\delta^{i}} - \BB\lrp{k\delta^{i}}}} \circ E^{i}_{k}
  \end{alignat*}
  where the last line is because $\delta^i = 2 \delta^{i+1}$ and because $E^{i+1}_{2k} := \party{x^{i}_{k}}{x^{i+1}_{2k}} \lrp{E^{i}_{k}}$ from \eqref{d:x^i_k}.

  Thus
  \begin{alignat*}{1}
      & {\lrp{\BB\lrp{\lrp{k+1}\delta^{i}} - \BB\lrp{k\delta^{i}}}} \circ E^{i}_{k} - \party{x^{i+1}_{2k}}{x^i_k} \lrp{{\lrp{\BB\lrp{\lrp{2k+2}\delta^{i+1}} - \BB\lrp{2k\delta^{i+1}}}} \circ E^{i+1}_{2k}}\\
      =& {\lrp{\BB\lrp{\lrp{k+1}\delta^{i}} - \BB\lrp{k\delta^{i}}}} \circ E^{i}_{k} -{\lrp{\BB\lrp{\lrp{k+1}\delta^{i}} - \BB\lrp{k\delta^{i}}}} \circ \party{x^{i+1}_{2k}}{x^i_k}E^{i}_{k}\\
      =& 0
  \end{alignat*}

  We can thus bound via Young's Inequality:
  \begin{alignat*}{1}
      \lrn{u(0) - v(0)}^2 \leq 2 {\delta^i}^2 {L_\beta'}^2 \dist\lrp{x^i_k,x^{i+1}_{2k}}^2
  \end{alignat*}
  Finally, noting that $\lrn{\gamma'(0)}=i\dist\lrp{x^i_k,x^{i+1}_{2k}}$,
  \begin{alignat*}{1}
      & 2\lin{\gamma'(0), v(0) - u(0)} 
      = 2\lin{\gamma'(0), \party{x^{i+1}_{2k}}{x^i_k}\delta^i \beta(x^{i+1}_{2k}) - \delta^i \beta(x^i_k) }
      \leq 2\delta^i L_\beta' \dist\lrp{x^i_k, x^{i+1}_{2k}}^2 
  \end{alignat*}
  Plugging into \ref{e:t:pmcqwd:1}
  \begin{alignat*}{1}
      & \dist\lrp{x^i_{k+1}, \Exp_{x^{i+1}_{2k}}\lrp{a_k + b_k}}^2 
      \leq \lrp{1+ 4 \C_k^2 e^{4\C_k} + 64e^{\C_k}\delta^iL_\beta'}\dist\lrp{x^i_k,x^{i+1}_{2k}}^2
  \end{alignat*}

  We now bound the second term of \eqref{e:t:rjlqkwn:1}. Let us introduce two more convenient definitions:
  \begin{alignat*}{1}
      & b'_k := \delta^i \beta\lrp{x^{i+1}_{2k+1}} + {\lrp{\BB\lrp{\lrp{2k+2}\delta^{i+1}} - \BB\lrp{\lrp{2k+1}\delta^{i+1}}}} \circ E^{i+1}_{2k+1}\\
      & z := \Exp_{x^{i+1}_{2k}}\lrp{a_k}
  \end{alignat*}
  It follows from definition that 
  \begin{alignat*}{1}
      x^{i+1}_{2k+2} = \Exp_{z}\lrp{b'_k}
  \end{alignat*}

  We break the bound on $\dist\lrp{\Exp_{x^{i+1}_{2k}}\lrp{a_k + b_k},x^{i+1}_{2k+2}}$ into two terms:
  \begin{alignat*}{1}
      \dist\lrp{\Exp_{x^{i+1}_{2k}}\lrp{a_k + b_k},x^{i+1}_{2k+2}}
      \leq& \dist\lrp{\Exp_{x^{i+1}_{2k}}\lrp{a_k + b_k},\Exp_{z}\lrp{ \party{x^{i+1}_{2k}}{x^{i+1}_{2k+1}}b_k}} + \dist\lrp{\Exp_{z}\lrp{ \party{x^{i+1}_{2k}}{x^{i+1}_{2k+1}}b_k},x^{i+1}_{2k+2}}\\
      =& \dist\lrp{\Exp_{x^{i+1}_{2k}}\lrp{a_k + b_k},\Exp_{z}\lrp{ \party{x^{i+1}_{2k}}{x^{i+1}_{2k+1}}b_k}} + \dist\lrp{\Exp_z\lrp{ \party{x^{i+1}_{2k}}{x^{i+1}_{2k+1}}b_k}, \Exp_z\lrp{b'_k}}
  \end{alignat*}
  To bound the first term, we apply Lemma \ref{l:triangle_distortion} (from \citep{sun2019escaping}) with $x = x^{i}_{2k}$, $a = a_k$, $y = b_k$, to get
  \begin{alignat*}{1}
      & \dist\lrp{\Exp_{x^{i+1}_{2k}}\lrp{a_k + b_k},\Exp_{z}\lrp{ \party{x^{i+1}_{2k}}{x^{i+1}_{2k+1}}b_k}}\\
      \leq& L_R \lrn{a_k}\lrn{b_k}\lrp{\lrn{a_k} + \lrn{b_k}} e^{\sqrt{L_R} \lrp{\lrn{a_k}+\lrn{b_k}}}
  \end{alignat*}

  To bound the second term, we apply Lemma \ref{l:discrete-approximate-synchronous-coupling} (with $x=y=z$), so that
  \begin{alignat*}{1}
      & \dist\lrp{\Exp_z\lrp{ \party{x^{i+1}_{2k}}{x^{i+1}_{2k+1}}b_k}, \Exp_z\lrp{b'_k}}^2 \\
      \leq& 32 e^{\C_k'} \lrn{\party{x^{i+1}_{2k}}{x^{i+1}_{2k+1}}b_k - b'_k}^2\\
      \leq& 64 e^{\C_k'} {\delta^{i+1}}^2 {L_\beta'}^2 \dist\lrp{x^{i+1}_{2k}, x^{i+1}_{2k+1}}^2 \\
      \leq& 128 e^{\C_k'} {\delta^{i+1}}^2 {L_\beta'}^2 \lrn{a_k}_2^2
  \end{alignat*}
  where we define $\C_k':= \sqrt{L_R} \lrp{\lrn{b_k} + \lrn{b_k'}}$.

  Plugging everything into \eqref{e:t:rjlqkwn:1},
  \begin{alignat*}{1}
      \dist\lrp{x^i_{k+1}, x^{i+1}_{2k+2}}^2
      \leq& \lrp{1 + \frac{1}{2K}}\lrp{1+ 4 \C_k^2 e^{4\C_k} + 64e^{\C_k}\delta^iL_\beta'}\dist\lrp{x^i_k,x^{i+1}_{2k}}^2\\
      &\qquad + 32K L_R^2 \lrp{\lrn{a_k}^6 + \lrn{b_k}^6} e^{2\sqrt{L_R} \lrp{\lrn{a_k}+\lrn{b_k}}} + 256K e^{\C_k'} {\delta^{i+1}}^2 {L_\beta'}^2 \lrn{a_k}^2
  \end{alignat*}

  In fact, if we consider any $t\in [k\delta^i, (k+1)\delta^i)$, and using the definition of $x^i(t)$ from \eqref{d:x^i(t)} as the linear interpolation between $x^i_k$ and $x^i_{k+1}$, we can extend the bound to
  \begin{alignat*}{1}
      & \sup_{t\in [k\delta^i, (k+1)\delta^i)} \dist\lrp{x^i(t), x^{i+1}(t)}^2 \\\leq& \lrp{1 + \frac{1}{2K}}\lrp{1+ 4 \C_k^2 e^{4\C_k} + 64e^{\C_k}\delta^iL_\beta'}\dist\lrp{x^i_k,x^{i+1}_{2k}}^2\\
      &\qquad + 32K L_R^2 \lrp{\lrn{a_k}^6 + \lrn{b_k}^6} e^{2\sqrt{L_R} \lrp{\lrn{a_k}+\lrn{b_k}}} + 256K e^{\C_k'} {\delta^{i+1}}^2 {L_\beta'}^2 \lrn{a_k}^2
      \elb{e:t:qomlkqm}
  \end{alignat*}

  Let us define
  \begin{alignat*}{1}
      r_0 =& 0\\
      r_{k+1}^2 
      :=& \lrp{1 + \frac{1}{2K}}\lrp{1+ 4 \C_k^2 e^{4\C_k} + 64e^{\C_k}\delta^iL_\beta'}r_k^2\\
      &\qquad + 32K L_R^2 \lrp{\lrn{a_k}^6 + \lrn{b_k}^6} e^{2\sqrt{L_R} \lrp{\lrn{a_k}+\lrn{b_k}}} + 256K e^{\C_k'} {\delta^{i+1}}^2 {L_\beta'}^2 \lrn{a_k}^2
  \end{alignat*}
  It follows from \eqref{e:t:qomlkqm} that $r_k \geq \sup_{t\in [(k-1)\delta^i, k\delta^i)} \dist\lrp{x^i(t), x^{i+1}(t)}$ and that $r_{k+1} \geq r_k$ with probability 1, for all $k$, so that $\sup_{t\leq T} \dist\lrp{x^i(t), x^{i+1}(t)} \leq r_K$. We will now bound $\E{r_K^2}$, and then apply Markov's Inequality. Let us define $\F_k$ to be the $\sigma$-field generated by $\BB(t)$ for $t\in [0, k\delta^i)$. Then
  \begin{alignat*}{1}
      \Ep{\F_k}{r_{k+1}^2} 
      \leq& \Ep{\F_k}{\lrp{1 + \frac{1}{2K}}\lrp{1+ 4 \C_k^2 e^{4\C_k} + 64e^{\C_k}\delta^iL_\beta'}} r_k^2 \\
      &\qquad + \Ep{\F_k}{32K L_R^2 \lrp{\lrn{a_k}^6 + \lrn{b_k}^6} e^{2\sqrt{L_R} \lrp{\lrn{a_k}+\lrn{b_k}}} + 256K e^{\C_k'} {\delta^{i+1}}^2 {L_\beta'}^2 \lrn{a_k}^2}
  \end{alignat*}
  We will bound the terms above one by one. First, note from definition that \\
  $\C_k \leq \sqrt{L_R}\lrp{2 \delta^i L_\beta + 2\lrn{\BB((k+1)\delta^i)-\BB(k\delta^i)}_2}$. Let $\eta^i_k:=\BB((k+1)\delta^i)-\BB(k\delta^i)$.
  
  For sufficiently large $i$, $\delta^i \leq {\sqrt{L_R} L_\beta/8}$. Simplifying,
  \begin{alignat*}{1}
      & \Ep{\F_k}{\lrp{1 + \frac{1}{2K}}\lrp{1+ 4 \C_k^2 e^{4\C_k} + 64e^{\C_k}\delta^iL_\beta'}}\\
      \leq& 1 + \frac{1}{2K} + 16{\delta^i}^2 L_R L_\beta^2 + 16 L_R \E{\lrn{\eta^i_k}^2} + 16{\delta^i}^2 L_R L_\beta^2 \E{e^{2\sqrt{L_R} \lrn{\eta^i_k}}}\\
      &\quad + \frac{8L_R}{\delta^i d} \E{\lrn{\eta^i_k}^4} + 8L_R \delta^i d \E{e^{4\sqrt{L_R} \lrn{\eta^i_k}}} + 128\delta^i L_\beta' \E{e^{2\sqrt{L_R} \lrn{\eta^i_k}}}\\
      \leq& 1 + \frac{1}{2K} + 16 L_R \lrp{{\delta^i}^2 L_\beta^2 + \E{\lrn{\eta^i_k}^2} + \frac{1}{\delta^i d}\E{\lrn{\eta^i_k}^4}} + 8\delta^i \lrp{\delta^i L_R L_\beta^2 + L_R d + 16 L_\beta'}\E{e^{4\sqrt{L_R} \lrn{\eta^i_k}}}\\
      \leq& 1 + \frac{1}{2K} + 40 \lrp{{\delta^{i}}^2 L_R L_\beta^2 + 2\delta^i L_R d + \delta^i L_\beta'}
  \end{alignat*}
  where we use
  \begin{alignat*}{1}
      & \E{\lrn{\eta^i_k}^2} = \delta^i d\\
      & \E{\lrn{\eta^i_k}^2} \leq 2 {\delta^i}^2 d^2\\
      & \E{e^{4\sqrt{L_R} \lrn{\eta^i_k}}} \leq 2\E{e^{8L_R \lrn{\eta^i_k}^2}} \leq 2e^{16L_R \delta^i d} \leq 4 
  \end{alignat*}
  where we use Lemma \ref{l:subexponential-chi-square}, and the fact that $\delta^i \leq \frac{1}{32L_R d}$ for sufficiently large $i$. 

  Next, we bound $\Ep{\F_k}{32K L_R^2 \lrp{\lrn{a_k}^6 + \lrn{b_k}^6} e^{2\sqrt{L_R} \lrp{\lrn{a_k}+\lrn{b_k}}}}$. Note that $\lrn{a_k}\leq \frac{\delta^i}{2} L_\beta + \lrn{\eta^{i+1}_{2k}}$ and $\lrn{b_k}\leq \frac{\delta^i}{2} L_\beta + \lrn{\eta^{i+1}_{2k+1}}$. By similar argument as above, 
  \begin{alignat*}{1}
      & \Ep{\F_k}{32K L_R^2 \lrp{\lrn{a_k}^6 + \lrn{b_k}^6} e^{2\sqrt{L_R} \lrp{\lrn{a_k}+\lrn{b_k}}}}\\
      \leq& 2048KL_R^2 e^{2\sqrt{L_R}\delta^i L_\beta} \Ep{\F_k}{\lrp{2^{-5} {\delta^i}^6 L_\beta^6 + \lrn{\eta^{i+1}_{2k}}^6 + \lrn{\eta^{i+1}_{2k}}^6} \cdot e^{2\sqrt{L_R} \lrp{\lrn{\eta^{i+1}_{2k}} + \lrn{\eta^{i+1}_{2k+1}}}}}\\
      \leq& K L_R^2 \lrp{512 {\delta^i}^6 L_\beta^6 + 2048 {\delta^i}^3 d^3 }
  \end{alignat*}

  Finally, note that $\C_k' \leq 2\sqrt{L_R} \lrp{L_\beta + \lrn{\eta^{i+1}_{2k+1}}}$, so that
  \begin{alignat*}{1}
      \Ep{\F_k}{256K e^{\C_k'} {\delta^{i+1}}^2 {L_\beta'}^2 \lrn{a_k}^2}
      \leq& 256K{\delta^i}^2 {L_\beta'}^2 \lrp{{\delta^i}^2 L_\beta^2 + \delta^i d}
  \end{alignat*}

  Put together,
  \begin{alignat*}{1}
      &\Ep{\F_k}{r_{k+1}^2} \\
      \leq& \lrp{1 + \frac{1}{2K} + 40 \lrp{{\delta^{i}}^2 L_R L_\beta^2 + 2\delta^i L_R d + \delta^i L_\beta'}} \lrp{K L_R^2 \lrp{512 {\delta^i}^6 L_\beta^6 + 2048 {\delta^i}^3 d^3 } + 256K{\delta^i}^2 {L_\beta'}^2 \lrp{{\delta^i}^2 L_\beta^2 + \delta^i d}}
      \elb{e:t:lqmf:0}
  \end{alignat*}
  Applying the above recursively and simplifying,
  \begin{alignat*}{1}
      \E{r_{K}^2} 
      \leq e^{40K{\delta^i}^2 L_R L_\beta^2 + 2K\delta^i L_R d + K\delta^i L_\beta'} \lrp{K {\delta^i}}^2 \lrp{{\delta^i}^4 L_R^2 L_\beta^6 + {\delta^i} L_R^2 d^3 + {\delta^i}^2 {L_\beta'}^2 L_\beta^2 + {\delta^i} {L_\beta'}^2d} \cdot 2^{10}.
      \elb{e:t:lqmf}
  \end{alignat*}
  Recall that $r_k \geq \sup_{t\in [(k-1)\delta^i, k\delta^i)} \dist\lrp{x^i(t), x^{i+1}(t)}$ (see \eqref{e:t:qomlkqm}), and that $r_{k+1} \geq r_k$ with probability 1, for all $k$, so that $\sup_{t\leq T} \dist\lrp{x^i(t), x^{i+1}(t)} \leq r_K$. This proves the first claim of the lemma.

  By Markov's Inequality, and recalling that $r_k$ is w.p. 1 non-decreasing and \\ 
  $\sup_{t\leq T} \dist\lrp{x^i(t), x^{i+1}(t)} \leq R_K$,
  \begin{alignat*}{1}
      & \Pr{\sup_{t\in[0,T]} \dist\lrp{x^i(t), x^{i+1}(t)}^2 \geq 2^{- i/2 -4}}\\
      \leq& \E{r_{K}^2} \cdot 2^{i/2 + 4}\\
      \leq& e^{\lrp{40\delta^i {L_R} L_\beta^2 + 2 L_R d + L_\beta'}T} T^2 \cdot \lrp{{\delta^i}^3 L_R^2 L_\beta^6 + L_R^2 d^3 + {\delta^i} {L_\beta'}^2 L_\beta^2 +{L_\beta'}^2d} \cdot 2^{-i/2 + 14}
      \elb{e:t:rjlqkwn:3}
  \end{alignat*}
  This proves the second claim of the lemma.
\end{proof}

\begin{lemma}
    \label{l:x(t)_is_brownian_motion}
    Let $x\in M$ be some initial point and $E$ an orthonormal basis of $T_xM$. Let $\BB(t)$ be a Brownian motion in $\Re^d$; and $\beta(x)$ a vector field satisfying Assumption~\ref{ass:beta_lipschitz}. Let $T\in \Re^+$, for $t\in[0,T]$ and let $x^i(t)$ be constructed as per~\eqref{d:x^i(t):0}.
    Then with probability 1, there is a limit $x(t)$ such that for all $\epsilon$, there exists an integer $N$ such that for all $i \geq N$,
    \begin{alignat*}{1}
      \sup_{t\in[0,T]} \dist\lrp{x^i(t), x(t)} \leq \epsilon.
    \end{alignat*}
  \end{lemma}

\begin{proof}[Proof of Lemma \ref{l:x(t)_is_brownian_motion}]
  Let us define $L_0 := \lrn{\beta(x(0))}$.

  \textbf{Step 1: Bounding the probability of deviation between $x^i$ and $x^{i+1}$}\\

  We would like to apply Lemma \ref{l:key_brownian_limit_lemma}. However, note that Lemma \ref{l:key_brownian_limit_lemma} assumes that $\lrn{\beta(x)} \leq L_\beta$ globally, which we do not assume here. We must therefore approximate $\beta$ by a sequence of Lipschitz vector fields. 

  Let us define 
  \begin{alignat*}{1}
      \beta^j(x):= \twocase{\beta(x)}{\lrn{\beta(x)} \leq 2^{j/2}}{\beta(x)\cdot\frac{2^{j/2}}{\lrn{\beta(x)}}}{\lrn{\beta(x)} > 2^{j/2}}
  \end{alignat*}

  Let us denote by $L_{\beta^j} := 2^{j/2}$. We verify that for all $x,y\in M$, $\lrn{\beta^j(x)} \leq L_{\beta^j}$ and \\
  $\lrn{\beta^j(x) - \party{y}{x} \beta^j(y)} \leq L_\beta' \lrn{x-y}$.

  Finally, for any let $\t{x}^{i,j}(t)$ be as defined in \eqref{d:x^i(t)}, with $\beta$ replaced by $\beta^j$. Lemma \ref{l:key_brownian_limit_lemma} immediately implies that, for all $i \geq \C$ (where $\C$ is some constant depending on $L_R, T, d$),
  \begin{alignat*}{1}
      & \Pr{\sup_{t \in [0,T]} \dist\lrp{\t{x}^{i,i}(t), \t{x}^{i+1,i}(t)} \geq 2^{- \frac{i}{4} - 2}} \\
      &\quad \leq e^{\lrp{40T {L_R} + 2 L_R d + L_\beta'}T}  T^2 \cdot \lrp{T^3L_R^2 + L_R^2 d^3 + T {L_\beta'}^2  +{L_\beta'}^2d} \cdot 2^{-i/2 + 14}
  \end{alignat*}
  where we use the fact that $\delta^i {L_\beta^i}^2 = T$ by definition.

  Recalling that $\beta^j(x) = \beta(x)$ unless $\lrn{\beta(x)} \geq 2^{j/2}$,
  \begin{alignat*}{1}
      \Pr{\exists_{t\in[0,T]} x^i(t) \neq \t{x}^{i,i}(t)} = \Pr{\exists_{k\in \lrbb{0...2^i}} x^i_k \neq  \t{x}^{i,i}_k} \leq \Pr{\sup_{k\in \lrbb{0...2^i}} \lrn{\beta(x^i_k)} \geq 2^{i/2}}
  \end{alignat*}

  We can bound $\lrn{\beta(x)} \leq L_0 + L_\beta' \dist\lrp{x,x_0}$, so that
  \begin{alignat*}{1}
      & \Pr{\sup_{k\in \lrbb{0...2^i}} \lrn{\beta(x^i_k)} \geq 2^{i/2}}\\
      \leq& \Pr{\sup_{k\in \lrbb{0...2^i}} \dist\lrp{x^i_k} \geq \frac{2^{i/2} - L_0}{L_\beta'}}\\
      \leq& \exp\lrp{2 + 8T L_\beta' +T L_R d + TL_R L_0^2} \cdot \lrp{2T d + 4T^2 L_0^2} \cdot {L_\beta'}^2 \cdot 2^{-i+2}
  \end{alignat*}
  where we use Lemma \ref{l:near_tail_bound_L2}, with $K=2^i$, and assume that $i$ satisfies $2^{i/2} \geq L_0$ and $2^i \geq T$.

  Using identical steps, we can also bound
  \begin{alignat*}{1}
      \Pr{\exists_{t\in[0,T]} x^{i+1}(t) \neq \t{x}^{i+1,i}(t)}
      \leq \exp\lrp{2 + 8T L_\beta' +T L_R d + T L_R L_0^2} \cdot \lrp{2T d + 4T^2 L_0^2} \cdot {L_\beta'}^2 \cdot 2^{-i+2}
  \end{alignat*}

  Put together,
  \begin{alignat*}{1}
      & \Pr{\sup_{t \in [0,T]} \dist\lrp{{x}^{i}(t), {x}^{i+1}(t)} \geq 2^{- \frac{i}{4} - 2}} \\
      \leq& \Pr{\sup_{t \in [0,T]} \dist\lrp{\t{x}^{i,i}(t), \t{x}^{i+1,i}(t)} \geq 2^{- \frac{i}{4} - 2}}  + \Pr{\exists_{t\in[0,T]} x^{i}(t) \neq \t{x}^{i,i}(t)} + \Pr{\exists_{t\in[0,T]} x^{i+1}(t) \neq \t{x}^{i+1,i}(t)}\\
      \leq& \C_2 \cdot 2^{-i/2}
  \end{alignat*}
  where $\C_2$ is a constant that depends on $T, L_R, L_\beta', L_0, d$, but does not depend on $i$.

  \textbf{Step 2: Apply Borel-Cantelli to show uniformly-Cauchy sequence with probability 1}\\
  Thus
  \begin{alignat*}{1}
      \sum_{i=\C_1}^{\infty} \Pr{\sup_{t}\dist\lrp{x^i(t), x^{i+1}(t)} \geq 2^{- \frac{i}{4}}} < \infty
  \end{alignat*}
  By the Borel-Cantelli Lemma,
  \begin{alignat*}{1}
      \Pr{\sup_{t}\dist\lrp{x^i(t), x^{i+1}(t)} \geq 2^{- \frac{i}{4}} \text{ for infinitely many $i$}} = 0
  \end{alignat*}
  Equivalently, with probability $1$, for all $\epsilon$, there exists a $N$ such that for all $i\geq N$,\\
  $\sup_{t}\dist\lrp{x^i(t), x^{i+1}(t)} \leq 2^{- \frac{i}{4}}$. For any $j \geq i \geq N$, it then follows that
  \begin{alignat*}{1}
      \sup_t \dist\lrp{x^i(t), x^j(t)}
      \leq& \sum_{\ell=i}^j \dist\lrp{x^\ell(t), x^{\ell+1}(t)}\\
      \leq& \sum_{\ell=i}^j 2^{- \frac{\ell}{4}}\\
      \leq& 6 \cdot 2^{- i/4}
  \end{alignat*}

  \textbf{Step 3: Uniform-Cauchy sequence implies uniform convergence to limit using standard arguments}
  Therefore, with probability 1, $x^i(t)$ is a uniformly Cauchy sequence. Let $x(t)$ be the point-wise limit of $x^i(t)$, as $i\to \infty$. It follows \footnote{A nice clean proof can be seen at \url{https://math.stackexchange.com/questions/1287669/uniformly-cauchy-sequences}} that with probability 1, for any $\epsilon$, there exists a $N$ such that for all $i \geq N$, 
  \begin{alignat*}{1}
      \sup_{t\in[0,T]} \dist\lrp{x^i(t), x(t)} \leq \epsilon
  \end{alignat*}
\end{proof}
  \begin{lemma}
    \label{l:discretization-approximation-lipschitz}
    Let $\beta(\cdot)$ be a vector field satisfying Assumption \ref{ass:beta_lipschitz}. Assume also that there exists $L_\beta$ such that $\lrn{\beta(x)}\leq L_\beta$ for all $x$. Consider arbitrary $x_0\in M$ and let $E$ be an orthonormal basis of $T_{x_0} M$. Let $\BB$ be a standard Brownian motion in $\Re^d$. Let $x^i(t)= \overline{\Phi}(t;x,E,\beta,\BB,i)$ and $x(t)= \Phi(t;x,E,\beta,\BB)$ as defined in \eqref{d:x^i(t)} and \eqref{d:x(t)} respectively. (Existence of $x(t)$ follows from Lemma \ref{l:x(t)_is_brownian_motion}).
    
    Then for any non-negative integer $\ell$,
    \begin{alignat*}{1}
        \E{\sup_{t\in[0,T]}\dist\lrp{x^\ell(t),x(t)}^2} 
        \leq& 2^{14}e^{40T\delta^\ell L_R L_\beta^2 + 2T L_R d + T L_\beta'} T^3\lrp{{\delta^\ell}^3 L_R^2 L_\beta^6 + L_R^2 d^3 + {\delta^\ell} {L_\beta'}^2 L_\beta^2 + {L_\beta'}^2d} \cdot 2^{-\ell}
    \end{alignat*}
    where $\delta^i := 2^{-i} T$
\end{lemma}
\begin{proof}
    Consider any fixed $i$, let $\delta^i := T/2^i$ and let $K:= T/\delta^i = 2^i$ as in \eqref{d:x^i_k}.

    By the first claim of Lemma \ref{l:key_brownian_limit_lemma}, we can bound
    \begin{alignat*}{1}
        &\E{\sup_{t\in[0,T]}\dist\lrp{x^i(t),x^{i+1}(t)}^2} \\
        \leq& 2^{10} \cdot e^{40K{\delta^i}^2 L_R L_\beta^2 + 2K\delta^i L_R d + K\delta^i L_\beta'} \lrp{K {\delta^i}}^2 \lrp{{\delta^i}^4 L_R^2 L_\beta^6 + {\delta^i} L_R^2 d^3 + {\delta^i}^2 {L_\beta'}^2 L_\beta^2 + {\delta^i} {L_\beta'}^2d}\\
        =& \underbrace{2^{10}e^{40T\delta^i L_R L_\beta^2 + 2T L_R d + T L_\beta'} T^3\lrp{{\delta^i}^3 L_R^2 L_\beta^6 + L_R^2 d^3 + {\delta^i} {L_\beta'}^2 L_\beta^2 + {L_\beta'}^2d}}_{:=s_i} \cdot 2^{-i}
    \end{alignat*}
    where we use the fact that $K \delta^i =T$ by definition.

    By repeated application of Young's Inequality, we can bound, for any $\ell$ and any $j\geq \ell$,
    \begin{alignat*}{1}
        \E{\sup_{t\in[0,T]}\dist\lrp{x^{\ell}(t),x^{j}(t)}^2} 
        \leq& \sum_{i=\ell}^{j-1} 3 \lrp{\frac{3}{2}}^{i-\ell}\E{\sup_{t\in[0,T]}\dist\lrp{x^i(t), x^{i+1}(t)}^2}\\
        \leq& \sum_{i=\ell}^{j-1} 3 \lrp{\frac{3}{2}}^{i-\ell}\cdot 2^{-i} s_i \\
        \leq& 12 \cdot 2^{-\ell} \cdot s_\ell
    \end{alignat*}

    Since the above holds for any $j$, we can take the limit of $j\to\infty$ and 
    \begin{alignat*}{1}
        \E{\sup_{t\in[0,T]}\dist\lrp{x^\ell(t),x(t)}^2} 
        \leq& 12 \cdot 2^{-\ell} \cdot s_\ell
    \end{alignat*}
    where we use the fact that $\dist\lrp{x^j_{2^j},x(T)}$ converges almost surely to $0$, from Lemma \ref{l:x(t)_is_brownian_motion}.
\end{proof}

\subsection{Discretization Error of Euler Murayama}
\label{ss:langevin_mcmc_on_manifold}
Given the results of the previous section, we are now ready to prove Lemma \ref{l:discretization-approximation-lipschitz-derivative}, which is informally stated as Lemma \ref{l:informal_discretization-approximation-lipschitz-derivative} in the Section \ref{Brownian Motion Section}. The proof of Lemma \ref{l:discretization-approximation-lipschitz-derivative} works by summing, for all $i$, the distance between $x^i(t)$ and $x^{i+1}(t)$ (which is bounded in Lemma \ref{l:key_brownian_limit_lemma}). Extra care must be taken to ensure that iterates do not stray too far from the initial error.

The crucial analysis corresponding to Step 1 (synchronous coupling) and Step 2 (rolling without slipping) discussed in the proof sketch of Lemma \ref{l:informal_discretization-approximation-lipschitz-derivative} in Section \ref{Brownian Motion Section} can be found in the proof of Lemma \ref{l:key_brownian_limit_lemma} in Section \ref{ss:Existence} above.

\begin{lemma}
    \label{l:discretization-approximation-lipschitz-derivative}
    Let $M$ satisfy Assumption \ref{ass:sectional_curvature_regularity}. Let $\beta(\cdot)$ be a vector field satisfying Assumption \ref{ass:beta_lipschitz}. Consider arbitrary $x(0) \in M$. Let $L_1$ be any constant such that $L_1 \geq \lrn{\beta(x(0))}$ and let $T$ be a step-size satisfying $T\leq \min\lrbb{\frac{1}{16{L_\beta'}},  \frac{1}{16L_R d}, \frac{1}{16\sqrt{L_R} L_1}}$.
    
    Let $x(t)$ denote the solution to \eqref{e:intro_euler_murayama}, initialized at $x(0)$. Let $x^0(t) := \Exp_{x(0)}\lrp{t \beta(x(0)) + \sqrt{t} \zeta}$, where $\zeta \sim \N_{x(0)}(0,I)$. Then there exists a coupling between $x^0(T)$ and $x(T)$ such that
    \begin{alignat*}{1}
      \E{\dist\lrp{x^{0}(T),x(T)}^2} \leq 2^{20}\lrp{T^4 L_1^4 \lrp{1+L_R}+ T^4 {L_\beta'}^4 + T^{3} \lrp{d^{3}\lrp{L_R + {L_\beta'}^2/{L_1^2}} + {L_\beta'}^2 d}}
    \end{alignat*}
  \end{lemma}

\begin{proof}[Proof of Lemma \ref{l:discretization-approximation-lipschitz-derivative}]
    Let $E$ be an orthonormal basis of $T_{x(0)} M$. Following the definition of $\overline{\Phi}$ in \eqref{d:x^i(t)}, we verify that $x^0(T)= \overline{\Phi}(T;x(0),E,\beta,\BB,0)$, where $\BB$ is some Brownian motion, and equality is in the sense of distribution. On the other hand, by Lemma \ref{l:Phi_is_diffusion}, $x(t) = \Phi(t;x(0),E,\beta,\BB)$, where $\Phi$ is the limit of $\overline{\Phi}$ as defined in \eqref{d:x(t)}. 
    
    \textbf{Step 1: Bounding the distance between $x^i(T)$ and $x^{i+1}(T)$}\\
    Let us consider some fixed $i$. Let $\delta^i = T/2^i$ denote the stepsize, and let $K := 2^i$ so that $T= K\delta^i$. Let $x^i(t)= \overline{\Phi}(T;x(0),E,\beta,\BB,i)$ be as defined in \eqref{d:x^i(t)}. Recall that $x^i(t)$ is by definition the linear interpolation of $x^i_k$ (defined in \eqref{d:x^i_k} or equivalently \eqref{d:x^i(t):0-appendix}), which are marginally Euler-Murayama sequences with stepsize $\delta^i$. 
    
    Our goal is to bound $\E{\dist\lrp{x^i(T),x^{i+1}(T)}^2}$. Lemma \ref{l:key_brownian_limit_lemma} almost gives us what we want; the problem is that Lemma \ref{l:key_brownian_limit_lemma} assumes that for all $x\in M$, $\lrn{\beta(x)}\leq L_\beta$, but we do not make that assumption in this lemma. In order to get around this issue, \textbf{we use an argument based on truncating $\beta$ at larger and larger norms}.

    Let us define $L_0 := \lrn{\beta(x(0))}$ and
    \begin{alignat*}{1}
        \beta^j(y):= \twocase{\beta(y)}{\lrn{\beta(y)} \leq L_1 2^{j/2+1}}{\beta(y)\cdot\frac{L_1 2^{j/2+1}}{\lrn{\beta(y)}}}{\lrn{\beta(y)} > L_1 2^{j/2+1}},
        \elb{e:t:oqmflwkf:10}
    \end{alignat*}
    i.e. $\beta^j$ is the truncated version of $\beta$, so that the norm of $\beta^j$ is globally upper bounded by $L_1 2^{j/2+1}$.

    Given this definition of $\beta^j$, we now define, for all $j\in \Z^+$, $\t{x}^{i,j}(t):= \overline{\Phi}(t;x(0),E,\beta^j,\BB,i)$, which is the (interpolated) Euler-Murayama discretization from \eqref{d:x^i(t)}, with step-size $\delta^i$, and drift $\beta^j$. In other words, $x^i(t)$ and $\t{x}^{i,j}(t)$ are both Euler Murayama discretizations with stepsize $\delta^i$, but the former has drift $\beta$ whereas the latter has drift $\beta^j$. We also let $\t{x}^{\cdot,j}(t):= {\Phi}(t;x(0),E,\beta^i,\BB)$ denote the limit, as $i\to \infty$, of $\t{x}^{i,j}(t)$ (see definition in \eqref{d:x(t)}).
    
    By Young's Inequality,
    \begin{alignat*}{1}
        \dist\lrp{x^i(t),x^{i+1}(t)}^2
        \leq& 8\dist\lrp{x^i(t),\t{x}^{i,i}(t)}^2 + 8\dist\lrp{x^{i+1}(t),\t{x}^{i+1,i}(t)}^2 + 8\dist\lrp{\t{x}^{i,i}(t),\t{x}^{i+1,i}(t)}^2
        \elb{e:t:oqmflwkf:1}
    \end{alignat*}
    Before proceeding, we briefly explain the intuition behind the decomposition in \eqref{e:t:oqmflwkf:1}. Notice that we let $j=i$ in $\t{x}^{i,i}$, i.e. as $i$ increases, two this happen: $\delta^i$ becomes smaller, and $\beta^i$ becomes truncated at a larger norm (see \eqref{e:t:oqmflwkf:10}), and is thus closer to the true un-truncated $\beta$.

    The first and second term on the right hand side of \eqref{e:t:oqmflwkf:1} correspond to error due to truncating $\beta$ to $\beta^i$. These two terms would equal $0$ if $\sup_t \dist\lrp{x^i(t), \t{x}^{i,i}(t)} \leq \frac{L_1 2^i/2}{L_\beta'}$, as that implies (via Assumption \ref{ass:beta_lipschitz}) that $\sup_t \lrn{\beta(x^i(t))} \leq L_1 2^{i/2+1}$ (so along the entire path, $\beta$ was never large enough to require truncation). As $i$ increases, $\frac{L_1 2^i/2}{L_\beta'} \to \infty$, and these two truncation errors are $0$ with increasing probability.
    
    The last term in \eqref{e:t:oqmflwkf:1} corresponds to error due to different discretization stepsize ($\delta^i$ vs $\delta^{i+1}$). 

    \textbf{Step 1.1: Bounding distance of truncated-drift sequences $\t{x}^{i,i}$ and $\t{x}^{i+1,i}$}\\
    We first bound the last term of \eqref{e:t:oqmflwkf:1}. By the first claim of Lemma \ref{l:key_brownian_limit_lemma}, we can bound
    \begin{alignat*}{1}
        &\E{\dist\lrp{\t{x}^{i,i}(T),\t{x}^{i+1,i}(T)}^2}\\
        \leq& 2^{10}e^{40T^2 L_R L_1^2 + 2T L_R d + T L_\beta'} T^3\lrp{T^3 L_R^2 L_1^6 + L_R^2 d^3 + T {L_\beta'}^2 L_1^2 + {L_\beta'}^2d} \cdot 2^{-i}\\
        \leq& 2^{14-i} T^3\lrp{T^3 L_R^2 L_1^6 + L_R^2 d^3 + T {L_\beta'}^2 L_1^2 + {L_\beta'}^2d},
    \end{alignat*}
    where the first inequality uses the fact that $\delta^i = T/2^i\leq T$ and that $\delta^i L_1 2^{i+1} = 2T L_1$, and second inequality is by our assumed upper bound on $T$. Here, we crucially use the fact that, the effect of $\delta^i$ halving with $i$ "cancels out" the effect of $\lrn{\beta^i(x)}$ becoming larger with $i$.

    \textbf{Step 1.2: Bounding error due to truncation}\\
    We now bound the first two terms of \eqref{e:t:oqmflwkf:1}. Once again recall from the definition in \eqref{d:x^i(t):0} and \eqref{d:x^i(t)} that $x^i(t)$ (resp $\t{x}^{i,i}(t)$) are linear interpolations of the discrete sequence $x^i_k$ (resp $\t{x}^{i,i}_k$) as defined in \eqref{d:x^i_k}. Under the event $\sup_{k\in\lrbb{0...2^i}} \dist\lrp{x^i_k,x(0)} \leq \frac{2^{i/2}L_1}{L_\beta'}$, we verify that $x^i_k = \t{x}^{i,i}_k$ for all $k\in\lrbb{0...2^i}$ -- this is because we can then bound, for all $k$, $\lrn{\beta(x^i_k)} \leq L_\beta' \dist\lrp{x^i_k,x(0)} + L_0 \leq 2^{i/2+1} L_1$, which in turn implies that for all $k$, $\beta(x^i_k)$ equals the truncated version $\beta^i(x^i_k)$, which in turn implies that $\dist\lrp{x^i(T),\t{x}^{i,i}(T)}=0$. Therefore,
    \begin{alignat*}{1}
        &\E{\dist\lrp{x^i(T),\t{x}^{i,i}(T)}^2} \\
        =& \E{\ind{\sup_{k\in\lrbb{0..2^i}} \dist\lrp{x^i_k,x(0)}> \frac{2^{i/2}L_1}{L_\beta'}} \dist\lrp{x^i(T),\t{x}^{i,i}(T)}^2}\\
        \leq& 2\sqrt{\Pr{\sup_{k\in\lrbb{0..2^i}} \dist\lrp{x^i_k,x(0)}> \frac{2^{i/2}L_1}{L_\beta'}}} \cdot \lrp{\sqrt{\E{\dist\lrp{x^i(T),x(0)}^4}} + \sqrt{\E{\dist\lrp{\t{x}^{i,i}(T),x(0)}^4}}}.
        \elb{e:t:oqmflwkf:2}
    \end{alignat*}
    where the second line follows from Young's inequality and Cauchy Schwarz.

    From Lemma \ref{l:near_tail_bound_L4}, and our assumed bound on $T$,
    \begin{alignat*}{1}
        \sqrt{\Pr{\sup_{k\in\lrbb{0...2^i}} \dist\lrp{x^i_k,x(0)} > \frac{2^{i/2}L_1}{L_\beta'}}}
        \leq& \frac{{L_\beta'}^2}{L_1^2 2^i}\exp\lrp{1 + 8 T L_\beta' + 2T L_R d + 2T\delta^i L_R L_0^2}\lrp{3T d + 8T^2 L_0^2}\\
        \leq& \frac{{L_\beta'}^2\lrp{Td + T^2 L_0^2}}{L_1^2} \cdot{2^{4-i}}.
    \end{alignat*}

    Also from Lemma \ref{l:near_tail_bound_L4}, for any $k$,
    \begin{alignat*}{1}
        \E{\dist\lrp{x_k,x(0)}^4} 
        \leq& \exp\lrp{2 + 16 T L_\beta' + 4T L_R d + 3T^2 L_R L_0^2}\lrp{T^2 d^2 + 64T^4 L_0^4}\\
        \leq& 4 \lrp{T^2 d^2 + 64T^4 L_0^4}.
    \end{alignat*}

    The same upper bound also applies to $\E{\dist\lrp{\t{x}^{i,i}(T),x(0)}^4}$. Plugging into \eqref{e:t:oqmflwkf:2}, 
    \begin{alignat*}{1}
        \E{\dist\lrp{x^i(T),\t{x}^{i,i}(T)}^2} \leq 2^{12-i} \frac{{L_\beta'}^2\lrp{Td + T^2 L_0^2}^{3}}{L_1^2}.
    \end{alignat*}

    Note that the bound in Lemma \ref{l:near_tail_bound_L4} is strictly stronger for $x^{i+1}_k$ compared to $x^i_k$. Thus by exactly identical steps, we can also upper bound
    \begin{alignat*}{1}
        \E{\dist\lrp{x^{i+1}(T),\t{x}^{i+1,i}(T)}^2} \leq 2^{12-i} \frac{{L_\beta'}^2\lrp{Td + T^2 L_0^2}^{3}}{L_1^2}.
    \end{alignat*}

    Plugging into \eqref{e:t:oqmflwkf:1}, and applying Young's Inequality, we finally have our bound between the Euler-Murayama sequences $x^i(T)$ and $x^{i+1}(T)$ (without truncation of $\beta$).
    \begin{alignat*}{1}
        \E{\dist\lrp{x^i(T),x^{i+1}(T)}^2} 
        \leq& 2^{14-i} T^3\lrp{T^3 L_R^2 L_1^6 + L_R^2 d^3 + T {L_\beta'}^2 L_1^2 + {L_\beta'}^2d} + 2^{13-i} \frac{{L_\beta'}^2\lrp{Td + T^2 L_0^2}^{3}}{L_1^2} \\
        \leq& 2^{16-i}\lrp{T^4 L_1^4 \lrp{1+L_R}+ T^4 {L_\beta'}^4 + T^{3} \lrp{d^{3}\lrp{L_R + {L_\beta'}^2/{L_1^2}} + {L_\beta'}^2 d}},
        \elb{e:t:oqmflwkf:11}
    \end{alignat*}
    where we use the assumed upper bound on $T$ in the lemma statement. 

    \textbf{Step 2: Summing over $i$}\\
    For any $\ell \in \Z^+$, we can summing \eqref{e:t:oqmflwkf:11} for $i\in \lrbb{0,1,2...\ell}$ to bound
    \begin{alignat*}{1}
        \E{\dist\lrp{x^{0}(T),x^{\ell}(T)}^2} 
        \leq& \sum_{i=0}^{\ell} 3 \cdot \lrp{\frac{3}{2}}^{i} 2^{16-i}\lrp{T^4 L_1^4 \lrp{1+L_R}+ T^4 {L_\beta'}^4 + T^{3} \lrp{d^{3}\lrp{L_R + {L_\beta'}^2/{L_1^2}} + {L_\beta'}^2 d}}\\
        \leq& 2^{20}\lrp{T^4 L_1^4 \lrp{1+L_R}+ T^4 {L_\beta'}^4 + T^{3} \lrp{d^{3}\lrp{L_R + {L_\beta'}^2/{L_1^2}} + {L_\beta'}^2 d}}.
        \elb{e:t:oqmflwkf:3}
    \end{alignat*}
    The first inequality uses triangle inequality and Young's inequality recursively: for any $i$, 
    \begin{alignat*}{1}
        \dist\lrp{x^i(T),x^{\ell}(T)}^2 \leq \frac{3}{2} \dist\lrp{x^{i+1}(T),x^{\ell}(T)}^2 + 3\dist\lrp{x^i(T),x^{i+1}(T)}^2.
    \end{alignat*}

    Since \eqref{e:t:oqmflwkf:3} holds for all $\ell$, we take the limit of $\ell \to \infty$. By dominated convergence together with Lemma \ref{l:x(t)_is_brownian_motion}, 
    \begin{alignat*}{1}
        \E{\dist\lrp{x^{0}(T),x(T)}^2} \leq 2^{20}\lrp{T^4 L_1^4 \lrp{1+L_R}+ T^4 {L_\beta'}^4 + T^{3} \lrp{d^{3}\lrp{L_R + {L_\beta'}^2/{L_1^2}} + {L_\beta'}^2 d}}.
    \end{alignat*}
\end{proof}

\subsection{Proof of Theorem \ref{t:langevin_mcmc}}\label{ss:proof_of_t:langevin_mcmc}
Below, we provide the full proof of Theorem \ref{t:langevin_mcmc}, which was sketched in Section \ref{s:langevin MCMC}. The main results used are Lemma \ref{l:g_contraction_without_gradient_lipschitz} (contraction of Lyapunov function under exact SDE) and Lemma \ref{l:discretization-approximation-lipschitz-derivative} (bound on Euler Murayama discretization error).

\begin{proof}[Proof of Theorem \ref{t:langevin_mcmc}]\ \\
    \textbf{Step 0: Defining some Key Constants}\\
    In this step, we define a radius $r$, an event $A_k$ based on $r$, and an upper bound on $\delta$.
    \begin{alignat*}{1}
        & c_0 := \log\lrp{\frac{L_\beta'}{m}} + \log \lrp{\frac{\lrp{1+L_R}d}{m}} + \log \R + \log\lrp{K}, \\
        & r_0 = 32\sqrt{\frac{{L_\beta'}^2\R^2}{m} + \frac{L_R d^2}{m^2} + \frac{d}{m}\cdot c_0},\\
        & r = r_0 + 32\sqrt{\frac{d}{m} \cdot \log\lrp{1/\delta}}.
        \elb{e:t:wlqekwla:12}
    \end{alignat*}
    Let $A_k$ denote the event $\max_{i\leq k} \dist\lrp{x_i,x^*} \leq r$. The value of $r$ and $A_k$ are chosen so that \eqref{e:t:wlqekwla:0} holds in Step 1 below. Note that $r$ depends on $\log(1/\delta)$ on the last line.

    We now define a suitable upperbound on the stepsize $\delta$. To do so, we first define the constants $c_1,c_5,c_6,c_7$:
    \begin{alignat*}{1}
        c_1 := \min\lrbb{\frac{1}{16{L_\beta'}},  \frac{1}{16L_R d}, \frac{m}{64 d\sqrt{L_R} L_\beta' r_0}, \frac{m}{64 d\sqrt{L_R} L_\beta' \sqrt{\log \lrp{d\sqrt{L_R} L_\beta'/m}}}},
        \elb{e:t:wlqekwla:7}
    \end{alignat*}
    \begin{alignat*}{1}
        & c_5 := \frac{1}{64}\min\lrbb{\frac{m}{{L_\beta'}^2}, \frac{d}{m}},\\
        & c_6 := \frac{1}{64}\min\lrbb{\frac{m}{{L_\beta'}^2\sqrt{L_R} r_0}, \frac{d}{m\sqrt{L_R} r_0}, \frac{d^2}{m^2 r_0^2}},\\
        & c_7 := \frac{1}{64}\min\lrbb{\sqrt{\frac{m^3}{d{L_\beta'}^4L_R^2\log\lrp{\nicefrac{{d L_\beta'}^4L_R^2}{m^3}}}}, \sqrt{\frac{d}{m L_R \log \lrp{\nicefrac{m L_R}{d}}}}, {\frac{d}{m \log\lrp{\nicefrac{d}{m}}}}}.
        \elb{e:t:wlqekwla:11}
    \end{alignat*}

    For our proof, we require that $\delta$ satisfies
    \begin{alignat*}{1}
        \delta \leq \min\lrbb{c_1,c_5,c_6,c_7}
        \elb{e:t:wlqekwla:5}
    \end{alignat*}
    Thus $\C_0$ from the theorem statement is explicitly $\C_0 = \frac{1}{\min\lrbb{c_1,c_5,c_6,c_7}}$. The motivation for this upper bound on $\delta$ is so that $\delta$ satisfies the conditions in Lemma \ref{l:discretization-approximation-lipschitz-derivative} and Lemma \ref{l:far-tail-bound-truncated-sgld} which are used in Step 1 and Step 3. We provide details in Step 4 below. Note that the upper bound on $\delta$ depends only on $r_0$ but not $r$, so the definition is not circular.

    \textbf{Step 1: Tail Bound:}\\
    We now show that with high probability, the discretization sequence $x_k$ never steps outside the ball of radius $r$ centered at $x^*$ (this is exactly the event $A_k$ that we defined in the previous step).

    By Lemma \ref{l:far-tail-bound-truncated-sgld} with $\sigma = 0$ and $\t{\beta} = \beta$, for any $(\delta,r)$ satisfying \\
    $\delta \leq \min\lrbb{\frac{m}{16 {L_\beta'}^2 \lrp{1 + \sqrt{L_R} r}}, \frac{d}{m \lrp{1 + \sqrt{L_R} r}}, \frac{32 {d^2}}{m^2 r^2}}$, we can bound the probability of $x_k$ stepping outside the ball of radius $r$ centered at $x^*$, for any $k\leq K$, as
    \begin{alignat*}{1}
        \Pr{A_k^c} = \Pr{\max_{k\leq K} \dist\lrp{x_k,x^*} \geq r} \leq 32K\delta m \exp\lrp{\frac{2{L_\beta'}^2\R^2}{d} + \frac{64 L_R {d} }{m} - \frac{m r^2}{256 {d} }}.
        \elb{e:t:wlqekwla:3}
    \end{alignat*}
    
    Furthermore, we can bound the fourth-moment of the distance between $x_K$ and $x^*$: by Lemma \ref{l:far-tail-bound-l2} with $\xi_k(x_k) = \zeta_k$ and $\sigma_\xi = 2\sqrt{d}$, we can bound
    \begin{alignat*}{1}
        \E{\dist\lrp{x_K,x^*}^2} \leq\frac{2^{13} L_R {L_\beta'}^4d^2}{m^{6}}  + \frac{16 {L_\beta'} \R^2}{m} + \frac{16 d}{m}
    \end{alignat*}
    Similarly, we can use Lemma \ref{l:far-tail-bound-l2-brownian} to bound
    \begin{alignat*}{1}
        & \E{\dist\lrp{y(K\delta),x^*}^2} \leq \frac{2^{13} L_R {L_\beta'}^4 d^2}{m^{6}}+ \frac{16 {L_\beta'} \R^2}{m} + \frac{32 d}{m}.
    \end{alignat*}
    
    Together, we can bound the expected distance between $y(K\delta)$ and $x_K$ under the low probability event $A_K^c$. By choosing $r$ to be sufficiently large in \eqref{e:t:wlqekwla:12}, we make \eqref{e:t:wlqekwla:3} sufficiently small. We then apply triangle inequality to bound $\dist\lrp{y(K\delta),x_K} \leq \dist\lrp{y(K\delta),x^*}+\dist\lrp{x_K,x^*}$ followed by Cauchy Schwarz, and verify that
    \begin{alignat*}{1}
        \E{\ind{A_K^c} \dist\lrp{y(K\delta),x_K}} \leq \sqrt{\delta}
        \elb{e:t:wlqekwla:0}
    \end{alignat*}

    \textbf{Step 2: Continuous Time Contraction}\\
    Having established a bound in the distance under the event $A_K^c$, we now turn our attention to the high-probability event $A_K$. 
    
    Consider some fixed but arbitrary $k\leq K$. Let us define a continuous time SDE $\bar{x}^k(t)$, for $t\in [k\delta, (k+1)\delta]$, as the solution to the exact Langevin diffusion \eqref{e:intro_euler_murayama} initialized at $\bar{x}^k(k\delta) := x_k$. 

    The goal of this step is to bound $\E{\ind{A_{k}} \dist\lrp{y((k+1)\delta),\bar{x}^k((k+1)\delta)}}$ in terms of $\E{\ind{A_k} \dist\lrp{y(k\delta),\bar{x}^k(k\delta)}} := \E{\ind{A_k} \dist\lrp{y(k\delta),x_k}}$. We apply Lemma \ref{l:g_contraction_without_gradient_lipschitz}, which guarantees that there exists a coupling between the two exact SDE processes $y$ and $\bar{x}^k$, such that $\E{f(\dist\lrp{y(t), \bar{x}^k(t)})}$ contracts with rate $\alpha$, i.e.
    \begin{alignat*}{1}
        \E{\ind{A_{k}} f\lrp{\dist\lrp{y((k+1)\delta),\bar{x}^k((k+1)\delta)}}} \leq e^{-\alpha \delta} \E{\ind{A_k}f\lrp{\dist\lrp{y(k\delta),x_{k}}}},
        \elb{e:t:wlqekwla:1}
    \end{alignat*}
    where $f$ is a Lyapunov function satisfying $f(r) \geq \frac{1}{2}\exp\lrp{- \lrp{q + L_{\Ric}/2} \R^2/2} r$ and $\lrabs{f'(r)} \leq 1$, and $\alpha := \min\lrbb{\frac{m-L_{Ric}/2}{16}, \frac{1}{2 \R^2}} \cdot \exp\lrp{- \frac{1}{2}\lrp{q + L_{\Ric}/2} \R^2}$ are as defined in Lemma \ref{l:g_contraction_without_gradient_lipschitz}, and are consistent with the definition in our theorem statement.

    \textbf{Step 3: Euler Murayama Error}\\
    Next, we bound the distance between $x_{k+1}$ and $\bar{x}^k((k+1)\delta)$. This represents the discretization error between a single Euler-Murayama step with stepsize $\delta$, and the exact Langevin diffusion over $\delta$ time.
    
    We will apply Lemma \ref{l:discretization-approximation-lipschitz-derivative} with $L_1 = L_\beta' r$. We verify that under the event $A_k$ and by Assumption \ref{ass:beta_lipschitz}, $\lrn{\beta(x_k)}$ is indeed bounded by $L_1$. Thus by Lemma \ref{l:discretization-approximation-lipschitz-derivative},
    \begin{alignat*}{1}
        &\E{\ind{A_k}\dist\lrp{x_{k+1},\bar{x}^k((k+1)\delta)}} \\
        \leq& \sqrt{2^{20}\lrp{\delta^4 L_1^4 \lrp{1+L_R}+ \delta^4 {L_\beta'}^4 + \delta^{3} \lrp{d^{3}\lrp{L_R + {L_\beta'}^2/{L_1^2}} + {L_\beta'}^2 d}}}\\
        \leq& \sqrt{2^{20}\lrp{\delta^4 {L_\beta'}^4 r^4 \lrp{1+L_R}+ \delta^4 {L_\beta'}^4 + \delta^{3} \lrp{d^{3}\lrp{L_R + 1/{\R^2}} + {L_\beta'}^2 d}}}\\
        \leq& \t{O}\lrp{\delta^{3/2}}
        \elb{e:t:wlqekwla:2}
    \end{alignat*}
    where $\t{O}$ hides polynomial dependency on $L_\beta', d, L_R, \R, \log K, \log (1/\delta)$. 
    
    Combining \eqref{e:t:wlqekwla:1} and \eqref{e:t:wlqekwla:2} and using triangle inequality and the fact that $\lrabs{f'} \leq 1$, along with the fact that $A_{k+1} \subset A_k$, we can bound
    \begin{alignat*}{1}
        \E{\ind{A_{k+1}}{f\lrp{\dist\lrp{y((k+1)\delta),x_{k+1}}}}} \leq e^{- \alpha \delta}\E{\ind{A_k}f\lrp{\dist\lrp{y(k\delta),x_{k}}}} + \t{O}\lrp{\delta^{3/2}},
        \elb{e:t:wlqekwla:4}
    \end{alignat*}
    where $\t{O}$ hides polynomial dependency on $L_\beta', d, L_R, \R, \log K, \log (1/\delta)$. This shows that, in one $\delta$-time step, the Lyapunov function of the distance contracts with rate $\alpha$, plus a discretization error of order $\delta^{3/2}$. Applying \eqref{e:t:wlqekwla:4} recursively, we can bound
    \begin{alignat*}{1}
        \E{\ind{A_{K}}{f\lrp{\dist\lrp{y(K\delta),x_{K}}}}} \leq e^{-\alpha K\delta}\E{f\lrp{\dist\lrp{y(0),x_{0}}^2}} + \frac{1}{\alpha} \cdot \t{O}\lrp{\delta^{1/2}}.
        \elb{e:t:wlqekwla:6}
    \end{alignat*}
    Combining \eqref{e:t:wlqekwla:0} and \eqref{e:t:wlqekwla:6}, and using the fact that $f(r) \leq r$, gives
    \begin{alignat*}{1}
        \E{{f\lrp{\dist\lrp{y(K\delta),x_{K}}}}} \leq e^{-\alpha K\delta}\E{f\lrp{\dist\lrp{y(0),x_{0}}^2}} + \frac{1}{\alpha} \cdot \t{O}\lrp{\delta}.
    \end{alignat*}
    Using the fact that $\frac{1}{2}\exp\lrp{- \lrp{q + L_{\Ric}/2} \R^2/2} r \leq f(r) \leq r$, we have
    \begin{alignat*}{1}
        \E{\dist\lrp{y(K\delta),x_{K}}} 
        \leq& e^{-\alpha K \delta + \lrp{q + L_{Ric}/2}\R^2/2}\E{\dist\lrp{y(0),x_{0}}} + \frac{\exp\lrp{\lrp{q + L_{Ric}/2}\R^2/2}}{\alpha} \cdot \t{O}\lrp{\delta^{1/2}}\\
        =& e^{-\alpha K \delta + \lrp{q + L_{Ric}/2}\R^2/2}\E{\dist\lrp{y(0),x_{0}}} + \exp\lrp{\lrp{q + L_{Ric}/2}\R^2} \cdot \t{O}\lrp{\delta^{1/2}},
    \end{alignat*}
    where $\t{O}$ hides polynomial dependency on $L_\beta', d, L_R, \R, \frac{1}{m - L_{Ric}/2}, \log K, \log \frac{1}{\delta}$. This concludes the proof of Theorem \ref{t:langevin_mcmc}.

    \textbf{Step 4: Verifying Conditions on $\delta$}\\
    In the proof above, we applied Lemma \ref{l:discretization-approximation-lipschitz-derivative} and Lemma \ref{l:far-tail-bound-truncated-sgld}. Each of these requires certain bounds on $\delta$. In this step, we verify that the conditions on $\delta$ for each of these lemmas is satisfied by \eqref{e:t:wlqekwla:5}.

    \textbf{Lemma \ref{l:discretization-approximation-lipschitz-derivative} with $L_1 = L_\beta' r$ requires}
    \begin{alignat*}{1}
        \delta \leq \frac{1}{16} \min\lrbb{\frac{1}{{L_\beta'}},  \frac{1}{L_R d}, \frac{1}{\sqrt{L_R} L_\beta' r}}
    \end{alignat*}
    We verify that this follows from \eqref{e:t:wlqekwla:7}. We specifically verify that $\delta \leq \frac{1}{16\sqrt{L_R} L_\beta' r}$ is satisfied due to the last two terms in \eqref{e:t:wlqekwla:7}, and Lemma \ref{l:useful_xlogx}.

    \textbf{Lemma \ref{l:far-tail-bound-truncated-sgld} requires} requires $\delta \leq \min\lrbb{\frac{m}{16 {L_\beta'}^2 \lrp{1 + \sqrt{L_R} r}}, \frac{d + \sigma^2}{m \lrp{1 + \sqrt{L_R} r}}, \frac{32 \lrp{d^2 + \sigma^4}}{m^2 r^2}}$. This is satisfied by \eqref{e:t:wlqekwla:11} and Lemma \ref{l:useful_xlogx}.
    
\end{proof}

\subsection{Proof of Theorem \ref{t:SGLD}}
\label{ss:proof_of_t:sgld}
Below, we provide the full proof of Theorem \ref{t:SGLD}, which was sketched in Section \ref{s:SGLD}. The main results used are Lemma \ref{l:sgld-lemma} (contraction of distance under Euler Murayama step) and Lemma \ref{l:discretization-approximation-lipschitz-derivative} (bound on Euler Murayama discretization error).

\begin{proof}[Proof of Theorem \ref{t:SGLD}]\ \\
    \textbf{Step 0: Defining some Key Constants}\\
    In this step, we define a radius $r$, an event $A_k$ based on $r$, and an upper bound on $\delta$.
    \begin{alignat*}{1}
        & c_0 := \log\lrp{\frac{L_\beta'}{m}} + \log \lrp{\frac{\lrp{1+L_R}\lrp{d+\sigma^2}}{m}} + \log \R + \log\lrp{K}, \\
        & r_0 = 32\sqrt{\frac{{L_\beta'}^2\R^2}{m} + \frac{L_R \lrp{d^2+\sigma^4}}{m^2} + \frac{d}{m}\cdot c_0},\\
        & r = r_0 + 32\sqrt{\frac{d}{m} \cdot \log\lrp{1/\delta}}.
        \elb{e:t:asokdmasldm:12}
    \end{alignat*}
    Let $A_k$ denote the event $\max_{i\leq k} \max\lrbb{\dist\lrp{x_i,x^*}, \dist\lrp{y(i\delta),x^*}} \leq r$. The value of $r$ and $A_k$ are chosen so that \eqref{e:t:asokdmasldm:0} holds in Step 1 below. Note that $r$ depends on $\log(1/\delta)$ on the last line.

    We now define a suitable upperbound on the stepsize $\delta$. To do so, we first define the constants $c_1,c_2,c_3,c_4,c_5,c_6,c_7$:
    \begin{alignat*}{1}
        c_1 := \min\lrbb{\frac{1}{16{L_\beta'}},  \frac{1}{16L_R d}, \frac{m}{64 d\sqrt{L_R} L_\beta' r_0}, \frac{m}{64 d\sqrt{L_R} L_\beta' \sqrt{\log \lrp{d\sqrt{L_R} L_\beta'/m}}}},
        \elb{e:t:asokdmasldm:7}
    \end{alignat*}

    \begin{alignat*}{1}
        c_2 := \min\lrbb{\frac{m - L_{Ric}/2}{128{L_\beta'}^2}, \frac{m - L_{Ric}/2}{32 L_R \sigma^2},\frac{m - L_{Ric}/2}{2^{13} L_R L_\beta' d},\frac{(m - L_{Ric}/2)^2}{2^{24} {L_R'}^2 d^3} ,\frac{m - L_{Ric}/2}{2^{17} d^2 L_R^2 }},
        \elb{e:t:asokdmasldm:8}
    \end{alignat*}
    \begin{alignat*}{1}
        c_3 :=& \min\lrbb{ \frac{m - L_{Ric}/2}{32 L_R {L_\beta'}^2 r_0^2}, \sqrt{\frac{m - L_{Ric}/2}{2^{14} \lrp{{L_\beta'}^3 r_0^3} L_R'}}},\\
        c_4 :=& \frac{1}{128} \min\lrbb{\frac{m(m - L_{Ric}/2)}{32 d L_R {L_\beta'}^2 \log\lrp{\nicefrac{ d L_R {L_\beta'}^2 }{m(m - L_{Ric}/2)}}}, \frac{(m - L_{Ric}/2^{1/2}) m^{3/4}}{{L_\beta'}^{1/2} {L_R'}^{1/2} d^{3/4} \log\lrp{\nicefrac{{L_\beta'}^{1/2} {L_R'}^{1/2} d^{3/4}}{((m - L_{Ric}/2)^{1/2} m^{3/4})}}}},
        \elb{e:t:asokdmasldm:10}
    \end{alignat*}
    \begin{alignat*}{1}
        & c_5 := \frac{1}{64}\min\lrbb{\frac{m}{{L_\beta'}^2}, \frac{d}{m}},\\
        & c_6 := \frac{1}{64}\min\lrbb{\frac{m}{{L_\beta'}^2\sqrt{L_R} r_0}, \frac{d}{m\sqrt{L_R} r_0}, \frac{d^2}{m^2 r_0^2}},\\
        & c_7 := \frac{1}{64}\min\lrbb{\sqrt{\frac{m^3}{d{L_\beta'}^4L_R^2\log\lrp{\nicefrac{{d L_\beta'}^4L_R^2}{m^3}}}}, \sqrt{\frac{d}{m L_R \log \lrp{\nicefrac{m L_R}{d}}}}, {\frac{d}{m \log\lrp{\nicefrac{d}{m}}}}}.
        \elb{e:t:asokdmasldm:11}
    \end{alignat*}

    For our proof, we require that $\delta$ satisfies
    \begin{alignat*}{1}
        \delta \leq \min\lrbb{c_1,c_2,c_3,c_4,c_5,c_6,c_7}.
        \elb{e:t:asokdmasldm:5}
    \end{alignat*}
    Thus $\C_1$ from the theorem statement is explicitly $\C_1 = \frac{1}{\min\lrbb{c_1,c_2,c_3,c_4,c_5,c_6}}$. The motivation for this upper bound on $\delta$ is so that $\delta$ satisfies the conditions in Lemma \ref{l:discretization-approximation-lipschitz-derivative}, Lemma \ref{l:sgld-lemma} and Lemma \ref{l:far-tail-bound-truncated-sgld} which are used in Steps 1-3. We provide details in Step 4 below. Note that the upper bound on $\delta$ depends only on $r_0$ but not $r$, so the definition is not circular.

    \textbf{Step 1: Tail Bound:}\\
    We now show that with high probability, the discretization sequence $x_k$ and the exact SDE $y(t)$ never step outside the ball of radius $r$ centered at $x^*$ (this is exactly the event $A_k$ that we defined in the previous step).

    By Lemma \ref{l:far-tail-bound-truncated-sgld}, for any $(\delta,r)$ satisfying $\delta \leq \min\lrbb{\frac{m}{16 {L_\beta'}^2 \lrp{1 + \sqrt{L_R} r}}, \frac{d + \sigma^2}{m \lrp{1 + \sqrt{L_R} r}}, \frac{32 \lrp{d^2 + \sigma^4}}{m^2 r^2}}$, we can bound the probability of $x_k$ stepping outside the ball of radius $r$ centered at $x^*$, for any $k\leq K$, as
    \begin{alignat*}{1}
        \Pr{\max_{k\leq K} \dist\lrp{x_k,x^*} \geq r} \leq 32K\delta m \exp\lrp{\frac{2{L_\beta'}^2\R^2}{d+\sigma^2} + \frac{64 L_R \lrp{d + \sigma^2} }{m} - \frac{m r^2}{256 \lrp{d + \sigma^2} }}.
        \elb{e:t:asokdmasldm:3}
    \end{alignat*}
    On the other hand, by Lemma \ref{l:Phi_is_diffusion}, for $T = K\delta$, there is a family of discrete sequences, $y^i_k$, corresponding to \eqref{e:intro_euler_murayama} with stepsize $\delta^i = T/2^i$, whose linear interpolation $y^i(t)$ converges to $y(t)$ uniformly almost surely. By Lemma \ref{l:far-tail-bound-truncated-sgld} with $\t{\beta} = \beta$, and $\sigma=0$, and stepsize $\delta^i$, and iteration number $K^i := 2^i$ (so that $T = K^i\delta^i$), for $i$ sufficiently large, we can bound $\Pr{\max_{k\leq K^i} \dist\lrp{y^i_k,x^*} \geq r} \leq 32T m \exp\lrp{\frac{2{L_\beta'}^2\R^2}{d} + \frac{64 L_R d }{m} - \frac{m r^2}{256 d}}$. Taking the limit of $i \to \infty$, we can bound
    \begin{alignat*}{1}
        \Pr{\max_{k\leq K} \dist\lrp{y(k\delta),x^*} \geq r} \leq 32T m \exp\lrp{\frac{2{L_\beta'}^2\R^2}{d} + \frac{32 L_R d}{m} - \frac{m r^2}{256 d}}.
        \elb{e:t:asokdmasldm:4}
    \end{alignat*}

    Combining \eqref{e:t:asokdmasldm:3} and \eqref{e:t:asokdmasldm:4}, we can bound
    \begin{alignat*}{1}
        & \E{\ind{A_k^c}} 
        = \Pr{A_k^c} 
        \leq  64K\delta m \exp\lrp{\frac{2{L_\beta'}^2\R^2}{d} + \frac{64 L_R \lrp{d + \sigma^2} }{m} - \frac{m r^2}{256 d}}.
    \end{alignat*}

    Furthermore, we can bound the fourth-moment of the distance between $x_K$ and $x^*$: by Lemma \ref{l:far-tail-bound-l2} with $\xi_k(x_k) = \sqrt{\delta} \lrp{\t{\beta}_k(x_k) - \beta(x_k)} + \zeta_k$ and $\sigma_\xi = 2\lrp{\sigma + \sqrt{d}}$, we can bound
    \begin{alignat*}{1}
        \E{\dist\lrp{x_K,x^*}^4} \leq\frac{2^{26} L_R^2 {L_\beta'}^8\lrp{\sigma^8 + d^4}}{m^{12}}  + \frac{128 {L_\beta'}^2 \R^4}{m^2} + \frac{2^{10} \lrp{d^2 + \sigma^4}}{m^2}.
    \end{alignat*}
    Similarly, we can use Lemma \ref{l:far-tail-bound-l2-brownian} to bound
    \begin{alignat*}{1}
        & \E{\dist\lrp{y(K\delta),x^*}^4} \leq \frac{2^{26} L_R^2 {L_\beta'}^8 d^4}{m^{12}}+ \frac{128 {L_\beta'}^2 \R^4}{m^2} + \frac{512 d^2}{m^2}.
    \end{alignat*}

    Combining the above, we can bound the expected squared distance between $y(K\delta)$ and $x_K$ under the low probability event $A_K^c$. By choosing $r$ to be sufficiently large in \eqref{e:t:asokdmasldm:12}, we make \eqref{e:t:asokdmasldm:4} sufficiently small. We then apply Young's inequality to bound $\dist\lrp{y(K\delta),x_K}^2\leq 2\dist\lrp{y(K\delta),x^*}^2+2\dist\lrp{x_K,x^*}^2$ followed by Cauchy Schwarz, to verify that
    \begin{alignat*}{1}
        \E{\ind{A_K^c} \dist\lrp{y(K\delta),x_K}^2} \leq \sqrt{\delta}.
        \elb{e:t:asokdmasldm:0}
    \end{alignat*}

    \textbf{Step 2: Discrete Contraction}\\
    Having established a bound in the distance under the event $A_K^c$, we now turn our attention to the high-probability event $A_K$. 
    
    Consider some fixed but arbitrary $k\leq K$. We define a useful intermediate variable, representing a single Euler-Murayama step initialized at $y(k\delta)$ (recall that $y(t)$ corresponds to the exact SDE):
    \begin{alignat*}{1}
        \bar{y}_{k+1} := \Exp_{y(k\delta)}\lrp{\delta \beta(y(k\delta)) + \sqrt{\delta} \bar{\zeta}_k}.
    \end{alignat*}
    where $\bar{\zeta}_k \sim \N_{y(k\delta)}(0,I)$. Note that $\bar{y}_{k+1}$ evolves according to the Euler-Murayama step with \emph{exact drift} $\beta(y(k\delta))$, whereas $x_{k+1}$ evolves according to the Euler-Murayama step with \emph{stochastic drift} $\t{\beta}_k(x_k)$. 

    Under the strong-convexity-like condition due to Assumption \ref{ass:distant-dissipativity} with $\R=0$, we can show that a single discrete Euler Murayama step leads to contraction in distance between $\bar{y}_{k+1}$ and $x_{k+1}$. From Lemma \ref{l:sgld-lemma}, there exists a coupling such that
    \begin{alignat*}{1}
        \E{\ind{A_{k}} \dist\lrp{\bar{y}_{k+1},x_{k+1}}^2}\leq \lrp{1  - {\delta (m - L_{Ric}/2)}} \E{\ind{A_k} \dist\lrp{y(k\delta),x_{k}}^2} + 16 \delta^2 \sigma^2.
        \elb{e:t:asokdmasldm:1}
    \end{alignat*}

    \textbf{Step 3: Euler Murayama Error}\\
    Next, we bound the distance between $\bar{y}_{k+1}$ and $y((k+1)\delta)$. This represents the discretization error between a single Euler-Murayama step with stepsize $\delta$, and the exact Langevin diffusion over $\delta$ time.
    
    We will apply Lemma \ref{l:discretization-approximation-lipschitz-derivative} with $L_1 = L_\beta' r$. We verify that under the event $A_k$, $\lrn{\beta(y(k\delta))}$ is indeed bounded by $L_1$. Thus by Lemma \ref{l:discretization-approximation-lipschitz-derivative},
    \begin{alignat*}{1}
        \E{\ind{A_k}\dist\lrp{y((k+1)\delta),\bar{y}_{k+1}}^2} 
        \leq& 2^{20}\lrp{\delta^4 L_1^4 \lrp{1+L_R}+ \delta^4 {L_\beta'}^4 + \delta^{3} \lrp{d^{3}\lrp{L_R + {L_\beta'}^2/{L_1^2}} + {L_\beta'}^2 d}}\\
        \leq& 2^{20}\lrp{\delta^4 {L_\beta'}^4 r^4 \lrp{1+L_R}+ \delta^4 {L_\beta'}^4 + \delta^{3} \lrp{d^{3}\lrp{L_R + 1/{\R^2}} + {L_\beta'}^2 d}}\\
        \leq& \t{O}\lrp{\delta^{3}}
        \elb{e:t:asokdmasldm:2}
    \end{alignat*}
    where $\t{O}$ hides polynomial dependency on $L_\beta', d, L_R, \R, \log K, \log (1/\delta)$. This shows that, in one $\delta$-time step, the distance contracts with rate $m-L_{Ric}/2$, plus a discretization error of order $\delta^2$.
    
    Combining \eqref{e:t:asokdmasldm:1} and \eqref{e:t:asokdmasldm:2} and using Young's inequality inequality, together with the fact that $A_{k+1} \subset A_k$.
    \begin{alignat*}{1}
        \E{\ind{A_{k+1}}{\dist\lrp{y((k+1)\delta),x_{k+1}}}^2} \leq e^{- \delta \lrp{m - L_{Ric}/2}}\E{\ind{A_k}{\dist\lrp{y(k\delta),x_{k}}}^2} + \t{O}\lrp{\delta^2},
    \end{alignat*}
    where $\t{O}$ hides polynomial dependency on $L_\beta', d, L_R, \R, \sigma, \log K, \log (1/\delta)$.

    Applying the above recursively, we can bound
    \begin{alignat*}{1}
        \E{\ind{A_{K}}{\dist\lrp{y(K\delta),x_{K}}}^2} \leq e^{- K \delta \lrp{m - L_{Ric}/2}}\E{{\dist\lrp{y_{0},x_{0}}}^2} + \frac{1}{m - L_{Ric}/2} \cdot \t{O}\lrp{\delta}.
        \elb{e:t:asokdmasldm:6}
    \end{alignat*}
    Combining \eqref{e:t:asokdmasldm:0} with \eqref{e:t:asokdmasldm:6} gives
    \begin{alignat*}{1}
        \E{{\dist\lrp{y(K\delta),x_{K}}}^2} \leq e^{- K\delta \lrp{m-L_{Ric}/2}}\E{{\dist\lrp{y_{0},x_{0}}}^2} + \frac{1}{m - L_{Ric}/2} \cdot \t{O}\lrp{\delta},
    \end{alignat*}
    where $\t{O}$ hides polynomial dependency on $L_\beta', d, L_R, \R, \sigma, \frac{1}{m-L_{Ric}/2}, \log K, \log\frac{1}{\delta}$. This concludes the proof of Theorem \ref{t:SGLD}.

    \textbf{Step 4: Verifying Conditions on $\delta$}\\
    In the proof above, we applied Lemma \ref{l:discretization-approximation-lipschitz-derivative}, Lemma \ref{l:sgld-lemma} and Lemma \ref{l:far-tail-bound-truncated-sgld}. Each of these requires certain bounds on $\delta$.  In this step, we verify that the conditions on $\delta$ for each of these lemmas is satisfied by \eqref{e:t:asokdmasldm:5}.

    \textbf{Lemma \ref{l:discretization-approximation-lipschitz-derivative} with $L_1 = L_\beta' r$ requires}
    \begin{alignat*}{1}
        \delta \leq \frac{1}{16} \min\lrbb{\frac{1}{{L_\beta'}},  \frac{1}{L_R d}, \frac{1}{\sqrt{L_R} L_\beta' r}}.
    \end{alignat*}
    We verify that this follows from \eqref{e:t:asokdmasldm:7}. We specifically verify that $\delta \leq \frac{1}{16\sqrt{L_R} L_\beta' r}$ is satisfied due to the last two terms in \eqref{e:t:asokdmasldm:7}, and Lemma \ref{l:useful_xlogx}.

    \textbf{Lemma \ref{l:sgld-lemma} with $L_1 = r L_\beta'$ requires}
    \begin{alignat*}{1}
        & \delta \\
        \leq& \min\lrbb{\frac{m - L_{Ric}/2}{128{L_\beta'}^2}, \frac{m - L_{Ric}/2}{32 L_R \sigma^2},\frac{m - L_{Ric}/2}{2^{13} L_R L_\beta' d},\frac{(m - L_{Ric}/2)^2}{2^{24} {L_R'}^2 d^3} ,\frac{m - L_{Ric}/2}{2^{17} d^2 L_R^2 }, \frac{m - L_{Ric}/2}{32 L_R L_\beta^2}, \sqrt{\frac{m - L_{Ric}/2}{2^{14} \lrp{L_\beta^3 + \sigma^3} L_R'}} }\\
        =& \min\lrbb{\frac{m - L_{Ric}/2}{128{L_\beta'}^2}, \frac{m - L_{Ric}/2}{32 L_R \sigma^2},\frac{m - L_{Ric}/2}{2^{13} L_R L_\beta' d},\frac{(m - L_{Ric}/2)^2}{2^{24} {L_R'}^2 d^3},\frac{m - L_{Ric}/2}{2^{17} d^2 L_R^2 }, \frac{m - L_{Ric}/2}{32 L_R {L_\beta'}^2 r^2} , \sqrt{\frac{m - L_{Ric}/2}{2^{14} \lrp{{L_\beta'}^3 r^3 + \sigma^3} L_R'}} }
        \elb{e:t:asokdmasldm:9}
    \end{alignat*}
    The first 5 bounds in \eqref{e:t:asokdmasldm:9} follow from \eqref{e:t:asokdmasldm:8}. The last two bounds in \eqref{e:t:asokdmasldm:9} follow from \eqref{e:t:asokdmasldm:10} and Lemma \ref{l:useful_xlogx}.
    
    \textbf{Lemma \ref{l:far-tail-bound-truncated-sgld} requires} requires $\delta \leq \min\lrbb{\frac{m}{16 {L_\beta'}^2 \lrp{1 + \sqrt{L_R} r}}, \frac{d + \sigma^2}{m \lrp{1 + \sqrt{L_R} r}}, \frac{32 \lrp{d^2 + \sigma^4}}{m^2 r^2}}$. This is satisfied by \eqref{e:t:asokdmasldm:11} and Lemma \ref{l:useful_xlogx}.
    
\end{proof}

The following lemma shows that, for any two initial points $x$ and $y$, if $x$ undergoes an \emph{exact} Euler Murayama step with drift $\beta$, and $y$ undergooes a \emph{stochastic} Euler Murayama step with drift $\t{\beta}$, then their expected squared distance contracts, with rate $m - L_{Ric}/2$, plus an additional error of $\delta^2 \sigma^2$, where $\sigma = \lrn{\t{\beta} - \beta}$. This lemma is somewhat analogous to Lemma \ref{l:g_contraction_without_gradient_lipschitz} which shows contraction under the exact SDE, though Lemma \ref{l:sgld-lemma} also requires a fair amount of additional discretization analysis.

The key result used in the proof of Lemma \ref{l:sgld-lemma} is Lemma \ref{l:discrete-approximate-synchronous-coupling-ricci}.

\begin{lemma}\label{l:sgld-lemma}
    Let $M$ satisfy Assumption \ref{ass:ricci_curvature_regularity}, Assumption \ref{ass:sectional_curvature_regularity}. Assume in addition that there exists a constant $L_R'$ such that for all $x\in M$, $u,v,w,z,a\in T_x M$, $\lin{(\nabla_a R)(u,v)w,z} \leq L_R'\lrn{u}\lrn{v}\lrn{w}\lrn{z}\lrn{a}$. Let $\beta$ be a deterministic vector field satisfying Assumption \ref{ass:beta_lipschitz} and Assumption \ref{ass:distant-dissipativity}, with $\R = 0$. Let $x,y\in M$ be arbitrary. Let $\t{\beta}$ be a random vector field such that $\E{\t{\beta}} = \beta$. Assume that there exists $\sigma \in \Re^+$ such that $\lrn{\t{\beta}(y) - \beta(y)} \leq \sigma$. Let $L_\beta :=\max\lrbb{\lrn{\beta(x)}, \lrn{\beta(y)}}$. Let $x' = \Exp_{x}\lrp{\delta \beta(x) + \sqrt{\delta} \zeta}$. Let $y' = \Exp_{y} \lrp{\delta \t{\beta}(y) + \sqrt{\delta}\t{\zeta}}$, where $\zeta \sim \N_x(0,I)$ and $\t{\zeta} \sim \N_y(0,I)$.
    
    Assume that 
    \begin{alignat*}{1}
        \delta \leq \min\lrbb{\frac{m - L_{Ric}/2}{128{L_\beta'}^2}, \frac{m - L_{Ric}/2}{32 L_R L_\beta^2}, \frac{m - L_{Ric}/2}{32 L_R \sigma^2},\frac{m - L_{Ric}/2}{2^{13} L_R L_\beta' d},\frac{(m - L_{Ric}/2)^2}{2^{24} {L_R'}^2 d^3} , \sqrt{\frac{m - L_{Ric}/2}{2^{14} \lrp{L_\beta^3 + \sigma^3} L_R'}},\frac{m - L_{Ric}/2}{2^{17} d^2 L_R^2 } }.
    \end{alignat*}
    Then there is a coupling (synchronous coupling) between $\zeta$ and $\t{\zeta}$ such that
    \begin{alignat*}{1}
        \E{\dist\lrp{x', y'}^2} \leq \lrp{1  - {\delta (m - L_{Ric}/2)}} \dist\lrp{x,y}^2 + 16 \delta^2 \sigma^2
    \end{alignat*}
\end{lemma}
\textbf{Note:} elsewhere in this paper, we have used $L_\beta$ do denote a Lipschitz constant for $\beta$; the use of $L_\beta$ in Lemma \ref{l:sgld-lemma} is different (but related).
\begin{proof}
    Let $\gamma(s):[0,1] \to M$ be a minimizing geodesic between $x$ and $y$ with $\gamma(0) = x$ and $\gamma(1) = y$, such that $\linp{\party{y}{x}{\beta(y)-\beta(x)}, \gamma'(0)} \leq -m \dist\lrp{x,y}^2$. (Assumption \ref{ass:distant-dissipativity} guarantees the existence of such a $\gamma$.) 

    \textbf{Step 1: Synchronous Coupling of $\zeta$ and $\t{\zeta}$}\\
    We will now define a coupling between $\zeta$ and $\t{\zeta}$. Let $E$ be an orthonormal basis at $T_{x} M$, and let $F$ be the parallel transport of $E$ along $\gamma$, i.e. $F$ is an orthonormal basis for $T_y M$. Let $\zzeta \sim \N(0,I)$ be a standard Gaussian random variable in $\Re^d$, and define $\zeta := \zzeta \circ E$, and it follows by definition that $\zeta$ so defined has distribution $\N_x(0,I)$. Let $\t{\zeta} := \t{\zzeta} \circ F$, it follows by definition that $\t{\zeta}$ has distribution $\N_y(0,I)$.
    
    \textbf{Step 2: Applying Lemma \ref{l:discrete-approximate-synchronous-coupling-ricci} and Simplifications}\\
    We will apply Lemma \ref{l:discrete-approximate-synchronous-coupling-ricci} with $u = \delta \beta(x) + \sqrt{\delta} \zeta$ and $v = \delta \t{\beta}(y) + \sqrt{\delta} \t{\zeta}$.
    Then
    \begin{alignat*}{1}
        &\dist\lrp{\Exp_{x}(u), \Exp_y(v)}^2 - \dist\lrp{x,y}^2\\
        \leq& \underbrace{2\lin{\gamma'(0), v(0) - u(0)}}_{\circled{1}} + \underbrace{\lrn{v(0) - u(0)}^2}_{\circled{2}}\\
        &\quad \underbrace{-\int_0^1 \lin{R\lrp{\gamma'(s),(1-s) u(s) + s v(s)}(1-s) u(s) + s v(s),\gamma'(s)} ds}_{\circled{3}} \\
        &\quad + \underbrace{\lrp{2\C^2 e^{\C} + 18\C^4 e^{2\C}} \lrn{v(0) - u(0)}^2 + \lrp{18\C^4 e^{2\C} + 4\C'} \dist\lrp{x,y}^2 + 4 \C^2 e^{2\C} \dist\lrp{x,y} \lrn{v(0) - u(0)}}_{\circled{4}}
        \elb{e:t:bound_circled_0}
    \end{alignat*}
    where $\C := \sqrt{L_R} \lrp{\lrn{u} + \lrn{v}}$ and $\C' := L_R' \lrp{\lrn{u} + \lrn{v}}^3$, and $u(t)$ and $v(t)$ are parallel trannsport of $u$ and $v$ along $\gamma$, as defined in Lemma \ref{l:discrete-approximate-synchronous-coupling-ricci}. Some notes on notation: \\
    1. We will use $\party{x}{y}$ and $\party{\gamma(s)}{\gamma(t)}$ to denote parallel transport along $\gamma$.\\
    2. In subsequent parts, for $i=1,2....$, we will use
        \begin{enumerate}
            \item $\tau_i$ to denote terms which depend super-linearly on $\delta$. 
            \item $\xi_i$, to denote terms which have $0$ expectation.
            \item $\theta_i$ to denote terms which depend linearly on $\delta$, and have non-zero expectation (i.e. the important terms).
        \end{enumerate}
    \textbf{Step 2.1: Simplifying $\circled{1}$}\\
    By definition, $v(0) - u(0) = \delta \lrp{\party{y}{x} \t{\beta}(y) - \beta(x)}$, thus
    \begin{alignat*}{1}
        \circled{1} 
        =& \underbrace{2\delta \lin{\party{y}{x} \beta(y) - \beta(x), \gamma'(0)}}_{:=\theta_1} + \underbrace{2\delta \lin{\party{y}{x} \t{\beta}(y) - \party{y}{x} {\beta}(y), \gamma'(0)}}_{:=\xi_1}
        \elb{e:t:bound_circled_1}
    \end{alignat*}
    \textbf{Step 2.2: Simplifying $\circled{2}$}\\
    By similar algebra as Step 2.1,
    \begin{alignat*}{1}
        \circled{2} = \underbrace{\delta^2 \lrn{\party{y}{x}\t{\beta}(y) - \beta(x)}^2}_{:=\tau_1} \leq \delta^2 {L_\beta'}^2 \dist\lrp{x,y}^2 + \delta^2 \sigma^2
        \elb{e:t:bound_circled_2}
    \end{alignat*}

    \textbf{Step 2.3: Simplifying $\circled{3}$}\\
    \begin{alignat*}{1}
        \circled{3} =& -\int_0^1 \lin{R\lrp{\gamma'(s),(1-s) u(s) + s v(s)}(1-s) u(s) + s v(s),\gamma'(s)} ds\\
        =& \underbrace{-\delta \int_0^1 \lin{R\lrp{\gamma'(s), \party{x}{\gamma(s)}\zeta} \party{x}{\gamma(s)}\zeta,\gamma'(s)} ds}_{:=\theta_2}\\
        &\quad \underbrace{-\delta^2 \int_0^1 \lin{R\lrp{\gamma'(s), \party{x}{\gamma(s)}\beta(x)}\party{x}{\gamma(s)}\beta(x) ,\gamma'(s)} ds}_{:=\tau_2}\\
        &\quad \underbrace{-2\delta^{3/2} \int_0^1 \lin{R\lrp{\gamma'(s), \party{x}{\gamma(s)}\zeta} \party{x}{\gamma(s)}\beta(x),\gamma'(s)} ds}_{:=\xi_2}\\
        &\quad \underbrace{-2\delta^2 \int_0^1 s \lin{R\lrp{\gamma'(s),  \party{x}{\gamma(s)}\beta(x)}\party{y}{\gamma(s)}\t{\beta}(y) - \party{x}{\gamma(s)}\beta(x),\gamma'(s)} ds}_{:=\tau_3}\\
        &\quad \underbrace{-2\delta^{3/2} \int_0^1 s \lin{R\lrp{\gamma'(s), \party{x}{\gamma(s)} \zeta }\party{y}{\gamma(s)}\t{\beta}(y) - \party{x}{\gamma(s)}\beta(x),\gamma'(s)} ds}_{:=\xi_3}.
        \elb{e:t:bound_circled_3}
    \end{alignat*}

    We will now bound $\tau_2$ and $\tau_3$. By Assumption \ref{ass:beta_lipschitz}, 
    \begin{alignat*}{1}
        & \lrabs{\tau_2} \leq 2\delta^2 L_R L_\beta^2 \dist\lrp{x,y}^2\\
        & \lrabs{\tau_3} \leq 4\delta^2 L_R \lrp{L_\beta^2 + \sigma^2}  \dist\lrp{x,y}^2
    \end{alignat*}

    \textbf{Step 2.4: Simplifying \circled{4}}\\
    Since $\circled{4}$ has quite a few terms, we will bound them one by one:
    \begin{alignat*}{1}
        \circled{4} 
        =& \underbrace{2\C^2 e^{\C} \lrn{v(0) - u(0)}^2}_{:=\tau_5} + \underbrace{18\C^4 e^{2\C}\lrn{v(0) - u(0)}^2}_{:=\tau_6} + \underbrace{18\C^4 e^{2\C} \dist\lrp{x,y}^2}_{:=\tau_7}\\
        &\quad + \underbrace{4\C' \dist\lrp{x,y}^2}_{:=\tau_8} + \underbrace{4 \C^2 e^{2\C} \dist\lrp{x,y} \lrn{v(0) - u(0)}}_{:=\tau_9}.
    \end{alignat*}
    Recall that $\C := \sqrt{L_R} \lrp{\lrn{u} + \lrn{v}}$ and $\C' := L_R' \lrp{\lrn{u} + \lrn{v}}^3$. Following previous calculations, we can bound $\lrn{v(0) - u(0)} \leq \delta \sigma + \delta {L_\beta'} \dist\lrp{x,y}$. We can also bound $\lrn{u} \leq \delta L_\beta + \sqrt{\delta}\lrn{\zeta}$ and $\lrn{v} \leq \delta L_\beta + \delta \sigma + \sqrt{\delta} \lrn{\zeta}$. Thus
    \begin{alignat*}{1}
        \tau_5 =& 2\C^2 e^\C \lrn{v(0) - u(0)}^2\\
        \leq& 4 L_R \lrp{\lrn{u}^2 + \lrn{v}^2} e^{\sqrt{L_R} \lrp{\lrn{u} + \lrn{v}}} \lrn{v(0) - u(0)}^2\\
        \leq& 16 L_R \lrp{\delta^2 \lrp{2L_\beta^2 + \sigma^2} + \delta \lrn{\zeta}^2 } \cdot \exp\lrp{\sqrt{L_R} \lrp{\delta\lrp{2 L_\beta + \sigma} + 2\sqrt{\delta} \lrn{\zeta}}}\\
        &\quad \cdot \lrp{ \delta L_\beta + \sqrt{\delta}\lrn{\zeta}}^2\\
        \leq& 128 L_R e^{2\sqrt{\delta} \lrn{\zeta}} \cdot\lrp{ \delta^{2}\lrn{\zeta}^4 + \delta^{4} \lrp{L_\beta^4 + \sigma^4}},\\
        {\E{\tau_5^4}}^{1/4}
        \leq& 512 L_R \cdot\lrp{ \delta d + \delta^{2} \lrp{L_\beta^2 + \sigma^2}} \lrp{ \delta^2 \sigma^2 + \delta^2 {L_\beta'}^2 \dist\lrp{x,y}^2},
    \end{alignat*}
    where for the third line, we use the fact that our bound on $\delta$ implies that $\delta \leq \min\lrbb{\frac{1}{32\sqrt{L_R} L_\beta}, \frac{1}{32\sqrt{L_R} \sigma}}$.

    By similar algebra, we verify that
    \begin{alignat*}{2}
        &\tau_6 
        &&= 18\C^4 e^{2\C} \lrn{v(0) - u(0)}^2\\
        & && \leq 2048 L_R^2 e^{4\sqrt{\delta} \lrn{\zeta}} \cdot\lrp{ \delta^{2}\lrn{\zeta}^4 + \delta^{4} \lrp{L_\beta^4 + \sigma^4}}\lrp{ \delta^2 \sigma^2 + \delta^2 {L_\beta'}^2 \dist\lrp{x,y}^2},\\
        &\tau_7 
        &&= 18\C^4 e^{2\C} \dist\lrp{x,y}^2\\
        & && \leq 2048 L_R^2 e^{4\sqrt{\delta} \lrn{\zeta}} \cdot\lrp{ \delta^{2}\lrn{\zeta}^4 + \delta^{4} \lrp{L_\beta^4 + \sigma^4}} \cdot \dist\lrp{x,y}^2,\\
        &\tau_8 
        &&= 4\C' \dist\lrp{x,y}^2\\
        & && \leq 128 L_R' e^{4\sqrt{\delta} \lrn{\zeta}} \cdot\lrp{ \delta^{3/2}\lrn{\zeta}^3 + \delta^3 \lrp{L_\beta^3 + \sigma^3}} \cdot \dist\lrp{x,y}^2,\\
        &\tau_9 
        &&=  4 \C^2 e^{2\C} \dist\lrp{x,y} \lrn{v(0) - u(0)}\\
        & && \leq 128 L_R e^{4\sqrt{\delta} \lrn{\zeta}} \cdot\lrp{ \delta \lrn{\zeta}^2 + \delta^2 \lrp{L_\beta^2 + \sigma^2}} \cdot \lrp{\delta {L_\beta'} \dist\lrp{x,y}^2 + \delta \sigma \dist\lrp{x,y}}, \\
        & {\E{\tau_6^4}}^{1/4} 
        &&\leq 2^{14} L_R^2 \cdot\lrp{ \delta^2d^2 + \delta^{4} \lrp{L_\beta^4 + \sigma^4}}\lrp{ \delta^2 \sigma^2 + \delta^2 {L_\beta'}^2 \dist\lrp{x,y}^2},\\
        & {\E{\tau_7^4}}^{1/4} 
        &&\leq 2^{14} L_R^2 \cdot \lrp{ \delta^2 d^2 + \delta^{4} \lrp{L_\beta^4 + \sigma^4}} \cdot \dist\lrp{x,y}^2,\\
        & {\E{\tau_8^4}}^{1/4} 
        &&\leq 512 L_R' \cdot\lrp{ \delta^{3/2}d^{3/2} + \delta^{3} \lrp{L_\beta^3 + \sigma^3}} \cdot \dist\lrp{x,y}^2,\\
        & {\E{\tau_9^4}}^{1/4} 
        &&\leq 512 L_R \cdot\lrp{ \delta d + \delta^{2} \lrp{L_\beta^2 + \sigma^2}} \cdot \lrp{\delta {L_\beta'} \dist\lrp{x,y}^2 + \delta \sigma \dist\lrp{x,y}},
        \elb{e:t:bound_t8_t9_t10_t11}
    \end{alignat*}
    where we use Lemma \ref{l:subexponential-chi-square} and the fact that our assumption on $\delta$ implies that $\delta \leq \frac{1}{2^{14} L_R d}$.

    \textbf{Step 3: Putting Things Together}\\
    Let $\E{\cdot}$ denote expectation wrt the $\zeta$ and $\t{\beta}$ (which is a random function).  By Assumption \ref{ass:distant-dissipativity},
    \begin{alignat*}{1}
        \E{\theta_1} \leq - 2 \delta m \dist\lrp{x,y}^2.
    \end{alignat*}
    By definition of Ricci curvature and by Assumption \ref{ass:ricci_curvature_regularity}, 
    \begin{alignat*}{1}
        \E{\theta_2} \leq \delta L_{Ric}.
    \end{alignat*}

    Using our assumed bound on $\delta$, we verify that
    \begin{alignat*}{1}
        \E{\tau_1 + \tau_2 + \tau_3 + \tau_5 + \tau_6 + \tau_7 + \tau_8 + \tau_9} \leq \frac{\delta (m - L_{Ric}/2)}{2} \dist\lrp{x,y}^2 + 16 \delta^2 \sigma^2.
    \end{alignat*}
    We omit the proof for this fact, since it follows by basic but long algebra, but note that we need to apply Young's inequality at several points. We also verify that
    \begin{alignat*}{1}
        \E{\xi_1} = \E{\xi_2} = \E{\xi_3} = 0.
    \end{alignat*}

    Plugging everything into \eqref{e:t:bound_circled_0}, we get
    \begin{alignat*}{1}
        \E{\dist\lrp{x', y'}^2} \leq \lrp{1  - \delta \lrp{m - L_{Ric}/2}} \dist\lrp{x,y}^2 + 16 \delta^2 \sigma^2.
    \end{alignat*}

\end{proof}

\section{Distance Contraction under Kendall Cranston Coupling}
\label{sec:distance-contraction}
In this section, we prove Lemma~\ref{l:g_contraction_without_gradient_lipschitz}, which is the main tool for proving mixing of manifold diffusion processes under the distant dissipativity assumption. We note again that the proof is entirely based on existing results from \citep{eberle2016reflection,hsu2002stochastic}, and is only included for completeness.

\subsection{The Kendall Cranston Coupling}
\label{ss:The Kendall Cranston Coupling}

\begin{lemma}\label{l:kendall-cranston}
    Let $T\in \Re^+$ be some fixed time. Assume that there is are constants $L_\beta, L_\beta'$ such that for all $x,y\in M$, $\lrn{\beta(x)} \leq L_\beta$ and $\lrn{\beta(x) - \party{y}{x} \beta(y)} \leq L_\beta' \lrn{x-y}$. Let $i$ be some integer satisfying 
    $i \geq \max\lrbb{\log_2 \lrp{32T \sqrt{L_R} L_\beta}, \log_2 \lrp{32Td}, \log_2\lrp{32L_\beta T}}$.

    For any $x,y$, let $\Lambda(x,y)$ denote the set of minimizing geodesics from $x$ to $y$, i.e. for any $\gamma \in \Lambda(x,y)$, $\gamma(0) = x$, $\gamma(1) = y$, $\forall t, \nabla_{\gamma'(t)}\gamma'(t) = 0$ and $\dist\lrp{x,y} = \lrn{\gamma'(0)}$. Let $\kappa(r):= \frac{1}{r^2} \sup_{\dist\lrp{x,y} =  r} \inf_{\gamma \in \Lambda(x,y)} \lin{\party{y}{x} \beta(y) - \beta(x), \gamma'(0)}$. 
    
    Let $x,y\in M$ and let $E^x$ be an arbitrary orthonormal basis of $T_x M$ and let $E^y$ be an arbitrary orthonormal basis of $T_y$.  let $x^i(t):= \overline{\Phi}(t;x,E^x,\beta,\BB^x,i)$ and $y^i(t) := \overline{\Phi}(t;y,E^y,\beta,\BB^y,i)$ where $\BB^x$ and $\BB^y$ are standard Brownian motion in $\Re^d$, and where $\overline{\Phi}$ is as defined in \eqref{d:x^i(t)}. 
    
    For any $\epsilon$, there exists a coupling between $\BB^x$ and $\BB^y$, and Brownian motion $\WW^i$ over $\Re$, such that for all $k\in \lrbb{0...2^i}$,
    \begin{alignat*}{1}
        \dist\lrp{x^i_{k+1}, y^i_{k+1}}^2 
        \leq& \lrp{1+\delta^i \lrp{2\kappa\lrp{\dist\lrp{x^i_{k}, y^i_{k}}} + L_{Ric}}} \dist\lrp{x^i_{k}, y^i_{k}}^2 \\
        &\quad + \ind{\dist\lrp{x^i_{k}, y^i_{k}} > \epsilon}\lrp{4 \delta^i- 4 \dist\lrp{x^i_{k}, y^i_{k}}  \lrp{\WW^i\lrp{(k+1)\delta^i} - \WW^i\lrp{k\delta^{i}} }}\\
        &\quad + \tau^i_k
    \end{alignat*}
    where $\tau^i_k$ satisfies
    \begin{alignat*}{1}
        & \Ep{\F_k}{\lrabs{\tau^i_k}} \leq \C_1 {\delta^i}^{3/2} \lrp{1 + L_\beta^4}\lrp{1+ \dist\lrp{x^i_k,y^i_k}^2}\\
        & \Ep{\F_k}{{{\tau}^i_k}^2 }\leq \C_1 {\lrp{1+\dist\lrp{x^i_k,y^i_k}^4} {\delta^i}^{2}}
    \end{alignat*}
    where $\C_1$ is a constant depending on $L_R, L_R', d, T$.
\end{lemma}

\begin{proof}
    We set up some notation: throughout this proof, consider a fixed $i$. Recall that $\delta^i:= T/2^i$, and assume $i$ is large enough such that $\delta^i \leq \frac{1}{32\sqrt{L_R}L_\beta}$. Let $x^i_k$ be as defined in \eqref{d:x^i_k} so that $x^i_k = x^i(k\delta^i)$. Let us also define $K := 2^i$, so that $T = K\delta^i$.

    \textbf{Step 1: defining the coupling}
    By definition, for any $k\in \lrbb{0...K}$,
    \begin{alignat*}{1}
        & x^{i}_{k+1} := \Exp_{x^{i}_{k}}\lrp{\delta^i \beta\lrp{x^i_k} + {\lrp{\BB^x\lrp{(k+1)\delta^i} - \BB^x\lrp{k\delta^{i}}}} \circ E^{i}_{k}}\\
        & y^{i}_{k+1} := \Exp_{y^{i}_{k}}\lrp{\delta^i \beta\lrp{y^i_k} + {\lrp{\BB^y\lrp{(k+1)\delta^i} - \BB^y\lrp{k\delta^{i}}}} \circ \t{E}^{i}_{k}}
    \end{alignat*}

    Let $\gamma^i_k:[0,1] \to M$ denote a minimizing geodesic from $x^i_k$ to $y^i_k$.

    Let $F^i_k$ be an orthonormal basis at $T_{y^i_k} M$, obtained from the parallel transport of $E^i_k$ along $\gamma^i_k$, i.e. for all $j=1...d$,
    \begin{alignat*}{1}
        F^{i,j}_k = \party{\gamma^i_k}{} E^{i,j}_k
    \end{alignat*}
    Let us define $\MM^i_k \in \Re^{d\times d}$ as matrix whose $a,b$ entry is 
    \begin{alignat*}{1}
        \lrb{\MM^i_k}_{a,b} = \lin{F^{i,a}_k, \t{E}^{i,b}_k}
    \end{alignat*}
    one can verify that $\MM^i_k$ is an orthogonal matrix, and that for all $\vv\in \Re^d$, $\vv \circ F^i_k = \MM \vv \circ \t{E}^i_k$. 

    Let us define $\bar{\nnu}^i_k$ denote the unique coordinates of $\frac{{\gamma^i_k}'(1)}{\lrn{{\gamma^i_k}'(1)}}$ wrt $F^i_k$ (equivalently the coordinates of $\frac{{\gamma^i_k}'(0)}{\lrn{{\gamma^i_k}'(0)}}$ wrt $E^i_k$). We define $\nnu^i_k := \ind{\dist\lrp{x^i_{k}, y^i_{k}} > \epsilon} \bar{\nnu}^i_k$.
    
    We now define a coupling between $\BB^x(t)$ and $\BB^y(t)$ as follows:
    \begin{alignat*}{1}
        \BB^y(t) := \int_{0}^{T} \ind{t\in[k\delta^i,(k+1)\delta^i)}\MM^i_k  \lrp{I - 2 \nnu^i_k {\nnu^i_k}^T}  d \BB^x(t)
    \end{alignat*}
    For this to be a valid coupling, it suffices to verify that \\
    $\int_{0}^{T} \ind{t\in[k\delta^i,(k+1)\delta^i)} \MM^i_k \lrp{I - 2 \nnu^i_k {\nnu^i_k}^T} d \BB^x(t)$ is indeed a standard Brownian motion. This can be done by verifying that the definition satisfies Levy's characterization of Brownian motion. We omit the proof, but highlight two important facts: 1. $\int_{0}^{T} \ind{t\in[k\delta^i,(k+1)\delta^i)} \MM^i_k\lrp{I - 2 \nnu^i_k {\nnu^i_k}^T} $ is adapted to the natural filtration of $\BB^x(t)$, and 2. $\MM^i_k\lrp{I - 2 \nnu^i_k {\nnu^i_k}^T} $ is an orthogonal matrix. We have thus defined a coupling between $\BB^x$ and $\BB^y$, and consequently, a coupling between $x^i(t)$ and $y^i(t)$ for all $t$.

    \textbf{Step 2: Applying Lemma \ref{l:discrete-approximate-synchronous-coupling-ricci}}

    Having defined a coupling between $x^i_k$ and $y^i_k$, we bound $\E{\dist\lrp{x^i_K, y^i_K}^2}$ for $K := T/\delta^i = 2^i$ by applying Lemma \ref{l:discrete-approximate-synchronous-coupling-ricci} , with $x=x^i_k$, $y=y^i_k$, $u= \delta^i \beta\lrp{x^i_k} + {\lrp{\BB\lrp{(k+1)\delta^i} - \BB\lrp{k\delta^{i}}}} \circ E^{i}_{k}$, $v = \delta^i \beta\lrp{y^i_k} + {\lrp{\t{\BB}\lrp{(k+1)\delta^i} - \t{\BB}\lrp{k\delta^{i}}}} \circ \t{E}^{i}_{k}$ and $\gamma := \gamma^i_k$.

    Following the notation in Lemma \ref{l:discrete-approximate-synchronous-coupling-ricci}, let $u(t)$ and $v(t)$ be the parallel transport of $u$ and $v$ along $\gamma(t)$. We verify that $u(s) = \delta^i \party{x^i_k}{\gamma^i_k(s)} \beta(x^i_k) + \lrp{\BB\lrp{(k+1)\delta^i} - \BB\lrp{k\delta^{i}}} \circ \party{x^i_k}{\gamma^i_k(s)} E^i_k$ and that
    \begin{alignat*}{1}
        v(s) 
        =& \delta^i \party{y^i_k}{\gamma^i_k(s)} \beta(y^i_k) +  \MM^i_k\lrp{I - 2 \nnu^i_k {\nnu^i_k}^T}\lrp{\BB\lrp{(k+1)\delta^i} - \BB\lrp{k\delta^{i}}} \circ \party{x^i_k}{\gamma^i_k(s)} \t{E}^i_k\\
        =& \delta^i \party{y^i_k}{\gamma^i_k(s)} \beta(y^i_k) + \lrp{I - 2 \nnu^i_k {\nnu^i_k}^T } \lrp{\BB\lrp{(k+1)\delta^i} - \BB\lrp{k\delta^{i}}} \circ \party{y^i_k}{\gamma^i_k(s)} F^i_k\\
        =& \delta^i \party{y^i_k}{\gamma^i_k(s)} \beta(y^i_k) + \lrp{I - 2 \nnu^i_k {\nnu^i_k}^T }\lrp{\BB\lrp{(k+1)\delta^i} - \BB\lrp{k\delta^{i}}} \circ \party{x^i_k}{\gamma^i_k(s)} E^i_k\\
        =& \delta^i \party{y^i_k}{\gamma^i_k(s)} \beta(y^i_k) + \lrp{\BB\lrp{(k+1)\delta^i} - \BB\lrp{k\delta^{i}}} \circ \party{x^i_k}{\gamma^i_k(s)} E^i_k\\
        &\quad  - 2 \lin{\nnu^i_k, \BB\lrp{(k+1)\delta^i} - \BB\lrp{k\delta^{i}}} \frac{{\gamma^i_k}'(s)}{\lrn{{\gamma^i_k}'(s)}}
    \end{alignat*}
    where the second equality is by definition of $\MM^i_k$, the third equality is by definition of $F^i_k$, the fourth equality is by definition of $\nnu^i_k$ and the fact that $\gamma^i_k$ is a geodesic. It is convenient subsequently to note the following:
    \begin{alignat*}{1}
        & v(s) - u(s)\\
        =& \delta^i \lrp{\party{y^i_k}{\gamma^i_k(s)} \beta(y^i_k) - \party{x^i_k}{\gamma^i_k(s)} \beta(x^i_k)} - 2 \lin{\nnu^i_k, \BB\lrp{(k+1)\delta^i} - \BB\lrp{k\delta^{i}}} \frac{{\gamma^i_k}'(s)}{\lrn{{\gamma^i_k}'(s)}}
    \end{alignat*}
    and
    \begin{alignat*}{1}
        & (1-s) u(s) + s v(s) \\
        =& (1-s) \delta^i \party{x^i_k}{\gamma^i_k(s)} \beta(x^i_k) + s \delta^i \party{y^i_k}{\gamma^i_k(s)} \beta(y^i_k)\\
        &\quad + \lrp{\BB\lrp{(k+1)\delta^i} - \BB\lrp{k\delta^{i}}} \circ \party{x^i_k}{\gamma^i_k(s)} E^i_k \\
        &\quad - 2 s \lin{\nnu^i_k, \BB\lrp{(k+1)\delta^i} - \BB\lrp{k\delta^{i}}} \frac{{\gamma^i_k}'(s)}{\lrn{{\gamma^i_k}'(s)}}
    \end{alignat*}

    \textbf{Step 3: Reorganizing Lemma \ref{l:discrete-approximate-synchronous-coupling-ricci}}\\
    With $u,v$ as defined above, Lemma \ref{l:discrete-approximate-synchronous-coupling-ricci} implies that
    \begin{alignat*}{1}
        &\dist\lrp{x^i_{k+1}, y^i_{k+1}}^2 - \dist\lrp{x^i_{k}, y^i_{k}}^2\\
        \leq& 2\lin{{\gamma^i_k}'(0), v(0) - u(0)} + \lrn{v(0) - u(0)}^2 \\
        &\quad -\int_0^1 \lin{R\lrp{{\gamma^i_k}'(s),(1-s) u(s) + s v(s)}(1-s) u(s) + s v(s),{\gamma^i_k}'(s)} ds \\
        &\quad + \lrp{2\C^2 e^{\C} + 18\C^4 e^{2\C}} \lrn{v(0) - u(0)}^2 + \lrp{18\C^4 e^{2\C} + 4\C'} \dist\lrp{x^i_{k}, y^i_{k}}^2\\
        &\quad + 4 \C^2 e^{2\C} \dist\lrp{x^i_{k}, y^i_{k}} \lrn{v(0) - u(0)}
        \elb{e:l:kendall-cranston:step1.1}
    \end{alignat*}
    where $\C := \sqrt{L_R} \lrp{\lrn{u} + \lrn{v}}$ and $\C' := L_R' \lrp{\lrn{u} + \lrn{v}}^3$.

    Below, we bound each of the terms above
    \begin{alignat*}{2}
        & 2\lin{{\gamma^i_k}'(0), v(0) - u(0)} 
        &&= 2 \delta^i \lin{{\gamma^i_k}'(0), \party{y^i_k}{x^i_k} \beta(y^i_k) - \beta(x^i_k)} - 4 \lrn{{\gamma^i_k}'(0)}\lin{\nnu^i_k, \BB\lrp{(k+1)\delta^i} - \BB\lrp{k\delta^{i}}}\\
        & \lrn{v(0) - u(0)}^2 &&\leq 4 \lin{\nnu^i_k, \BB\lrp{(k+1)\delta^i} - \BB\lrp{k\delta^{i}}}^2\\
        & &&\quad + \underbrace{{\delta^i}^2 L_\beta^2 + 4 \delta^i L_\beta \lrabs{\lin{\nnu^i_k, \BB\lrp{(k+1)\delta^i} - \BB\lrp{k\delta^{i}}}}}_{\tau^i_{k,1}}
    \end{alignat*}
    \begin{alignat*}{1}
        & - \int_0^1 \lin{R\lrp{{\gamma^i_k}'(s),(1-s) u(s) + s v(s)}(1-s) u(s) + s v(s),{\gamma^i_k}'(s)} ds \\
        \leq& -\int_0^1 \lin{R\lrp{{\gamma^i_k}'(s),\lrp{\BB\lrp{(k+1)\delta^i} - \BB\lrp{k\delta^{i}}}\circ \party{x^i_k}{\gamma^i_k(s)} E^i_k }\lrp{\BB\lrp{(k+1)\delta^i} - \BB\lrp{k\delta^{i}}}\circ \party{x^i_k}{\gamma^i_k(s)} E^i_k  , {\gamma^i_k}'(s)} ds\\
        &\quad + \underbrace{{\delta^i}^2 L_R \dist\lrp{x^i_k,y^i_k}^2 L_\beta^2 + 4 \delta^i L_R \dist\lrp{x^i_k,y^i_k}^2 L_\beta \lrn{\BB\lrp{(k+1)\delta^i} - \BB\lrp{k\delta^{i}}}_2}_{\tau^i_{k,2}}
    \end{alignat*}
    In the first equality above, we crucially use the fact that $\lin{\nnu^i_k, \BB\lrp{(k+1)\delta^i} - \BB\lrp{k\delta^{i}}} \frac{{\gamma^i_k}'(s)}{\lrn{{\gamma^i_k}'(s)}}$ is a scalar multiple of $\gamma_k'(s)$, and the fact that $\lin{R(u,u)v,u} = \lin{R(u,v)u,u} = 0$ for all $u,v$ by symmetry of the Riemannian curvature tensor.

    Finally, we will take the remaining terms, and denote them by
    \begin{alignat*}{1}
        \tau^i_{k,3} :=& \lrp{2\C^2 e^{\C} + 18\C^4 e^{2\C}} \lrn{v(0) - u(0)}^2 \\
        &\quad + \lrp{18\C^4 e^{2\C} + 4\C'} \dist\lrp{x^i_{k}, y^i_{k}}^2 + 4 \C^2 e^{2\C} \dist\lrp{x^i_{k}, y^i_{k}} \lrn{v(0) - u(0)}
    \end{alignat*}

    We claim that under our assumption on $i$,
    \begin{alignat*}{1}
        \Ep{\F_k}{\lrabs{\tau^i_{k,1} + \tau^i_{k,2} + \tau^i_{k,3}}} = O\lrp{{\delta^i}^{3/2} \lrp{1 + L_\beta^4}\lrp{1+ \dist\lrp{x^i_k,y^i_k}^2}}
    \end{alignat*}
    where $O()$ hides dependencies on $L_R, L_R', d, T$.

    We omit the proof for the above claim, which involves some tedious but straightforward algebra, but we note that the proof uses $\E{\lrn{\BB\lrp{(k+1)\delta^i} - \BB\lrp{k\delta^{i}}}_2^j} = O\lrp{{\delta^i}^{j/2}}$ (for all integer $j$) and that $\E{\exp\lrp{a\lrn{\BB\lrp{(k+1)\delta^i} - \BB\lrp{k\delta^{i}}}_2}} \leq 4\exp\lrp{2a^2 \delta^i d} \leq 8$ for $\delta^i a^2 \leq 1/32$ (see Lemma \ref{l:subexponential-chi-square}). It is also important to use our assumption on $\delta^i$ in the lemma statement.


    We simplify \eqref{e:l:kendall-cranston:step1.1} to
    \begin{alignat*}{1}
        & \dist\lrp{x^i_{k+1}, y^i_{k+1}}^2 - \dist\lrp{x^i_{k}, y^i_{k}}^2\\
        \leq& 2 \delta^i \kappa\lrp{\dist\lrp{x^i_{k}, y^i_{k}}}\dist\lrp{x^i_{k}, y^i_{k}}^2 - 4 \dist\lrp{x^i_{k}, y^i_{k}}\lin{\nnu^i_k, \BB\lrp{(k+1)\delta^i} - \BB\lrp{k\delta^{i}}}\\
        &\quad + 4 \lin{\nnu^i_k, \BB\lrp{(k+1)\delta^i} - \BB\lrp{k\delta^{i}}}^2\\
        &\quad - \int_0^1 \lin{R\lrp{{\gamma^i_k}'(s),\lrp{\BB\lrp{(k+1)\delta^i} - \BB\lrp{k\delta^{i}}}\circ \party{x^i_k}{\gamma^i_k(s)} E^i_k }\lrp{\BB\lrp{(k+1)\delta^i} - \BB\lrp{k\delta^{i}}}\circ \party{x^i_k}{\gamma^i_k(s)} E^i_k  , {\gamma^i_k}'(s)} ds\\
        &\quad + \tau^i_{k,1} + \tau^i_{k,2} + \tau^i_{k,3}
        \elb{e:l:kendall-cranston:step1.2}
    \end{alignat*}

    \textbf{Step 4: Pulling out the expectation}\\
    We will further simplify \eqref{e:l:kendall-cranston:step1.2} by replacing a few terms by their expectations. Define
    \begin{alignat*}{1}
        & {\tau}^i_{k,4}:= \int_0^1 \lin{R\lrp{{\gamma^i_k}'(s),\lrp{\BB\lrp{(k+1)\delta^i} - \BB\lrp{k\delta^{i}}}\circ \party{x^i_k}{\gamma^i_k(s)} E^i_k }, \lrp{\BB\lrp{(k+1)\delta^i} - \BB\lrp{k\delta^{i}}}\circ \party{x^i_k}{\gamma^i_k(s)} E^i_k , {\gamma^i_k}'(s)}\\
        &\qquad - \delta^iRic\lrp{{\gamma^i_k}'(s)} ds \\
        & {\tau}^i_{k,5}:= \delta^i - \lin{\nnu^i_k, \BB\lrp{(k+1)\delta^i} - \BB\lrp{k\delta^{i}}}^2\\
    \end{alignat*}

    By definition of Ricci Curvature and by Assumption \ref{ass:ricci_curvature_regularity},
    \begin{alignat*}{1}
        & - \Ep{\F_k}{\int_0^1 \lin{R\lrp{{\gamma^i_k}'(s),\lrp{\BB\lrp{(k+1)\delta^i} - \BB\lrp{k\delta^{i}}}\circ \party{x^i_k}{\gamma^i_k(s)} E^i_k }, \lrp{\BB\lrp{(k+1)\delta^i} - \BB\lrp{k\delta^{i}}}\circ \party{x^i_k}{\gamma^i_k(s)} E^i_k , {\gamma^i_k}'(s)} ds}\\
        \leq& - \delta^i L_{Ric} \lrn{{\gamma^i_k}'(s)}^2 = \leq - \delta^i L_{Ric} \dist\lrp{x^i_k,y^i_k}^2
    \end{alignat*}
    where $Ric$ denotes the Ricci curvature tensor.

    By definition of $\nnu^i_k$,  $\E{\lin{\nnu^i_k, \BB\lrp{(k+1)\delta^i} - \BB\lrp{k\delta^{i}}}^2} = \delta^i \ind{\dist\lrp{x^i_k,y^i_k}>\epsilon}$. 
    
    Let $\tau^i_k:= \tau^i_{k,1}+\tau^i_{k,2}+\tau^i_{k,3}$.
    We can thus further simplify \eqref{e:l:kendall-cranston:step1.2} to
    \begin{alignat*}{1}
        & \dist\lrp{x^i_{k+1}, y^i_{k+1}}^2 - \dist\lrp{x^i_{k}, y^i_{k}}^2\\
        \leq& 2 \delta^i \kappa\lrp{\dist\lrp{x^i_{k}, y^i_{k}}}\dist\lrp{x^i_{k}, y^i_{k}}^2 - 4 \dist\lrp{x^i_{k}, y^i_{k}}\lin{\nnu^i_k, \BB\lrp{(k+1)\delta^i} - \BB\lrp{k\delta^{i}}}\\
        &\quad + 4 \delta^i \ind{\dist\lrp{x^i_k,y^i_k}>\epsilon} - \delta^i L_{Ric} \dist\lrp{x^i_k,y^i_k}^2\\
        &\quad + \tau^i_k\\
        \leq& \delta^i \lrp{2\kappa\lrp{\dist\lrp{x^i_{k}, y^i_{k}}} + L_{Ric}} \dist\lrp{x^i_{k}, y^i_{k}}^2\\
        &\quad - 4 \dist\lrp{x^i_{k}, y^i_{k}}\lin{\nnu^i_k, \BB\lrp{(k+1)\delta^i} - \BB\lrp{k\delta^{i}}} + 4 \delta^i \ind{\dist\lrp{x^i_k,y^i_k}>\epsilon}\\
        &\quad + \tau^i_k
        \elb{e:l:kendall-cranston:step2}
    \end{alignat*}  
    the conclusion follows by defining $\WW^i(t):= \int_0^t \ind{t\in[k\delta^i,(k+1)\delta^i]}\lin{\bar{\nnu}^i_k, \BB\lrp{(k+1)\delta^i} - \BB\lrp{k\delta^{i}}}$ and verifying that it is a Brownian motion. (Recall our definition that $\nnu^i_k := \ind{\dist\lrp{x^i_{k}, y^i_{k}} > \epsilon} \bar{\nnu}^i_k$)
\end{proof}

\subsection{Lyapunov function and its smooth approximation}
\label{ss:Lyapunov function and its smooth approximation}
In this section, we consider a Lyapunov function $f$ taken from \cite{eberle2016reflection}. By analyzing how $f(\dist\lrp{x^i_k,y^i_k})$ evolves under the dynamic in Lemma \ref{l:kendall-cranston}, one can demonstrate that the distance function contracts.

Let $\L,\R\in \Re^+$. We will see later that $\L$ and $\R$ will correspond to distant-dissipativity parameters in \eqref{ass:distant-dissipativity}.

Let $\epsilon\in[0,\infty)$. One should think of $\epsilon$ as being arbitrarily small, as eventually we are only interested in the limit as $\epsilon \to 0$.
    
Define functions $\psi_\epsilon(r)$, $\Psi_\epsilon(r)$ and $\nu(r)$, all from $ \Re^+$ to $\Re$:
\begin{align*}
&\mu_\epsilon(r) = \threecase{1}{r\leq \R}{1 - (r-\R)/\lrp{\epsilon}}{r\in{\R, \R + \epsilon}}{0}{r \geq \R + \epsilon}\\
&\nu_\epsilon(r) := 1- \frac{1}{2}
\frac{\int_0^{r}\frac{ \mu_\epsilon(s) \Psi_\epsilon(s)}{\psi_\epsilon(s)} ds}{\int_0^{\infty}\frac{ \mu_\epsilon(s) \Psi_\epsilon(s)}{\psi_\epsilon(s)}ds}
&\psi_\epsilon(r) := e^{- \frac{\L \int_0^r r \mu_\epsilon(r) dr}{2}}\\
&\Psi_\epsilon(r) := \int_0^r \psi_\epsilon(s) ds, \\
\end{align*}

We defined an $\epsilon$-smoothed Lyapunov function as
\begin{definition}
    \label{d:f_epsilon}
    \begin{alignat*}{1}
        & f_\epsilon(r):= \int_0^r \psi_\epsilon(s) \nu_\epsilon(s) ds\\
        & g_\epsilon(s) = f_\epsilon\lrp{\sqrt{s + \epsilon}}
    \end{alignat*}
\end{definition}
The case when $\epsilon=0$ (when there is no smoothing) will be of particular interest to us:
\begin{definition}
    \label{d:f}
    \begin{alignat*}{1}
        & f(r):= f_0(r) = g_0(r)
    \end{alignat*}
\end{definition}

\begin{remark}
    The Lyapunov function from \cite{eberle2016reflection} is more general, but for the specific case of $\L , \R$ distant dissipative functions, it is equal to $f$ as defined in \eqref{d:f}.
\end{remark}
\begin{lemma}
    \label{l:fproperties}
    Assume $\epsilon \in [0, 1/(4\sqrt{\L})]$, then $f_\epsilon$ as defined in \eqref{d:f_epsilon} satisfies
    \begin{alignat*}{2}
        &1.\ f_\epsilon(r) \in [\frac{1}{2}\exp\lrp{-  (1+\epsilon) \L \R^2/2} r, r] \qquad && \text{for all $r$} \\
        &2.\ f_\epsilon'(r) \in [\frac{1}{2}\exp\lrp{-  (1+\epsilon) \L \R^2/2}, 1] \qquad && \text{for all $r$} \\
        &3.\ f_\epsilon''(r) \in [-4{\L}^{3/2}, 0] \qquad && \text{for all $r$} \\
        &4.\ f_\epsilon''(r) + \L r f_\epsilon'(r)  \leq -\frac{\exp\lrp{- (1+\epsilon)\L \R^2/2}}{\lrp{1+\epsilon}^2 \R^2} f_\epsilon(r) \qquad && \text{for $r\in [0,\R]$}
    \end{alignat*}
    If in addition, $\epsilon >0$, $f_\epsilon$ satisfies
    \begin{alignat*}{1}
        &5.\ \lrabs{f_\epsilon'''(r)} \leq \frac{256 \sqrt{\L}}{\epsilon} \qquad \text{for all $r$}
    \end{alignat*}
\end{lemma}
\begin{proof}
    We can verify that
    \begin{alignat*}{2}
        & f_\epsilon'(r) &&= \psi_\epsilon(r) \nu_\epsilon(r)\\
        & f_\epsilon''(r) &&= \psi_\epsilon'(r) \nu_\epsilon(r) + \psi_\epsilon(r) \nu_\epsilon'(r)\\
        & &&= - \L \mu_\epsilon(r) r\psi_\epsilon(r) \nu_\epsilon(r) + \psi_\epsilon(r) \nu_\epsilon'(r)\\
        & f_\epsilon'''(r) &&= -\L \psi_\epsilon(r) \nu_\epsilon(r) + \L r \psi_\epsilon(r) \mu_\epsilon'(r) + \L^2 r^2 \psi_\epsilon(r) \nu_\epsilon(r) - 2\L r \psi_\epsilon(r) \nu_\epsilon'(r) + \psi_\epsilon(r) \nu_\epsilon''(r)
    \end{alignat*}
    1. follows from integrating 2.
    
    2. follows from $\nu_\epsilon(r) \in [1/2,1]$ and $\psi_\epsilon \in [\exp\lrp{-  (1+\epsilon) \L \R^2/2}, 1]$ and the expression for $f_\epsilon'(r)$ above.

    3. follows from $\mu_\epsilon, \psi_\epsilon, \nu_\epsilon \geq 0$ and $\nu_\epsilon' \leq 0$, and the fact that $r \psi_\epsilon(r) \leq 2 \sqrt{\L}$ and \eqref{e:t:qoidmas}.

    4. is a little more involved. First note that over $r\in [0,\R]$, $\mu_\epsilon(r) = 1$. This will simplify some calculations. From the expression for $f_\epsilon''$ above, we verify
    \begin{alignat*}{1}
        f_\epsilon''(r) + \L r f_\epsilon'(r) 
        = \psi_\epsilon(r) \nu_\epsilon'(r)
        = - \frac{\Psi_\epsilon(r)}{2\int_0^{\infty}\frac{ \mu_\epsilon(s) \Psi_\epsilon(s)}{\psi_\epsilon(s)}ds}
    \end{alignat*}
    We can bound the denominator as
    \begin{alignat*}{1}
        & \int_0^{\infty}\frac{ \mu_\epsilon(s) \Psi_\epsilon(s)}{\psi_\epsilon(s)}ds
        \leq \int_0^{\R + \epsilon}\frac{\Psi_\epsilon(s)}{\psi_\epsilon(s)}ds
        \leq \frac{\int_0^{\R + \epsilon} \Psi_\epsilon(s) ds}{\psi\lrp{\R + \epsilon}} 
        \leq \frac{(1+\epsilon)^2\R^2}{2\exp\lrp{-\L (1+\epsilon) \R^2/2}}\\
    \end{alignat*}
    where the first inequality is by $\mu_\epsilon(s) \leq 1$, and $\mu_\epsilon(r) = 0$ for $r\geq \R + \epsilon$ the second inequality is by $\psi_\epsilon(r)$ being monotonically decreasing, and the third inequality is by $\Psi_\epsilon(r) \leq r$.Finally, note that $\Psi_\epsilon(r) \geq f_\epsilon(r)$. Put together,
    \begin{alignat*}{1}
        f_\epsilon''(r) + \L r f_\epsilon'(r)  \leq -\frac{\exp\lrp{- (1+\epsilon)\L \R^2/2}}{\lrp{1+\epsilon}^2 \R^2} f_\epsilon(r)
    \end{alignat*}

    We now prove the bound for 5. It is useful to recall that $\psi_\epsilon(r) \leq 1$ and $\nu_\epsilon(r) \leq 1$. 
    \begin{alignat*}{1}
        \Psi_\epsilon(r) = \int_0^r \exp\lrp{-\L s^2} ds \leq \frac{4}{\sqrt{\L}}
    \end{alignat*}
    \begin{alignat*}{1}
        \int_0^{\infty}\frac{ \mu_\epsilon(s) \Psi_\epsilon(s)}{\psi_\epsilon(s)}ds
        \geq& \int_0^{\R}\frac{\Psi_\epsilon(s)}{\psi_\epsilon(s)}ds
        \geq \frac{1}{2}\int_0^{1/\sqrt{2\L}} \Psi_\epsilon(s) ds \geq \frac{1}{16\L}
        \elb{e:t:qoidmas}
    \end{alignat*}
    \begin{alignat*}{1}
        \lrabs{\psi_\epsilon(r) \nu'_\epsilon(r)} \leq \frac{\Psi\lrp{\R + \epsilon}}{ 2\int_0^{\infty}\frac{ \mu_\epsilon(s) \Psi_\epsilon(s)}{\psi_\epsilon(s)}ds}
        \leq 8 \sqrt{\L}
    \end{alignat*}

    For $r\in [0, \R + \epsilon]$ ($\nu_\epsilon'' = 0$ outside this range),
    \begin{alignat*}{1}
        \lrabs{\psi_\epsilon(r) \nu_\epsilon''(r)} \leq \frac{\frac{1}{\epsilon}\Psi\lrp{r} + r \psi_\epsilon(r)/\epsilon + \psi_\epsilon(r) + 2r\Psi_\epsilon(r)/\psi_\epsilon(r)}{ 2\int_0^{\infty}\frac{ \mu_\epsilon(s) \Psi_\epsilon(s)}{\psi_\epsilon(s)}ds} \leq 32\L \cdot \Psi_\epsilon(r) \cdot \lrp{\frac{2}{\epsilon} + 2\L \R} \leq \frac{128 \sqrt{\L}}{\epsilon}
    \end{alignat*}

    We can thus bound $\lrabs{f'''(r)}$ as
    \begin{alignat*}{1}
        \lrabs{f_\epsilon'''(r)} \leq 2\L + 16 \L^{3/2} \R + \frac{128 \sqrt{\L}}{\epsilon} \leq \frac{256 \sqrt{\L}}{\epsilon}
    \end{alignat*}
\end{proof}

\begin{lemma}
    \label{l:gproperties}
    Assume $\epsilon \in (0, 1/(4\sqrt{\L})]$
    \begin{alignat*}{2}
        &1.\ {g_\epsilon'(s)} = \frac{1}{2\sqrt{s+\epsilon}} f_\epsilon'(\sqrt{s+\epsilon})\\
        &2.\ {g_\epsilon''(s)} = \frac{1}{4\lrp{s+\epsilon}} f_\epsilon''(\sqrt{s+\epsilon}) - \frac{1}{4\lrp{s+\epsilon}^{3/2}} f_\epsilon'(\sqrt{s+\epsilon})\\
        &3.\ {g_\epsilon'''(s)} = \frac{1}{8\lrp{s+\epsilon}^{3/2}} f_\epsilon'''(\sqrt{s+\epsilon}) - \frac{1}{8\lrp{s+\epsilon}^{2}} f_\epsilon''(\sqrt{s+\epsilon}) + \frac{1}{6\lrp{s+\epsilon}^{5/2}} f_\epsilon'(\sqrt{s+\epsilon})\\
        &4.\ \lrabs{g_\epsilon'''(s)} \leq O\lrp{\epsilon^{-5/2}} \qquad \text{for all $s$}
    \end{alignat*}
    where $O()$ notation hides dependency on $\L$ and $\R$.
\end{lemma}
\begin{proof}
    The first 3 points follow from chain rule.

    The last point follows from point 5 from Lemma \ref{l:fproperties}. 
    \begin{alignat*}{1}
        \lrabs{g_\epsilon'''(s)} \leq \frac{64\sqrt{\L}}{\epsilon^{5/2} \R} + \frac{\sqrt{\L}}{\epsilon^2} + \frac{1}{\epsilon^{5/2}}
    \end{alignat*}
\end{proof}

\subsection{Contraction of Lyapunov Function under Kendall Cranston Coupling}
\label{ss:Evolution_of_Lyapunov_Function_under_Kendall_Cranston_Coupling}

\begin{lemma}
    \label{l:g_epsiilon_evolution_beta_lipschitz}
    Consider the same setup as Lemma \ref{l:kendall-cranston}. For any $x,y$, let $\Lambda(x,y)$ denote the set of minimizing geodesics from $x$ to $y$, i.e. for any $\gamma \in \Lambda(x,y)$, $\gamma(0) = x$, $\gamma(1) = y$, $\forall t, \nabla_{\gamma'(t)}\gamma'(t) = 0$ and $\dist\lrp{x,y} = \lrn{\gamma'(0)}$. Let $\kappa(r):= \frac{1}{r^2} \sup_{\dist\lrp{x,y} =  r} \inf_{\gamma \in \Lambda(x,y)} \lin{\party{y}{x} \beta(y) - \beta(x), \gamma'(0)}$. 
    
    Assume there exists $\R\geq 0, q\leq 0$ such that $\kappa(r) \leq q$ for all $r\leq \R$. Let $\L = q + L_{Ric}/2$. Let  $\epsilon \in (0, 1/(4\sqrt{\L})]$. Let $g_\epsilon$ be as defined in \ref{d:f_epsilon} with parameters $\L$ and $\R$. Let $\F_k$ denote the natural filtration generated by $x^i_k$ and $y^i_k$.
    
    There exists a constant $c_1$, depending on $L_\beta, L_\beta', L_R, T, d$, and some constant $c_2$, depending on $L_\beta', L_{Ric}, \R$ such that for any $i>c_1$ and $\epsilon>c_2$, there exists a coupling between $x^i_k$ and $y^i_k$ such that
    \begin{alignat*}{1}
        &\E{g_\epsilon\lrp{\dist\lrp{x^i_{k+1}, y^i_{k+1}}^2}} \\
        \leq& \E{\ind{r > \R}\delta^i \lrp{\lrp{\kappa(r_k) + L_{Ric}/2}\exp\lrp{-  (1+\epsilon) \L \R^2/2}/8} g_\epsilon\lrp{\dist\lrp{x^i_{k+1}, y^i_{k+1}}^2}}\\
        &\quad - \frac{\exp\lrp{- (1+\epsilon)\L \R^2/2}}{2\lrp{1+\epsilon}^2 \R^2} \delta^i \E{\ind{r \leq \R} g_\epsilon\lrp{\dist\lrp{x^i_{k+1}, y^i_{k+1}}^2}} + O\lrp{\delta^i{\epsilon^{1/2}} + \epsilon^{-5/2} {\delta^i}^{3/2}}
    \end{alignat*}
    where $O\lrp{}$ hides dependency on $L_R, L_\beta', T, d$.
\end{lemma}
\begin{proof}
    Let us define, for convenience, $r_k := \dist\lrp{x^i_{k}, y^i_{k}}$. By Lemma \ref{l:kendall-cranston}, for any $i$ and any $\epsilon$, there exists a coupling satisfying
    \begin{alignat*}{1}
        r_{k+1}^2 
        \leq& \lrp{1+\delta^i \lrp{2\kappa\lrp{r_k} + L_{Ric}}} r_k^2 \\
        &\quad +\ind{r_k > \epsilon^{1/3}} \lrp{4 \delta^i - 4 r_k \WW^i\lrp{(k+1)\delta^i} - \WW^i\lrp{k\delta^{i}}} + \tau^i_k
    \end{alignat*}
    where $\tau^i_k$ satisfies
    \begin{alignat*}{1}
        & \E{\lrabs{\tau^i_k}} \leq O \lrp{{\delta^i}^{3/2} \lrp{1 + L_\beta^4}\lrp{1+ r_k^2}} \qquad \E{{{\tau}^i_k}^2 }\leq O\lrp{{\lrp{1+r_k^4} {\delta^i}^{2}}}
    \end{alignat*}
    where $O\lrp{}$ hides dependencies on $L_R, L_R', d, T$.

    By third order Taylor expansion,
    \begin{alignat*}{1}
        & \E{g_{\epsilon}\lrp{r_{k+1}^2}}\\
        =& \E{g_{\epsilon}\lrp{r_k^2}}\\
        &\quad + \E{g_\epsilon'\lrp{r_k^2} \cdot  \lrp{\delta^i \lrp{2\kappa\lrp{r_k} + L_{Ric}}} r_k^2}\\
        &\quad + \E{g_\epsilon'\lrp{r_k^2} \cdot 4\delta^i }\\
        &\quad + \E{\frac{1}{2} g_\epsilon''\lrp{r_k^2} \cdot \lrp{ 4 r_k \ind{r_k > \epsilon^{1/3}} \lrp{\WW^i\lrp{(k+1)\delta^i} - \WW^i\lrp{k\delta^{i}} }}^2}\\
        &\quad + O\lrp{\epsilon^{-5/2} {\delta^i}^{3/2}}
        \elb{e:t:qknqoiwdn:0}
    \end{alignat*}
    The last line uses two facts:
    \begin{enumerate}
        \item From Lemma \ref{l:near_tail_bound_one_step}, for any $j$, there exists a constant $\C$, depending on $T,d, L_R, L_\beta'$, but independent of $L_\beta$, such that for all $i,k$, $\E{\dist\lrp{x^i_k,x_0}^{2j}} < \C$ and $\E{\dist\lrp{y^i_k,y_0}^{2j}} < \C$.
        \item Roughly speaking, $\E{\dist\lrp{x^i_{k+1},y^i_{k+1}}^2 - \dist\lrp{x^i_{k},y^i_{k}}^2} = O\lrp{{\delta^i}^{3/2}}$. More specifically:
        \begin{alignat*}{1}
            & \lrabs{\dist\lrp{x^i_{k+1},y^i_{k+1}} - \dist\lrp{x^i_{k},y^i_{k}}}\\
            & \leq 2\dist\lrp{x^i_k x^i_{k+1}} + 2\dist\lrp{y^i_k y^i_{k+1}} \\
            & \leq 2\delta^i\lrp{\lrn{\beta(x_0)} + \lrn{\beta(y_0)} + L_\beta' \dist\lrp{\dist\lrp{x^i_k,x_0}} + L_\beta' \dist\lrp{\dist\lrp{x^i_k,x_0}}}\\
            &\quad + 4 \lrn{\BB((k+1)\delta^i) - \BB(k\delta^i)}_2
        \end{alignat*}
    \end{enumerate}

    Plugging in the definition of $g_\epsilon'$ and $g_\epsilon''$,
    \begin{alignat*}{1}
        & g_\epsilon'\lrp{r_k^2} \cdot  \lrp{\delta^i \lrp{\lrp{2\kappa\lrp{r_k} + L_{Ric}}} r_k^2+ 4\ind{r_k > \epsilon^{1/3}}\delta^i}\\
        =& \frac{\delta^i}{2\sqrt{r_k^2 + \epsilon}} f_\epsilon'\lrp{\sqrt{r_k^2+\epsilon}} \lrp{\lrp{2\kappa\lrp{r_k} + L_{Ric}} r_k^2 + 4\ind{r_k > \epsilon^{1/3}}}\\
        \leq&  \frac{\delta^i}{2\sqrt{r_k^2 + \epsilon}} f_\epsilon'\lrp{\sqrt{r_k^2+\epsilon}} \lrp{\lrp{2\kappa\lrp{r_k} + L_{Ric}} r_k^2} + 2\ind{r_k> \epsilon^{1/3}} \frac{\delta^i f_\epsilon'\lrp{\sqrt{r_k^2+\epsilon}}}{\sqrt{r_k^2 + \epsilon}}
    \end{alignat*}
    where we use the assumption that $\epsilon \leq \frac{1}{4\R^2}$ and $\epsilon < 1$.

    On the other hand, 
    \begin{alignat*}{1}
        & \Ep{\F_k}{\frac{1}{2} g_\epsilon''\lrp{r_k^2} \cdot \lrp{ 4 r_k \ind{r_k > \epsilon^{1/3}} \lrp{\WW^i\lrp{(k+1)\delta^i} - \WW^i\lrp{k\delta^{i}} }}^2}\\
        =& 8 \delta^i r_k^2 g_\epsilon''\lrp{r_k^2} \cdot \ind{r_k>\epsilon^{1/3}}\\
        =& \ind{r_k>\epsilon^{1/3}}\frac{2 \delta^i r_k^2}{r_k^2 + \epsilon} f_\epsilon''\lrp{\sqrt{r_k^2+\epsilon}} - \ind{r_k>\epsilon^{1/3}} \frac{2\delta^i r_k^2f_\epsilon'\lrp{\sqrt{r_k^2+\epsilon}}}{\lrp{r_k^2 + \epsilon}^{3/2}}
    \end{alignat*}

    Note that $r_k > \epsilon^{1/3}$ implies that $\frac{r_k^2}{\lrp{r_k^2 + \epsilon}^{3/2}} \geq \frac{1}{1+\epsilon^{1/3}}$. Thus
    \begin{alignat*}{1}
        2\ind{r_k> \epsilon^{1/3}} \frac{\delta^i f_\epsilon'\lrp{\sqrt{r_k^2+\epsilon}}}{\sqrt{r_k^2 + \epsilon}}- \ind{r_k>\epsilon^{1/3}} \frac{2\delta^i r_k^2f_\epsilon'\lrp{\sqrt{r_k^2+\epsilon}}}{\lrp{r_k^2 + \epsilon}^{3/2}}\leq 4\delta^i \epsilon^{1/3}
        \elb{e:t:qknqoiwdn:1}
    \end{alignat*}
    where we use the fact that $\lrabs{f_\epsilon'}\leq 1$.

    We now bound $\frac{\delta^i}{2\sqrt{r_k^2 + \epsilon}} f_\epsilon'\lrp{\sqrt{r_k^2+\epsilon}} \lrp{\lrp{2\kappa\lrp{r_k} + L_{Ric}} r_k^2} + \ind{r_k>\epsilon^{1/3}}\frac{2 \delta^i r_k^2}{r_k^2 + \epsilon} f_\epsilon''\lrp{\sqrt{r_k^2+\epsilon}}$. Consider three cases:
    \begin{enumerate}
        \item $r_k \leq \epsilon^{1/3}$: 
        \begin{alignat*}{1}
            & \frac{\delta^i}{2\sqrt{r_k^2 + \epsilon}} f_\epsilon'\lrp{\sqrt{r_k^2+\epsilon}} \lrp{\lrp{2\kappa\lrp{r_k} + L_{Ric}} r_k^2}
            \leq {\delta^i \lrp{q + L_{Ric}/2}}{\epsilon^{1/2}} 
        \end{alignat*}
        \item $r_k \in (\epsilon^{1/3},\R]$: 
        \begin{alignat*}{1}
            & \frac{\delta^i}{2\sqrt{r_k^2 + \epsilon}} f_\epsilon'\lrp{\sqrt{r_k^2+\epsilon}} \lrp{\lrp{2\kappa\lrp{r_k} + L_{Ric}} r_k^2} + \frac{2 \delta^i r_k^2}{r_k^2 + \epsilon} f_\epsilon''\lrp{\sqrt{r_k^2+\epsilon}}\\
            \leq& \frac{\delta^i r_k^2}{{r_k^2 + \epsilon}}\lrp{\L f_\epsilon'\lrp{\sqrt{r_k^2 + \epsilon}}\sqrt{r_k^2 + \epsilon} + 2f_\epsilon''\lrp{\sqrt{r_k^2+\epsilon}}}\\
            \leq& - \frac{\exp\lrp{- (1+\epsilon)\L \R^2/2}}{2\lrp{1+\epsilon}^2 \R^2}  \delta^i f_\epsilon\lrp{\sqrt{r_k^2 + \epsilon}}
        \end{alignat*}
        where we use Lemma \ref{l:fproperties} and the definition of $\L$.
        \item $r_k > \R$: We use the fact that $f_\epsilon''(r) \leq 0$ for all $r \geq \R \geq \epsilon$. Thus
        \begin{alignat*}{1}
            & \frac{\delta^i}{2\sqrt{r_k^2 + \epsilon}} f_\epsilon'\lrp{\sqrt{r_k^2+\epsilon}} \lrp{\lrp{2\kappa\lrp{r_k} + L_{Ric}} r_k^2} + \frac{2 \delta^i r_k^2}{r_k^2 + \epsilon} f_\epsilon''\lrp{\sqrt{r_k^2+\epsilon}}\\
            \leq& \frac{\delta^i}{2\sqrt{r_k^2 + \epsilon}} f_\epsilon'\lrp{\sqrt{r_k^2+\epsilon}} \lrp{\lrp{2\kappa\lrp{r_k} + L_{Ric}} r_k^2}\\
            \leq& - \frac{\delta^i \lrp{\lrp{\kappa(r_k) + L_{Ric}/2}} r_k^2 f_\epsilon'\lrp{\sqrt{r_k^2 + \epsilon}}}{8\sqrt{r_k^2 + \epsilon}}\\
            \leq& - \frac{1}{8}\delta^i \lrp{\lrp{\kappa(r_k) + L_{Ric}/2}\exp\lrp{-  (1+\epsilon) \L \R^2/2}} f_\epsilon\lrp{\sqrt{r_k^2 + \epsilon}}
        \end{alignat*}
    \end{enumerate}
\end{proof}

\begin{proof}[Proof of Lemma \ref{l:g_contraction_without_gradient_lipschitz}]
    Let $E^x$ be an orthonormal basis of $T_{x(0)} M$, $E^y$ be an orthonormal basis of $T_{y(0)} M$, and let $\BB^x$ and $\BB^y$ denote two Brownian motions which may be coupled in a non-trivial way. By definition of $\Phi$ in \eqref{d:x(t)} and by Lemma \ref{l:Phi_is_diffusion}, $x(t) = \Phi(t;x(0),E^x,\beta,\BB^x)$ and $y(t) = \Phi(t;y(0),E^y,\beta,\BB^y)$, where equivalence is in the sense of distribution.

    Lemma \ref{l:g_epsiilon_evolution_beta_lipschitz} almost gives us what we need. However, because we assumed that $\beta$ satisfies Assumption \ref{ass:distant-dissipativity}, the assumption that $\lrn{\beta(x)} \leq L_\beta$ cannot possibly hold. We thus need to approximate $\beta$ by a sequence of increasingly non-Lipschitz functions.

    Consider a fixed $i$. Let $s^j$ be a sequence of increasing radius, such that $s^j \to \infty$ as $j\to \infty$. Let $\beta^j$ denote the truncation of $\beta$ to norm $s^j$, i.e.
    \begin{alignat*}{1}
        \beta^j(x):= \twocase{\beta(x)}{\lrn{\beta(x)} \leq s^j}{\beta(x)\cdot\frac{s^j}{\lrn{\beta(x)}}}{\lrn{\beta(x)} > s^j}.
    \end{alignat*}
    We verify that $\beta^j$ also satisfies Assumption \ref{ass:beta_lipschitz} with the same $L_\beta'$ as $\beta$.

    Consider some fixed $j$. Let us now define the Euler Murayama discretization of $x(t)$ and $y(t)$ as
    \begin{alignat*}{1}
        & x^i(t) := \overline{\Phi}(t;x(0),E^x,\beta,\BB^x,i) \\
        & y^i(t) := \overline{\Phi}(t;y(0),E^y,\beta,\BB^y,i).
    \end{alignat*}
    Where $\overline{\Phi}$ is as defined in \eqref{d:x^i(t)}, and is a short-hand for the (interpolated) Euler Murayama sequence with stepsize $\delta^i = T/2^i$, defined in \eqref{d:x^i(t):0} (equivalently \eqref{d:x^i(t):0-appendix}). It is by definition that $x(t) = \lim_{i\to \infty} x^i(t)$ (and similarly for $y(t)$ and $y^i(t)$).
    
    Furthermore, define, for all $i,j$,
    \begin{alignat*}{1}
        & \t{x}^{i,j}(t) := \overline{\Phi}(t;x(0),E^x,\beta^j,\BB^x,i), \\
        & \t{y}^{i,j}(t) := \overline{\Phi}(t;y(0),E^y,\beta^j,\BB^y,i), \\
        & \t{x}^{\cdot,j}(t) := {\Phi}(t;x(0),E^x,\beta^j,\BB^x), \\
        & \t{y}^{\cdot,j}(t) := {\Phi}(t;y(0),E^y,\beta^j,\BB^y). 
    \end{alignat*}
    Note that the above definition implies a non-trivial coupling between $\t{x}^{i,j}(t)$ and $x^{i}(t)$, via the shared Brownian motion $\BB^x$. In words, $\t{x}^{\cdot,j}(t)$ denotes the exact Langevin SDE, but with drift given by $\beta^j$ (the truncated version of $\beta$), and $\t{x}^{i,j}$ denotes the (interpolated) Euler Murayama discretization of $\t{x}^{\cdot,j}$ with stepsize $\delta^i$.

    Let $L_0 := \max\lrbb{\lrn{\beta(x(0))},\lrn{\beta(y(0))}}$. Let us define ${\t{r}^{i,j}_k} := \dist\lrp{\t{x}^{i,j}_k,\t{y}^{i,j}_k}$. Let $\kappa(r)$ be as defined in the statement of Lemma \ref{l:g_epsiilon_evolution_beta_lipschitz}.
    
    Using Assumption \ref{ass:distant-dissipativity}, we verify that
    \begin{alignat*}{1}
        & \ind{{\t{r}^{i,j}_k} > \R}\lrp{\kappa({\t{r}^{i,j}_k}) + L_{Ric}/2} \\
        <& \ind{\R < {\t{r}^{i,j}_k} \leq \frac{s^j - L_0}{L_\beta'}} \lrp{- m + L_{Ric}/2} + \ind{\R < {\t{r}^{i,j}_k}, \frac{s^j - L_0}{L_\beta'}\leq {\t{r}^{i,j}_k}} \lrp{\frac{s^j}{{\t{r}^{i,j}_k}} + L_{Ric}/2}.
    \end{alignat*}

    
    Let $q,\R$ be the parameters in Assumption \ref{ass:distant-dissipativity}. This implies that $\kappa(r) \leq q$ for all $r\leq \R$. Let $\L := q + L_{Ric}/2$ and $\epsilon$ be as defined in Lemma \ref{l:g_epsiilon_evolution_beta_lipschitz}. Let $g_\epsilon$ be as defined in Definition \ref{d:f_epsilon} with parameters $\L$ and $\R$. Then 
    \begin{alignat*}{1}
        &\E{g_\epsilon\lrp{r_{k+1}^2}} \\
        \leq& \E{\ind{r > \R}\delta^i \lrp{\lrp{\kappa({\t{r}^{i,j}_k}) + L_{Ric}/2}\exp\lrp{-  (1+\epsilon) \L \R^2/2}/8} g_\epsilon\lrp{r_{k+1}^2}} \\
        &\quad - \frac{\exp\lrp{- (1+\epsilon)\L \R^2/2}}{2\lrp{1+\epsilon}^2 \R^2} \delta^i \E{\ind{r \leq \R} g_\epsilon\lrp{r_{k+1}^2}} + O\lrp{\delta^i{\epsilon^{1/2}} + \epsilon^{-5/2} {\delta^i}^{3/2}}\\
        \leq& -\frac{\delta^i (m-L_{Ric}/2) \exp\lrp{-  (1+\epsilon) \L \R^2/2}}{16} \E{\ind{\R < {\t{r}^{i,j}_k} \leq \frac{s^j - L_0}{L_\beta'}} g_\epsilon\lrp{r_{k+1}^2}} \\
        &\quad - \frac{\delta^i \exp\lrp{- (1+\epsilon)\L \R^2/2}}{2\lrp{1+\epsilon}^2 \R^2} \E{\ind{r \leq \R} g_\epsilon\lrp{r_{k+1}^2}}\\
        &\quad + \delta^i \E{\ind{\frac{s^j - L_0}{L_\beta'}< {\t{r}^{i,j}_k}} \lrp{{s^j}+ \frac{1}{2}L_{Ric}{\t{r}^{i,j}_k}}} + O\lrp{\delta^i{\epsilon^{1/2}} + \epsilon^{-5/2} {\delta^i}^{3/2}}\\
        \leq& - \alpha_\epsilon \delta^i \E{g_\epsilon\lrp{r_{k+1}^2}} + \delta^i \E{\ind{\frac{s^j - L_0}{L_\beta'}< {\t{r}^{i,j}_k}} \lrp{{s^j}+ \lrp{m + L_{Ric}/2}{\t{r}^{i,j}_k}}} + O\lrp{\delta^i{\epsilon^{1/2}} + \epsilon^{-5/2} {\delta^i}^{3/2}}
        \elb{e:t:rwkenmnf:0}
    \end{alignat*}
    where we define $\alpha_\epsilon := \min\lrbb{\frac{m-L_{Ric}/2}{16}, \frac{1}{2\lrp{1+\epsilon}^2 \R^2}}\exp\lrp{-  \frac{1}{2}(1+\epsilon) \L \R^2}$. The first line follows from Lemma \ref{l:g_epsiilon_evolution_beta_lipschitz}, the second line simply splits $\ind{r > \R}$ into two cases: $\ind{\R < {\t{r}^{i,j}_k} \leq \frac{s^j - L_0}{L_\beta'}}$ and $\ind{\frac{s^j - L_0}{L_\beta'}< {\t{r}^{i,j}_k}}$, and bounds $\kappa(r) \leq -m$ when $r>\R$ under Assumption \ref{ass:distant-dissipativity}. The third inequality is by definition of $\alpha_\epsilon$.

    Applying \eqref{e:t:rwkenmnf:0} recursively for $k=0...K$, where $K = T/2^i$, we get that
    \begin{alignat*}{1}
        &\E{g_{\epsilon}\lrp{\lrp{\t{r}^{i,j}_K}^2}} \\
        \leq& \exp\lrp{-\alpha_\epsilon K\delta^i}\E{g_{\epsilon}\lrp{r_{0}^2}} + O\lrp{T \epsilon^{1/2} + T\epsilon^{-5/2} {\delta^i}^{1/2}}\\
        &\quad + \delta^i \sum_{k=0}^K \E{\ind{\frac{s^j - L_0}{L_\beta'}< {\t{r}^{i,j}_k}} \lrp{{s^j}+ \lrp{m + L_{Ric}/2}{\t{r}^{i,j}_k}}} + O\lrp{T{\epsilon^{1/2}} + \epsilon^{-5/2} T {\delta^i}^{1/2}}
        \elb{e:t:rwkenmnf:1}
    \end{alignat*}
    We now bound the second term of \eqref{e:t:rwkenmnf:1} more carefully:
    \begin{alignat*}{1}
        & \delta^i \sum_{k=0}^K \E{\ind{\frac{s^j - L_0}{L_\beta'}\leq {\t{r}^{i,j}_k}} \lrp{{s^j}+ \lrp{m + L_{Ric}/2}{\t{r}^{i,j}_k}}}\\
        \leq& \delta^i \E{ \lrp{\max_{k\leq K} \ind{\frac{s^j - L_0}{L_\beta'}\leq {\t{r}^{i,j}_k}}} \sum_{k=0}^K\lrp{{s^j}+ \lrp{m + L_{Ric}/2}{\t{r}^{i,j}_k}}}\\
        \leq& \delta^i \E{ \lrp{\ind{\frac{s^j - L_0}{L_\beta'}\leq \max_{k\leq K}{\t{r}^{i,j}_k}}} \sum_{k=0}^K\lrp{{s^j}+ \lrp{m + L_{Ric}/2}{\t{r}^{i,j}_k}}}\\
        \leq& \delta^i \sqrt{\Pr{\frac{s^j - L_0}{L_\beta'}\leq \max_{k\leq K}{\t{r}^{i,j}_k}}} \cdot \sqrt{\E{\lrp{\sum_{k=0}^K\lrp{{s^j}+ \lrp{m + L_{Ric}/2}{\t{r}^{i,j}_k}}}^2}}\\
        \leq& O\lrp{\frac{1}{s^j}}
        \elb{e:t:rwkenmnf:2}
    \end{alignat*}
    The last line is because of the following: from Lemma \ref{l:near_tail_bound_L4}, $\delta^i \Pr{\frac{s^j - L_0}{L_\beta'}\leq \sup_{k\leq K} {\t{r}^{i,j}_k}}^{1/2} = O\lrp{\frac{{L_\beta'}^2}{(s^j - L_0)^2}} = O\lrp{\frac{1}{s_j^2}}$ assuming $j$ sufficiently large. Also from Lemma \ref{l:near_tail_bound_L4}, $\delta^i \sqrt{\E{\lrp{\sum_{k=0}^{K}\lrp{{s^j}+ \lrp{m + L_{Ric}/2}{\t{r}^{i,j}_k}}}^2}} = O\lrp{T}$. 
    
    Plugging \eqref{e:t:rwkenmnf:2} into \eqref{e:t:rwkenmnf:1}, and recalling the definition of $\t{r}$, and the fact that $\dist\lrp{\t{x}^i_K,\t{y}^i_K}:= \dist\lrp{\t{x}^i(T),\t{y}^i(T)}$,
    \begin{alignat*}{1}
        \E{g_{\epsilon}\lrp{\dist\lrp{\t{x}^{i,j}(T),\t{y}^{i,j}(T}^2}} 
        \leq& \exp\lrp{-\alpha_\epsilon K\delta^i}g_{\epsilon}\lrp{\dist\lrp{x(0),y(0)}^2} + O\lrp{\epsilon^{1/2} + \epsilon^{-5/2} {\delta^i}^{1/2} + \frac{1}{s^j}}
    \end{alignat*}
    where $O(\cdot)$ hides $T$ dependency as well.

    First, by taking the limit of $i$ to infinity (e.g. for each $i$, we see that for any $j$ and any $\epsilon$,
    \begin{alignat*}{1}
        \lim_{i \to \infty} \E{g_{\epsilon}\lrp{\dist\lrp{\t{x}^{i,j}(T),\t{y}^{i,j}(T)}^2}}  \leq& \exp\lrp{-\alpha_\epsilon K\delta^i}g_{\epsilon}\lrp{\dist\lrp{x(0),y(0)}^2}  + O\lrp{\epsilon^{1/2} + \frac{1}{s^j}}
    \end{alignat*}
    Let us define $\t{x}^{\cdot,j}(t)$ as the almost sure limit of $\t{x}^{i,j}(t)$, as $i\to \infty$, whose existence is shown in Lemma \ref{l:x(t)_is_brownian_motion} (similarly for $\t{y}^{\cdot,j}(t)$). It follows that $g_{\epsilon}\lrp{\dist\lrp{\t{x}^{i,j}(T),\t{y}^{i,j}(T}^2}$ converges almost surely to $g_{\epsilon}\lrp{\dist\lrp{\t{x}^{\cdot,j}(T),\t{y}^{\cdot,j}(T)}^2}$ as $i\to \infty$. By dominated convergence (Lemma \ref{l:near_tail_bound_L2} implies a single constant upper bounds $\E{\dist\lrp{\t{x}^{i,j}(T),\t{y}^{i,j}(T}^2}$ for all $i$), $\E{g_{\epsilon}\lrp{\dist\lrp{\t{x}^{i,j}(T),\t{y}^{i,j}(T)}^2}}$ converges to $\E{g_{\epsilon}\lrp{\dist\lrp{\t{x}^{\cdot,j}(T),\t{y}^{\cdot,j}(T)}^2}}$ as $i\to \infty$. Let $\Omega$ denote the set of all couplings between $\t{x}^{\cdot,j}(t)$ and $\t{y}^{\cdot,j}(t)$. Then 
    \begin{alignat*}{1}
        & \inf_{\Omega} \E{g_{\epsilon}\lrp{\dist\lrp{\t{x}^{\cdot,j}(T),\t{y}^{\cdot,j}(T)}^2}}\\
        \leq & \lim_{i \to \infty} \E{g_{\epsilon}\lrp{\dist\lrp{\t{x}^{i,j}(T),\t{y}^{i,j}(T)}^2}} \\
        \leq& \exp\lrp{-\alpha_\epsilon K\delta^i}g_{\epsilon}\lrp{\dist\lrp{x(0),y(0)}^2}  + O\lrp{\epsilon^{1/2} + \frac{1}{s^j}},
        \elb{e:t:pmfalkfma}
    \end{alignat*}
    where the first inequality uses the fact that $\t{x}^{\cdot,j}(t)$ (resp $\t{y}^{\cdot,j}(t)$) is the limit, as $i \to \infty$, of $\t{x}^{i,j}(t)$ (resp $\t{y}^{i,j}(t)$).




    From Lemma \ref{l:near_tail_bound_L4}, we know that
    \begin{alignat*}{1}
        \Pr{\sup_{t} \dist\lrp{x(t),x(0)} \geq s} \leq O\lrp{\frac{1}{s^4}},
    \end{alignat*}
    where we use the fact that, by definition in \eqref{d:x^i(t)}, $x^i(t)$ are linear interpolations of $x^i(k)$. Next, notice that when $\sup_{t\in[0,T]} \dist\lrp{x(t), x(0)} \leq \frac{s^j - L_0}{L_\beta'}$, $x(t) = \t{x}^{\cdot,j}(t)$ for all $t\in [0,T]$. It thus follows that as $s^j \to \infty$, $\E{g_\epsilon\lrp{\dist\lrp{\t{x}^{\cdot,j}(T), \t{y}^{\cdot,j}(T)}}}$ converges to $\E{g_\epsilon\lrp{\dist\lrp{x(T),y(T)}}}$ almost surely. Thus taking limit of \eqref{e:t:pmfalkfma} as $j\to \infty$, i.e. $s^j \to \infty$,
    \begin{alignat*}{1}
        & \inf_{\Omega} \E{g_{\epsilon}\lrp{\dist\lrp{x(T),y(T)}^2}}\\
        \leq& \exp\lrp{-\alpha_\epsilon K\delta^i}g_{\epsilon}\lrp{\dist\lrp{x(0),y(0)}^2}  + O\lrp{\epsilon^{1/2}}.
    \end{alignat*}
    Finally, take the limit of $\epsilon \to 0$. Note that $g_\epsilon(r^2) \to g_0(r^2) = f(r)$, where $f$ is defined in Definition \ref{d:f}. Note also that $\alpha_\epsilon \to \alpha$ as defined in the lemma statement. Finally, the properties of $f$ follows from Lemma \ref{l:fproperties}.

\end{proof}

\section{Tail Bounds}
In this section, we establish various probability and moment bounds for  and \eqref{e:intro_sde}, \eqref{e:intro_euler_murayama} and \eqref{e:intro_sgld}. These bounds are used at many places in our proofs.

\subsection{One-Step Distance Bounds}
\subsubsection{Under Lipschitz Continuity}
\begin{lemma}[One-step distance evolution under Lipschitz Continuity]\label{l:near_tail_bound_one_step}
    Let $\beta$ be a vector field satisfying \ref{ass:beta_lipschitz}. Let $\delta \in \Re^+$ be a stepsize satisfying $\delta \leq {\frac{1}{16{L_\beta'}}}$. Let $x_0\in M$ be arbitrary, let $x_k$ denote the iterative process
    \begin{alignat*}{1}
        x_{k+1} = \Exp_{x_k}\lrp{\delta \beta(x_k) + \sqrt{\delta} \xi_k(x_k)}
    \end{alignat*}
    where $\xi_k$ denote a random variable that possibly depends on $x_k$. Let $\gamma_k(t): [0,1] \to M$ denote any minimizing geodesic with $\gamma(0) = x_k$ and $\gamma(1) = x_0$. Then for any positive integer $K$ and for any $k\leq K$, we can bound,
    \begin{alignat*}{1}
        \dist\lrp{x_{k+1}, x_0}^2 
        \leq&  \lrp{1 + 8\delta L_\beta' + \frac{1}{2K} + \delta L_R \lrn{\xi_k(x_k)}^2 + \delta^2 L_R L_0^2}\dist\lrp{x_k, x_0}^2 +  2\delta \lrn{\xi_k(x_k)}^2 + 8K\delta^2 L_0^2\\
        & \quad + \ind{\dist\lrp{x_k,x_0} \leq \frac{1}{\delta \sqrt{L_R} {L_\beta'}}}\lrp{- 2\lin{\sqrt{\delta} \xi_k(x_k), \gamma_k'(0)}}
    \end{alignat*}
\end{lemma}
\begin{proof}
    We will be using the bound from \cite{zhang2016first} (see Lemma \ref{l:zhang2016}). Let $v := \delta \beta(x_k) + \sqrt{\delta} \xi_k(x_k)$. Then Lemma \ref{l:zhang2016} bounds
    \begin{alignat*}{1}
        \dist\lrp{x_{k+1}, x}^2
        \leq& \dist\lrp{x_{k}, x_0}^2 - 2\lin{v, \gamma_k'(0)} + {\tc\lrp{\sqrt{L_R}\dist\lrp{x_k,x_0}}} \lrn{v}^2\\
        \leq& \dist\lrp{x_{k}, x_0}^2 - 2\lin{v, \gamma_k'(0)} + \lrp{1+\sqrt{L_R} \dist(x_k,x_0)} \lrn{v}^2
        \elb{e:triangle_lemma_far_copy:1}
    \end{alignat*}
    where $\zeta(r) := \frac{r}{\tanh(r)}$.
    
    We will consider two cases:\\
    \textbf{Case 1: $\dist\lrp{x_k,x_0} \leq \frac{1}{\delta\sqrt{L_R} {L_\beta'}}$}.\\
    From \eqref{e:triangle_lemma_far_copy:1}:
    \begin{alignat*}{1}
        &\dist\lrp{x_{k+1}, x_0}^2 \\
        \leq& \dist\lrp{x_{k}, x_0}^2 - 2\lin{\delta \beta(x_k) + \sqrt{\delta} \xi_k(x_k), \gamma_k'(0)} + \lrp{1+\sqrt{L_R} \dist(x_k,x_0)} \lrn{\delta \beta(x_k) + \sqrt{\delta} \xi_k(x_k)}^2\\
        \leq& \dist\lrp{x_{k}, x_0}^2 - 2\lin{\delta \beta(x_k) + \sqrt{\delta} \xi_k(x_k), \gamma_k'(0)} + \delta^2\lrp{{L_0}^2 +  {L_\beta'}^2 \dist\lrp{x_k,x_0}^2} + \delta \lrn{\xi_k(x_k)}^2\\
        &\quad + \delta^2 \sqrt{L_R} \lrp{L_0^2\dist\lrp{x_k,x_0} + {L_\beta'}^2 \dist\lrp{x_k,x_0}^3 }+ \delta \sqrt{L_R} \lrn{{\xi_k(x_k)}}^2 \dist\lrp{x_k,x_0}\\
        \leq& \dist\lrp{x_{k}, x_0}^2 + \delta L_\beta' \dist\lrp{x_{k}, x_0}^2  + K\delta^2 L_0^2 + \frac{1}{4K} \dist\lrp{x_{k}, x_0}^2 - 2\lin{\sqrt{\delta} \xi_k(x_k), \gamma_k'(0)}\\
        &\quad + \delta^2 L_0^2 + \delta L_\beta' \dist\lrp{x_k,x_0}^2 + \delta \lrn{\xi_k(x_k)}^2\\
        &\quad + \delta^2 L_R L_0^2 \dist\lrp{x_,x_0}^2 + \delta^2 L_0^2 + \delta L_\beta' \dist\lrp{x_k,x_0}^2 + \delta L_R \lrn{\xi_k(x_k)}^2 \dist\lrp{x_k,x_0}^2 + {\delta \lrn{\xi_k(x_k)}^2}\\
        \leq& \lrp{1 + 3\delta L_\beta'+ \delta L_R \lrn{\xi_k(x_k)}^2 + \frac{1}{4K} + \delta^2 L_R L_0^2}\dist\lrp{x_{k}, x_0}^2 - 2\lin{\sqrt{\delta} \xi_k(x_k), \gamma_k'(0)}\\
        &\quad + \lrp{2K\delta^2 L_0^2 + \delta \lrn{\xi_k(x_k)}^2 + {\delta \lrn{\xi_k(x_k)}^2}}
    \end{alignat*}
    where the third inequality uses the definition of Case 1, and the fourth inequality is by several applications of Young's Inequality.
    
    \textbf{Case 2: $\dist\lrp{x_k,x_0} > \frac{1}{4\delta \sqrt{L_R} {L_\beta'}}$}.\\
    Let us define 
    \begin{alignat*}{1}
        z(t) := \Exp_{x_k}\lrp{t\lrp{\delta \beta(x_k) + \sqrt{\delta}\xi_k(x_k)}}
    \end{alignat*}
    I.e. $z(t)$ interpolates between $x_k$ and $x_{k+1}$. We verify that $z'(t) = \party{z(0)}{z(t)} \lrp{\delta \beta(x_k) + \sqrt{\delta}\xi_k(x_k)}$. Let us also define a family of geodesics $\gamma_t$, where for each $t$, $\gamma_t$ is a minimizing geodesic with $\gamma_t(0) = z(t)$ and $\gamma_t(1) = x_0$. If such a minimizing geodesic is not unique, any choice will do. We verify that
    \begin{alignat*}{1}
        \frac{d}{dt} \dist\lrp{z(t), x_0}^2
        \leq& -2\lin{\gamma_t'(0), z'(t)}\\
        \leq& \underbrace{- 2 \lin{\delta \beta(z(t)), \gamma_t'(0)} }_{\circled{1}} \\
        &\quad + \underbrace{2\lin{\delta \beta(z(t)) - \party{z(0)}{z(t)}\lrp{\delta \beta(x_k) }, \gamma_t'(0)}}_{\circled{2}} - \underbrace{2 \lin{\party{z(0)}{z(t)}\lrp{\sqrt{\delta}\xi_k(x_k)}, \gamma_t'(0)} }_{\circled{3}}.
    \end{alignat*}

    Let's upper bound the terms one by one.
    
    We first bound $\circled{2}$, which represents the "discretization error in drift":
    \begin{alignat*}{1}
        \circled{2} :=& 2 \lin{\delta \beta(z(t)) - \party{z(0)}{z(t)}\lrp{\delta \beta(x_k) }, \gamma_t'(0)}\\
        \leq& 2 \lrn{\delta \beta(z(t)) - \party{z(0)}{z(t)}\lrp{\delta \beta(x_k) }} \dist\lrp{z(t), x_0} \\
        \leq& 2 \delta {L_\beta'} \dist\lrp{z(t), x_k} \dist\lrp{z(t), x_0} \\
        \leq& 2\delta L_\beta' \dist\lrp{z(t),x_k}^2 + 2\delta L_\beta' \dist\lrp{z(t),x_0}^2
    \end{alignat*}
    
    By definition of $z(t)$, we know that $\dist\lrp{z(t), x_k} \leq \lrn{\delta \beta(x_k) + \sqrt{\delta} \xi_k(x_k)} \leq \delta L_0 + \delta {L_\beta'} \dist\lrp{x_k,x_0} + \sqrt{\delta} \lrn{\xi_k(x_k)}$, so that $2\delta L_\beta' \dist\lrp{z(t),x_k}^2 \leq 8 \delta^3 {L_\beta'}^3 \dist\lrp{x_k,x_0}^2 + 8\delta^2 {L_\beta'} \lrn{\xi_k(x_k)}^2 + 8 \delta^3 L_\beta' L_0^2\leq \delta L_\beta' \dist\lrp{x_k,x_0}^2 + \delta \lrn{\xi_k(x_k)}^2 + \delta^2 L_0^2$, so that
    \begin{alignat*}{1}
        \circled{2} 
        \leq& \delta L_\beta' \dist\lrp{z(t),x_k}^2 + \delta L_\beta' \dist\lrp{x_k,x_0}^2 + \delta \lrn{\xi_k(x_k)}^2 + \delta^2 L_0^2
    \end{alignat*}
    
    Next, we bound $\circled{3}$, which is the most significant error term. From the definition of Case 2, $\dist\lrp{x_k,x_0} > \frac{1}{\delta\sqrt{L_R} {L_\beta'}}$,
    \begin{alignat*}{1}
        \circled{3} \leq & 2 \lin{\party{z(0)}{z(t)}\lrp{\sqrt{\delta}\xi_k(x_k)}, \gamma_t'(0)} \\
        \leq& 2 \sqrt{\delta} \lrn{\xi_k(x_k)} \dist\lrp{z(t), x_0}\\
        \leq& \delta L_\beta' \dist\lrp{z(t),x_0}^2 + \frac{1}{L_\beta'} \lrn{\xi_k(x_k)}^2\\
        \leq& \delta L_\beta' \dist\lrp{z(t),x_0}^2 + {\delta L_R} \lrn{\xi_k(x_k)}^2 \dist\lrp{x_k,x_0}^2
    \end{alignat*}
    where we use our assumption that $\delta \leq \frac{1}{ {L_\beta'}}$.

    Finally, we bound $\circled{1}$ as
    \begin{alignat*}{1}
        - 2 \lin{\delta \beta(z(t)), \gamma_t'(0)} 
        \leq& 4\delta L_\beta' \dist\lrp{z(t),x_0}^2 + 4K\delta^2 L_0^2 + \frac{1}{4K} \dist\lrp{z(t),x_0}^2
    \end{alignat*}

    Putting everything together,
    \begin{alignat*}{1}
        \frac{d}{dt} \dist\lrp{z(t), x_0}^2
        \leq& \lrp{6\delta L_\beta' + \frac{1}{4K}} \dist\lrp{z(t), x_0}^2 + \lrp{\delta L_\beta' + \delta L_R \lrn{\xi_k(x_k)}^2} \dist\lrp{x_k, x_0}^2 + \delta \lrn{\xi_k(x_k)}^2 + 4K\delta^2 L_0^2
    \end{alignat*}
    By Gronwall's Lemma (integrating from $t=0$ to $t=1$), 
    \begin{alignat*}{1}
        &\dist\lrp{x_{k+1}, x_0}^2 \\
        =& \dist\lrp{z(1),x_0}^2 \\
        \leq& \exp\lrp{6\delta L_\beta' + \frac{1}{4K} } \dist\lrp{x_k, x_0}^2 + \lrp{2\delta L_\beta' + 2\delta L_R \lrn{\xi_k(x_k)}^2} \dist\lrp{x_k, x_0}^2 + 2\delta \lrn{\xi_k(x_k)}^2 + 8K\delta^2 L_0^2\\
        \leq& \lrp{1 + 8\delta L_\beta' + \frac{1}{2K} + \delta L_R \lrn{\xi_k(x_k)}^2}\dist\lrp{x_k, x_0}^2 +  2\delta \lrn{\xi_k(x_k)}^2 + 8K\delta^2 L_0^2
    \end{alignat*}
    where we use the assumption that $\delta \leq \frac{1}{8 L_\beta'}$.

    \textbf{Combining Case 1 and Case 2:}
    \begin{alignat*}{1}
        \dist\lrp{x_{k+1}, x_0}^2 
        \leq&  \lrp{1 + 8\delta L_\beta' + \frac{1}{2K} + \delta L_R \lrn{\xi_k(x_k)}^2 + \delta^2 L_R L_0^2}\dist\lrp{x_k, x_0}^2 +  2\delta \lrn{\xi_k(x_k)}^2 + 8K\delta^2 L_0^2\\
        & \quad + \ind{\dist\lrp{x_k,x_0} \leq \frac{1}{\delta \sqrt{L_R} {L_\beta'}}}\lrp{- 2\lin{\sqrt{\delta} \xi_k(x_k), \gamma_k'(0)}}
    \end{alignat*}
\end{proof}
\subsubsection{Under Dissipativity}
\begin{lemma}[One-step distance evolution under Dissipativity]\label{l:far_tail_bound_one_step}
    Let $\beta$ be a vector field satisfying \ref{ass:beta_lipschitz}. Let $\delta \in \Re^+$ be a stepsize satisfying $\delta \leq {\frac{m}{128 {L_\beta'}^2}}$. Let $x^*$ be some point with $\beta(x^*)=0$. Let $x_0 \in M$ be arbitrary. Let $x_k$ be the iterative process
    \begin{alignat*}{1}
        x_{k+1} = \Exp_{x_k}\lrp{\delta \beta(x_k) + \sqrt{\delta} \xi_k(x_k)}
    \end{alignat*}
    Assume that for all $x$ such that $\dist\lrp{x,x^*} \geq \R$, there exists a minimizing geodesic $\gamma : [0,1]\to M$ with $\gamma(0) = x, \gamma(1) = x^*$, and
    \begin{alignat*}{1}
        \lin{\beta(x), \gamma'(0)} \leq - m \dist\lrp{x,x^*}^2,
    \end{alignat*}
    and let $\gamma_k$ denote such a geodesic for $x=x_k$. Then for any $k$, 
    \begin{alignat*}{1}
        \dist\lrp{x_{k+1}, x^*}^2 
        \leq& \lrp{1 - \delta m}\dist\lrp{x_k, x^*}^2 + {\frac{2048 \delta L_R {L_\beta'}^4}{m^5} \lrn{\xi_k(x_k)}^4 + 4\delta {L_\beta'} \R^2}\\
        & \quad + \ind{\dist\lrp{x_k,x^*} \leq \frac{m}{4\delta\sqrt{L_R} {L_\beta'}^2}}\lrp{- 2\lin{\sqrt{\delta} \xi_k(x_k), \gamma_k'(0)}}
    \end{alignat*}
\end{lemma}
\begin{proof}
    Throughout the proof, it is useful to note that by our assumptions, it must be that $m\leq L_\beta'$. We will be using the bound from \cite{zhang2016first} (see Lemma \ref{l:zhang2016}). Let $v := \delta \beta(x_k) + \sqrt{\delta} \xi_k(x_k)$. Then Lemma \ref{l:zhang2016} bounds
    \begin{alignat*}{1}
        \dist\lrp{x_{k+1}, x}^2
        \leq& \dist\lrp{x_{k}, x^*}^2 - 2\lin{v, \gamma_k'(0)} + {\tc\lrp{\sqrt{L_R}\dist\lrp{x_k,x^*}}} \lrn{v}^2\\
        \leq& \dist\lrp{x_{k}, x^*}^2 - 2\lin{v, \gamma_k'(0)} + \lrp{1+\sqrt{L_R} \dist(x_k,x^*)} \lrn{v}^2
        \elb{e:triangle_lemma_far_copy}
    \end{alignat*}
    where $\zeta(r) := \frac{r}{\tanh(r)}$.
    
    We will consider two cases:\\
    \textbf{Case 1: $\dist\lrp{x_k,x^*} \leq \frac{m}{4\delta\sqrt{L_R} {L_\beta'}^2}$}.\\
    From \eqref{e:triangle_lemma_far_copy}:
    \begin{alignat*}{1}
        &\dist^2\lrp{x_{k+1}, x^*}^2 \\
        \leq& \dist^2\lrp{x_{k}, x^*} - 2\lin{\delta \beta(x_k) + \sqrt{\delta} \xi_k(x_k), \gamma_k'(0)} + \lrp{1+\sqrt{L_R} \dist(x_k,x^*)} \lrn{\delta \beta(x_k) + \sqrt{\delta} \xi_k(x_k)}^2\\
        \leq& \dist^2\lrp{x_{k}, x^*} - 2\lin{\delta \beta(x_k) + \sqrt{\delta} \xi_k(x_k), \gamma_k'(0)} + \delta^2 {L_\beta'}^2 \dist\lrp{x_k,x^*}^2 + \delta \lrn{\xi_k(x_k)}^2\\
        &\quad + \delta^2 \sqrt{L_R} {L_\beta'}^2 \dist\lrp{x_k,x^*}^3 + \delta \sqrt{L_R} \lrn{{\xi_k(x_k)}}^2 \dist\lrp{x_k,x^*}\\
        \leq& \lrp{1 + \delta m/2}\dist^2\lrp{x_{k}, x^*} - 2\lin{\delta \beta(x_k) + \sqrt{\delta} \xi_k(x_k), \gamma_k'(0)} + \frac{4\delta L_R}{m} \lrn{{\xi_k(x_k)}}^4
        \elb{e:t:dakkmd:1}
    \end{alignat*}
    where we use our assumptions that $\delta \leq m/\lrp{16{L_\beta'}^2}$ and the inequality under Case 1. We used Cauchy Schwarz a few times.

    We can further bound
    \begin{alignat*}{1}
        2\lin{\delta \beta(x_k), \gamma_k'(0)}
        \leq& \ind{\dist\lrp{x_k,x^*}\geq \R}\lrp{-2m\dist\lrp{x_k,x^*}^2} + \ind{\dist\lrp{x_k,x^*}\leq \R}\lrp{2L_\beta' \dist\lrp{x_k,x^*}^2}\\
        \leq& -2\delta m\dist\lrp{x_k,x^*}^2 + 2\delta \lrp{m + L_\beta'}\R^2
    \end{alignat*}
    Thus
    \begin{alignat*}{1}
        \dist^2\lrp{x_{k+1}, x^*}^2 
        \leq& \lrp{1 -\delta m}\dist^2\lrp{x_{k}, x^*} - 2\lin{\sqrt{\delta} \xi_k(x_k), \gamma_k'(0)} \\
        &\quad + 2\delta \lrp{m + L_\beta'}\R^2 + \frac{4\delta L_R}{m} \lrn{{\xi_k(x_k)}}^4
        \elb{e:t:omfl:1}
    \end{alignat*}
    where we use the fact that $e^{\delta m} \leq e^{\frac{m^2}{16{L_\beta'}^2}} \leq 2$.

    \textbf{Case 2: $\dist\lrp{x_k,x^*} > \frac{m}{4\delta \sqrt{L_R} {L_\beta'}^2}$}.\\
    Let us define 
    \begin{alignat*}{1}
        z(t) := \Exp_{x_k}\lrp{t\lrp{\delta \beta(x_k) + \sqrt{\delta}\xi_k(x_k)}}
    \end{alignat*}
    I.e. $z(t)$ interpolates between $x_k$ and $x_{k+1}$. We verify that $z'(t) = \party{z(0)}{z(t)} \lrp{\delta \beta(x_k) + \sqrt{\delta}\xi_k(x_k)}$. Let us also define a family of geodesics $\gamma_t$, where for each $t$, $\gamma_t$ is a minimizing geodesic with $\gamma_t(0) = z(t)$ and $\gamma_t(1) = x^*$. If such a minimizing geodesic is not unique, any choice will do. We also verify that 
    \begin{alignat*}{1}
        \frac{d}{dt} \dist\lrp{z(t), x^*}^2
        \leq& -2\lin{\gamma_t'(0), z'(t)}\\
        \leq& \underbrace{- 2 \lin{\delta \beta(z(t)), \gamma_t'(0)} }_{\circled{1}} \\
        &\quad + \underbrace{2\lin{\delta \beta(z(t)) - \party{z(0)}{z(t)}\lrp{\delta \beta(x_k) }, \gamma_t'(0)}}_{\circled{2}} - \underbrace{2 \lin{\party{z(0)}{z(t)}\lrp{\sqrt{\delta}\xi_k(x_k)}, \gamma_t'(0)} }_{\circled{3}}
    \end{alignat*}
    Let's upper bound the terms one by one.
    
    We first bound $\circled{2}$, which represents the "discretization error in drift":
    \begin{alignat*}{1}
        \circled{2} :=& 2 \lin{\delta \beta(z(t)) - \party{z(0)}{z(t)}\lrp{\delta \beta(x_k) }, \gamma_t'(0)}\\
        \leq& 2 \lrn{\delta \beta(z(t)) - \party{z(0)}{z(t)}\lrp{\delta \beta(x_k) }} \dist\lrp{z(t), x^*} \\
        \leq& 2 \delta {L_\beta'} \dist\lrp{z(t), x_k} \dist\lrp{z(t), x^*} \\
        \leq& \frac{\delta m}{4} \dist\lrp{z(t),x^*}^2 + \frac{4\delta {L_\beta'}^2}{m} \dist\lrp{z(t),x_k}^2
    \end{alignat*}
    
    By definition of $z(t)$, we know that $\dist\lrp{z(t), x_k} \leq \lrn{\delta \beta(x_k) + \sqrt{\delta} \xi_k(x_k)} \leq \delta {L_\beta'} \dist\lrp{x_k,x^*} + \sqrt{\delta} \lrn{\xi_k(x_k)}$, so that $\frac{4\delta {L_\beta'}^2}{m} \dist\lrp{z(t),x_k}^2 \leq \frac{4\delta^3 {L_\beta'}^4}{m} \dist\lrp{x_k,x^*}^2 + \frac{4\delta^2 {L_\beta'}^2}{m} \lrn{\xi_k(x_k)}^2 \leq \frac{\delta m}{8} \dist\lrp{x_k,x^*}^2 + \delta \lrn{\xi_k(x_k)}^2$, so that
    \begin{alignat*}{1}
        \circled{2} 
        \leq& \frac{\delta m}{4} \dist\lrp{z(t),x^*}^2 + \frac{\delta m}{8} \dist\lrp{x_k,x^*}^2 + \delta \lrn{\xi_k(x_k)}^2
    \end{alignat*}
    
    Next, we bound $\circled{3}$, which is the most significant error term. From the definition of Case 2, $\dist\lrp{x_k,x^*} > \frac{1}{4\delta\sqrt{L_R} {L_\beta'}}$,
    \begin{alignat*}{1}
        \circled{3} \leq & 2 \lin{\party{z(0)}{z(t)}\lrp{\sqrt{\delta}\xi_k(x_k)}, \gamma_t'(0)} \\
        \leq& 2 \sqrt{\delta} \lrn{\xi_k(x_k)} \dist\lrp{z(t), x^*}\\
        \leq& \frac{\delta m}{8} \dist\lrp{z(t),x^*}^2 + \frac{8}{m} \lrn{\xi_k(x_k)}^2\\
        \leq& \frac{\delta m}{8} \dist\lrp{z(t),x^*}^2 + \frac{32\delta \sqrt{L_R} {L_\beta'}^2}{m^2} \lrn{\xi_k(x_k)}^2 \dist\lrp{x_k,x^*}\\
        \leq& \frac{\delta m}{8} \dist\lrp{z(t),x^*}^2 + 
        \frac{\delta m}{8} \dist\lrp{x_k,x^*}^2 + \frac{2048 \delta L_R {L_\beta'}^4}{m^5} \lrn{\xi_k(x_k)}^4
    \end{alignat*}
    where we use our assumption that $\delta \leq \frac{m}{128 {L_\beta'}^2}$.

    Finally, we bound $\circled{1}$ as
    \begin{alignat*}{1}
        &- 2 \lin{\delta \beta(z(t)), \gamma_t'(0)} \\
        \leq& \ind{\dist\lrp{z(t),x^*}^2 \geq \R}\lrp{-2 \delta m \dist\lrp{z(t),x^*}^2} + \ind{\dist\lrp{z(t),x^*} \leq \R}\lrp{2 \delta L_\beta' \dist\lrp{z(t)x^*}}\\
        \leq& 4 \delta L_\beta' \R^2 - 2\delta m \dist\lrp{\dist\lrp{z(t),x^*}^2}
    \end{alignat*}

    Putting everything together,
    \begin{alignat*}{1}
        \frac{d}{dt} \dist\lrp{z(t), x^*}^2
        \leq& - \frac{3\delta m}{2} \dist\lrp{z(t), x^*}^2 + \frac{\delta m}{4} \dist\lrp{x_k, x^*}^2 +  \frac{2048 \delta L_R {L_\beta'}^4}{m^5} \lrn{\xi_k(x_k)}^4 + 4\delta {L_\beta'} \R^2
    \end{alignat*}
    By Gronwall's Lemma (integrating from $t=0$ to $t=1$), 
    \begin{alignat*}{1}
        \dist\lrp{x_{k+1}, x^*}^2 
        =& \dist\lrp{z(1),x^*}^2 \\
        \leq& \exp\lrp{-{3\delta m/2}} \dist\lrp{x_k, x^*}^2 + \frac{\delta m}{4} \dist\lrp{x_k, x^*}^2 + \frac{2048 \delta L_R {L_\beta'}^4}{m^5} \lrn{\xi_k(x_k)}^4 + 4\delta{L_\beta'} \R^2 \\
        \leq& \lrp{1 - \delta m}\dist\lrp{x_k, x^*}^2 + \frac{2048 \delta L_R {L_\beta'}^4}{m^5} \lrn{\xi_k(x_k)}^4 + 4\delta{L_\beta'} \R^2
        \elb{e:t:omfl:2}
    \end{alignat*}
    where we use the assumption that $\delta \leq \frac{1}{128 m}$ so that $\exp\lrp{-3\delta m/2} \leq 1-5\delta m/4$.

    \textbf{Combining Case 1 and Case 2:}

    Combining \eqref{e:t:omfl:1} and \eqref{e:t:omfl:2},
    \begin{alignat*}{1}
        \dist\lrp{x_{k+1}, x^*}^2 
        \leq& \lrp{1 - \delta m}\dist\lrp{x_k, x^*}^2 + \lrp{\frac{2048 \delta L_R {L_\beta'}^4}{m^5} \lrn{\xi_k(x_k)}^4 + 4\delta {L_\beta'} \R^2}\\
        & \quad + \ind{\dist\lrp{x_k,x^*} \leq \frac{1}{4\delta \sqrt{L_R} {L_\beta'}}}\lrp{- 2\lin{\sqrt{\delta} \xi_k(x_k), \gamma_k'(0)}}
    \end{alignat*}
\end{proof}

\subsection{$L_p$ Bounds}
\subsubsection{Under Lipschitz Continuity}
\begin{lemma}[L2 Bound and Chevyshev under Lipschitz Continuity]\label{l:near_tail_bound_L2}
    Consider the same setup as Lemma \ref{l:near_tail_bound_one_step}. Assume in addition that there exists $\sigma_\xi \in \Re^+$ such that for all $x$ and for all $k$, $\E{\lrn{\xi_k(x)}^2} \leq \sigma_\xi^2$. Then for any positive integer $K$, and for all $k\leq K$,
    \begin{alignat*}{1}
        \E{\dist\lrp{x_k,x_0}^2} \leq 4\exp\lrp{8K\delta L_\beta' + K\delta L_R \sigma_\xi^2 + K\delta^2 L_R L_0^2} \cdot \lrp{2K\delta \sigma_\xi^2 + 8K^2\delta^2 L_0^2}
    \end{alignat*}
    and
    \begin{alignat*}{1}
        \Pr{\max_{k\leq K} \dist\lrp{x_k,x_0} \geq s}\leq \frac{4}{s^2}\exp\lrp{8K\delta L_\beta' + K\delta L_R \sigma_\xi^2 + K\delta^2 L_R L_0^2} \cdot \lrp{2K\delta \sigma_\xi^2 + 8K^2\delta^2 L_0^2}
    \end{alignat*}
\end{lemma}
\begin{proof}
    Let $\F_k$ denote the $\sigma$-field generated by $\xi_0...\xi_{k-1}$.

    To bound the first claim, take expectation of the bound from Lemma \ref{l:near_tail_bound_one_step} wrt $\F_k$:
    \begin{alignat*}{1}
        &\Ep{\F_k}{\dist\lrp{x_{k+1},x_0}^2}\\
        \leq& \Ep{\F_k}{\lrp{1 + 8\delta L_\beta' + \frac{1}{2K} + \delta L_R \lrn{\xi_k(x_k)}^2 + \delta^2 L_R L_0^2}{\dist\lrp{x_k, x_0}^2}} +  2\delta \Ep{\F_k}{\lrn{\xi_k(x_k)}^2} + 8K\delta^2 L_0^2\\
        & \quad + \Ep{\F_k}{\ind{\dist\lrp{x_k,x_0} \leq \frac{1}{\delta \sqrt{L_R} {L_\beta'}}}\lrp{- 2\lin{\sqrt{\delta} \xi_k(x_k), \gamma_k'(0)}}}\\
        \leq& \lrp{1 + 8\delta L_\beta' + \frac{1}{2K} + \delta L_R \sigma_\xi^2 + \delta^2 L_R L_0^2}{\dist\lrp{x_k, x_0}^2} +  2\delta \sigma_\xi^2 + 8K\delta^2 L_0^2\\
        \leq& \exp\lrp{8\delta L_\beta' + \frac{1}{2K} + \delta L_R \sigma_\xi^2 + \delta^2 L_R L_0^2}{\dist\lrp{x_k, x_0}^2} +  2\delta \sigma_\xi^2 + 8K\delta^2 L_0^2
    \end{alignat*}
    In line 3 above, $\gamma_k(t)$ is a minimizing geodesic from $x_k$ to $x_0$, as defined in Lemma \ref{l:near_tail_bound_one_step}.

    Applying the above recursively,
    \begin{alignat*}{1}
        \E{\dist\lrp{x_K,x_0}^2} \leq \exp\lrp{1 + 8K\delta L_\beta' + K\delta L_R \sigma_\xi^2 + K\delta^2 L_R L_0^2} \cdot \lrp{2K\delta \sigma_\xi^2 + 8K^2\delta^2 L_0^2}
    \end{alignat*}
    The above upper bound clearly also holds for $\E{\dist\lrp{x_k,x_0}^2}$ for all $k\leq K$. This proves our first claim.

    To prove the second claim, let us define
    \begin{alignat*}{1}
        & r_0^2 := 0 \\
        & r_{k+1}^2 := \lrp{1 + 8\delta L_\beta' + \frac{1}{4K} + \delta L_R \lrn{\xi_k(x_k)}^2 + \delta^2 L_R L_0^2}r_k^2 +  2\delta \lrn{\xi_k(x_k)}^2 + 8K\delta^2 L_0^2\\
        & \quad + \ind{\dist\lrp{x_k,x_0} \leq \frac{1}{\delta \sqrt{L_R} {L_\beta'}}}\lrp{- 2\lin{\sqrt{\delta} \xi_k(x_k), \gamma_k'(0)}}
    \end{alignat*}

    We verify that $r_k$ as defined above is a sub-martingale. Thus by Doob's martingale inequality,
    \begin{alignat*}{1}
        \Pr{\max_{k\leq K} r_k^2 \geq s} \leq \frac{\E{r_K^2}}{s}
    \end{alignat*}
    Furthermore, notice that 
    \begin{alignat*}{1}
        & r_0^2 = \dist\lrp{x_0,x_0}^2 = 0\\
        & r_{k+1}^2 - \dist\lrp{x_{k+1}^2,x_0}^2 \geq \lrp{1 + 8\delta L_\beta' + \frac{2}{K} + \delta L_R \lrn{\xi_k(x_k)}^2 + \delta^2 L_R L_0^2} \lrp{r_k^2 - \dist\lrp{x_k,x_0}^2} \geq 0
    \end{alignat*}
    so that $r_k \geq \dist\lrp{x_k,x_0}$ with probability 1, for all $k$.

    Thus
    \begin{alignat*}{1}
        &\Pr{\max_{k\leq K} \dist\lrp{x_k,x_0} \geq s} \leq \Pr{\max_{k\leq K} r_k^2 \geq s^2} \leq \frac{\E{r_K^2}}{s^2} \\
        &\leq \frac{1}{s^2}\exp\lrp{1 + 8K\delta L_\beta' + K\delta L_R \sigma_\xi^2 + K\delta^2 L_R L_0^2} \cdot \lrp{2K\delta \sigma_\xi^2 + 8K^2\delta^2 L_0^2}
    \end{alignat*}

    The proof for the bound on $r_K^2$ is identical to the proof of the first claim. We conclude our proof of the second claim
\end{proof}

\begin{lemma}[L4 Bound and Chevyshev under Lipschitz Continuity]\label{l:near_tail_bound_L4}
    Let $\beta$ be a vector field satisfying \ref{ass:beta_lipschitz}. Assume in addition that $\delta \in \Re^+$ satisfies $\delta \leq \min\lrbb{\frac{1}{16{L_\beta'}}, \frac{1}{16\sqrt{L_R} L_0}, \frac{1}{L_R d}}$. Let $L_0 := \lrn{\beta(x_0)}$. Let $x_k$ be the following stochastic process:
    \begin{alignat*}{1}
        x_{k+1} = \Exp_{x_k}\lrp{\delta \beta(x_k) + \sqrt{\delta} \xi_k(x_k)}
    \end{alignat*}
   
    Assume in addition that for all $x$ and for all $k$, $\E{\lrn{\xi_k(x)}^4} \leq 2d^2$. Then for any positive $K \geq 4$, and for all $k\leq K$,
    \begin{alignat*}{1}
        \E{\dist\lrp{x_k,x_0}^4} \leq \exp\lrp{2 + 16 K\delta L_\beta' + 4K\delta L_R d + 3K\delta^2 L_R L_0^2}\lrp{5K^2\delta^2 d^2 + 64K^4\delta^4 L_0^4}
    \end{alignat*}
    and
    \begin{alignat*}{1}
        &\Pr{\max_{k\leq K} \dist\lrp{x_k,x_0} \geq s}\leq \frac{1}{s^4}\exp\lrp{2 + 16 K\delta L_\beta' + 4K\delta L_R d + 3K\delta^2 L_R L_0^2}\lrp{5K^2\delta^2 d^2 + 64K^4\delta^4 L_0^4}
    \end{alignat*}
\end{lemma}
\begin{proof}
    Let $\F_k$ denote the $\sigma$-field generated by $\xi_0...\xi_{k-1}$.

    We will use the following inequality from Lemma \ref{l:near_tail_bound_one_step}:
    \begin{alignat*}{1}
        \dist\lrp{x_{k+1}, x_0}^2 
        \leq&  \lrp{1 + 8\delta L_\beta' + \frac{1}{2K} + \delta L_R \lrn{\xi_k(x_k)}^2 + \delta^2 L_R L_0^2}\dist\lrp{x_k, x_0}^2 +  2\delta \lrn{\xi_k(x_k)}^2 + 8K\delta^2 L_0^2\\
        & \quad + \ind{\dist\lrp{x_k,x_0} \leq \frac{1}{\delta \sqrt{L_R} {L_\beta'}}}\lrp{- 2\lin{\sqrt{\delta} \xi_k(x_k), \gamma_k'(0)}}
    \end{alignat*}

    Squaring both sides,
    \begin{alignat*}{1}
        & \dist\lrp{x_{k+1},x_0}^4 \\
        \leq& \lrp{1 + 16 \delta L_\beta' + 3\delta L_R \lrn{\xi_k\lrp{x_k}}^2 + 3\delta^2 L_R L_0^2 + 2\delta^2 L_R^2 \lrn{\xi_k(x_k)}^4 + \frac{2}{K}} \dist\lrp{x_k,x_0}^4\\
        &\quad + 5K\delta^2 \lrn{\xi_k(x_k)}^4 + 64 K^3\delta^4L_0^4\\
        &\quad + \circled{*}
        \elb{e:t:alkdsmalsc}
    \end{alignat*}
    where $\circled{*}$ has $0$-mean, and we used a few times Cauchy Schwarz and Young's Inequality.

    Taking expectation wrt $\F_k$,
    \begin{alignat*}{1}
        &\Ep{\F_k}{\dist\lrp{x_{k+1},x_0}^4}\\
        \leq& \lrp{1 + 16 \delta L_\beta' + 3\delta L_R d + 3\delta^2 L_R L_0^2 + 2\delta^2 L_R^2 d^2 + \frac{2}{K}} \dist\lrp{x_k,x_0}^4 + 5K\delta^2 d^2 + 64K^3\delta^4 L_0^4\\
        \leq& \lrp{1 + 16 \delta L_\beta' + 4\delta L_R d + 3\delta^2 L_R L_0^2 + \frac{2}{K}} \dist\lrp{x_k,x_0}^4 + 5K\delta^2 d^2 + 64K^3\delta^4 L_0^4
    \end{alignat*}

    Applying the above recursively,
    \begin{alignat*}{1}
        \E{\dist\lrp{x_{K},x_0}^4} 
        \leq& \exp\lrp{2 + 16 K\delta L_\beta' + 4K\delta L_R d + 3K\delta^2 L_R L_0^2}\lrp{5K^2\delta^2 d^2 + 64K^4\delta^4 L_0^4}
    \end{alignat*}

    To prove the second claim, define
    \begin{alignat*}{1}
        & r_0^4 := 0 \\
        & r_{k+1}^4 := \lrp{1 + 16 \delta L_\beta' + 3\delta L_R \lrn{\xi_k\lrp{x_k}}^2 + 3\delta^2 L_R L_0^2 + 2\delta^2 L_R^2 \lrn{\xi_k(x_k)}^4 + \frac{2}{K}} r_k^4\\
        &\quad + 5K\delta^2 \lrn{\xi_k(x_k)}^4 + 64 K^3\delta^4L_0^4\\
        &\quad + \circled{*}
    \end{alignat*}
    where $\circled{*}$ is the same term as \eqref{e:t:alkdsmalsc}.

    We verify that $r_k$ as defined above is a sub-martingale. Thus by Doob's martingale inequality,
    \begin{alignat*}{1}
        \Pr{\max_{k\leq K} r_k^2 \geq s} \leq \frac{\E{r_K^2}}{s}
    \end{alignat*}
    Furthermore, notice that 
    \begin{alignat*}{1}
        r_0^2 =& \dist\lrp{x_0,x_0}^2 = 0\\
        r_{k+1}^2 - \dist\lrp{x_{k+1}^2,x_0}^2 \geq& \lrp{1 + 16 \delta L_\beta' + 3\delta L_R \lrn{\xi_k\lrp{x_k}}^2 + 3\delta^2 L_R L_0^2 + 2\delta^2 L_R^2 \lrn{\xi_k(x_k)}^4 + \frac{2}{K}} \lrp{r_k^2 - \dist\lrp{x_k,x_0}^2} \\
        \geq& 0
    \end{alignat*}
    so that $r_k \geq \dist\lrp{x_k,x_0}$ with probability 1, for all $k$.

    Thus
    \begin{alignat*}{1}
        &\Pr{\max_{k\leq K} \dist\lrp{x_k,x_0} \geq s} \leq \Pr{\max_{k\leq K} r_k \geq s} \leq \frac{\E{r_K^4}}{s^4} \\
        &\leq \frac{1}{s^4}\exp\lrp{2 + 16 K\delta L_\beta' + 4K\delta L_R d + 3K\delta^2 L_R L_0^2}\lrp{5K^2\delta^2 d^2 + 64K^4\delta^4 L_0^4}
    \end{alignat*}

    The proof for the bound on $r_K^2$ is identical to the proof of the first claim. We conclude our proof of the second claim
\end{proof}
\subsubsection{Under Dissipativity}
\begin{lemma}[L4 Bound under Dissipativity, Discretized SDE]
    \label{l:far-tail-bound-l2}
    Let $\beta$ be a vector field satisfying \ref{ass:beta_lipschitz}. Let $x^*$ be some point with $\beta(x^*)=0$. Assume that for all $x$ such that $\dist\lrp{x,x^*} \geq \R$, there exists a minimizing geodesic $\gamma : [0,1]\to M$ with $\gamma(0) = x, \gamma(1) = x^*$, and
    \begin{alignat*}{1}
        \lin{\beta(x), \gamma'(0)} \leq - m \dist\lrp{x,x^*}^2
    \end{alignat*}. Assume in addition that $\delta \in \Re^+$ satisfies $\delta \leq {\frac{m}{128 {L_\beta'}^2}}$
    Let $x_k$ be the following stochastic process:
    \begin{alignat*}{1}
        x_{k+1} = \Exp_{x_k}\lrp{\delta \beta(x_k) + \sqrt{\delta} \xi_k(x_k)}
    \end{alignat*}
    where $\xi_k$ is a random vector field satisfying $\E{\xi_k(x)} = 0$ and $\E{\lrn{\xi_k}^4} \leq \lrp{\sigma_\xi}^4$.

    For any $k$,
    \begin{alignat*}{1}
        \E{\dist\lrp{x_K,x^*}^4} \leq e^{-K \delta m } \E{\dist\lrp{x_0,x^*}^4} + \frac{2^{24} L_R^2 {L_\beta'}^8}{m^{12}}\sigma_\xi^8  + \frac{64 {L_\beta'}^2 \R^4}{m^2} + \frac{128}{m^2} \sigma_\xi^4
    \end{alignat*}
\end{lemma}
\begin{proof}
    Let $\gamma_k$ denote a geodesic with $\gamma_k(0) = x_k, \gamma(1) = x^*$, and
    \begin{alignat*}{1}
        \lin{\beta(x_k), \gamma_k'(0)} \leq - m \dist\lrp{x_k,x^*}^2
    \end{alignat*}

    From Lemma \ref{l:far_tail_bound_one_step}, 
    \begin{alignat*}{1}
        \dist\lrp{x_{k+1}, x^*}^2 
        \leq& \lrp{1 - \delta m}\dist\lrp{x_k, x^*}^2 + {\frac{2048 \delta L_R {L_\beta'}^4}{m^5} \lrn{\xi_k(x_k)}^4 + 4\delta {L_\beta'} \R^2}\\
        & \quad + \ind{\dist\lrp{x_k,x^*} \leq \frac{m}{4\delta \sqrt{L_R} {L_\beta'}^2}}\lrp{- 2\lin{\sqrt{\delta} \xi_k(x_k),\gamma_k'(0)}}
    \end{alignat*}
    Squaring both sides and taking expectation wrt $\xi_k$ (and applying Young's Inequality),
    \begin{alignat*}{1}
        & \E{\dist\lrp{x_{k+1}, x^*}^4}\\
        \leq& \lrp{1 - \frac{3 \delta m}{2} }\dist\lrp{x_k, x^*}^4 + \frac{2^{24} \delta L_R^2 {L_\beta'}^8}{m^{11}} \E{\lrn{\xi_k(x_k)}^8}  + \frac{64\delta {L_\beta'}^2 \R^4}{m} \\
        & \quad + \frac{\delta m}{2} \dist\lrp{x_k,x^*}^4 + \frac{128\delta}{m} \E{\lrn{\xi_k(x_k)}^4}\\
        \leq& \lrp{1 - \delta m}\dist\lrp{x_k, x^*}^4 + \frac{2^{24} \delta L_R^2 {L_\beta'}^8}{m^{11}} \E{\lrn{\xi_k(x_k)}^8}  + \frac{64\delta {L_\beta'}^2 \R^4}{m} + \frac{128\delta}{m} \E{\lrn{\xi_k(x_k)}^4}
    \end{alignat*}
    Applying the above recursively, 
    \begin{alignat*}{1}
        \E{\dist\lrp{x_K,x^*}^4} \leq e^{-K \delta m } \E{\dist\lrp{x_0,x^*}^4} + \frac{2^{24} L_R^2 {L_\beta'}^8}{m^{12}}\sigma_\xi^8  + \frac{64 {L_\beta'}^2 \R^4}{m^2} + \frac{128}{m^2} \sigma_\xi^4
    \end{alignat*}
\end{proof}

\begin{lemma}[L4 Bound under Dissipativity, Exact SDE]
    \label{l:far-tail-bound-l2-brownian}
    Let $\beta$ be a vector field satisfying \ref{ass:beta_lipschitz}. Let $x^*$ be some point with $\beta(x^*)=0$. Assume that for all $x$ such that $\dist\lrp{x,x^*} \geq \R$, there exists a minimizing geodesic $\gamma : [0,1]\to M$ with $\gamma(0) = x, \gamma(1) = x^*$, and
    \begin{alignat*}{1}
        \lin{\beta(x), \gamma'(0)} \leq - m \dist\lrp{x,x^*}^2
    \end{alignat*}.
    Let $x(t)$ denote the solution to \eqref{e:intro_euler_murayama}.
    For any $k$,
    \begin{alignat*}{1}
        \E{\dist\lrp{x(T),x^*}^4} \leq \exp\lrp{- K\delta m}\E{\dist\lrp{x^i_0,x^*}^4} + \frac{2^{26} L_R^2 {L_\beta'}^8}{m^{12}}d^4  + \frac{64 {L_\beta'}^2 \R^4}{m^2} + \frac{256}{m^2} d^2
    \end{alignat*}
\end{lemma}
\begin{proof}
    For $i \in \Z^+$, let $\delta^i$ be a sequence of stepsizes going to $0$, let $K^i := T/\delta^i$, and let $x^i_k$ be a discretization of $x(t)$ with stepsize $\delta^i$ of the form \eqref{e:intro_euler_murayama}, i.e.
    \begin{alignat*}{1}
        x^i_{k+1} = \Exp_{x^i_k}\lrp{\delta^i \beta(x^i_k) + \sqrt{\delta^i}\zeta_k}
    \end{alignat*}
    where $\zeta_k\sim \N_{x^i_k}(0,I)$.
    Applying Lemma \ref{l:far-tail-bound-l2} to $x^i_k$ with $\xi_k(x_k) = \zeta_k$ and $\sigma_\xi = \sqrt{2} d$, and for $i$ sufficiently large, gives the bound
    \begin{alignat*}{1}
        \E{\dist\lrp{x^i_{K^i},x^*}^4} \leq \exp\lrp{- K^i\delta^i m}\E{\dist\lrp{x^i_0,x^*}^4} + \frac{2^{26} L_R^2 {L_\beta'}^8}{m^{12}}d^4  + \frac{64 {L_\beta'}^2 \R^4}{m^2} + \frac{256}{m^2} d^2
    \end{alignat*}
    Our conclusion follows by Lemma \ref{l:Phi_is_diffusion} as $x^i_{K^i}$ converges to $x(T)$ almost surely as $i\to \infty$.
\end{proof}

\subsection{Subgaussian Bounds}

\subsubsection{Under Dissipativity}
\begin{lemma}[Subgaussian Bound under Dissipativity, Discrete Time Semimartingale, Adaptive Stepsize]\label{l:far-tail-bound-gaussian-adaptive}
    Let $\beta$ be a vector field satisfying \ref{ass:beta_lipschitz}. Let $x^*$ be some point with $\beta(x^*)=0$. Assume that for all $x$ such that $\dist\lrp{x,x^*} \geq \R$, there exists a minimizing geodesic $\gamma : [0,1]\to M$ with $\gamma(0) = x, \gamma(1) = x^*$, and
    \begin{alignat*}{1}
        \lin{\beta(x), \gamma'(0)} \leq - m \dist\lrp{x,x^*}^2.
    \end{alignat*}
    Let $x_k$ be the following stochastic process:
    \begin{alignat*}{1}
        x_{k+1} = \Exp_{x_k}\lrp{\delta_k \beta(x_k) + \sqrt{\delta_k} \xi_k(x_k)}
    \end{alignat*}
    where $\xi_k$ is a random vector field. Assume that for all $x$, $\E{\xi_k(x)} = 0$, and that for any $\rho \leq \frac{1}{8}$, $\E{\exp\lrp{\rho \lrn{\xi_k(x)}^2}} \leq \exp\lrp{\rho \sigma_\xi^2}$. For each $k$, $\delta_k$ is a positive stepsize that depends only on $x_k$ and satisfies
    $\delta_k \leq \min\lrbb{\frac{m}{16 {L_\beta'}^2 \lrp{1 + \sqrt{L_R} \dist\lrp{x_k,x^*}}}, \frac{\sigma_\xi^2}{m \lrp{1 + \sqrt{L_R} \dist\lrp{x_k,x^*}}}, \frac{32 \sigma_\xi^4}{m^2 \dist\lrp{x_k,x^*}^2}}$. Assume that $\dist\lrp{x_0,x^*} \leq 2\R$. Finally, assume that there exists $\delta \in \Re^+$ such that for all $k$, $\delta_k \leq \delta$. Then
    \begin{alignat*}{1}
        \Pr{\max_{i\leq K} \dist\lrp{x_k,x^*} \geq t} 
        \leq& 32K\delta \lambda \exp\lrp{\frac{2L_\beta' \R^2}{\sigma_\xi^2} + \frac{16 L_R \sigma_\xi^2}{m} - \frac{mt^2}{64 \sigma_\xi^2}}
    \end{alignat*}
\end{lemma}
\begin{proof}
    For each $k$, let $\gamma_k$ be a minimizing geodesic with $\gamma_k(0) = x_k$ and $\gamma_k(1) = x^*$ satisfying $\lin{\beta(x_k), \gamma_k'(0)} \leq - m \dist\lrp{x_k,x^*}^2$.

    Using the result from Corollary 8 of \cite{zhang2016first} (see Lemma \ref{l:zhang2016}),
    \begin{alignat*}{1}
        \dist\lrp{x_{k+1}, x^*}^2
        \leq& \dist^2\lrp{x_{k}, x^*} - 2\lin{\delta_k \beta(x_k) + \sqrt{\delta_k} \xi_k(x_k), \gamma_k'(0)}\\
        &\quad + \lrp{1+\sqrt{L_R} \dist(x_k,x^*)} \lrn{\delta_k \beta(x_k) + \sqrt{\delta_k} \xi_k(x_k)}^2
    \end{alignat*}
    By Assumption \ref{ass:distant-dissipativity} and Assumption \ref{ass:beta_lipschitz}, $\lin{\beta(x_k), \gamma_k'(0)} \leq - m \dist\lrp{x_k,x^*}^2 + 2L_\beta' \R^2$. Applying Cauchy Schwarz and simplifying,
    \begin{alignat*}{1}
        & \dist\lrp{x_{k+1}, x^*}^2\\
        \leq& \lrp{-2m \delta_k + 2\delta_k^2 {L_\beta'}^2 +  2 \delta_k^2 \sqrt{L_R} {L_\beta'}^2 \dist\lrp{x_k, x^*}}\dist\lrp{x_k, x^*}^2  + 4\delta_k L_\beta' \R^2 \\
        &\quad + 2\sqrt{\delta_k} \lin{\xi_k(x_k), \gamma_k'(0)} + 2 \delta_k \lrp{1+\sqrt{L_R} \dist(x_k,x^*)} \lrn{\xi_k(x_k)}^2\\
        \leq& -\delta_k m \dist\lrp{x_k, x^*}^2 + 4\delta_k L_\beta'\R^2 + 2\sqrt{\delta_k} \lin{\xi_k(x_k), \gamma_k'(0)} + 2 \delta_k \lrp{1+\sqrt{L_R} \dist(x_k,x^*)} \lrn{\xi_k(x_k)}^2
    \end{alignat*}
    where we used our assumption that $\delta_k \leq \frac{m}{16 {L_\beta'}^2 \lrp{1 + \sqrt{L_R} \dist\lrp{x_k,x^*}}}$.

    Let $s := \frac{m}{32 \sigma_\xi^2 }$ We will now apply Lemma \ref{l:doob_maximal_3} with
    \begin{alignat*}{1}
        & q_k = s \dist\lrp{x_k,x^*}^2 \qquad \nu_k = s\lin{\xi_k(x_k), \gamma_k'(0)} + 2 \sqrt{\delta_k} s \lrp{1+\sqrt{L_R} \dist(x_k,x^*)} \lrn{\xi_k(x_k)}^2 \\
        & \lambda = m \qquad \gamma  =  4sL_\beta' \R^2 \qquad \mu = 2 s \sigma_\xi^2 + \frac{64 s L_R \sigma_\xi^4}{m}
        \elb{e:t:qijdkfwnjf:0}
    \end{alignat*}
    We will verify Lemma \ref{l:doob_maximal_3}'s condition regarding $\E{\exp\lrp{\sqrt{\delta_k} \nu_k}}$. Taking expectation conditioned on $\xi_0(x_0)...\xi_{k-1}(x_{k-1})$,
    \begin{alignat*}{1}
        & \E{\exp\lrp{\sqrt{\delta_k} \nu_k}} \\
        =& \E{\exp\lrp{\sqrt{\delta_k} s\lin{\xi_k(x_k), \gamma_k'(0)} + 2 \delta_k s\lrp{1+\sqrt{L_R} \dist(x_k,x^*)} \lrn{\xi_k(x_k)}^2}} \\
        \leq& \E{\exp\lrp{\sqrt{\delta_k} 2s\lin{\xi_k(x_k), \gamma_k'(0)} }}^{1/2} \cdot \E{\exp\lrp{ 4 \delta_k s\lrp{1+\sqrt{L_R} \dist(x_k,x^*)} \lrn{\xi_k(x_k)}^2}}^{1/2}\\
        \elb{e:t:qijdkfwnjf:1}
    \end{alignat*} 
    By our assumption on $\xi_k$ and Lemma \ref{l:hoeffding}, and our assumption that {$\delta_k \leq \frac{32 \sigma_\xi^4}{m^2 \dist\lrp{x_k,x^*}^2}$}, 
    \begin{alignat*}{1}
        \E{\exp\lrp{\sqrt{\delta_k} 2s\lin{\xi_k(x_k), \gamma_k'(0)} }}
        \leq& \E{\exp\lrp{8\delta_k s^2 \dist\lrp{x_k,x^*}^2 \lrn{\xi_k(x_k)}^2}} \\
        \leq& \E{\exp\lrp{8\delta_k s^2 \dist\lrp{x_k,x^*}^2 \sigma_\xi^2}}\\
        \leq& \E{\exp\lrp{ \frac{s m \delta_k}{2} \dist\lrp{x_k,x^*}^2}}
    \end{alignat*}
    On the other hand, by our assumption on $\xi_k$ and {$\delta_k \leq  \frac{\sigma_\xi^2}{m \lrp{1+\sqrt{L_R} \dist(x_k,x^*)}}$},
    \begin{alignat*}{1}
        \E{\exp\lrp{ 4 \delta_k s\lrp{1+\sqrt{L_R} \dist(x_k,x^*)} \lrn{\xi_k(x_k)}^2}} 
        \leq& \exp\lrp{ 4 \delta_k s\lrp{1+\sqrt{L_R} \dist(x_k,x^*)} \sigma_\xi^2}\\
        \leq& \exp\lrp{\frac{s m \delta_k}{2} \dist\lrp{x_k,x^*}^2 + \frac{128 \delta_k s L_R \sigma_\xi^4}{m} + 4 \delta_k s \sigma_\xi^2}
    \end{alignat*}

    Plugging both of the above into \eqref{e:t:qijdkfwnjf:1},
    \begin{alignat*}{1}
        \E{\exp\lrp{\sqrt{\delta_k} \nu_k}}  \leq \exp\lrp{\frac{\delta_k m s}{4} \dist\lrp{x_k,x^*}^2 +  \frac{64 \delta_k s L_R \sigma_\xi^4}{m} + 2 \delta_k s \sigma_\xi^2}
    \end{alignat*}
    We thus verify that $(\nu_k,\lambda,\mu)$ satisfy the requirement for Lemma \ref{l:doob_maximal_3}, which bounds
    \begin{alignat*}{1}
        \Pr{\max_{i\leq K} q_{i} \geq t^2} 
        \leq& 8K\delta \lambda \exp\lrp{q_0 + \frac{8(\gamma + \mu)}{\lambda} - \frac{t^2}{2}}\\
        \leq& 8K\delta \lambda \exp\lrp{\frac{m\R^2}{16 \sigma_\xi^2} + \frac{L_\beta' \R^2}{\sigma_\xi^2} + \frac{1}{2} + \frac{16 L_R \sigma_\xi^2}{m} - \frac{t^2}{2}}\\
        \leq& 16K\delta \lambda \exp\lrp{\frac{2L_\beta' \R^2}{\sigma_\xi^2} + \frac{16 L_R \sigma_\xi^2}{m} - \frac{t^2}{2}}
    \end{alignat*}
    where the second inequality plugs in definitions from \eqref{e:t:qijdkfwnjf:0}, uses our assumption on $x_0$. Finally, using the fact that $q_k = s \dist\lrp{x_k,x^*}$,
    \begin{alignat*}{1}
        \Pr{\max_{i\leq K} \dist\lrp{x_i,x^*} \geq t} \leq& 32K\delta \lambda \exp\lrp{\frac{2L_\beta' \R^2}{\sigma_\xi^2} + \frac{16 L_R \sigma_\xi^2}{m} - \frac{mt^2}{64 \sigma_\xi^2}}
    \end{alignat*}
\end{proof}

\begin{lemma}[Subgaussian Bound under Dissipativity, SGLD, fixed stepsize]\label{l:far-tail-bound-truncated-sgld}
    Let $\beta$ be a vector field satisfying Assumption \ref{ass:beta_lipschitz}. Let $x^*$ be some point with $\beta(x^*)=0$. Assume that for all $x$ such that $\dist\lrp{x,x^*} \geq \R$, there exists a minimizing geodesic $\gamma : [0,1]\to M$ with $\gamma(0) = x, \gamma(1) = x^*$, and
    \begin{alignat*}{1}
        \lin{\beta(x), \gamma'(0)} \leq - m \dist\lrp{x,x^*}^2
    \end{alignat*}.
    Let $r \in \Re^+$ denote an arbitrary radius, and assume that $\delta$ is a stepsize satisfying
    \begin{alignat*}{1}
        \delta \leq \min\lrbb{\frac{m}{16 {L_\beta'}^2 \lrp{1 + \sqrt{L_R} r}}, \frac{d + \sigma^2}{m \lrp{1 + \sqrt{L_R} r}}, \frac{32 \lrp{d^2 + \sigma^4}}{m^2 r^2}}
    \end{alignat*}. 
    Let $x_k$ be the following stochastic process:
    \begin{alignat*}{1}
        x_{k+1} = \Exp_{x_k}\lrp{\delta \t{\beta}_k(x_k) + \sqrt{\delta} \zeta_k(x_k)}
    \end{alignat*}
    where $\zeta_k(x_k) \sim \N_{x_k}\lrp{0,I}$ and $\t{\beta}_k(x)$ satisfies, for all $x$, $\E{\t{\beta}_k(x)} = \beta(x)$ and $\lrn{\t{\beta}_k(x) - \beta(x)}\leq \sigma$. Assume that $\dist\lrp{x_0,x^*} \leq 2\R$. Then
    \begin{alignat*}{1}
        \Pr{\max_{k\leq K} \dist\lrp{x_k,x^*} \geq r} \leq 32K\delta m \exp\lrp{\frac{2{L_\beta'}^2\R^2}{d+\sigma^2} + \frac{64 L_R \lrp{d + \sigma^2} }{m} - \frac{m r^2}{256 \lrp{d + \sigma^2} }}
    \end{alignat*}
\end{lemma}
\begin{proof}
    We begin by defining $\xi_k(x_k) := \zeta_k(x_k) + \sqrt{\delta} \lrp{\t{\beta}(x_k) - \beta(x_k)}$. We verify that
    \begin{alignat*}{1}
        x_{k+1} = \Exp_{x_k}\lrp{\delta \t{\beta}_k(x_k) + \sqrt{\delta} \xi_k(x_k)}.
    \end{alignat*}
    We verify that $\E{\xi_k} = 0$ and that for any $\lambda \leq \frac{1}{8}$,
    \begin{alignat*}{1}
        \E{\exp\lrp{\lambda\lrn{\xi_k(x_k)}^2}} \leq \E{\exp\lrp{2\lambda\lrn{\zeta_k(x_k)}^2 + 2\lambda \sigma^2}} \leq \exp\lrp{4 \lambda d + 2\lambda \sigma^2}
    \end{alignat*}
    where we use Lemma \ref{l:subexponential-chi-square}. Let us define $\sigma_\xi := \sqrt{4d + \sigma^2}$.

    Next, let us define, for analysis purposes, the following process:
    \begin{alignat*}{1}
        \t{x}_{k+1} = \Exp_{\t{x}_k}\lrp{\delta_k \beta(\t{x}_k) + \sqrt{\delta_k} \xi_k(\t{x}_k)}
    \end{alignat*}
    initialized at $\t{x}_0 = x_0$ and where 
    \begin{alignat*}{1}
        \delta_k := \min\lrbb{\delta, \frac{m}{16 {L_\beta'}^2 \lrp{1 + \sqrt{L_R} \dist\lrp{\t{x}_k,x^*}}}, \frac{\sigma_\xi^2}{m \lrp{1 + \sqrt{L_R} \dist\lrp{\t{x}_k,x^*}}},\frac{32 \sigma_\xi^4}{m^2 \dist\lrp{\t{x}_k,x^*}^2}}
    \end{alignat*}
    Define the event $A_k := \lrbb{\max_{i\leq k} \dist\lrp{\t{x}_i,x^*} \leq r}$. Under the event $A_k$, $\delta_i = \delta$ for all $i\leq k$, and consequently, $\t{x}_i = x_i$ for all $i \leq k$. Therefore, $A_k = \lrbb{\max_{i\leq k} \dist\lrp{{x}_i,x^*} \leq r}$, and thus
    \begin{alignat*}{1}
        & \Pr{\max_{k\leq K} \dist\lrp{x_k,x^*} \geq r}\\
        =& \Pr{A_k^c}\\
        =& \Pr{\max_{k\leq K} \dist\lrp{\t{x}_k,x^*} \geq r}\\
        \leq& 32K\delta m \exp\lrp{\frac{2{L_\beta'}^2\R^2}{\sigma_\xi^2} + \frac{16 L_R \sigma_\xi^2 }{m} - \frac{m r^2}{64 \sigma_\xi^2 }}
    \end{alignat*}
    where the last inequality follows from Lemma \ref{l:far-tail-bound-gaussian-adaptive} with $\xi_k$ and $\sigma_\xi$ as defined above.
\end{proof}

\subsection{Misc}
\begin{lemma}\label{l:subexponential-chi-square}
    For $\lambda \leq \frac{1}{4}$ and $\xi \sim \N(0,I_{d\times d})$,
    \begin{alignat*}{1}
        \E{\exp\lrp{\lambda \lrn{\xi}_2^2}} \leq \exp\lrp{\lambda d + 2 \lambda^2 d} \leq \exp\lrp{2\lambda d}
    \end{alignat*}
\end{lemma}
\begin{proof}
    Consequence of $\chi^2$ distribution being subexponential.
\end{proof}
\begin{lemma}[Hoeffding's Lemma] \label{l:hoeffding}
    Let $\eta_k$ be a $0$-mean random variable. Then for all $\lambda$, 
    \begin{alignat*}{1}
        \Ep{\eta}{\exp\lrp{\lambda \eta}} \leq \Ep{\eta}{\exp\lrp{2\lambda^2 \eta^2}}
    \end{alignat*}
\end{lemma}
\begin{proof}
    \begin{alignat*}{1}
        \Ep{\eta}{\exp\lrp{\lambda \eta}}
        =& \Ep{\eta}{\exp\lrp{\lambda \eta - \Ep{\eta'}{\lambda \eta'}}}\\
        \leq& \Ep{\eta,\eta'}{\exp\lrp{\lambda \lrp{\eta - \eta'}}}\\
        =& \Ep{\eta,\eta',\epsilon}{\exp\lrp{\lambda \epsilon \lrp{\eta - \eta'}}}\\
        \leq& \Ep{\eta,\eta'}{\exp\lrp{\lambda^2 \lrp{\eta - \eta'}^2/2}}\\
        \leq& \Ep{\eta,\eta'}{\exp\lrp{\lambda^2 \lrp{\eta^2 + {\eta'}^2}}}\\
        =& \Ep{\eta}{\exp\lrp{2\lambda^2 \eta^2}}
    \end{alignat*}
    where $\epsilon$ is a Rademacher random variable.
\end{proof}

\begin{lemma}[Corollary of Doob's maximal inequality]
    \label{l:doob_maximal}
    Let $K$ be any positive integer. For any $k\leq K$, let $a_k, b_k, c_k, d_k \in \Re^+$ be arbitrary positive constants. Assume that for all $k$, $a_k \leq \frac{1}{4}$ and ${a_k+c_k} \leq \frac{1}{4}$. Let $q_k$ be a semi-martingale of the form
    \begin{alignat*}{1}
        q_{k+1} \leq \lrp{1 + a_k}q_k + b_k + \eta_k
    \end{alignat*}
    where $\eta_k$ are random variables. Assume that for all $k$, $\eta_k$ satisfy
    \begin{alignat*}{1}
        \E{\exp\lrp{\eta_k} | \eta_0...\eta_{k-1}} \leq \exp\lrp{c_k q_{k} + d_k}
    \end{alignat*}

    Assume in addition that $q_k \geq 0$ almost surely, for all $k$. Finally, assume that $\sum_{k=0}^K a_k + c_k \leq\frac{1}{8}$. Then
    \begin{alignat*}{1}
        \Pr{\max_{k\leq K} q_{k} \geq t^2} \leq \exp\lrp{q_0 + \sum_{k=0}^K (b_k+d_k) - \frac{t^2}{2}}
    \end{alignat*}
\end{lemma}
\begin{proof}
    Let us first define
    \begin{alignat*}{1}
        r_0 :=& q_0\\
        r_{k+1} 
        :=& \lrp{1 + a_k} r_k + b_k  + \eta_k
    \end{alignat*}
    i.e. $r_k$ is very similar to $q_k$, only difference being that we replaced $\leq$ by $=$.
    
    We first verify that for all $k\leq K$, $r_k \geq q_k$. For $k=0$, by definition, $r_0 = q_0$. Now assume that $r_k \geq q_k$ for some $k$. Then for $k+1$, 
    \begin{alignat*}{1}
        r_{k+1} 
        :=& \lrp{1 + a_k} r_k + b_k  + \eta_k\\
        \geq& \lrp{1 + a_k} q_k + b_k  + \eta_k\\
        \geq& q_{k+1}
    \end{alignat*}

    We verify below that $\exp\lrp{r_k}$ is a sub-martingale: conditioning on $\eta_0...\eta_{k-1}$, and taking expectation wrt $\eta_k$, 
    \begin{alignat*}{1}
        & \E{\exp\lrp{r_{k+1}} | \eta_0...\eta_k}\\
        =& \E{\exp\lrp{\lrp{1 + a_k} r_k + b_k + \eta_k} | \eta_0...\eta_k}\\
        =& \exp\lrp{(1+a_k)r_k + b_k} \cdot \E{\exp\lrp{\eta_k} | \eta_0...\eta_k}\\
        \geq& \exp\lrp{(1+a_k)r_k + b_k}\\
        \geq& \exp\lrp{ r_k}
    \end{alignat*}
    where the first inequality is by convexity of $\exp$, and $\E{\eta_k}= 0$, and Jensen's inequality. 
    
    Let us now define $s_k := \prod_{i=0}^{k-1} \lrp{1+a_i+c_i}^{-1}$. We can upper bound
    \begin{alignat*}{1}
        & \E{\exp\lrp{s_{k+1} r_{k+1}}}\\
        =& \E{\exp\lrp{s_{k+1} \lrp{(1+a_k) r_{k} + b_k + \eta_k}}}\\
        =& \E{\exp\lrp{s_{k+1}(1+a_k) r_k + s_{k+1} b_k} \cdot \E{\exp\lrp{s_{k+1} \eta_k}| \eta_0...\eta_k}}\\
        \leq& \E{\exp\lrp{s_{k+1}(1+a_k) r_k + s_{k+1} b_k} \cdot \E{\exp\lrp{s_{k+1} \eta_k}| \eta_0...\eta_k}}\\
        \leq& \E{\exp\lrp{s_{k+1}(1+a_k) r_k + s_{k+1} b_k} \cdot \lrp{\E{\exp\lrp{\eta_k}| \eta_0...\eta_k}}^{s_{k+1}}}\\
        \leq& \E{\exp\lrp{s_{k+1}(1+a_k) r_k + s_{k+1} \lrp{b_k + c_kq_k + d_k}}}\\
        =& \E{\exp\lrp{s_{k}r_k}} \cdot \exp\lrp{s_{k+1} \lrp{b_k + d_k}}
    \end{alignat*}
    where the second inequality is by Lemma \ref{l:hoeffding}, the third inequality is by the fact that $s_k \leq 1$ for all $k$ and by Jensen's inequality, the fourth inequality uses our assumption on $\eta_k$ in the Lemma statement, as well as the fact that $s_k \leq 1$ for all $k$. The last equality is by definition of $s_k$ and because $q_k \leq r_k$. Applying this recursively gives
    \begin{alignat*}{1}
        \E{\exp\lrp{s_K r_{K}}} \leq \exp\lrp{r_0} \cdot \exp\lrp{\sum_{k=0}^K s_{k+1} (b_k+d_k)} \leq \exp\lrp{r_0 + \sum_{k=0}^K (b_k+d_k)}
    \end{alignat*}
    
    By Doob's maximal inequality (recall that we $e^{r_k}$ is a sub-martingale),
    \begin{alignat*}{1}
        \Pr{\max_{k\leq K} q_{k} \geq t^2}
        \leq& \Pr{\max_{k\leq K} r_{k} \geq t^2}\\
        =& \Pr{\max_{k\leq K} \exp\lrp{s_K r_{k}} \geq \exp\lrp{s_K t^2}}\\
        \leq& \E{\exp\lrp{s_K r_K } } \cdot \exp\lrp{- \frac{t^2}{2}}\\
        \leq& \exp\lrp{r_0 + \sum_{k=0}^K s_{k+1} (b_k+d_k) - \frac{t^2}{2}}
        \elb{e:t:wkqndasq:4}
    \end{alignat*}
    The first equality uses our assumption that $s_K = \prod_{k=0}^{K-1} \lrp{1+ a_k + c_k}^{-1} \geq e^{- \sum_{k=0}^{K-1} {4(a_k+c_k)}} \geq e^{- \frac{1}{4}} \geq \frac{1}{2}$ and the fact that $r_K \geq q_k \geq 0$.
    \end{proof}

    \begin{lemma}[Uniform Bound]
        \label{l:doob_maximal_3}
        Let $K$ be any positive integer. For $k\leq K$, let $\delta_k \in \Re^+$, let $\lambda, \gamma, \mu \in \Re^+$. Let $q_k$ be a sequence of random numbers of the form
        \begin{alignat*}{1}
            q_{k+1} \leq \lrp{1 - \delta_k \lambda}q_k + \delta_k \gamma + \sqrt{\delta_k} \nu_k
        \end{alignat*}
        where $\nu_k$ are random variables. Assume that $q_k$ and $\nu_k$ are measurable wrt some filtration $\F_k$. Assume that for all $k$, $\nu_k$ satisfy
        \begin{alignat*}{1}
            & \E{\exp\lrp{\sqrt{\delta_k} \nu_k}| \F_k} \leq \exp\lrp{\delta_k \lambda q_k/2 + \delta_k \mu}
        \end{alignat*}
        Assume that there is a constant $\delta$ such that for all $k$, $\delta_k \leq \delta \leq \frac{1}{8\lambda}$. Assume in addition that $q_k \geq 0$ almost surely, for all $k$. Then
        \begin{alignat*}{1}
            \Pr{\max_{i\leq K} q_{i} \geq t^2} 
            \leq& 8K\delta \lambda \exp\lrp{q_0 + \frac{8(\gamma + \mu)}{\lambda} - \frac{t^2}{2}}
        \end{alignat*}
    \end{lemma}
    \begin{proof}
        For any $s \leq 1$, 
        \begin{alignat*}{1}
            \E{\exp\lrp{s q_{k+1}}}
            \leq& \E{\exp\lrp{s \lrp{(1-\delta_k \lambda) q_k + \delta_k\gamma + \sqrt{\delta_k} \nu_k}}}\\
            =& \E{\exp\lrp{s \lrp{(1-\delta_k \lambda) q_k + \delta_k\gamma}} \cdot \E{\exp\lrp{s \sqrt{\delta_k} \nu_k}}}\\
            \leq& \E{\exp\lrp{s \lrp{(1-\delta_k \lambda) q_k + \delta_k\gamma}} \cdot \lrp{\E{\exp\lrp{\sqrt{\delta_k} \nu_k}}}^{s}}\\
            \leq& \E{\exp\lrp{s \lrp{(1-\delta_k \lambda/2) q_k + \delta_k(\gamma+\mu)}}}
        \end{alignat*}

        Applying the above recursively, for any $k$, we can bound
        \begin{alignat*}{1}
            \E{\exp\lrp{q_k}}
            \leq& \E{\exp\lrp{(1 - \delta_k \lambda/2) q_{k-1} + \delta_k (\gamma+\mu) }}\\
            \leq& ...\\
            \leq& \E{\exp\lrp{\prod_{i=0}^{k-1}\lrp{1-\delta_i \lambda/2} q_0 + \sum_{i=0}^{k-1} \prod_{j=i}^{k-1} \lrp{1-\delta_j \lambda/2 }\lrp{\delta_i (\gamma+\mu)}}}\\
            \leq& \E{\exp\lrp{e^{-\frac{\lambda}{2} \sum_{i=0}^{k-1}\delta_i } q_0 + (\gamma+\mu) \sum_{i=0}^{k-1} e^{-\frac{\lambda}{2} \sum_{j=i}^{k-1} \delta_j}\delta_i}}
            \elb{e:t:qokdsda:1}
        \end{alignat*}
        Let us define $t_k := \sum_{i=0}^{k} \delta_i$. By our assumption that $\delta_i \leq \frac{1}{4\lambda}$, we can verify that $\sum_{i=0}^{k-1} e^{\frac{\lambda}{2} \sum_{j=i}^{k-1} \delta_j}\delta_i \leq 2 \int_0^{t_k} e^{-\frac{\lambda(t_k - t)}{2}} dt \leq \frac{4}{\lambda}$. Therefore, for all $k$,
        \begin{alignat*}{1}
            \E{\exp\lrp{q_{k}}} \leq \exp\lrp{e^{-\frac{\lambda}{2} \sum_{i=0}^{k-1}\delta_i }q_0 + \frac{4(\gamma+\mu)}{\lambda}}
        \end{alignat*}
        Let us now define $N := \left\lceil \frac{1}{4\delta \lambda}\right\rceil \geq 1$ (inequality is because $\delta \leq \delta_k \leq \frac{1}{8\lambda}$). We verify that $q_{k+1} \leq q_{k-1} + \delta_k (\gamma+\mu) + \eta_k$. Let us now apply Lemma \ref{l:doob_maximal} with $\eta_k = \sqrt{\delta_k} \nu_k$, $a_k=0, b_k = \delta_k \gamma, c_k = \lambda/2, d_k=\mu$ and the fact that $\delta_k \lambda \leq 1/4$ to bound, for any $k$,
        \begin{alignat*}{1}
            \Pr{\max_{i\leq N} q_{k+i} \geq t^2} 
            \leq& \E{\exp\lrp{q_k + (\gamma+\mu) \sum_{i=0}^N \delta_{i+N} - \frac{t^2}{2}}} 
            \leq \exp\lrp{q_0 + \frac{8(\gamma+\mu)}{\lambda} - \frac{t^2}{2}}
        \end{alignat*}
        where we use the fact that $N \leq \frac{1}{2\delta \lambda}$

        Applying union bound over the events $\lrbb{\max_{i\leq N} q_{k+i} \geq t^2}$ for $k = 0, N, 2N...$, we can bound, for any positive integer $M$,
        \begin{alignat*}{1}
            \Pr{\max_{i\leq MN} q_{i} \geq t^2} 
            \leq& \sum_{j=0}^{M-1} \Pr{\max_{i\leq N} q_{jN + i} \geq t^2}\\
            \leq& \sum_{j=0}^{M-1} \exp\lrp{e^{-\frac{\lambda}{2} \sum_{i=0}^{jN}\delta_i } q_0 + \frac{8(\gamma+\mu)}{\lambda} - \frac{t^2}{2}}\\
            \leq& M \exp\lrp{q_0 + \frac{8(\gamma+\mu)}{\lambda} - \frac{t^2}{2}}
        \end{alignat*}
        Plugging in $M = \frac{K}{N} \leq 8K\delta \lambda$, it follows that for any $K$, 
        \begin{alignat*}{1}
            \Pr{\max_{i\leq K} q_{i} \geq t^2} 
            \leq& 8K\delta \lambda \exp\lrp{q_0 + \frac{8(\gamma+\mu)}{\lambda} - \frac{t^2}{2}}
        \end{alignat*}

    \end{proof}

    \begin{lemma}\label{l:zhang2016}
        Let $M$ satisfy Assumption \ref{ass:sectional_curvature_regularity}. For any 3 points $x,y,z \in M$, let $u, v \in T_y M$ be such that $z = \Exp_y(v)$ and $x = \Exp_y(u)$. Assume in addition that $\lrn{u} = \dist\lrp{x,y}$ (i.e. $t\to \Exp_x(tu)$ is a minimizing geodesic). Then
        \begin{alignat*}{1}
            \dist\lrp{z, x}^2
            \leq& \dist\lrp{y, x}^2 - 2\lin{v, u} + {\tc\lrp{\sqrt{L_R}\dist\lrp{y,x}}} \lrn{v}^2
        \end{alignat*}    
        where $\zeta(r) := \frac{r}{\tanh(r)}$.
    \end{lemma}
    The above lemma is a restatement of Corollary 8 from \cite{zhang2016first}. Although \cite{zhang2016first} required minimizizng geodesics to be unique, their proof, based on Lemma 6 of the same paper, works even if minimizing geodesics are not unique.

    \section{Fundamental Manifold Results}
    
    In this section, we provide Taylor expansion style inequalities for the evolution of geodesics on manifold. By making use of our bounds for matrix ODEs in Section \ref{s:matrix_exponent}, we can bound the distance between two points along geodesics under various conditions. Most of our analysis is based on some variant of the Jacobi equation $D_t^2 J + R(J,\gamma')\gamma' = 0$.

    Notably,
    \begin{enumerate}
        \item Lemma \ref{l:discrete-approximate-synchronous-coupling} quantifies the distance evolution between $x(t) = \Exp_x(tu)$ and $y(t)=\Exp_y(tv)$. This is the key to proving Lemma \ref{l:discretization-approximation-lipschitz-derivative}, which bounds the distance between Euler Murayama discretization \eqref{e:intro_euler_murayama} and \eqref{e:intro_sde}.
        \item Lemma \ref{l:discrete-approximate-synchronous-coupling-ricci} is a more refined version of Lemma \ref{l:discrete-approximate-synchronous-coupling}. Lemma \ref{l:discrete-approximate-synchronous-coupling-ricci} is key to proving Lemma \ref{l:sgld-lemma}, which is in turn key to proving Theorem \ref{t:SGLD}. Lemma \ref{l:discrete-approximate-synchronous-coupling-ricci} is also used to analyze the distance evolution of two processes under the Kendall-Cranston coupling in Lemma \ref{l:g_contraction_without_gradient_lipschitz}.
    \end{enumerate}

    \subsection{Jacobi Field Approximations}
    In the following lemma, we consider a variation field $\Lambda(s,t)$, where for each $s$, $\Lambda(s,t)$ is a geodesic. We bound the error between $\Lambda(s,t)$, and its Taylor approximation of various orders. This lemma is key to proving Lemma \ref{l:discrete-approximate-synchronous-coupling} and Lemma \ref{l:discrete-approximate-synchronous-coupling}.

    \begin{lemma}\label{l:jacobi_field_norm_bound}
        Let $\Lambda(s,t) : [0,1]\times[0,1]\to M$ be a variation field, where for each fixed $s$, $t \to \Lambda(s,t)$ is a geodesic. Let us define $\C := \sqrt{L_R \lrn{\del_t \Lambda(s,0)}^2}$. Then for all $s,t\in[0,1]$, 
        \begin{alignat*}{1}
            &\lrn{\del_s \Lambda(s,t)}
            \leq \cosh\lrp{{\C}} \lrn{\del_s \Lambda(s,0)} + \frac{\sinh\lrp{{\C}}}{{\C}} \lrn{D_t \del_s \Lambda(s,0)}\\
            &\lrn{\del_s \Lambda(s,t) - \party{\Lambda(s,0)}{\Lambda(s,t)}\lrp{\del_s \Lambda(s,0) + t D_t \del_s \Lambda(s,0)}}
            \leq \lrp{\cosh\lrp{{\C}} - 1} \lrn{\del_s \Lambda(s,0)} + \lrp{\frac{\sinh\lrp{{\C}}}{{\C}}  - 1} \lrn{D_t \del_s \Lambda(s,0)}\\
            & \lrn{\del_s \Lambda(s,t) - \party{\Lambda(s,0)}{\Lambda(s,t)}\lrp{\del_s \Lambda(s,0)}}
            \leq \lrp{\cosh\lrp{{\C}} - 1} \lrn{\del_s \Lambda(s,0)} + \frac{\sinh\lrp{{\C} }}{{\C}}\lrn{D_t \del_s \Lambda(s,0)}\\
            & \lrn{D_t \del_s \Lambda(s,t)}
            \leq {\C}  \sinh\lrp{{\C}} \lrn{\del_s \Lambda(s,0)} + \cosh\lrp{{\C}} \lrn{D_t \del_s \Lambda(s,0)}\\
            &\lrn{D_t \del_s \Lambda(s,t) - \party{\Lambda(s,0)}{\Lambda(s,t)}\lrp{D_t \del_s \Lambda(s,0)}} 
            \leq {\C}  \sinh\lrp{{\C}}\lrn{\del_s \Lambda(s,0)}  + \lrp{\cosh\lrp{{\C}} - 1} \lrn{D_t \del_s \Lambda(s,0)}
        \end{alignat*}

        If in addition, the derivative of the Riemannian curvature tensor is globally bounded by $L_R'$, then
        \begin{alignat*}{1}
            & \lrn{D_t \del_s \Lambda(s,t) - \party{\Lambda(s,0)}{\Lambda(s,t)}\lrp{D_t \del_s \Lambda(s,0)} - t\party{\Lambda(s,0)}{\Lambda(s,t)}\lrp{D_t D_t \del_s \Lambda(s,0)}}\\
            \leq& \lrp{L_R'\lrn{\del_t\Lambda(s,0)}^3 + \C^4} e^{\C} \lrn{\del_s \Lambda(s,0)} + \lrp{L_R'\lrn{\del_t\Lambda(s,0)}^3 + \C^2}e^{\C} \lrn{D_t \del_s \Lambda(s,0)}
        \end{alignat*}
    \end{lemma}
    \begin{proof}
        For any fixed $s$, let $E_i(s,0)$ be a basis of $T_{\Lambda(s,0)} M$.
        
        Let $E_i(s,t)$ denote an orthonormal frame along $\gamma_s(t) := \Lambda(s,t)$, by parallel transporting $E_i(s,0)$.
        
        Let $\JJ(s,t) \in \Re^d$ denote the coordinates of $\del_s \Lambda(s,t)$ wrt $E_i(s,t)$. Let $\KK(s,t) \in \Re^d$ denote the coordinates of $D_t \del_s \Lambda(s,t)$ wrt $E_i(s,t)$. Let $\aa(s,t) \in \Re^d$ denote the coordinates of $\del_t \Lambda(s,t)$ wrt $E_i(s,t)$ (this is constant for fixed $s$, for all $t$). Let $\RR(s,t) \in \Re^{4d}$ be such that $\RR^i_{jkl} = \lin{R(E_j(s,t), E_k(s,t)) E_l(s,t), E_i(s,t)}$. Let $\MM(s,t)$ denote the matrix with \\
        $\MM_{i,j}(s,t) := -\sum_{k,l} \RR^i_{jkl}(s,t) \aa_k(s,t) \aa_l(s,t)$. 
        
        Notice that $\MM_{i,j}$ is symmetric, since $\MM_{i,j} = \lin{R(E_j, \aa)\aa, E_i} = \lin{R(\aa, E_i)E_j, \aa} = \lin{R(E_i,\aa)\aa,E_j}$, where the first equality is by interchange symmetry and second equality is by skew symmetry of the Riemannian curvature tensor. Therefore, by definition of $L_R$ in Assumption \ref{ass:sectional_curvature_regularity}, it follows that $\lrn{\MM(s,t)}_2 \leq L_R \lrn{\del_t \Lambda(s,t)}^2 = L_R \lrn{\del_t \Lambda(s,0)}^2$.
        
        The Jacobi Equation states that $D_t D_t \del_s \Lambda(s,t) = - R(\del_s \Lambda(s,t), \del_t \Lambda(s,t)) \del_t \Lambda(s,t)$. We verify that $-\lin{R(\del_s , \del_t) \del_t, E_i} = -\sum_{j,k,l} \RR^i_{jkl} \JJ_j \aa_k \aa_l = \lrb{\MM(s,t)\JJ(s,t)}_i$, thus $D_t D_t \del_s \Lambda(s,t) = \sum_{i=1}^d \lrb{\MM(s,t) \JJ(s,t)}_{i} \cdot E_i(s,t)$.

        We verify that $\frac{d}{dt} \JJ_i(s,t) = D_t \lin{\del_s \Lambda(s,t), E_i(s,t)} = \lin{D_t \del_s \Lambda(s,t), E_i(s,t)} = \KK_i(s,t)$. We also verify that $\frac{d}{dt} \KK_i(s,t) = D_t \lin{D_t \del_s \Lambda(s,t), E_i(s,t)} = \lin{D_t D_t \del_s \Lambda(s,t), E_i(s,t)} = \lrb{\MM(s,t)\JJ(s,t)}_i$.

        Let us now consider a fixed $s$. To simplify notation, we drop the $s$ dependence. The Jacobi Equation, in coordinate form, corresponds to the following second-order ODE:
        \begin{alignat*}{1}
            & \frac{d}{dt} \JJ(t) = \KK(t)\\
            & \frac{d}{dt} \KK(t) = \MM(t) \JJ(t) dt
        \end{alignat*}

        Define $L_{\MM} := L_R \lrn{\del_t \Lambda(s,0)}^2 = \C^2$. We verify that $L_{\MM} \geq \max_{t\in[0,1]} \lrn{\MM(t)}_2$. Then from \ref{l:formal-matrix-exponent}, we see that
        \begin{alignat*}{1}
            \cvec{\JJ(t)}{\KK(t)} = \emat\lrp{t; \MM} \cvec{\JJ(0)}{\KK(0)}
        \end{alignat*}
        From Lemma \ref{l:matrix-exponent-block-bounds},
        \begin{alignat*}{1}
            \emat\lrp{t; \bmat{0 & I \\ \MM(t) & 0}} = \bmat{\AA(t) & \BB(t) \\ \CC(t) & \DD(t)} 
        \end{alignat*}
        where each block is $\Re^{2d}$, and can be bounded as
        \begin{alignat*}{1}
            & \lrn{\AA(t)}_2 \leq \cosh\lrp{{\C} t} \leq \cosh\lrp{{\C}}\\
            & \lrn{\BB(t)}_2 \leq \frac{\sinh\lrp{{\C} t}}{{\C}} \leq \frac{\sinh\lrp{{\C} }}{{\C}} \\
            & \lrn{\CC(t)}_2 \leq {\C}  \sinh\lrp{{\C} t} \leq {\C}  \sinh\lrp{{\C}}\\
            & \lrn{\DD(t)}_2 \leq \cosh\lrp{{\C}t} \leq \cosh\lrp{{\C}}\\
            & \lrn{\AA(t) - I}_2 \leq \cosh\lrp{{\C} t} - 1 \leq \cosh\lrp{{\C}} - 1\\
            & \lrn{\BB(t) - tI}_2 \leq \frac{\sinh\lrp{{\C} t}}{{\C}}  - t \leq \frac{\sinh\lrp{{\C}}}{{\C}}  - 1\\
            & \lrn{\DD(t) - I}_2 \leq \cosh\lrp{{\C} t} - 1 \leq \cosh\lrp{{\C}} - 1
        \end{alignat*}
        where we use the fact that $\cosh(r)$, $\sinh(r)$ and $\frac{\sinh(r)}{r}$ are monotonically increasing and that $\frac{\sinh(r)}{r} - 1 \geq 0$ for positive $r$.

        It follows that
        \begin{alignat*}{1}
            \JJ(t) =& \AA(t) \JJ(0) + \BB(t) \KK(0)\\
            \KK(t) =& \CC(t) \JJ(0) + \DD(t) \KK(0)
            \elb{e:t:kjqnmd:1}
        \end{alignat*}

        Thus
        \begin{alignat*}{1}
            & \lrn{\del_s \Lambda(s,t)}\\
            =& \lrn{\JJ(t)}\\
            =& \lrn{\AA(t) \JJ(0) + \BB(t) \KK(0)}\\
            \leq& \cosh\lrp{{\C}} \lrn{\del_s \Lambda(s,0)} + \frac{\sinh\lrp{{\C}}}{{\C}} \lrn{D_t \del_s \Lambda(s,0)}
        \end{alignat*}
        and
        \begin{alignat*}{1}
            & \lrn{\del_s \Lambda(s,t) - \lrbb{\del_s \Lambda(s,0) + t D_t \del_s \Lambda(s,0)}^{\to\Lambda(s,t)}}\\
            =& \lrn{\JJ(t) - \JJ(0) - t \KK(0)}_2\\
            =& \lrn{\lrp{\AA(t) - I} \JJ(0) + \lrp{\BB(t) - tI} \KK(0)}_2\\
            \leq& \lrp{\cosh\lrp{{\C}} - 1} \lrn{\del_s \Lambda(s,0)} + \lrp{\frac{\sinh\lrp{{\C}}}{{\C}}  - 1} \lrn{D_t \del_s \Lambda(s,0)}
        \end{alignat*}
        and
        \begin{alignat*}{1}
            & \lrn{\del_s \Lambda(s,t) - \lrbb{\del_s \Lambda(s,0)}^{\to\Lambda(s,t)}}\\
            \leq& \lrn{\lrp{\AA(t) - I} \JJ(0) + \BB(t) \KK(0)}_2\\
            \leq& \lrp{\cosh\lrp{{\C}} - 1} \lrn{\del_s \Lambda(s,0)} + \frac{\sinh\lrp{{\C} }}{{\C}}\lrn{D_t \del_s \Lambda(s,0)}
        \end{alignat*}

        Similarly,
        \begin{alignat*}{1}
            & \lrn{D_t \del_s \Lambda(s,t)}\\
            =& \lrn{\KK(t)}_2\\
            \leq& \lrn{\CC(t) \JJ(0)}_2 + \lrn{\DD(t)\KK(0)}_2\\
            \leq& {\C}  \sinh\lrp{{\C}} \lrn{\del_s \Lambda(s,0)} + \cosh\lrp{{\C}} \lrn{D_t \del_s \Lambda(s,0)}
        \end{alignat*}
        and
        \begin{alignat*}{1}
            & \lrn{D_t \del_s \Lambda(s,t) - \lrbb{D_t \del_s \Lambda(s,0)}^{\to \Lambda(s,t)}}\\
            =& \lrn{\KK(t) - \KK(0)}_2\\
            =& \lrn{\CC(t) \JJ(0) + \lrp{\DD(t) - I} \KK(0)}_2\\
            \leq& {\C}  \sinh\lrp{{\C}}\lrn{\del_s \Lambda(s,0)}  + \lrp{\cosh\lrp{{\C}} - 1} \lrn{D_t \del_s \Lambda(s,0)}
        \end{alignat*}

        To prove the last bound, let us define $L_{\MM}' := L_R' \lrn{\del_t \Lambda(s,0)}^3$. We verify that \\
        $L_{\MM}' \geq \max_{t\in[0,1]} \lrn{\MM(t) - \MM(0)}_2$.
        
        we know that
        \begin{alignat*}{1}
            & \lrn{D_t \del_s \Lambda(s,t) - \lrbb{D_t \del_s \Lambda(s,0)}^{\to \Lambda(s,t)} - t\lrbb{D_t D_t \del_s \Lambda(s,0)}^{\to \Lambda(s,t)}}\\
            =& \lrn{\KK(t) - \KK(0) - t \MM(0) \JJ(0)}_2\\
            =& \lrn{\int_0^t \MM(r) \JJ(r) - \MM(0) \JJ(0) dr}_2\\
            \leq& \int_0^t \lrn{\MM(r) - \MM(0)}_2 \lrn{\JJ(0)}_2 dr + \int_0^t \lrn{\MM(r)}_2 \lrn{\JJ(r) - \JJ(0)}_2 dr\\
            \leq& \int_0^t L_\MM' \lrn{\del_s \Lambda(s,r)} + L_\MM \lrn{\del_s \Lambda(s,r) - \party{\Lambda(s,0)}{\Lambda(s,r)}\lrp{\del_s \Lambda(s,0)}} dr
        \end{alignat*}
        where the last line follows from \eqref{e:t:kjqnmd:1}. From our earlier results in this lemma,
        \begin{alignat*}{2}
            & \lrn{\del_s \Lambda(s,r)} 
            &&\leq \cosh\lrp{{\C}} \lrn{\del_s \Lambda(s,0)} + \frac{\sinh\lrp{{\C}}}{{\C}} \lrn{D_t \del_s \Lambda(s,0)}\\
            & &&\leq e^{\C} \lrn{\del_s \Lambda(s,0)} + e^{\C} \lrn{D_t \del_s \Lambda(s,0)}\\
            & \lrn{\del_s \Lambda(s,t) - \party{\Lambda(s,0)}{\Lambda(s,t)}\lrp{\del_s \Lambda(s,0)}}
            &&\leq \lrp{\cosh\lrp{{\C}} - 1} \lrn{\del_s \Lambda(s,0)} + \frac{\sinh\lrp{{\C} }}{{\C}}\lrn{D_t \del_s \Lambda(s,0)}\\
            & &&\leq {\C}^2 e^{\C} \lrn{\del_s \Lambda(s,0)} + e^{\C} \lrn{D_t \del_s \Lambda(s,0)}
        \end{alignat*}
        where the simplifications are from Lemma \ref{l:sinh_bounds}. Put together,
        \begin{alignat*}{1}
            & \lrn{D_t \del_s \Lambda(s,t) - \lrbb{D_t \del_s \Lambda(s,0)}^{\to \Lambda(s,t)} - t\lrbb{D_t D_t \del_s \Lambda(s,0)}^{\to \Lambda(s,t)}}\\
            \leq& \lrp{L_R'\lrn{\del_t\Lambda(s,0)}^3 + \C^4} e^{\C} \lrn{\del_s \Lambda(s,0)} + \lrp{L_R'\lrn{\del_t\Lambda(s,0)}^3 + \C^2}e^{\C} \lrn{D_t \del_s \Lambda(s,0)}
        \end{alignat*}
    \end{proof}
    
    \begin{lemma}\label{l:jacobi_field_divergence}
        Let $x,y,z \in M$, with $x = \Exp_z(u)$, $y = \Exp_z(v)$. 
        \begin{alignat*}{1}
            \dist(x,y) \leq  \frac{\sinh\lrp{\sqrt{L_R}\lrp{\lrn{u}+\lrn{v}} t}}{\sqrt{L_R}\lrp{\lrn{u}+\lrn{v}}} \lrn{v - u}
        \end{alignat*}
        
    \end{lemma}
    \begin{proof}
        Let us define the variational field
        \begin{alignat*}{1}
            \Lambda(s,t) = \Exp_z\lrp{t \lrp{u + s (v-u)}}
        \end{alignat*}
        We verify that
        \begin{alignat*}{1}
            \del_s \Lambda(s,0) =& 0\\
            \del_t \Lambda(s,0) =& u + s (v-u)\\
            D_t \del_s \Lambda(s,0) =& v-u
        \end{alignat*}

        Lemma \ref{l:jacobi_field_norm_bound} then immediately gives
        \begin{alignat*}{1}
            & \lrn{\del_s \Lambda(s,t)}
            \leq \frac{\sinh\lrp{\sqrt{L_R}\lrp{\lrn{u}+\lrn{v}} t}}{\sqrt{L_R}\lrp{\lrn{u}+\lrn{v}}} \lrn{v - u}
        \end{alignat*}

    \end{proof}

    \subsection{Discrete Coupling Bounds}
    \label{s:discrete_coupling_bounds}
    This section presents two key lemmas which play an important role in many of our proofs. Lemma \ref{l:discrete-approximate-synchronous-coupling} analyzes the distance between $\Exp_x(u)$ and $\Exp_y(v)$, as a function of $x,y$ and $u,v$. Notably, Lemma \ref{l:discrete-approximate-synchronous-coupling} implies that when $u,v$ are "parallel", i.e. $u - \party{y}{x} v = 0$, then the distance between $\dist\lrp{\Exp_{x}(u), \Exp_y(v)}$ is not much larger than $\dist\lrp{x,y}$. The proof of Lemma \ref{l:discrete-approximate-synchronous-coupling} is based on a first-order expansion of the Jacobi equation for Riemannian manifolds.

    The second key lemma is Lemma \ref{l:discrete-approximate-synchronous-coupling-ricci}. It considers a similar problem setup as Lemma \ref{l:discrete-approximate-synchronous-coupling}, but is based on a second-order expansion of the Jacobi equation. It thus requires an additional bound on the derivative of the Riemannian curvature tensor. The more refined distance bound in Lemma \ref{l:discrete-approximate-synchronous-coupling-ricci} is required to properly analyze the convergence of both continuous-time SDEs (Lemma \ref{l:g_contraction_without_gradient_lipschitz}) as well as discrete-time stochastic processes (Lemma \ref{l:sgld-lemma}). The term $\int_0^1 \lin{R...\gamma'(s)} ds$ in the upper bound of Lemma \ref{l:discrete-approximate-synchronous-coupling-ricci} gives rise to the Ricci curvature (as opposed to sectional curvature) dependencies in our results.

    \begin{lemma}\label{l:discrete-approximate-synchronous-coupling}
        Let $x,y\in M$. Let $\gamma(s):[0,1] \to M$ be a minimizing geodesic between $x$ and $y$ with $\gamma(0) = x$ and $\gamma(1) = y$. Let $u\in T_x M$ and $v\in T_y M$. Let $u(s)$ and $v(s)$ be the parallel transport of $u$ and $v$ along $\gamma$, with $u(0) = u$ and $v(1) = v$. 
        
        Then
        \begin{alignat*}{1}
            \dist\lrp{\Exp_{x}(u), \Exp_y(v)}^2 \leq& \lrp{1+ 4 \C^2 e^{4\C}} \dist\lrp{x,y}^2 + 32 e^{\C} \lrn{v(0) - u(0)}^2 + 2\lin{\gamma'(0), v(0) - u(0)} 
        \end{alignat*}
        where $\C := \sqrt{L_R} \lrp{\lrn{u} + \lrn{v}}$.

    \end{lemma}
    \begin{proof}
        Let us consider the length function $E(\gamma) = \int_0^1 \lrn{\gamma'(s)}^2 ds$. We define a variation of geodesics $\Lambda(s,t)$:
        \begin{alignat*}{1}
            \Lambda(s,t):= \Exp_{\gamma(s)}\lrp{t\lrp{u(s) + s (v(s) - u(s))}}
        \end{alignat*}
        We verify that
        \begin{alignat*}{1}
            \del_s \Lambda(s,0) =& \gamma'(s)\\
            \del_t \Lambda(s,0) =& u(s) + s (v(s)-u(s))\\
            D_t \del_s \Lambda(s,0) =& v(s)-u(s)
        \end{alignat*}
        
        Consider a fixed $t$, and let $\gamma_t(s):= \Lambda(s,t)$ (so $\gamma_t'(s)$ is the velocity wrt $s$). 
        \begin{alignat*}{1}
            & \frac{d}{dt} E(\gamma_t)\\
            =& \frac{d}{dt}\int_0^1 \lrn{\gamma_t'(s)}^2 ds\\
            =& \int_0^1 2\lin{\gamma_t'(s), D_t \gamma_t'(s)} ds\\
            =& \int_0^1 2\lin{\del_s \Lambda(s,t), D_t \del_s \Lambda(s,t)} ds\\
            =& \int_0^1 2\lin{\del_s \Lambda(s,0), D_t \del_s \Lambda(s,0)} ds \\
            &\quad + \int_0^1 2\lin{\del_s \Lambda(s,0), \lrbb{D_t \del_s \Lambda(s,t)}^{\to \Lambda(s,0)} - D_t \del_s \Lambda(s,0)} ds \\
            &\quad +\int_0^1 2\lin{\del_s \Lambda(s,t) - \lrbb{\del_s \Lambda(s,0)}^{\to \Lambda(s,t)}, D_t \del_s \Lambda(s,t)} ds 
            \numberthis \label{e:t:ddtE}
        \end{alignat*}

        For any $s$, and for $t=0$, $\del_s \Lambda(s,0) = \gamma'(s)$ and $D_t \del_s \Lambda(s,0) = v(s) - u(s)$. Using the fact that norms and inner products are preserved under parallel transport, the first term can be simplified as
        \begin{alignat*}{1}
            \int_0^1 2\lin{\del_s \Lambda(s,0), D_t \del_s \Lambda(s,0)} ds = 2\lin{\gamma'(0), v(0) - u(0)}
        \end{alignat*}

        To bound the second and third term, we use Lemma \ref{l:jacobi_field_norm_bound}:
        \begin{alignat*}{1}
            \lrn{\del_s \Lambda(s,0)}
            =& \lrn{\gamma'(0)}
        \end{alignat*}
        and
        \begin{alignat*}{1}
            &\lrn{D_t \del_s \Lambda(s,t) - \party{\Lambda(s,0)}{\Lambda(s,t)}\lrp{D_t \del_s \Lambda(s,0)}} \\
            \leq& \sqrt{L_R \lrn{\del_t \Lambda(s,0)}^2}  \sinh\lrp{\sqrt{L_R \lrn{\del_t \Lambda(s,0)}^2}}\lrn{\del_s \Lambda(s,0)}  + \lrp{\cosh\lrp{\sqrt{L_R \lrn{\del_t \Lambda(s,0)}^2}} - 1} \lrn{D_t \del_s \Lambda(s,0)}\\
            \leq& \C  \sinh\lrp{\C}\lrn{\gamma'(0)}  + \lrp{\cosh\lrp{\C} - 1} \lrn{v(0) - u(0)}
        \end{alignat*}
        where we use the fact that $\sqrt{L_R \lrn{\del_t \Lambda(s,0)}^2}\leq \C$.

        We can thus bound the second term of \eqref{e:t:ddtE} as
        \begin{alignat*}{1}
            &\lrabs{\int_0^1 2\lin{\del_s \Lambda(s,0), \lrbb{D_t \del_s \Lambda(s,t)}^{\to \Lambda(s,0)} - D_t \del_s \Lambda(s,0)} ds}\\
            \leq& 2\lrn{\gamma'(0)} \cdot \lrp{\C  \sinh\lrp{\C}\lrn{\gamma'(0)}  + \lrp{\cosh\lrp{\C} - 1} \lrn{v(0) - u(0)}}\\
            \leq& 4 \lrn{\gamma'(0)}^2 \lrp{\C \sinh\lrp{\C} + \lrp{\cosh\lrp{\C}-1}^2} + 4 \lrn{v(0) - u(0)}^2
        \end{alignat*}
        
        Finally, to bound the third term of \eqref{e:t:ddtE}, we again apply Lemma \ref{l:jacobi_field_norm_bound}:
        \begin{alignat*}{1}
            &\lrn{\del_s \Lambda(s,t) - \party{\Lambda(s,0)}{\Lambda(s,t)}\lrp{\del_s \Lambda(s,0)}}\\
            \leq& \lrp{\cosh\lrp{\sqrt{L_R \lrn{\del_t \Lambda(s,0)}^2} } - 1} \lrn{\del_s \Lambda(s,0)} + \lrp{\frac{\sinh\lrp{\sqrt{L_R \lrn{\del_t \Lambda(s,0)}^2}}}{\sqrt{L_R \lrn{\del_t \Lambda(s,0)}^2}}} \lrn{D_t \del_s \Lambda(s,0)}\\
            \leq& \lrp{\cosh\lrp{\sqrt{L_R} \lrp{\lrn{u(s)} + \lrn{v(s)}}} - 1} \lrn{\gamma'(s)} + \lrp{\frac{\sinh\lrp{\sqrt{L_R} \lrp{\lrn{u(s)} + \lrn{v(s)}}}}{\sqrt{L_R} \lrp{\lrn{u(s)} + \lrn{v(s)}}}} \lrn{v(s)-u(s)}\\
            =& \lrp{\cosh\lrp{\C} - 1} \lrn{\gamma'(0)} + \lrp{\frac{\sinh\lrp{\C}}{\C}} \lrn{v(0) - u(0)}
        \end{alignat*}
        and
        \begin{alignat*}{1}
            & \lrn{D_t \del_s \Lambda(s,t)}\\
            \leq& \sqrt{L_R \lrn{\del_t \Lambda(s,0)}^2} \sinh\lrp{\sqrt{L_R \lrn{\del_t \Lambda(s,0)}^2}} \lrn{\del_s \Lambda(s,0)} + \cosh\lrp{\sqrt{L_R \lrn{\del_t \Lambda(s,0)}^2}} \lrn{D_t \del_s \Lambda(s,0)}\\
            \leq& \sqrt{L_R}\lrp{\lrn{u} + \lrn{v}}  \sinh\lrp{\sqrt{L_R}\lrp{\lrn{u} + \lrn{v}}} \lrn{\gamma'(0)} + \cosh\lrp{\sqrt{L_R}\lrp{\lrn{u} + \lrn{v}}} \lrn{v(0) - u(0)}\\
            =& \C  \sinh\lrp{\C} \lrn{\gamma'(0)} + \cosh\lrp{\C} \lrn{v(0) - u(0)}
        \end{alignat*}
        for $0 \leq t \leq 1$, where we usse the fact that $\cosh(r)$ and $\frac{\sinh(r)}{r}$ are monotonically increasing in $r$.
        Put together, the third term of \eqref{e:t:ddtE} is bounded as
        \begin{alignat*}{1}
            & \lrabs{\int_0^1 2\lin{\del_s \Lambda(s,t) - \lrbb{\del_s \Lambda(s,0)}^{\to \Lambda(s,t)}, D_t \del_s \Lambda(s,t)} ds} \\
            \leq& 2 \lrp{\lrp{\cosh\lrp{\C} - 1} \lrn{\gamma'(0)} + \lrp{\frac{\sinh\lrp{\C}}{\C}} \lrn{v(0) - u(0)} }\cdot \lrp{\C  \sinh\lrp{\C} \lrn{\gamma'(0)} + \cosh\lrp{\C} \lrn{v(0) - u(0)}}\\
            \leq& 8 \lrn{v(0) - u(0)}^2 \lrp{\cosh\lrp{\C}^2 + \frac{\sinh\lrp{\C}^2}{\C^2}} + 8\lrn{\gamma'(0)}^2 \lrp{\lrp{\cosh\lrp{\C} - 1}^2 + \C^2 \sinh\lrp{\C}^2}
        \end{alignat*}

        Put together, we get
        \begin{alignat*}{1}
            &\lrabs{\frac{d}{dt} E(\gamma_t) - 2\lin{\gamma'(0), v(0) - u(0)}} \\
            \leq& 8 \lrn{v(0) - u(0)}^2 \lrp{\cosh\lrp{\C}^2 + \frac{\sinh\lrp{\C}^2}{\C^2}} + 8\lrn{\gamma'(0)}^2 \lrp{\lrp{\cosh\lrp{\C} - 1}^2 + \C^2 \sinh\lrp{\C}^2}\\
            &\quad + 4 \lrn{v(0) - u(0)}^2 + 4 \lrn{\gamma'(0)}^2 \lrp{\C \sinh\lrp{\C} + \lrp{\cosh\lrp{\C}-1}^2}\\
            \leq& 8 \lrn{v(0) - u(0)}^2 \lrp{\cosh\lrp{\C}^2 + \frac{\sinh\lrp{\C}^2}{\C^2} + 1} \\
            &\quad + 8\lrn{\gamma'(0)}^2 \lrp{2\lrp{\cosh\lrp{\C} - 1}^2 + \C^2 \sinh\lrp{\C}^2 + \C \sinh\lrp{\C}}
        \end{alignat*}

        Integrating for $t=\in[0,1]$, and noting that $E(\gamma_0) = \lrn{\gamma'(0)}$,
        \begin{alignat*}{1}
            E\lrp{\gamma_1} 
            \leq& \lrp{1+8 \lrp{2\lrp{\cosh\lrp{\C} - 1}^2 + \C^2 \sinh\lrp{\C}^2 + \C \sinh\lrp{\C}}} E\lrp{\gamma_0} \\
            &\qquad + 8 \lrn{v(0) - u(0)}^2 \lrp{\cosh\lrp{\C}^2 + \frac{\sinh\lrp{\C}^2}{\C^2} + 1}\\
            &\qquad + 2\lin{\gamma'(0), v(0) - u(0)} 
        \end{alignat*}

        From Lemma \ref{l:sinh_bounds}, we can upper bound
        \begin{alignat*}{1}
            8 \lrp{2\lrp{\cosh\lrp{\C} - 1}^2 + \C^2 \sinh\lrp{\C}^2 + \C \sinh\lrp{\C}}
            \leq& 8 r^4 e^{2r} + 8 r^4 e^{2r} + r^2 e^r\\
            \leq& 4 r^2 e^{4r}\\
            \cosh\lrp{\C}^2 + \frac{\sinh\lrp{\C}^2}{\C^2} + 1
            \leq& 4e^{2r}
        \end{alignat*}
        where we use the fact that $r^2 \leq e^{2r}/6$ for all $r \geq 0$. 

        The conclusion follows by noting that $\dist(x,y) = \sqrt{E\lrp{\gamma_0}}$ and $\dist\lrp{\Exp_{x}(u), \Exp_y(v)} \leq \sqrt{E\lrp{\gamma_1}}$.

    \end{proof}

    \begin{lemma}\label{l:discrete-approximate-synchronous-coupling-ricci}
        Let $x,y\in M$. Let $\gamma(s):[0,1] \to M$ be a minimizing geodesic between $x$ and $y$ with $\gamma(0) = x$ and $\gamma(1) = y$. Let $u\in T_x M$ and $v\in T_y M$. Let $u(s)$ and $v(s)$ be the parallel transport of $u$ and $v$ along $\gamma$. Let $u = u_1 + u_2$ and $v = v_1 + v_2$ be a decomposition such that $v_2 = \party{x}{y} u_2$, where the parallel transport is along $\gamma(s)$. 

        Let us define $u_1(s), u_2(s), v_1(s)$, all mapping from $[0,1] \to T_{\gamma(s)} M$, such that they are the parallel transport of $u_1, u_2, v_1$ along $\gamma(s)$ respectively ($u_1(0) = u_1$, $u_2(0) = u_2$, $v_1(1) = v_1$, $u_2(1) = v_2$)
        
        Then
        \begin{alignat*}{1}
            &\dist\lrp{\Exp_{x}(u), \Exp_y(v)}^2 - \dist\lrp{x,y}^2\\
            \leq& 2\lin{\gamma'(0), v(0) - u(0)} + \lrn{v(0) - u(0)}^2 \\
            &\quad -\int_0^1 \lin{R\lrp{\gamma'(s),(1-s) u(s) + s v(s)}(1-s) u(s) + s v(s),\gamma'(s)} ds \\
            &\quad + \lrp{2\C^2 e^{\C} + 18\C^4 e^{2\C}} \lrn{v(0) - u(0)}^2 + \lrp{18\C^4 e^{2\C} + 4\C'} \dist\lrp{x,y}^2 + 4 \C^2 e^{2\C} \dist\lrp{x,y} \lrn{v(0) - u(0)}
        \end{alignat*}
        where $\C := \sqrt{L_R} \lrp{\lrn{u} + \lrn{v}}$ and $\C' := L_R' \lrp{\lrn{u} + \lrn{v}}^3$.
    \end{lemma}
    \begin{proof}

        The proof is similar to Lemma \ref{l:discrete-approximate-synchronous-coupling}. Let us consider the length function $E(\gamma) = \int_0^1 \lrn{\gamma'(s)}^2 ds$. We define a variation of geodesics $\Lambda(s,t)$:
        \begin{alignat*}{1}
            \Lambda(s,t):= \Exp_{\gamma(s)}\lrp{t\lrp{u(s) + s (v(s) - u(s))}}
        \end{alignat*}
        We verify that
        \begin{alignat*}{1}
            \del_s \Lambda(s,0) =& \gamma'(s)\\
            \del_t \Lambda(s,0) =& u(s) + s (v(s)-u(s))\\
            D_t \del_s \Lambda(s,0) =& v(s)-u(s)
        \end{alignat*}

        Consider a fixed $t$, and let $\gamma_t(s):= \Lambda(s,t)$ (so $\gamma_t'(s)$ is the velocity wrt $s$). 

        \begin{alignat*}{1}
            & \frac{d}{dt} E(\gamma_t)\\
            =& \frac{d}{dt}\int_0^1 \lrn{\gamma_t'(s)}^2 ds\\
            =& \int_0^1 2\lin{\gamma_t'(s), D_t \gamma_t'(s)} ds\\
            =& \int_0^1 2\lin{\del_s \Lambda(s,t), D_t \del_s \Lambda(s,t)} ds
        \end{alignat*}
        and
        \begin{alignat*}{1}
            & \frac{d^2}{dt^2} E(\gamma_t)\\
            =& \int_0^1 2\lin{D_t \del_s \Lambda(s,t), D_t \del_s \Lambda(s,t)} ds + \int_0^1 2\lin{\del_s \Lambda(s,t), D_t D_t \del_s \Lambda(s,t)} ds\\
            =& \int_0^1 2\lrn{D_t \del_s \Lambda(s,t)}^2 ds -\int_0^1 2\lin{R\lrp{\del_s(\Lambda(s,t)), \del_t \Lambda(s,t)}\del_t \Lambda(s,t),\del_s \Lambda(s,t)} ds\\
            =& \int_0^1 2\lrn{D_t \del_s \Lambda(s,t)}^2 ds - \int_0^1 2\lin{R\lrp{\del_s(\Lambda(s,t)), \party{\Lambda(s,0)}{\Lambda(s,t)}\del_t \Lambda(s,0)}\party{\Lambda(s,0)}{\Lambda(s,t)}\del_t \Lambda(s,t),\del_s \Lambda(s,t)} ds\\
            \leq& \int_0^1 2\lrn{D_t \del_s \Lambda(s,t)}^2 ds - \int_0^1 2\lin{R\lrp{\del_s(\Lambda(s,0)),\del_t \Lambda(s,0)}\del_t \Lambda(s,t),\del_s \Lambda(s,0)} ds \\
            &\quad + \int_0^1 4 L_R \lrn{\del_t \Lambda(s,0)}^2 \lrn{\del_s \Lambda(s,t)} \lrn{\del_s \Lambda(s,t) - \party{\Lambda(s,0)}{\Lambda(s,t)}\del_s \Lambda(s,0)} + 4 L_R' \lrn{\del_t \Lambda(s,0)}^3 \lrn{\del_s \Lambda(s,0)}^2 ds
        \end{alignat*}
        where the second equality uses the Jacobi equation.

        The Riemannian curvature tensor term can be simplified as
        \begin{alignat*}{1}
            & -\int_0^1 2\lin{R\lrp{\del_s(\Lambda(s,0)),\del_t \Lambda(s,0)}\del_t \Lambda(s,t),\del_s \Lambda(s,0)} ds\\
            =& -2\int_0^1 \lin{R\lrp{\gamma'(s),(1-s) u(s) + s v(s)}(1-s) u(s) + s v(s),\gamma'(s)} ds
        \end{alignat*}

        We further bound
        \begin{alignat*}{1}
            & \int_0^1 2\lrn{D_t \del_s \Lambda(s,t)}^2 ds \\
            \leq& \int_0^1 2\lrp{{\C}  \sinh\lrp{{\C}} \lrn{\del_s \Lambda(s,0)} + \cosh\lrp{{\C}} \lrn{D_t \del_s \Lambda(s,0)}}^2 ds \\
            \leq& 2\cosh\lrp{\C}^2 \int_0^1 \lrn{D_t \del_s \Lambda(s,0)}^2 ds\\
            &\quad + 2{\C}^2  \sinh\lrp{{\C}}^2 \int_0^1 \lrn{\del_s \Lambda(s,0)}^2 ds +  4 \C \int_0^1 \sinh\lrp{\C} \cosh\lrp{\C} \lrn{ \del_s \Lambda(s,0)}\lrn{D_t \del_s \Lambda(s,0)} ds\\
            =& 2\cosh\lrp{\C}^2 \lrn{v(0)-u(0)}^2\\
            &\quad + 2{\C}^2  \sinh\lrp{{\C}}^2 \dist\lrp{x,y}^2 +  4 \C \int_0^1 \sinh\lrp{\C} \cosh\lrp{\C} \lrn{ \del_s \Lambda(s,0)}\lrn{D_t \del_s \Lambda(s,0)} ds\\
            \leq& 2 \lrn{v(0) - u(0)}^2 + 2 \lrp{\C^2 e^{\C} + \C^4 e^{2\C}} \lrn{v(0) - u(0)}^2\\
            &\quad + 2\C^4 e^{2\C} \dist\lrp{x,y}^2 + 4 \C^2 e^{2\C} \dist\lrp{x,y} \lrn{v(0) - u(0)}
        \end{alignat*}
        where we use Lemma \ref{l:jacobi_field_norm_bound} and Lemma \ref{l:sinh_bounds}.

        We also bound
        \begin{alignat*}{1}
            & \int_0^1 4 L_R \lrn{\del_t \Lambda(s,0)}^2 \lrn{\del_s \Lambda(s,t)} \lrn{\del_s \Lambda(s,t) - \party{\Lambda(s,0)}{\Lambda(s,t)}\del_s \Lambda(s,0)} ds\\
            \leq& 4\C^2 \lrp{\cosh\lrp{\C} \dist\lrp{x,y} + \frac{\sinh\lrp{\C}}{\C} \lrn{u(0) - v(0)}}\lrp{\lrp{\cosh\lrp{\C}-1}\dist\lrp{x,y} + \lrp{\frac{\sinh\lrp{\C}}{\C}-1} \lrn{u(0)-v(0)}}\\
            \leq& 16\C^4e^{2\C} \dist\lrp{x,y}^2 + 16\C^4e^{2\C} \lrn{u(0) - v(0)}^2
        \end{alignat*}
        where we use Lemma \ref{l:jacobi_field_norm_bound} and Lemma \ref{l:sinh_bounds}.

        We finally bound
        \begin{alignat*}{1}
            & \int_0^1 4 L_R' \lrn{\del_t \Lambda(s,0)}^3 \lrn{\del_s \Lambda(s,0)}^2 ds
            \leq 4\C' \dist\lrp{x,y}^2
        \end{alignat*}
        by definition of $\C'$.

        Combining the above bounds,
        \begin{alignat*}{1}
          &E\lrp{\gamma_1}
          = \at{\frac{d}{dt} E\lrp{\gamma_t}}{t=0} + \int_0^1 \int_0^r \frac{d^2}{dt^2} E\lrp{\gamma_t} dt dr\\
          &\quad\leq\  2\lin{\gamma'(0), v(0) - u(0)} + \lrn{v(0) - u(0)}^2\\
          &\quad- \int_0^1 \lin{R\lrp{\gamma'(s),(1-s) u(s) + s v(s)}(1-s) u(s) + s v(s),\gamma'(s)} ds \\
          &\quad + 2 \lrp{\C^2 e^{\C} + \C^4 e^{2\C}} \lrn{v(0) - u(0)}^2 + 2\C^4 e^{2\C} \dist\lrp{x,y}^2 + 4 \C^2 e^{2\C} \dist\lrp{x,y} \lrn{v(0) - u(0)}\\
          &\quad + 16\C^4e^{2\C} \dist\lrp{x,y}^2 + 16\C^4e^{2\C} \lrn{u(0) - v(0)}^2 + 4\C' \dist\lrp{x,y}^2\\
          &\quad= 2\lin{\gamma'(0), v(0) - u(0)} + \lrn{v(0) - u(0)}^2\\
          &\quad- \int_0^1 \lin{R\lrp{\gamma'(s),(1-s) u(s) + s v(s)}(1-s) u(s) + s v(s),\gamma'(s)} ds \\
          &\quad + \lrp{2\C^2 e^{\C} + 18\C^4 e^{2\C}} \lrn{v(0) - u(0)}^2 + \lrp{18\C^4 e^{2\C} + 4\C'} \dist\lrp{x,y}^2 + 4 \C^2 e^{2\C} \dist\lrp{x,y} \lrn{v(0) - u(0)}
        \end{alignat*}
        Our conclusion follows as $\dist\lrp{\Exp_x(u), \Exp_y(v)}^2 \leq E(\gamma_1)$.
    \end{proof}

    \begin{lemma}\label{l:sinh_ode}
        Let $a_t, b_t : t \to \Re^+$ satisfy
        \begin{alignat*}{1}
            & \frac{d}{dt} a_t = b_t\\
            & \frac{d}{dt} b_t \leq C a_t
        \end{alignat*}
        with initial conditions $a_0, b_0$, then for all $t$,
        \begin{alignat*}{1}
            & a_t \leq {a_0} \cosh\lrp{\sqrt{C} t} + \frac{b_0}{\sqrt{C}} \sinh\lrp{\sqrt{C} t}\\
            & b_t \leq \sqrt{C} \lrp{{a_0} \sinh\lrp{\sqrt{C} t} + \frac{b_0}{\sqrt{C}} \cosh\lrp{\sqrt{C} t}}
        \end{alignat*}
    \end{lemma}
    \begin{proof}
        Let $x_t := {a_0} \cosh\lrp{\sqrt{C} t} + \frac{b_0}{\sqrt{C}} \sinh\lrp{\sqrt{C} t}$ and $y_t := \sqrt{C} \lrp{{a_0} \sinh\lrp{\sqrt{C} t} + \frac{b_0}{\sqrt{C}} \cosh\lrp{\sqrt{C} t}}$. We verify that
        \begin{alignat*}{1}
            & \frac{d}{dt} x_t = \sqrt{C} \lrp{{a_0} \sinh\lrp{\sqrt{C} t} + \frac{b_0}{\sqrt{C}} \cosh\lrp{\sqrt{C} t}} = y_t\\
            & \frac{d}{dt} y_t = C \lrp{{a_0} \cosh\lrp{\sqrt{C} t} + \frac{b_0}{\sqrt{C}} \sinh\lrp{\sqrt{C} t}} = C x_t
        \end{alignat*}
        We further verify the initial conditions. Note that $\sinh(0) = 0$ and $cosh(0) = 1$. Thus
        \begin{alignat*}{1}
            & x_0 = a_0\\
            & y_0 = b_0
        \end{alignat*}
        Finally, we verify that $a_t \leq x_t$ and $b_t \leq y_t$ for all $t$:
        \begin{alignat*}{1}
            & \frac{d}{dt} x_t - a_t = y_t - b_t\\
            & \frac{d}{dt} y_t - b_t \geq C \lrp{x_t - a_t}
        \end{alignat*}
        
    \end{proof}

    \begin{lemma}\label{l:sinh_ode_with_offset}
        Let $a_t, b_t : t \to \Re^+$ satisfy
        \begin{alignat*}{1}
            & \frac{d}{dt} a_t = b_t\\
            & \frac{d}{dt} b_t \leq C a_t + Dt + E
        \end{alignat*}
        with initial conditions $a_0 = 0, b_0 = 0$, then for all $t$,
        \begin{alignat*}{1}
            & a_t \leq \frac{E}{C} \cosh(\sqrt{C} t) + \frac{D}{C^{3/2}} \sinh\lrp{\sqrt{C} t} - \frac{D}{C}t  - \frac{E}{C}\\
            & b_t \leq \frac{E}{\sqrt{C}} \sinh(\sqrt{C} t) + \frac{D}{C} \cosh\lrp{\sqrt{C} t} - \frac{D}{C}
        \end{alignat*}
    \end{lemma}
    \begin{proof}
        Let $x_t := \frac{E}{C} \cosh(\sqrt{C} t) + \frac{D}{C^{3/2}} \sinh\lrp{\sqrt{C} t} - \frac{D}{C}t  - \frac{E}{C}$ and \\
        $y_t :=  \frac{E}{\sqrt{C}} \sinh(\sqrt{C} t) + \frac{D}{C} \cosh\lrp{\sqrt{C} t} - \frac{D}{C}$.
        
        We verify that
        \begin{alignat*}{1}
            & \frac{d}{dt} x_t = \frac{E}{\sqrt{C}} \sinh(\sqrt{C} t) + \frac{D}{C} \cosh\lrp{\sqrt{C} t} - \frac{D}{C} = y_t\\
            & \frac{d}{dt} y_t = E \cosh(\sqrt{C} t) + \frac{D}{\sqrt{C}} \sinh\lrp{\sqrt{C} t} = C x_t + D t + E
        \end{alignat*}
        we also verify the initial conditions that $x_0 = 0$ and $y_0 = 0$.

    \end{proof}

    \begin{lemma}\label{l:identity_covariance_parallel_transport}
		Let $x,y\in M$, and let $E_1...E_d$ be an orthonormal basis at $T_xM$.
		Let $v\in T_x M$ be a random vector with $\E{\lin{E_i, v}\lin{E_j,v}} = \ind{i=j}$.
		Let $\gamma:[0,1] \to M$ be any smooth path between $x$ and $y$. Let $v(t)$ be the parallel transport of $v$ along $\gamma$. Then for any basis $E'_1...E'_d$ at $T_y M$,
		\begin{alignat*}{1}
			\E{\lin{v(t), E_i'}\lin{v(t), E_j'}} = \ind{i=j}
		\end{alignat*}
		In other words, if $v$ has identity covariance, then the parallel transport of $v$ has identity covariance.
	\end{lemma}
	\begin{proof}
		Let $E_i(t)$ be an orthonormal frame along $\gamma$ with $E_i(0) = E_i$.
		Under parallel transport, $\frac{d}{dt}\lin{v(t), E_i(t)}=0$. Thus for all $t$,
		\begin{alignat*}{1}
			\E{\lin{E_i(t), v}\lin{E_j(t),v}}
			= \E{\lin{E_i(0), v}\lin{E_j(0),v}} = \ind{i=j}
		\end{alignat*}
		Finally, consider any basis $E_i'$. Let $E_i' = \sum_k \alpha^i_k E_i(1)$, i.e. $\alpha^i_k = \lin{E_i', E_k(1)}$ Then
		\begin{alignat*}{1}
			& \lin{v, E_i'}\lin{v, E_j'}\\
			=& \sum_{k,l} \alpha^i_j \alpha^j_l \lin{v, E_k}\lin{v, E_l}\\
			=& \sum_{k,l} \alpha^i_j \alpha^j_l \ind{k=l}\\
			=& \sum_{k} \alpha^i_k \alpha^j_k\\
			=& \lin{E_i', E_j'}\\
			=& \ind{i=j}
		\end{alignat*}
	\end{proof}

    \begin{lemma}
        \label{l:symmetric_distribution_parallel_transport}
        Let $x,y\in M$, and let $E_1...E_d$ be an orthonormal basis at $T_xM$.

        Let $\alpha$ denote a spherically symmetric random variable in $\Re^d$, i.e. for any orthogonal matrix $G \in \Re^{d\times d}$
        \begin{alignat*}{1}
            \alpha \overset{d}{=} G \alpha
        \end{alignat*}
        Then for any $x\in M$, let $E_1... E_d$ and $E_1'...E_d'$ be two sets of orthonormal bases of $T_x M$. then
        \begin{alignat*}{1}
            \sum_{i=1}^d \alpha_i E_i \overset{d}{=} \sum_{i=1}^d \alpha_i E_i'
        \end{alignat*}

        Consequently, let $v\in T_x M:= \sum_{i=1}^d \alpha_i E_i$. Let $\gamma:[0,1] \to M$ be any differentiable path between $x$ and $y$. Let $v(t)$ be the parallel transport of $v$ along $\gamma$. Then for any orthogonal basis $E'_1...E'_d$ at $T_y M$,
		\begin{alignat*}{1}
			v(1) \overset{d}{=} \sum_{i=1}^d \alpha_i E_i'
		\end{alignat*}
        
    \end{lemma}

    \begin{proof}
        First, we verify that if $\alpha$ is spherically symmetric, and $E_1.. E_d$, $E'_1...E'_d$ are two sets of orthonormal basis at some point $x$, then
        \begin{alignat*}{1}
            \sum_{i=1}^d \alpha_i E_i \overset{d}{=} \sum_{i=1}^d \alpha_i E'_i
        \end{alignat*}
    
        To see this, notice that there exists an orthogonal matrix $G$, with $G_{i,j} = \lin{E_i, E'_j}$, such that
        \begin{alignat*}{1}
            E_i = \sum_{j=1}^d G_{j,i} E'_j
        \end{alignat*}
        We further verify that $G_{i,j}$ is orthogonal. It suffices to verify that $G G^T = I$.
        \begin{alignat*}{1}
            \ind{j=k} = \lin{E'_i, E'_j}
            =& \lin{\sum_{k=1}^d \lin{E'_i, E_k} E_j, \sum_{\ell=1}^d \lin{E'_j, E_\ell} E_\ell}\\
            =& \sum_{k,\ell} \lin{E'_i, E_k} \lin{E'_j, E_\ell} \lin{E_j, E_\ell}\\
            =& \sum_{k} \lin{E'_i, E_k} \lin{E'_j, E_k} \\
            =& \lin{G_{i,\cdot}, G_{j,\cdot}}
        \end{alignat*}
        Note that the inner product on the last line is dot product over $\Re^d$, and the inner product on preceding lines are over $T_x M$. The above implies that
        \begin{alignat*}{1}
            G G^T = I
        \end{alignat*}
        i.e. $G$ is orthogonal.
        
        Now consider any arbitrary function $f: T_x M \to \Re$, then
        \begin{alignat*}{1}
            \E{f\lrp{\sum_i \alpha_i E_i}}
            =& \E{f\lrp{\sum_i \sum_j \alpha_i G_{i,j} E'_j}}\\
            :=& \E{f\lrp{\sum_j \beta_j E'_j}}
        \end{alignat*}
        where we defined $\beta_j := \sum_i \alpha_i G_{i,j}$. We finally verify that $\beta \overset{d}{=} \alpha$. This follows from the fact that $\beta = G^T \alpha$, where $G$ is an orthogonal matrix, and the definition of spherical symmetry for $\alpha$.

        Consider an arbitrary line $\gamma(t) : [0,1] \to M$, with $x := \gamma(0)$, $y := \gamma(1)$. Let $E_i$ be an orthonormal basis at $T_x M$, and $E_i(t)$ be an orthonormal basis at $T_{\gamma(t)} M$ obtained from parallel transport of $E_i$. This proves the first claim.

        To verify the second claim, let $v \in T_x M$ be a random vector, given by 
        \begin{alignat*}{1}
            v = \sum_{i=1}^d \alpha_i E_i
        \end{alignat*}
        where $\alpha$ is some spherically random vector in $\Re^d$. Let $v(t)$ be the parallel transport of $v$ along $\gamma$. Let $\alpha(t):= \lin{v(t), E_i(t)}$. Then by definition of parallel transport, for all $i$,
        \begin{alignat*}{1}
            \frac{d}{dt} \lin{v(t), E_i(t)} = 0
        \end{alignat*}
        so that for all $t\in[0,1]$,
        \begin{alignat*}{1}
            \alpha(t) := \alpha
        \end{alignat*}
        the second claim then follows from the first claim.

    \end{proof}

    \section{Matrix ODE}\label{s:matrix_exponent}

    In this section we provide Gronwall-style inequality for matrix ODE. The results in this section are necessary for analyzing Jacobi Equation, whose coordinates with respect to some orthonormal frame can be viewed as an ODE in $\Re^d$. In particular, Lemma \ref{l:discrete-approximate-synchronous-coupling} and Lemma \ref{l:discrete-approximate-synchronous-coupling-ricci} rely on results in this section.

\begin{lemma}[Formal Matrix Exponent]\label{l:formal-matrix-exponent}
    Given $\MM(t): \Re^+ \to \Re^{d\times d}$, define $\emat\lrp{t;\MM}: \Re^+ \to \Re^{d\times d}$ as the solution to the matrix ODE
    \begin{alignat*}{1}
        \emat\lrp{0;\MM} =& I\\
        \ddt \emat\lrp{t;\MM} =& \MM(t) \emat\lrp{t;\MM}
    \end{alignat*}

    Then
    \begin{enumerate}
        \item Let $\xx(t)$ be the solution to the ODE $\ddt \xx(t) = \MM(t) \xx(t)$, for some $\MM$, then
        \begin{alignat*}{1}
            \xx(t) = \emat\lrp{t;\MM} \xx(0)
        \end{alignat*}
        \item Let $\zz(t)$ be the solution to $\ddt \zz(t) = \MM(t) \zz(t) + \vv(t)$, for some $\MM$, $\vv$, then
        \begin{alignat*}{1}
            \zz(T) = \int_0^T \emat\lrp{T-s;\NN_s} \vv(s) ds + \emat\lrp{T;\MM} \zz(0)
        \end{alignat*}
        where for any $s,t$, $\NN_s(t) := \MM(s+t)$.
    \end{enumerate}
\end{lemma}
\begin{proof}[Proof of Lemma \ref{l:formal-matrix-exponent}]
    Let $\yy_t := \emat\lrp{t;\MM} \xx(0)$. We verify that
    \begin{alignat*}{1}
        \yy(0) =& 0\\
        \ddt \yy(t) =& \lrp{\ddt \emat\lrp{t;\MM}} \xx(0) = \MM(t) \yy(t)
    \end{alignat*}
    Given the same dynamics and initial conditions, we conclude that $\xx(t) = \yy(t)$ for all $t$.

    To verify the second claim, note that 
    \begin{alignat*}{1}
        & \ddt  \int_0^t \emat\lrp{t-s;\NN_s} \vv(s) ds\\
        =& \emat\lrp{0;\NN_t} \vv(s) + \int_0^t \lrp{\frac{d}{dt} \emat\lrp{t-s;\NN_s}} \vv(s) ds\\
        =& \vv(s) + \int_0^t \NN_s(t-s) \emat\lrp{t-s;\NN_s} \vv(s) dt\\
        =& \vv(s) + \MM(t) \int_0^t \emat\lrp{t-s;\NN_s} \vv(s) ds
    \end{alignat*}

    Additionally, $\ddt  \emat\lrp{t;\MM} \zz(0) = \MM(t)  \emat\lrp{t;\MM} \zz(0)$, summing,
    \begin{alignat*}{1}
        \ddt \zz(t)
        =& \ddt \int_0^t \emat\lrp{t-s;\NN_s} \vv(s) ds + \emat\lrp{t;\MM} \zz(0) \\
        =& \vv(s) + \MM(t) \int_0^t \emat\lrp{t-s;\NN_s} \vv(s) ds + \MM(t) \emat\lrp{t;\MM} \zz(0)\\
        =& \vv(s) + \MM(t) \zz(t)
    \end{alignat*}
\end{proof}

\begin{lemma}\label{l:matrix-exponent-block-bounds}
    Let $\emat$ be as defined in Lemma \ref{l:formal-matrix-exponent}. Let 
    \begin{alignat*}{1}
        \bmat{\AA(t) & \BB(t) \\ \CC(t) & \DD(t)} := \emat\lrp{t; \bmat{0 & I \\ \MM(t) & 0}}
    \end{alignat*}
    for some $\MM(t)$. Assume $\lrn{\MM(t)}_2 \leq L_{\MM}$ for all $t$. Then for all $t$,
    \begin{alignat*}{1}
        & \lrn{\AA(t)}_2 \leq \cosh\lrp{\sqrt{L_{\MM}} t}\\
        & \lrn{\BB(t)}_2 \leq \frac{1}{\sqrt{L_{\MM}}} \sinh\lrp{\sqrt{L_{\MM}} t}\\
        & \lrn{\CC(t)}_2 \leq \sqrt{L_{\MM}} \sinh\lrp{\sqrt{L_{\MM}} t}\\
        & \lrn{\DD(t)}_2 \leq \cosh\lrp{\sqrt{L_{\MM}}t}
    \end{alignat*}

    and
    \begin{alignat*}{1}
        & \lrn{\AA(t) - I}_2 \leq \cosh(\sqrt{L_{\MM}} t) - 1\\
        & \lrn{\BB(t) - tI}_2 \leq \frac{1}{\sqrt{L_{\MM}}} \sinh\lrp{\sqrt{L_{\MM}} t} - t\\
        & \lrn{\DD(t) - I}_2 \leq \cosh\lrp{\sqrt{L_{\MM}} t} - 1
    \end{alignat*}

    \begin{alignat*}{1}
        \lrn{\AA(t) - I}_2 \leq& \frac{1}{2} L_{\MM}e^{L_{\MM}}\\
        \lrn{\BB(t) - tI}_2 \leq& \frac{1}{6} L_{\MM} e^{L{\MM}}\\
        \lrn{\DD(t) - I}_2 \leq& \frac{1}{2} L_{\MM}e^{L_{\MM}}\\
        \lrn{\CC(t) }_2 \leq& L_{\MM}e^{L_{\MM}}\\
    \end{alignat*}
\end{lemma}
\begin{proof}[Proof of Lemma \ref{l:matrix-exponent-block-bounds}]
    We first verify the first part of the lemma. Consider the ODE given by
    \begin{alignat*}{1}
        \ddt \cvec{\xx}{\yy}(t) = \bmat{0 & I \\ \MM(t) & 0} \cvec{\xx(t)}{\yy(t)}
    \end{alignat*}
    By Lemma \ref{l:formal-matrix-exponent}, $\bmat{\AA(t) & \BB(t) \\ \CC(t) & \DD(t)}$ satisfies
    \begin{alignat*}{1}
        \cvec{\xx(t)}{\yy(t)} = \bmat{\AA(t) & \BB(t) \\ \CC(t) & \DD(t)} \cvec{\xx(0)}{\yy(0)}
    \end{alignat*}
    By Cauchy Schwarz,
    \begin{alignat*}{1}
        & \frac{d}{dt} \lrn{\xx(t)}_2 \leq \lrn{\yy(t)}_2\\
        & \frac{d}{dt} \lrn{\yy(t)}_2 \leq L_{\MM}\lrn{\xx(t)}_2
    \end{alignat*}
    We apply Lemma \ref{l:sinh_ode}, with $a_t := \lrn{\xx(t)}_2$ and $b_t := \lrn{\yy(t)}_2$, $C := L_{\MM}$. Then
    \begin{alignat*}{1}
        & \lrn{\xx_t}_2 \leq {\lrn{\xx_0}_2} \cosh\lrp{\sqrt{L_{\MM}} t} + \frac{\lrn{\yy_0}_2}{\sqrt{L_{\MM}}} \sinh\lrp{\sqrt{L_{\MM}} t}\\
        & \lrn{\yy_t}_2 \leq \sqrt{L_{\MM}} \lrp{{\lrn{\xx_0}_2} \sinh\lrp{\sqrt{L_{\MM}} t} + \frac{\lrn{\yy_0}_2}{\sqrt{L_{\MM}}} \cosh\lrp{\sqrt{L_{\MM}} t}}
    \end{alignat*}
    This immediately implies that
    \begin{alignat*}{1}
        & \lrn{\AA(t)}_2 \leq \cosh\lrp{\sqrt{L_{\MM}} t}\\
        & \lrn{\BB(t)}_2 \leq \frac{1}{\sqrt{L_{\MM}}} \sinh\lrp{\sqrt{L_{\MM}} t}\\
        & \lrn{\CC(t)}_2 \leq \sqrt{L_{\MM}} \sinh\lrp{\sqrt{L_{\MM}} t}\\
        & \lrn{\DD(t)}_2 \leq \cosh\lrp{\sqrt{L_{\MM}}t}
    \end{alignat*}
    This proves the first claim of the Lemma.

    We now prove the second claim. We verify that 
    \begin{alignat*}{1}
        \ddt \cvec{\xx(t) - \xx(0) - t \yy(0)}{\yy(t) - \yy(0)} = \cvec{\yy(t) - \yy(0)}{\MM(t) \xx(t)}
    \end{alignat*}

    \begin{alignat*}{1}
        \ddt \cvec{\xx(t) - \xx(0) - t \yy(0)}{\yy(t) - \yy(0)} = \cvec{\yy(t) - \yy(0)}{\MM(t) \xx(t)}
        =& \cvec{\yy(t) - \yy(0)}{\MM(t) \xx(t)}\\
        =& \cvec{\yy(t) - \yy(0)}{\MM(t) \lrp{\xx(t) - \xx(0) - t\yy(0)}} + \cvec{0}{\MM(t) \lrp{\xx(0) + t\yy(0)}}
    \end{alignat*}
    Thus
    \begin{alignat*}{1}
        \ddt \lrn{\xx(t) - \xx(0) - t \yy(0)}_2 
        \leq& \lrn{\yy(t) - \yy(0)}_2\\
        \ddt \lrn{\yy(t) - \yy(0)}_2 
        \leq& L_{\MM} \lrn{\xx(t) - \xx(0) - t \yy(0)}_2 + L_{\MM} \lrp{\lrn{\xx(0)}_2 + t \lrn{\yy(0)}_2}
    \end{alignat*}

    We verify that
    \begin{alignat*}{1}
        \lrn{\yy(t) - \yy(0)}_2 \leq L_{\MM} \int_0^t \lrn{\xx(s) - s\yy(0)}_2 ds + \frac{t^2}{2} L_{\MM}
    \end{alignat*}
    Let us apply Lemma \ref{l:sinh_ode_with_offset} with $a_t = \lrn{\xx_t - \xx(0) - t\yy(0)}_2$, $b_t = \lrn{\yy(t) - \yy(0)}_2$, $C = L_{\MM}$, $D = L_{\MM}\lrn{\yy(0)}_2$ and $E = L_{\MM}\lrn{\xx(0)}_2$
    \begin{alignat*}{1}
        \lrn{\xx_t - \xx(0) - t\yy(0)}_2 
        \leq& \frac{L_{\MM}\lrn{\xx(0)}_2}{L_{\MM}} \cosh(\sqrt{L_{\MM}} t) + \frac{L_{\MM}\lrn{\yy(0)}_2}{L_{\MM}^{3/2}} \sinh\lrp{\sqrt{L_{\MM}} t}\\
        &\qquad  - \frac{L_{\MM}\lrn{\yy(0)}_2}{L_{\MM}}t  - \frac{L_{\MM}\lrn{\xx(0)}_2}{L_{\MM}}\\
        =& \lrn{\xx(0)}_2 \lrp{\cosh(\sqrt{L_{\MM}} t) - 1}+ \lrn{\yy(0)}_2 \lrp{\frac{1}{\sqrt{L_{\MM}}} \sinh\lrp{\sqrt{L_{\MM}} t} - t}\\
        \lrn{\yy(t) - \yy(0)}_2 
        \leq& \frac{L_{\MM}\lrn{\xx(0)}_2}{\sqrt{L_{\MM}}} \sinh(\sqrt{L_{\MM}} t) + \frac{L_{\MM}\lrn{\yy(0)}_2}{L_{\MM}} \cosh\lrp{\sqrt{L_{\MM}} t} - \frac{L_{\MM}\lrn{\yy(0)}_2}{L_{\MM}}\\
        =& \lrn{\xx(0)}_2 \sqrt{L_{\MM}}  \sinh(\sqrt{L_{\MM}} t) + \lrn{\yy(0)}_2 \lrp{\cosh\lrp{\sqrt{L_{\MM}} t} - 1}
        \elb{e:t:ksnlqdk}
    \end{alignat*}

    Again from Lemma \ref{l:formal-matrix-exponent}, we know that
    \begin{alignat*}{1}
        \cvec{\xx(t)}{\yy(t)} = \bmat{\AA(t) & \BB(t) \\ \CC(t) & \DD(t)} \cvec{\xx(0)}{\yy(0)}
    \end{alignat*}
    thus
    \begin{alignat*}{1}
        \cvec{\xx(t) - \xx(0) - t\yy(0)}{\yy(t) - \yy(0)} = \bmat{\AA(t) - I & \BB(t) - t I \\ \CC(t) & \DD(t) - I} \cvec{\xx(0)}{\yy(0)}
    \end{alignat*}
    combined with \eqref{e:t:ksnlqdk}, and using the fact that the above hold for all $\yy(0)$ and $\xx(0)$, we can bound
    \begin{alignat*}{1}
        & \lrn{\AA(t) - I}_2 \leq \cosh(\sqrt{L_{\MM}} t) - 1\\
        & \lrn{\BB(t) - tI}_2 \leq \frac{1}{\sqrt{L_{\MM}}} \sinh\lrp{\sqrt{L_{\MM}} t} - t\\
        & \lrn{\DD(t) - I}_2 \leq \cosh\lrp{\sqrt{L_{\MM}} t} - 1
    \end{alignat*}

    Finally, to prove the third claim, 
    \begin{alignat*}{1}
        & \frac{d}{dt} \cvec{\xx(t) - t \yy(0) - \frac{t^2}{2} \MM(0) \xx(0)}{\yy(t) - \yy(0) - t \MM(0) \xx(0)} \\
        =& \bmat{0 & I \\ \MM(t) & 0} \cvec{\xx(t)}{\yy(t)} - \bmat{0 & I \\ \MM(0) & 0} \cvec{\yy(0)}{\vv}\\
    \end{alignat*}

\end{proof}

\begin{lemma}\label{l:matrix-exponent-block-bounds-second-order}
    Let $\emat$ be as defined in Lemma \ref{l:formal-matrix-exponent}. Let 
    \begin{alignat*}{1}
        \bmat{\AA(t) & \BB(t) \\ \CC(t) & \DD(t)} := \emat\lrp{t; \bmat{0 & I \\ \MM(t) & 0}}
    \end{alignat*}
    for some $\MM(t)$. Assume $\lrn{\MM(t)}_2 \leq L_{\MM}$ for all $t$. Then for all $t$,
    \begin{alignat*}{1}
            \lrn{\CC(t) - t \MM(0)}_2 \leq \frac{\lrp{L_\MM'+\frac{1}{2} L_\MM^2}}{\sqrt{L_\MM}} \sinh(\sqrt{L_\MM} t)
    \end{alignat*}
\end{lemma}
\begin{proof}
    The proof is similar to Lemma \ref{l:matrix-exponent-block-bounds}. Consider the ODE given by
    \begin{alignat*}{1}
        \ddt \cvec{\xx}{\yy}(t) = \bmat{0 & I \\ \MM(t) & 0} \cvec{\xx(t)}{\yy(t)}
    \end{alignat*}
    with initial condition $\yy(0) = 0$.

    By Lemma \ref{l:formal-matrix-exponent}, $\bmat{\AA(t) & \BB(t) \\ \CC(t) & \DD(t)}$ satisfies
    \begin{alignat*}{1}
        \cvec{\xx(t)}{\yy(t)} = \bmat{\AA(t) & \BB(t) \\ \CC(t) & \DD(t)} \cvec{\xx(0)}{\yy(0)}
    \end{alignat*}
    \begin{alignat*}{1}
        & \ddt \cvec{\xx(t) - \xx(0) - \frac{t^2}{2} \MM(0) \xx(0)}{\yy(t) - t\MM(0) \xx(0)}\\
        =& \cvec{y(t) - t\MM(0) \xx(0)}{\MM(t) \xx(t) - \MM(0) \xx(0)}\\
        =& \cvec{y(t) - t\MM(0) \xx(0)}{\MM(t) \xx(t) - \MM(0) \xx(0)}\\
        =& \cvec{\yy(t)- t \MM(0) \xx(0)}{\MM(t) \lrp{\xx(t) - \xx(0) - \frac{t^2}{2} \MM(0)}} + \cvec{0}{\lrp{\MM(t) - \MM(0)} \xx(0)} + \cvec{0}{\frac{t^2}{2}\MM(t) \MM(0) \xx(0)}
    \end{alignat*}

    By Cauchy Schwarz, for all $t\leq 1$,
    \begin{alignat*}{1}
        & \frac{d}{dt} \lrn{\xx(t) - \xx(0) - \frac{t^2}{2} \MM(0) \xx(0)}_2 \leq \lrn{\yy(t)- t \MM(0) \xx(0)}_2\\
        & \frac{d}{dt} \lrn{\yy(t) - t\MM(0) \xx(0)}_2 \leq L_\MM \lrn{\xx(t) - \xx(0) - \frac{t^2}{2} \MM(0) \xx(0)}_2 + \lrp{L_\MM'+\frac{1}{2} L_\MM^2}\lrn{\xx(0)}_2
    \end{alignat*}

    Apply Lemma \ref{l:sinh_ode_with_offset} with $a_t = \lrn{\xx(t) - \xx(0) - \frac{t^2}{2} \MM(0) \xx(0)}_2$, $b_t = \lrn{\yy(t) - t\MM(0) \xx(0)}_2$, $C = L_{\MM}$, $D = 0$ and $E = \lrp{L_\MM'+\frac{1}{2} L_\MM^2}\lrn{\xx(0)}_2$ to get
    \begin{alignat*}{1}
        & a_t \leq \frac{\lrp{L_\MM'+\frac{1}{2} L_\MM^2}\lrn{\xx(0)}_2}{L_\MM} \lrp{\cosh(\sqrt{L_\MM} t) - 1}\\
        & b_t \leq \frac{\lrp{L_\MM'+\frac{1}{2} L_\MM^2}\lrn{\xx(0)}_2}{\sqrt{L_\MM}} \sinh(\sqrt{L_\MM} t)
    \end{alignat*}

    Finally, recall that
    \begin{alignat*}{1}
        \yy(t) - t \MM(0) \xx(0) = \lrp{\CC(t) - t \MM(0)} \xx(0)
    \end{alignat*}
    Since we have shown that $\lrn{\yy(t) - t \MM(0) \xx(0)}_2 \leq \frac{\lrp{L_\MM'+\frac{1}{2} L_\MM^2}\lrn{\xx(0)}_2}{\sqrt{L_\MM}} \sinh(\sqrt{L_\MM} t)$ for all $\xx(0)$, it follows that
    \begin{alignat*}{1}
        \lrn{\CC(t) - t \MM(0)}_2 \leq \frac{\lrp{L_\MM'+\frac{1}{2} L_\MM^2}}{\sqrt{L_\MM}} \sinh(\sqrt{L_\MM} t)
    \end{alignat*}
\end{proof}

\section{Miscellaneous Lemmas}
\begin{lemma}
    \label{l:useful_xlogx}
    Let $c\in \Re^+$ be such that $c \geq 3$. For any $x$ satisfying $x\geq 3 c \log c$, we have that 
    \begin{alignat*}{1}
        \frac{x} {\log x} \geq c
    \end{alignat*}
\end{lemma}

The following Lemma is taken from \cite{sun2019escaping}:
    \begin{lemma}\label{l:triangle_distortion}
        For any $x\in M$, $a,y\in T_x M$
        \begin{alignat*}{1}
            \dist\lrp{\Exp_x(y+a), \Exp_{\Exp_x(a)} \lrp{\party{x}{\Exp_x(a)}y}}
            \leq& L_R \lrn{a}\lrn{y}\lrp{\lrn{a} + \lrn{y}} e^{\sqrt{L_R} \lrp{\lrn{a}+\lrn{y}}} 
        \end{alignat*}
    \end{lemma}
    \begin{proof}
        From the proof of Lemma 3 from \cite{sun2019escaping} (which is in turn a refinement of the proof from \cite{karcher1977riemannian})
        \begin{alignat*}{1}
            &\dist\lrp{\Exp_x(y+a), \Exp_{\Exp_x(a)} \lrp{\party{x}{\Exp_x(a)}y}} \\
            \leq& \int_0^1 \frac{\cosh(\sqrt{L_R}\lrn{y + (1-t) a})-\frac{\sinh(\sqrt{L_R}\lrn{y + (1-t) a})}{\sqrt{L_R}\lrn{y + (1-t) a}}}{\lrn{y + (1-t) a}} dt \cdot \lrn{a}\lrn{y}\\
            \leq& \sqrt{L_R} \int_0^1 \sqrt{L_R} \lrn{y + (1-t) a} e^{\sqrt{L_R} \lrn{y + (1-t) a}} dt \cdot \lrn{a}\lrn{y}\\
            \leq& L_R \lrn{a}\lrn{y}\lrp{\lrn{a} + \lrn{y}} e^{\sqrt{L_R} \lrp{\lrn{a}+\lrn{y}}} 
        \end{alignat*}
        where we use the fact from Lemma \ref{l:sinh_bounds} that for all $r \geq 0$,
        \begin{alignat*}{1}
            \frac{\cosh(r)}{r} - \frac{\sinh(r)}{r^2} \leq r e^r
        \end{alignat*}
    \end{proof}

    \begin{lemma}\label{l:sinh_bounds}
        For all $r \geq 0$,
        \begin{alignat*}{1}
            & \sinh(r) \leq r e^r\\
            & \cosh(r) - 1 \leq \frac{r^2}{2} e^r\\
            & \frac{\cosh(r)}{r} - \frac{\sinh(r)}{r^2} \leq r e^r\\
            & \cosh(r) \leq e^r\\
            & \frac{\sinh(r)}{r} - 1 \leq r^2 e^3
        \end{alignat*}
    
    \end{lemma}
    \begin{proof}
        Elementary computation from power series.
    \end{proof}

\end{document}